\documentclass[11pt]{article}

\usepackage{epsfig,epsfig}
\usepackage{epstopdf}
\usepackage{amsmath}
\usepackage{mathrsfs}
\usepackage{bm}
\usepackage{amssymb}
\usepackage{booktabs}
\usepackage{color}
\usepackage{multirow}
\usepackage{graphicx}
\usepackage{subfigure}
\usepackage{enumerate}
\usepackage{float}
\usepackage{appendix}
\usepackage{tikz}
\usetikzlibrary{calc}
\usetikzlibrary{matrix}
\usepackage{algorithm}

\usepackage{lineno}
\usepackage{srcltx,graphicx}

\newtheorem{thm}{Theorem}[section]
\newtheorem{prop}[thm]{Proposition}

\newtheorem{proof}[thm]{Proof}

\newtheorem{rem}[thm]{Remark}
\newtheorem{example}[thm]{Example}

\newcommand{\beq}{\begin{equation}}
\newcommand{\eeq}{\end{equation}}

\usepackage{hyperref}

\numberwithin{equation}{section} \topmargin=-2cm \oddsidemargin=1cm
\begin{document}
\title{\textbf{A high-order well-balanced positivity-preserving moving mesh DG method for the shallow water equations with non-flat bottom topography}}
\author{Min Zhang\footnote{School of Mathematical Sciences, Xiamen University,
Xiamen, Fujian 361005, China.
E-mail: minzhang2015@stu.xmu.edu.cn.
},
~Weizhang Huang\footnote{Department of Mathematics, University of Kansas, Lawrence, Kansas 66045, USA. E-mail: whuang@ku.edu.
},
~and Jianxian Qiu\footnote{School of Mathematical Sciences and Fujian Provincial Key Laboratory of Mathematical Modeling and High-Performance Scientific Computing, Xiamen University, Xiamen, Fujian 361005, China.
E-mail: jxqiu@xmu.edu.cn.
}
}

\date{}
\maketitle

\noindent\textbf{Abstract:}
A rezoning-type adaptive moving mesh discontinuous Galerkin method is proposed for the numerical solution of the shallow water equations with non-flat bottom topography.
The well-balance property is crucial to the simulation of perturbation waves over the lake-at-rest steady state such as waves on a lake or tsunami waves in the deep ocean.
To ensure the well-balance and positivity-preserving properties, strategies are discussed in the use of slope limiting, positivity-preservation limiting, and data transferring between meshes.
Particularly, it is suggested that a DG-interpolation scheme be used for the interpolation of both the flow variables and bottom topography from the old mesh to the new one and after each application of the positivity-preservation limiting on the water depth, a high-order correction be made to the approximation of the bottom topography according to the modifications in the water depth.
Moreover, mesh adaptation based on the equilibrium variable and water depth is shown to give more desirable results than that based on the commonly used entropy function.
Numerical examples in one and two spatial dimensions are presented to demonstrate the well-balance and positivity-preserving properties of the method and its ability to capture small perturbations of the lake-at-rest steady state.

\vspace{5pt}

\noindent\textbf{The 2020 Mathematics Subject Classification:} 65M50, 65M60, 76B15, 35Q35
\vspace{5pt}


\newcommand{\h}{\hspace{1.cm}}
\newcommand{\hh}{\hspace{2.cm}}
\newtheorem{yl}{\hspace{1.cm}Lemma}
\newtheorem{dl}{\hspace{1.cm}Theorem}
\renewcommand{\sec}{\section*}
\renewcommand{\l}{\langle}
\renewcommand{\r}{\rangle}
\newcommand{\be}{\begin{eqnarray}}
\newcommand{\ee}{\end{eqnarray}}

\normalsize \vskip 0.2in
\newpage

\section{Introduction}
\label{sec:Introduction}

We are interested in the numerical solution of the shallow water equations (SWEs) modeling
the water flow over a surface that plays an important role
in the ocean and hydraulic engineering, such as hydraulic jumps/shocks, open-channel flows,
bore wave propagation, tidal flows in the estuary and coastal zones.
The SWEs can be derived by integrating the Navier-Stokes equations in depth
under the hydrostatic assumption when the depth of the flow is small compared to its horizontal dimensions.
In non-dimensional form they read as
\begin{equation}
\label{swe-2d}
U_t +  \nabla \cdot \mathbf{F}(U)  = S(h,B),
\end{equation}
where
\begin{align*}
&U = (h, hu, hv )^T = (h, m, w)^T,
\\&
\mathbf{F}(U)  = \big(F(U),G(U)\big),
\\ &
F(U) = (hu, h u^2+\frac{1}{2}gh^2, huv)^T
= (m, \frac{m^2}{h}+\frac{1}{2}gh^2, \frac{mw}{h})^T,
\\&G(U) = (hv, huv,h v^2+\frac{1}{2}gh^2)^T
= (w, \frac{mw}{h}, \frac{w^2}{h}+\frac{1}{2}gh^2)^T,
\\&
S(h,B) = \big{(}0,-hgB_x,-hgB_y\big{)}^T.
\end{align*}
Here, $h$ is the depth of water, $(u,v)$ and $(m, w) = (hu, hv)$ are the velocities and discharges
(in $x$ and $y$ directions), respectively,
$g$ is the gravitational acceleration constant, and $B$ is the bottom topography which is a given time-independent function.
A distinct feature of the SWEs is that they admit steady-state solutions in which
the flux gradients are balanced by the source term exactly.
An important example is the so-called ``lake at rest" solution,
\begin{equation}
\label{Lake-2d}
u=0,\quad v=0, \quad  h+B= C,
\end{equation}
where $C$ is a constant.
It is crucial that this solution is preserved by numerical methods for the SWEs.
Indeed, many physical phenomena, such as waves on a lake or tsunami waves in the deep ocean,
can be described as small perturbations of this lake-at-rest steady state, and
they are difficult to capture by a numerical method on an unrefined mesh
unless the method can preserve the steady-state solution.
This property is known as the C-property or well-balance.

The ``exact C-property" concept was first introduced in 1994 by Bermudez and Vazquez \cite{Bermudez-Vazquez-1994}. Since then, a number of well-balanced numerical methods have been developed for the SWEs,
for example, (Godunov-type) finite volume methods \cite{Audusse-etal-2004Siam, Bermudez-Vazquez-1994, LeVeque-1998JCP, Zhou-etal-2001JCP},
finite difference/volume WENO methods \cite{Xing-Shu-2005JCP,Xing-Shu-2006JCP, Xing-Shu-2006CiCP},
and discontinuous Galerkin (DG) methods \cite{Eskilsson-Sherwin-2004,Ern-etal-2008,Li-etal-2018JCAM,
Tumolo-2013JCP,Xing-Shu-2006JCP, Xing-Shu-2006CiCP,Xing-Zhang-Shu-2010,Xing-Zhang-2013JSC}.
DG methods have the advantages of high-order accuracy, high parallel efficiency, flexibility for $hp$-adaptivity and arbitrary geometry and meshes, and these make them particularly suited for the SWEs.

It should be pointed out that the SWEs form a nonlinear hyperbolic system and its solution can develop
discontinuities such as hydraulic jumps/shock waves, rarefaction waves, and stationary state transitions.
Small mesh spacings are required in the regions of these structures in order to resolve them,
and mesh adaptation becomes a necessary tool to improve the computational accuracy and efficiency
in the numerical solution of the SWEs. In the past decades, various work has been done in this direction.
For example, Remacle et al. \cite{Remacle-etal-2006} studied an $h$-adaptive meshing procedure
for the transient computation of the SWEs.
Lamby et al. \cite{Lamby-etal-2005} proposed an adaptive
multi-scale finite volume method for the SWEs with source terms, combining a quadtree grid generation
strategy based on B-splines and fully adaptive multi-resolution methods.
Tang \cite{Tang-2004} developed an adaptive moving structured mesh kinetic flux-vector splitting (KFVS) scheme
for the SWEs without treating the bed slope source terms in order to balance the source terms and flux gradients.
Although the well-balance is not considered in the work, the numerical results demonstrate that
the adaptive moving mesh method leads to more accurate solutions than methods based on fixed meshes.
Donat et al. \cite{Donat-etal-2014JCP} presented a well-balanced shock capturing adaptive
mesh refinement (AMR) scheme for shallow water flows.
They showed that the use of well-balanced interpolation operators is essential
in order to maintain the well-balance property in the numerical solution computed with the AMR code.
Zhou et al. \cite{Zhou-etal-2013} proposed a well-balanced adaptive moving mesh generalized Riemann problem (GRP)-based finite volume scheme for the SWEs with irregular bottom topography.
A geometrical conservative interpolation scheme was used to update the solutions from the old mesh to the new one.
Recently, Arpaia and Ricchiuto \cite{Arpaia-Ricchiuto-2018} considered several arbitrary Lagrangian-Eulerian (ALE) formulations of the SWEs on moving meshes and provided a discrete analog in the well-balanced finite volume and residual distribution framework.

In this work we study a moving mesh DG (MM-DG) method
for the numerical solution of SWEs with non-flat bottom topography in one and two spatial dimensions.
Our main objective is to show that the MM-DG method maintains high-order accuracy
of DG discretization while preserving lake-at-rest steady state solutions and nonnegativity of the water depth.
The method is based on the rezoning approach of moving mesh methods
and contains three basic components at each time step,
the adaptive mesh movement, the interpolation of the solution from the old mesh
to the new one, and the numerical solution of the SWEs on the new mesh that is fixed for the time step.
The adaptive mesh movement is based on the moving-mesh-partial-differential-equation (MMPDE) method \cite{Huang-etal-1994JCP,Huang-etal-1994Siam,Huang-2001JCP,Huang-Sun-2003JCP,Huang-2006CiCP,Huang-Russell-2011,Huang-Kamenski-2015JCP}
which is known to produce meshes free of tangling \cite{Huang-Kamenski-2018MC}.
With the MMPDE method, the size, shape, and orientation of the mesh elements are controlled through
a metric tensor, a symmetric and uniformly positive definite matrix-valued function defined on the physical domain
and computed typically using the recovered Hessian of a DG solution. Instead of the commonly used entropy/total energy in the context of adaptive mesh shock wave simulation, we propose to use the equilibrium variable
$\mathcal{E}=\frac{1}{2}(u^2+v^2)+g(h+B) $ and the water depth $h$ to construct the metric tensor.
A motivation for this is that the mesh is adapted to both the perturbations of the lake-at-rest steady state
through $\mathcal{E}$ and the water depth distribution through $h$.

We use a fixed mesh well-balanced Runge-Kutta DG method
\cite{Xing-Shu-2006CiCP, Xing-Zhang-Shu-2010,Xing-Zhang-2013JSC}
for the numerical solution of the SWEs on the new mesh,
a DG-interpolation scheme \cite{Zhang-Huang-Qiu-2019arXiv}
for the interpolation of the solution and bottom topography from the old mesh to the new one,
the TVB slope limiter \cite{DG-series5} to avoid spurious oscillations,
and a linear scaling limiter \cite{Liu-Osher1996,ZhangShu2010,ZhangXiaShu2012}
for positivity preservation. Although these are existing schemes/techniques,
great caution is needed in their use to warrant the well-balance and positivity-preserving (PP)
properties of the overall MM-DG method.

The interpolation of the solution from the old mesh to the new one is a key component
for the MM-DG method to maintain high-order accuracy, preserve the lake-at-rest steady
state solution, and conserve the mass. Several conservative interpolation schemes
between deforming meshes have been investigated; see, e.g. \cite{Li-Tang-2006JSC,TangTang-2003Siam}.
We use the recently developed DG-interpolation scheme
\cite{Zhang-Huang-Qiu-2019arXiv} for the purpose.
This scheme is known to work for arbitrary bounded mesh deformation, have high-order accuracy,
conserve the mass, and preserve constant solution,
and, with a linear scaling PP limiter \cite{Liu-Osher1996,ZhangShu2010,ZhangXiaShu2012}, can preserve
the nonnegativity of the function to be interpolated.
We also use the same scheme to update the bottom topography on the new mesh.
This update is necessary due to the movement of the mesh and must be done
using a same scheme as for $U$ to ensure the well-balance property.

Note that we need to ensure the nonnegativity of the water depth in the interpolation for the dry situation.
We use the PP-DG-interpolation scheme (the DG-interpolation scheme with the linear scaling PP limiter
\cite{Zhang-Huang-Qiu-2019arXiv})
for the water depth for this purpose.
However, the PP limiter will destroy the well-balance property. In order to restore the property,
we apply the same DG-interpolation scheme to the total surface level and obtain a new approximation
for the bottom topography by subtracting the water depth approximation from the total surface level approximation.
This new approximation of the bottom topography differs from the old one by a high-order correction
that corresponds to the modifications in the water depth due to the PP limiting.

Since spurious oscillations and even nonlinear instability can occur in numerical solutions
for the SWEs, a nonlinear slope limiter is needed after each Runge-Kutta stage.
We apply the TVB limiter \cite{DG-series2,DG-series3,DG-series5}
to the local characteristic variables based on the variables $\big((h+B),m, w\big)$.
Note that $\big((h+B),m, w\big)$, rather the original variables $\big(h,m, w\big)$, is used
to ensure the well-balance property.
Moreover, we need to apply the PP limiter to ensure the nonnegativity of the water depth
every time after we use the TVB limiter.
Since the TVB procedure preserves the cell averages of the water depth,
we can use the linear scaling limiter \cite{Liu-Osher1996,ZhangShu2010,ZhangXiaShu2012}
for this purpose. Once again, the PP limiter destroys the well-balance property.
To recover the property, we propose to make a high-order correction to the approximation of the bottom topography
according to the modifications in the water depth due to the PP limiting.
Numerical examples show that this strategy works out well.

The remainder of the paper is organized as follows.
\S\ref{sec:WB-MMDG} is devoted to the description of the overall procedure
of the well-balanced MM-DG method and its DG and Runge-Kutta discretization for the SWEs.
The DG-interpolation scheme and its properties are described in \S\ref{sec:DG-interp}.
In \S\ref{sec:mmpde}, the MMPDE moving mesh method for the generation of the new mesh is discussed.
Numerical results obtained with the MM-DG method for a selection of one- and two-dimensional
examples are presented in \S\ref{sec:numerical-results}.
Finally, \S\ref{sec:conclusions} contains conclusions and further comments.

\section{The well-balanced MM-DG scheme for the SWEs}
\label{sec:WB-MMDG}

In this section we describe the well-balanced MM-DG method for the numerical solution
of the SWEs on moving simplicial meshes. We use here a rezoning approach where the unknown
variables are interpolated from the old mesh to the new one and the SWEs are solved
on the new mesh. In this section we focus on the overall description of the method in two dimensions
(the method is similar in one dimension), the discretization of the SWEs using DG in space and
a Runge-Kutta scheme in time, and the well-balance property.
A DG-interpolation scheme is described in \S\ref{sec:DG-interp}
and the generation of the new mesh using the MMPDE moving mesh method is discussed in \S\ref{sec:mmpde}.

We assume that the numerical solutions $U^n_h = (h^n_h,m^n_h,w^n_h)^T$ at physical time $t_n$
on the mesh $\mathcal{T}_h^{n}$ are known. (Notice that, in $h_h$, $h$ is the water depth while
the subscript $h$ is used to indicate that $h_h$ is a numerical approximation to the water depth.)
We also assume that a new mesh $\mathcal{T}_h^{n+1}$, which has the same number of vertices and
elements and the same connectivity as $\mathcal{T}_h^{n}$, has been generated
based on $U^n_h $ and $\mathcal{T}_h^{n}$ (cf. \S\ref{sec:mmpde}).
We denote an interplant of $U^n_h =  (h^n_h,m^n_h,w^n_h)^T$ on the new mesh $\mathcal{T}_h^{n+1}$
by $\tilde{U}^n_h=(\tilde{h}^n_h,\tilde{m}^n_h,\tilde{w}^n_h)^T$.
As we will see later in this section, caution is needed in choosing the interpolation scheme
to make the MM-DG method to be conservative, positivity-preserving, and well-balanced.

We now describe the DG discretization of the SWEs on $\mathcal{T}_h^{n+1}$. Let
\begin{equation}\label{MM-Vh}
\mathcal{V}^{k, n+1}_h= \{q\in L^2(\Omega):\; q|_{K}\in \mathbb{P}^{k}(K),
\; \forall K\in \mathcal{T}_h^{n+1}\},
\end{equation}
where
\begin{equation*}
\mathbb{P}^k(K)=\text{span}\{\phi_K^{1},\phi_K^{2},...,\phi_K^{n_b}\}
\end{equation*}
is the set of polynomials of degree at most $k$ $(k\ge 1)$ on element $K$ and
$\phi_{K}^j = \phi_{K}^j(\bm{x}),~ j=1,..., n_b\equiv (k+1)(k+2)/2$, denote the basis functions of $\mathbb{P}^k(K)$.
Multiplying \eqref{swe-2d} by a test function $\phi\in \mathcal{V}^{k, n+1}_h$,
integrating the resulting equation over $K\in \mathcal{T}_h^{n+1}$, and using the divergence theorem,  we get
\begin{equation}\label{div}
\frac{d}{d t}\int_{K}U \phi d\bm{x}
- \int_{K}\mathbf{F}(U)\cdot \nabla \phi d\bm{x}
+ \int_{\partial K} \mathbf{F}(U) \cdot\bm{n}_K \phi  ds
= \int_{K}S(h,B)\phi d\bm{x},
\end{equation}
where $\bm{n}_K=(n_x,n_y)^T$ is the outward unit normal to the boundary $\partial K$.

Recall that  each interior edge is shared by two triangular elements. Thus, on any (interior) edge
of $K$, $U_h$ can be defined using its value in $K$ or in the other element sharing the
common edge with $K$.  These values are denoted by $U_{h,K}^{int}$ and $U_{h,K}^{ext}$, respectively.
Moreover, we define the global Lax-Friedrichs numerical flux to approximate the flux function
$\mathbf{F}(U) \cdot \bm{n}_K $ on the edge $e_K\in \partial K$ as
\begin{align}
&\hat{\mathbf{F}}|_{e_K}=\hat{\mathbf{F}}(U_{h,K}^{int}, U_{h,K}^{ext},\bm{n}^e_K)=
\frac{1}{2}\Big{(}
\big{(}\mathbf{F}(U_{h,K}^{int})+\mathbf{F}(U_{h,K}^{ext}) \big{)} \cdot \bm{n}^e_K
-\alpha (U_{h,K}^{ext}-U_{h,K}^{int})\Big{)},
\label{lf-flux}
\end{align}
where
$\alpha = \max\limits_{K, e}\Big{(} \max\limits_s (|\lambda^{s}(U_{h,K}^{int})|, |\lambda^{s}(U_{h,K}^{ext})|)\Big{)}$.
Here,
\begin{align}
\label{eigen-123}
\lambda^{1}(U)=un_x +vn_y -c,
\quad
\lambda^{2}(U)=un_x +vn_y ,
\quad
\lambda^{3}(U)=un_x +vn_y +c,
\end{align}
are the eigenvalues of the Jacobian matrix
\begin{equation}
\mathbf{F}'(U)\cdot \bm{n}_K=\left(
  \begin{array}{ccc}
    0 & n_x & n_y \\
    (c^2-u^2)n_x-uvn_y & 2un_x+vn_y & un_y\\
    -uvn_x+(c^2-v^2)n_y & vn_x & un_x+2vn_y \\
  \end{array}
\right),
\end{equation}
where $c = \sqrt{gh}$ is the sound speed.
The semi-discrete DG scheme is then to find $U_h \in \mathcal{V}^{k, n+1}_h$ such that,
$\forall \phi \in \mathcal{V}^{k, n+1}_h$,
\begin{align}
\frac{d}{d t}\int_{K}U_h \phi d\bm{x}
+\sum_{e_K \in \partial K}\int_{e_K} \phi \hat{ \mathbf{F}}|_{e_K} ds
&-\int_{K} \mathbf{F}(U_h)\cdot \nabla \phi d\bm{x}
 = \int_{K}S(h_h,B_h)\phi d\bm{x} ,
 \label{semi-DG-2d}
\end{align}
where $\hat{ \mathbf{F}}|_{e_K}$ is the numerical flux defined (\ref{lf-flux}).
Define the residual as
\begin{equation}
\label{wb-residual-2d+2}
\begin{split}
R_{h,K} (U_h, \phi, B_h)&
= \int_{K}S(h_h,B_h)\phi d\bm{x}+\int_{K} \mathbf{F}(U_h)\cdot \nabla \phi d\bm{x} \\
& \qquad -\sum_{e_K \in \partial K}\int_{e_K} \phi \hat{ \mathbf{F}}|_{e_K} ds, \quad \forall  K\in \mathcal{T}^{n+1}_h .
\end{split}
\end{equation}
Generally speaking, the residual does not vanish for the lake-at-rest steady state, i.e., the scheme (\ref{semi-DG-2d})
is not well-balanced. Our objective is to find a special numerical flux $\hat{ \mathbf{F}}^{*}|_{e_K}$ by modifying
$\hat{ \mathbf{F}}|_{e_K}$ so that the residual becomes zero
for the lake-at-rest steady state.
We employ the hydrostatic reconstruction technique \cite{Audusse-etal-2004Siam,Xing-Shu-2006CiCP,Xing-Zhang-Shu-2010,Xing-Zhang-2013JSC} to modify $\hat{ \mathbf{F}}|_{e_K}$.
It is noted that the bottom topography function $B(\bm{x})$ needs to be projected into the finite element space $\mathcal{V}_h^{k,n+1}$. We denote this approximation by $B^{n+1}_h$ or $B_h$ without confusion.
We will discuss the projection/interpolation of $B(\bm{x})$ later in this section and
a DG-interpolation scheme in \S\ref{sec:DG-interp}.
Let
\begin{equation}\label{reconst-h-2d}
\begin{split}
&h_{h,K}^{*,int}|_{e_K} =
\max \Big( 0, h_{h,K}^{int}|_{e_K}+B_{h,K}^{int}|_{e_K}
             -\max\big(B_{h,K}^{int}|_{e_K},B_{h,K}^{ext}|_{e_K}\big)
             \Big),
\\&
h_{h,K}^{*,ext}|_{e_K} =
\max \Big( 0, h_{h,K}^{ext}|_{e_K}+B_{h,K}^{ext}|_{e_K}
             -\max\big(B_{h,K}^{int}|_{e_K},B_{h,K}^{ext}|_{e_K}\big)
             \Big).
\end{split}
\end{equation}
Recall that $B_h$ is the DG approximation of $B$ and is discontinuous on element edges.
In \eqref{reconst-h-2d}, the value of $B_h$ on $e_K$ is taken as
$\max\big(B_{h,K}^{int}|_{e_K},B_{h,K}^{ext}|_{e_K}\big)$, which was first proposed by Audusse et al \cite{Audusse-etal-2004Siam} and is now considered as the standard choice.
Some other researchers suggest to use $\min\big(B_{h,K}^{int}|_{e_K},B_{h,K}^{ext}|_{e_K}\big)$ \cite{Noelle-etal-2007-movingwater,Xing-2014JCP-movingwater}.
The optimal choice of the value of $B_h$ on $e_K$ remains an open problem, which seems to be closely related to the non-uniqueness of the Riemann problem \cite{Chinnayya-etal-2004}.
With \eqref{reconst-h-2d}, we force $h_{h,K}^{*,int}|_{e_K} \ge 0$ and $h_{h,K}^{*,ext}|_{e_K}\ge 0$
while trying to satisfy
\begin{equation*}
\begin{split}
&h_{h,K}^{*,int} + \max\big(B_{h,K}^{int}|_{e_K},B_{h,K}^{ext}|_{e_K}\big) = h_{h,K}^{int}|_{e_K}+B_{h,K}^{int}|_{e_K},
\\&
h_{h,K}^{*,ext} + \max\big(B_{h,K}^{int}|_{e_K},B_{h,K}^{ext}|_{e_K}\big)=
h_{h,K}^{ext}|_{e_K}+B_{h,K}^{ext}|_{e_K}.
\end{split}
\end{equation*}
Then, we redefine the interior and exterior values of $U_h$ as
\begin{equation}
\label{red-U-2d}
U_{h,K}^{*,int}|_{e_K}=\left(
  \begin{array}{c}
     h_{h,K}^{*,int}|_{e_K}    \\
    \frac{ h_{h,K}^{*,int}|_{e_K}}{ h_{h,K}^{int}|_{e_K}} m_{h,K}^{int}|_{e_K}\\
    \frac{ h_{h,K}^{*,int}|_{e_K}}{ h_{h,K}^{int}|_{e_K}} w_{h,K}^{int}|_{e_K}\\
  \end{array}
\right)
\quad \hbox{and}\quad
U_{h,K}^{*,ext}|_{e_K}=\left(
  \begin{array}{c}
     h_{h,K}^{*,ext}|_{e_K}\\
     \frac{ h_{h,K}^{*,ext}|_{e_K}}{ h_{h,K}^{ext}|_{e_K}}m_{h,K}^{ext}|_{e_K}\\
     \frac{ h_{h,K}^{*,ext}|_{e_K}}{ h_{h,K}^{ext}|_{e_K}}w_{h,K}^{ext}|_{e_K}\\
  \end{array}
\right).
\end{equation}
The numerical flux on the edge $e_K \in \partial K$ is modified as
\begin{align}
\label{eK-flux}
&\hat{ \mathbf{F}}^{*}|_{e_K} =
\hat{ \mathbf{F}}(U_{h,K}^{*,int}, U_{h,K}^{*,ext},\bm{n}^e_K)+
\Delta^{*}_{e_K}\cdot \bm{n}_K^e,
\end{align}
where
\begin{equation}
\label{correct-2d}
\Delta^{*}_{e_K}=\left(
  \begin{array}{cc}
 0&
0 \\
   \frac{g}{2} (h_{h,K}^{int}|_{e_K})^2-\frac{g}{2} (h_{h,K}^{*,int}|_{e_K})^2 &0\\
   0&\frac{g}{2} (h_{h,K}^{int}|_{e_K})^2-\frac{g}{2} (h_{h,K}^{*,int}|_{e_K})^2\\
  \end{array}
\right).
\end{equation}
The correction term is added to ensure
\begin{equation}
\label{wb-flux-2}
\hat{ \mathbf{F}}^{*}|_{e_K}
=\mathbf{F}\big( U_{h,K}^{int}\big)\cdot \bm{n}_K^e
\end{equation}
when the lake-at-rest steady state is reached.
Replacing $\hat{ \mathbf{F}}|_{e_K}$ by $\hat{ \mathbf{F}}^{*}|_{e_K}$ in \eqref{semi-DG-2d}, we have
\begin{align}
\frac{d}{d t}\int_{K}U_h \phi d\bm{x}
+\sum_{e_K \in \partial K}\int_{e_K} \phi \hat{ \mathbf{F}}^{*}|_{e_K} ds
&-\int_{K} \mathbf{F}(U_h)\cdot \nabla \phi d\bm{x}
\notag
\\ &
 = \int_{K}S(h_h,B_h)\phi d\bm{x}, \quad \forall \phi \in \mathcal{V}^{k, n+1}_h.
 \label{well-semi-DG-2d}
\end{align}
Denote the residual for this scheme as
\begin{equation}
\label{wb-residual-2d}
\begin{split}
R_{h,K}^{*} (U_h, \phi, B_h)&
= \int_{K}S(h_h,B_h)\phi d\bm{x}+\int_{K} \mathbf{F}(U_h)\cdot \nabla \phi d\bm{x} \\
& \qquad -\sum_{e_K \in \partial K}\int_{e_K} \phi \hat{ \mathbf{F}}^{*} |_{e_K} ds, \quad \forall  K\in \mathcal{T}^{n+1}_h .
\end{split}
\end{equation}
From (\ref{wb-flux-2}), it can be shown that this residual vanishes
for the lake-at-rest steady state if suitable Gaussian quadrature rules are used to calculate all integrals
in the \eqref{wb-residual-2d} exactly. As a result, the scheme \eqref{well-semi-DG-2d} is well-balanced.

To see the convergence order of the scheme, we rewrite \eqref{well-semi-DG-2d} into
\begin{align}
& \frac{d}{d t}\int_{K}U_h \phi d\bm{x}
 +\sum_{e_K \in \partial K}\int_{e_K} \phi \hat{ \mathbf{F}}|_{e_K} ds
-\int_{K} \mathbf{F}(U_h)\cdot \nabla \phi d\bm{x}
\notag
\\
& \qquad = \int_{K}S(h_h,B_h)\phi d\bm{x}
+\sum_{e_K \in \partial K}\int_{e_K} \phi \big(\hat{ \mathbf{F}}^{*}- \hat{ \mathbf{F}}\big)|_{e_K} ds, \quad \forall \phi \in \mathcal{V}^{k, n+1}_h.
 \label{well-semi-DG-2d-2}
\end{align}
This is a standard DG scheme for the SWEs with a correction term (the last term). Notice that
\[
\begin{split}
\big|(h_{h,K}^{int}|_{e_K})^2-(h_{h,K}^{*,int}|_{e_K})^2)\big|
&=\big|(h_{h,K}^{int}+h_{h,K}^{*,int})|_{e_K}\cdot(h_{h,K}^{int}-h_{h,K}^{*,int})|_{e_K}\big|
\\&
\leq\big|(h_{h,K}^{int}+h_{h,K}^{*,int})|_{e_K}\big|\cdot\big|(h_{h,K}^{int}-h_{h,K}^{*,int})|_{e_K}\big|
=\mathcal{O}(a^{k+1}_{max}),
\end{split}
\]
where $a_{max}$ denotes the maximum element diameter of the mesh. This gives
$\hat{ \mathbf{F}}^{*}- \hat{ \mathbf{F}}=\mathcal{O}(a^{k+1}_{max})$.
Thus, the scheme \eqref{well-semi-DG-2d} is $(k+1)$-th-order in space.

We can rewrite \eqref{well-semi-DG-2d} into
\begin{equation}\label{well-semi-ode}
\frac{d}{d t}\int_{K}U_h \phi d\bm{x}=R_{h,K}^{*}(U_h,\phi,B_h), \quad \forall \phi \in \mathcal{V}_h^{k, n+1}.
\end{equation}
A third-order explicit Runge-Kutta scheme is used to integrate \eqref{well-semi-ode} in time. We have,
for any $K \in \mathcal{T}_h^{n+1}$,
\begin{equation}\label{third}
\begin{cases}
\begin{split}
\int_{K} U_h^{(1)}\phi d\bm{x}
& = \int_{K} \tilde{U}_h^n\phi d\bm{x}
+\Delta t_n R_{h,K}^{*}(\tilde{U}_h^n,\phi,B_h^{n+1}),\quad \forall \phi \in \mathcal{V}_h^{k, n+1}
\\
\int_{K} U_h^{(2)}\phi d\bm{x}
& = \frac{3}{4}\int_{K} \tilde{U}_h^n\phi d\bm{x} + \frac{1}{4} \int_{K}U_h^{(1)}\phi d\bm{x}
\\ &
\qquad \qquad +\frac{\Delta t_n}{4}  R_{h,K}^{*}(U_h^{(1)},\phi ,B_h^{n+1}), \quad \forall \phi \in \mathcal{V}_h^{k, n+1}
\\
\int_{K}U_h^{n+1}\phi d\bm{x}
& = \frac{1}{3}\int_{K} \tilde{U}_h^n\phi d\bm{x} +\frac{2}{3}\int_{K} U_h^{(2)}\phi d\bm{x}
\\
&
\qquad \qquad +\frac{2  \Delta t_n}{3} R_{h,K}^{*}(U_h^{(2)},\phi ,B_h^{n+1}), \quad \forall \phi \in \mathcal{V}_h^{k, n+1}
\end{split}
\end{cases}
\end{equation}
where $\Delta t_n= t_{n+1}-t_{n}$ and $B_h^{n+1}$ is a polynomial approximation of the bottom topography
on $\mathcal{T}_h^{n+1}$.
To ensure stability, we choose $\Delta t_n$ according to the CFL condition \cite{DG-review} as
\begin{equation}
\label{MMDG-cfl}
\Delta t_n =\frac{C_{cfl}}{\max\limits_{K, e}\Big{(} \max\limits_s (|\lambda^{s}((\tilde{U}^n_{h,K})^{int})|, |\lambda^{s}((\tilde{U}^n_{h,K})^{ext})|)\Big{)}}\cdot \min\big(a^{n}_{min},a^{n+1}_{min}\big) ,
\end{equation}
where $C_{cfl}$ is a constant and $a^{n}_{min}$ and $a^{n+1}_{min}$ are the minimum element heights
of $\mathcal{T}^{n}_h$ and $\mathcal{T}^{n+1}_h$, respectively.

We are now ready to discuss the interpolation of $U_h^{n} = (h_h^{n}, m_h^{n}, w_h^{n})^T$
from the old mesh $\mathcal{T}_h^n$ to the new one $\mathcal{T}_h^{n+1}$.
$L^2$-projection is inconvenient to use here since it requires finding
the intersection between elements in $\mathcal{T}_h^n$ and $\mathcal{T}_h^{n+1}$
and performing numerical integration thereon, which is known to be a daunting, if not impossible,
task in programming.
Moreover, for the dry situation, we need to ensure the nonnegativity of the water depth in the interpolation procedure.
Thus, we use the PP-DG-interpolation scheme (the DG-interpolation scheme with the linear scaling PP limiter),
see \cite{Zhang-Huang-Qiu-2019arXiv} or \S\ref{sec:DG-interp}) for the water depth
and the DG-interpolation scheme (without PP limiter) for the water discharges, i.e.,
\begin{equation}
\label{DG-interp-1}
\tilde{h}_h^{n} = \hbox{PP-DGInterp}(h_h^{n}), \quad
\tilde{m}_h^{n} = \hbox{DGInterp}(m_h^{n}), \quad
\tilde{w}_h^{n} = \hbox{DGInterp}(w_h^{n}) .
\end{equation}
Note that this interpolation scheme does not maintain the well-balance property in general.
To recover the property, we apply the same DG-interpolation scheme to the total surface
level $h_h^{n} + B_h^{n}$ and define the new approximation to the bottom topography
on the new mesh as
\begin{equation}
\label{B-update-1}
B_h^{n+1} = \hbox{DGInterp}(h_h^{n} + B_h^{n}) -\tilde{h}_h^{n} .
\end{equation}
From the fact that the DG-interpolation scheme preserves constant solutions (cf. \S\ref{sec:DG-interp}),
we know that this procedure
restores the well-balance property. Moreover, from the linearity of the interpolation scheme,
we can rewrite (\ref{B-update-1}) into
\[
B_h^{n+1} = \hbox{DGInterp}(B_h^{n}) - \left (\; \hbox{PP-DGInterp}(h_h^{n})-\hbox{DGInterp}(h_h^{n})\;\right ),
\]
which implies that $B_h^{n+1}$ differs
from $\hbox{DGInterp}(B_h^{n})$ by a high-order correction corresponding to the changes
in the water depth due to the PP limiting.

It is interesting to point out that the well-balance property cannot be achieved
if $B_h^{n+1}$ is computed by $L^2$-projecting $B$ onto the DG approximation space defined on $\mathcal{T}_h^{n+1}$
(cf. numerical results in \S\ref{sec:numerical-results}).

We now show that the fully discrete scheme \eqref{third} is well-balanced. To this end, we assume
that $\tilde{U}_h^n$ satisfies \eqref{Lake-2d} (with $B$ replaced by $B^{n+1}_h$).
Recall that $R^{*}_{h,K}$ vanishes for the lake-at-rest steady state.
Then, scheme \eqref{third} reduces to
\begin{equation}\label{third-lake}
\begin{cases}
\begin{split}
\int_{K} U_h^{(1)}\phi d\bm{x}
& = \int_{K} \tilde{U}_h^n\phi d\bm{x}, \quad \forall \phi \in \mathcal{V}_h^{k, n+1}
\\
\int_{K} U_h^{(2)}\phi d\bm{x}
& = \frac{3}{4}\int_{K} \tilde{U}_h^n\phi d\bm{x} + \frac{1}{4} \int_{K}U_h^{(1)}\phi d\bm{x},
\quad \forall \phi \in \mathcal{V}_h^{k, n+1}
\\
\int_{K}U_h^{n+1}\phi d\bm{x}
& = \frac{1}{3}\int_{K} \tilde{U}_h^n\phi d\bm{x} +\frac{2}{3}\int_{K} U_h^{(2)}\phi d\bm{x},
\quad \forall \phi \in \mathcal{V}_h^{k, n+1}.
\end{split}
\end{cases}
\end{equation}
This implies that $U_h^{n+1} \equiv U_h^{(2)} \equiv U_h^{(1)} \equiv \tilde{U}_h^n$.
Thus, $U_h^{n+1}$ also satisfies \eqref{Lake-2d} (with $B$ replaced by $B^{n+1}_h$)
and \eqref{third} is well-balanced.

Since spurious oscillations and even nonlinear instability can occur in numerical solutions,
we need to apply
a nonlinear limiter after each Runge-Kutta stage.
However, caution must be taken since this limiting procedure can destroy the well-balance property
and the nonnegativity of the water depth.
Following \cite{Audusse-etal-2004Siam,Xing-Shu-2006CiCP,Xing-Zhang-Shu-2010,Xing-Zhang-2013JSC,Zhou-etal-2001JCP},
we use the TVB limiter \cite{DG-series5} for the local characteristic variables based on
the variables $\big((h_h+B_h),m_h, w_h\big)^T$ (instead of $\big(h_h,m_h, w_h\big)^T$) to obtain $\big((h_h+B_h)^{mod},m^{mod}_h, w^{mod}_h\big)^T$.
Define
$h^{mod}_h = (h_h+B_h)^{mod}-B_h$.
It is known that this limiting maintains the well-balance property
and preserves the cell averages of $h_h$ but does not necessarily preserve
the nonnegativity of $h_h$.
Thus, after each application of the TVB limiter, we apply the linear scaling PP limiter
of \cite{Liu-Osher1996,ZhangShu2010,ZhangXiaShu2012} to $h^{mod}_h$ and denote
the result by $\mbox{PP}(h^{mod}_h)$.
Once again, this PP limiter destroys the well-balance property.
To restore the property, we make a high-order correction to the approximation of $B$ according to
the changes in the water depth due to the PP limiting, i.e.,
\begin{equation}
\label{B-update-2}
\hat{B}_h = B_h  - (\mbox{PP}(h^{mod}_h) - h^{mod}_h).
\end{equation}
Notice that $\hat{B}_h$ has the same cell averages as $B_h$.

\begin{rem}
{\em In practical implementation when the water depth is close to zero,
large velocities $u = (hu)/h$ and $v=(hv)/h$) can result in due to
small numerical error in $hu$ and $hv$, which in turn can lead to
very small time steps with the CFL condition.
Following \cite{Ricchiuto-Bollermann-2009JCP,Xing-Zhang-Shu-2010},
we set $u=0$ and $v=0$ when $h<10^{-6}$ in our computation.
}\end{rem}

To conclude this section, we summarize the procedure of the MM-DG method in Algorithm~\ref{MM-DG}.
From the above discussion,
we have seen that the interpolation and the fully discrete scheme \eqref{third}
are well-balanced. Hence, the MM-DG method is well-balanced.

\begin{algorithm}[htbp]
\caption{The MM-DG method for the SWEs on moving meshes.}\label{MM-DG}
\begin{itemize}
\item[0.] {\bf Initialization.}
 Project the initial physical variables and bottom topography into the DG space $\mathcal{V}_h^{k,0}$ to obtain $U^0_h =  (h^0_h,m^0_h,w^0_h)^T$ and $B_h^0$. For $n = 0, 1, ...$, do

\item[1.] {\bf Mesh adaptation.}
Generate the new mesh $\mathcal{T}_h^{n+1}$ using an MMPDE-based moving mesh method (cf. \S\ref{sec:mmpde}).

\item[2.] {\bf Solution interpolation.}
Use the DG-interpolation procedure to update $U^n_h = (h^n_h,m^n_h,w^n_h)^T$ and $B^n_h$ from $\mathcal{T}^n_h$ to $\mathcal{T}_h^{n+1}$ to obtain $\tilde{U}^n_h = (\tilde{h}^n_h,\tilde{m}^n_h,\tilde{w}^n_h)^T$ and $B^{n+1}_h$ (cf. \eqref{DG-interp-1} and \eqref{B-update-1}).

\item[3.] {\bf Solution of the SWEs on the new mesh.}
Integrate the SWEs from $t_n$ to $t_{n+1}$ on the new mesh $\mathcal{T}_h^{n+1}$ using the MM-DG scheme \eqref{third} to obtain $U^{n+1}_h = (h_h^{n+1},m^{n+1}_h,w^{n+1}_h)^T$.
At each Runge-Kutta stage, the TVB limiter is applied to the local characteristic variables based on
the variables $\big((h_h+B_h),m_h, w_h\big)^T$, followed by linear scaling PP limiter for the water depth
and a corresponding high-order correction to the approximation of the bottom topography (cf. (\ref{B-update-2}).

\end{itemize}
\end{algorithm}

\section{A conservative DG-interpolation scheme}
\label{sec:DG-interp}

In this section we briefly describe a DG-interpolation scheme \cite{Zhang-Huang-Qiu-2019arXiv}
for the interpolation of a numerical solution
$q_h^n$ from the old mesh $\mathcal{T}^n_h$ to the new one $\mathcal{T}_h^{n+1}$. The meshes are assumed
to have the same number of vertices and elements and the same connectivity and can be viewed
as a deformation from each other. The scheme works for arbitrary bounded mesh deformation.
It has high-order accuracy, conserves the mass, positivity-preserving, and satisfies the geometric conservation law (GCL) and thus preserves constant solutions, and, with a linear scaling PP limiter \cite{Liu-Osher1996,ZhangShu2010,ZhangXiaShu2012}, can preserve
the nonnegativity of the function to be interpolated.
The reader is referred to \cite{Zhang-Huang-Qiu-2019arXiv} for the detail.

The interpolation problem between $\mathcal{T}_h^{n}$ and $\mathcal{T}_h^{n+1}$
is mathematically equivalent to solving the linear convection equation
\begin{equation}
\label{pde}
\frac{\partial q}{\partial \varsigma} (\bm{x},\varsigma) =0,\quad (\bm{x},\varsigma)\in
\mathcal{D} \times(0,1]
\end{equation}
on the moving mesh $\mathcal{T}_h(\varsigma)$ that is defined as the linear interpolant of $\mathcal{T}_h^{n}$
and $\mathcal{T}_h^{n+1}$.
Specifically, $\mathcal{T}_h(\varsigma)$ has the same number of elements and vertices
and the same connectivity as $\mathcal{T}_h^{n}$ and $\mathcal{T}_h^{n+1}$ and its nodal positions and
displacements (or deformation) are given by
\begin{align}
\label{location+speed}
&\bm{x}_i(\varsigma) = (1-\varsigma)\bm{x}_i^{n}+\varsigma \bm{x}_i^{n+1}, \quad i =
1,...,N_v
\\
&\dot{\bm{x}}_i = \bm{x}_i^{n+1}-\bm{x}_i^{n}, \quad i = 1,...,N_v .
\label{location+speed+1}
\end{align}
We define the piecewise linear mesh deformation function as
\begin{equation}\label{Xdot-1}
\dot{\bm{X}}(\bm{x},\varsigma) = \sum_{i=1}^{N_v} \dot{\bm{x}}_i \phi_i(\bm{x},\varsigma),
\end{equation}
where $ \phi_i$ is the linear basis function associated with the vertex $\bm{x}_i$.

We use a quasi-Lagrangian MM-DG method \cite{Luo-Huang-Qiu-2019JCP,Zhang-Cheng-Huang-Qiu-2020CiCP} to solve \eqref{pde}. The semi-discrete scheme for \eqref{pde} is to find $q_h \in \mathcal{V}^{k}_h(\varsigma)$ such that
\begin{equation}
\label{semi-DG}
\frac{d}{d \varsigma}\int_{K} q_h \phi d\bm{x}
+\sum_{e \in \partial K} \int_{e} \phi F_{e}(q_{K}^{int}, q_{K}^{ext}) ds
+\int_{K} (q_h \dot{\bm{X}}) \cdot \nabla  \phi d\bm{x}=0, \quad
\forall  \phi\in \mathcal{V}^{k}_h(\varsigma)
\end{equation}
where $F_{e}(q_{K}^{int}, q_{K}^{ext}) \approx -q \dot{\bm{X}} \cdot \bm{n}_K$ is
a numerical flux defined on $e \in \partial K$, and $q_{K}^{int}$ and $q_{K}^{ext}$ denote the values of $q_h$ on $K$
and on the element (denoted by $K'$) sharing the common edge $e$ with $K$.
We use the local Lax-Friedrichs numerical flux, i.e.,
\begin{align}
F_{e}(q_{K}^{int}, q_{K}^{ext}) =
\frac{1}{2}\Big{(}
\big{(}-q_K^{int}\dot{\bm{X}}^{e}-q_K^{ext} \dot{\bm{X}}^{e} \big{)} \cdot \bm{n}^e_K
-\alpha_{e} (q_{K}^{ext}-q_{K}^{int})\Big{)},\quad \forall e \in \partial K
\label{llf-flux}
\end{align}
where $\dot{\bm{X}}^{e}$ denotes the restriction of $\dot{\bm{X}}$ on $e$ and $\alpha_{e} = \max\big{(}
|\dot{\bm{X}}^{e}\cdot \bm{n}^e_K |,
|\dot{\bm{X}}^{e}\cdot \bm{n}^e_{K'} |\big{)} $.
Note that the numerical flux is zero on the boundary due to $\dot{\bm{X}}^{e}\cdot \bm{n}^e_K =0$ if $e \in \partial \mathcal{D}$.

We consider the quasi-time instants
\[
0=\varsigma^0<\varsigma^1<...<\varsigma^{\nu}<\varsigma^{\nu+1}<... <\varsigma^{N_\varsigma}=1,
\quad \quad \Delta \varsigma^{\nu} = \varsigma^{\nu+1}- \varsigma^{\nu}.
\]
To ensure the  stability of the scheme, we choose $\Delta \varsigma$ according to the CFL condition
\begin{equation}
\label{DG-interp-cfl}
\Delta \varsigma = \frac{\mathcal{C}_{p}}{\max\limits_{e,K}|\dot{\bm{X}}^{e}\cdot \bm{n}^e_K |}
\cdot\min\big(a^{n}_{min},a^{n+1}_{min}\big) .
\end{equation}
where $\mathcal{C}_{p}$ is the CFL condition constant number and usually taken as $1/(2k+2)$  in our computation, unless stated otherwise.
A third-order Runge-Kutta scheme is used for time integration. Then,
the fully-discrete quasi-Lagrangian MM-DG scheme for \eqref{pde} is:
\begin{equation}
\label{third-interp}
\begin{cases}
\begin{split}
\int_{K^{\nu,(1)}}q_h^{\nu,(1)}\phi^{\nu,(1)} d\bm{x}
&= \int_{K^{\nu}}q_h^{\nu}\phi^{\nu} d\bm{x}+\Delta \varsigma^\nu \mathcal{A}(q_h^{\nu},\phi^{\nu})|_{K^\nu},
\\
\int_{K^{\nu,(2)}}q_h^{\nu,(2)}\phi^{\nu,(2)} d\bm{x}
&= \frac{3}{4}\int_{K^{\nu}}q_h^{\nu}\phi^{\nu}d\bm{x}  +\frac{1}{4} \int_{K^{\nu,(1)}}q_h^{\nu,(1)}\phi^{\nu,(1)} d\bm{x}
\\
& \qquad
+\frac{\Delta \varsigma^\nu}{4} \mathcal{A}(q_h^{\nu,(1)},\phi^{\nu,(1)})|_{K^{\nu,(1)}},
\\
\int_{K^{\nu+1}}q_h^{\nu+1}\phi^{\nu+1} d\bm{x}
&= \frac{1}{3}\int_{K^{\nu}}q_h^{\nu}\phi^{\nu} d\bm{x}
+ \frac{2}{3} \int_{K^{\nu,(2)}}q_h^{\nu,(2)} \phi^{\nu,(2)} d\bm{x}
\\
& \qquad
+\frac{2 \Delta \varsigma^\nu}{3} \mathcal{A}(q_h^{\nu,(2)},\phi^{\nu,(2)})|_{K^{\nu,(2)}} ,
\end{split}
\end{cases}
\end{equation}
where ($q_h^{\nu,(1)}$, $\phi^{\nu,(1)}$, $K^{\nu,(1)}$) are the stage values
at $\varsigma = \varsigma^{\nu + 1}$,
($q_h^{\nu,(2)}$, $\upsilon^{\nu,(2)}$, $K^{\nu,(2)}$) are the values
at $\varsigma = \varsigma^{\nu+\frac{1}{2}}$,
($q_h^{\nu+1}$, $\phi^{\nu+1}$, $K^{\nu+1}$) are at $\varsigma = \varsigma^{\nu + 1}$, and
\[
\mathcal{A}(q_h,\phi)|_{K} = -\sum_{e \in \partial K} \int_{e} \phi F_{e}(q_{K}^{int}, q_{K}^{ext}) ds
-\int_{K} (q_h \dot{\bm{X}}) \cdot \nabla  \phi d\bm{x} .
\]
It is emphasized that the volume of $K$ needs to be updated at these stages as
\begin{equation}
\label{Karea-1}
\begin{split}
|K^{\nu,(1)}|
&= |K^{\nu}|+\Delta \varsigma^\nu |K^{\nu}|\nabla\cdot\dot{\bm{X}}|_{K^{\nu}},
\\
|K^{\nu,(2)}|
&= \frac{3}{4}|K^{\nu}|
+\frac{1}{4}\big (|K^{\nu,(1)}|
+\Delta \varsigma^\nu |K^{\nu,(1)}|\nabla\cdot\dot{\bm{X}}|_{K^{\nu,(1)}}\big ),
\\
|K^{\nu+1}|
&= \frac{1}{3}|K^{\nu}|
+\frac{2}{3}\big ( |K^{\nu,(2)}|
+\Delta \varsigma^\nu |K^{\nu,(2)}|\nabla\cdot\dot{\bm{X}}|_{K^{\nu,(2)}}\big ).
\end{split}
\end{equation}
This ensures that the geometric conservation law be satisfied and constant solutions
be preserved.

\begin{prop}[\cite{Zhang-Huang-Qiu-2019arXiv}]
\label{fully-mass}
The DG-interpolation scheme \eqref{third-interp} conserves the mass, i.e.,
\begin{align}
\label{f3-4}
\sum_{K^{\nu+1}}\int_{K^{\nu+1}} q_h^{\nu+1}  d\bm{x}
=\sum_{K^{\nu}}\int_{K^{\nu}} q_h^{\nu}  d\bm{x}, \quad \nu = 0, 1, ...
\end{align}
\end{prop}


\begin{prop}
\label{WB-interp}
The DG-interpolation scheme \eqref{third-interp} preserves constant solutions,
i.e., for any arbitrary constant $C$, $q_h^{\nu} \equiv C$ implies $q_h^{\nu+1} \equiv C$ if the element volume is
updated according to \eqref{Karea-1}.
\end{prop}

\begin{proof} This proposition follows from the same proof for the geometric conservation law
(Lamma 2.1) in \cite{Zhang-Huang-Qiu-2019arXiv}.
\hfill
\end{proof}

\begin{prop}
[\cite{Zhang-Huang-Qiu-2019arXiv}]
\label{PP-interp}
When the linear scaling PP limiter \cite{Liu-Osher1996,ZhangShu2010,ZhangXiaShu2012} is further applied
at each step,
the DG-interpolation scheme \eqref{third-interp} preserves the solution positivity
in the sense that, for $\nu = 0, 1, ...$, if $q_h^{\nu}$ has nonnegative cell averages for all elements
and is nonnegative at a set of special points at each element of $\mathcal{T}_h(\varsigma^\nu)$,
then $q_h^{\nu+1}$ has nonnegative cell averages for all elements
and is nonnegative at the corresponding set of special points at each element of $\mathcal{T}_h(\varsigma^{\nu+1})$. The DG-interpolation with the linear scaling PP limiter we call as the PP-DG-interpolation.
\end{prop}

\begin{rem}
\label{rem-DG-interp}{\em
It is reasonable to assume that the mesh movement speed is finite during the course of computation.
Under this assumption, the deformation between $\mathcal{T}_h^{n}$ and $\mathcal{T}_h^{n+1}$ is proportional
to $\Delta t^n$, i.e.,
\[
\dot{\bm{x}}_i = \bm{x}_i^{n+1} - \bm{x}_i^{n} = \mathcal{O} (\Delta t^n) .
\]
From (\ref{Xdot-1}), we have
\[
\max\limits_{e,K}|\dot{\bm{X}}^{e}\cdot \bm{n}^e_K | = \mathcal{O} (\Delta t^n) .
\]
Combining this with (\ref{MMDG-cfl}) and (\ref{DG-interp-cfl}), we get
\[
\Delta \varsigma =\mathcal{O} \left ( \max\limits_{K, e}\Big{(} \max\limits_s (|\lambda^{s}((\tilde{U}^n_{h,K})^{int})|, |\lambda^{s}((\tilde{U}^n_{h,K})^{ext})|)\Big{)}\right ) ,
\]
which is close to being constant since the characteristic speeds (\ref{eigen-123}) are close to $c$.
As a result, only a constant number of steps are needed in integrating (\ref{pde}). Since each step of (\ref{third-interp})
requires $\mathcal{O}(N_v)$ operations, the total cost of the DG-interpolation is $\mathcal{O}(N_v)$.
The detailed discussion on the cost of DG-interpolation can be seen in \cite{Zhang-Huang-Qiu-2019arXiv}.
}
\end{rem}

\section{MMPDE moving mesh method}
\label{sec:mmpde}

In this section we describe the generation of the new mesh $\mathcal{T}^{n+1}_{h}$
from the old one $\mathcal{T}^{n}_{h}$ using the MMPDE moving mesh method \cite{Huang-etal-1994JCP,Huang-etal-1994Siam,Huang-Russell-2011}.
We use the $\bm{\xi}$-formulation of the method and its new implementation
proposed in \cite{Huang-Kamenski-2015JCP}.

For mesh generation purpose, we introduce a computational mesh
$\mathcal{T}_c= \{\bm{\xi}_1, \;...,\;\bm{\xi}_{N_v} \}$ which can be viewed as a deformation of the mesh $\mathcal{T}^{n}_{h}$ and serves as an intermediate variable.
We also assume that a reference computational mesh
$\hat{\mathcal{T}}_c = \{\hat{ \bm{\xi}}_1,\; ...,\; \hat{ \bm{\xi}}_{N_v} \}$ has been given.
This mesh is kept fixed in the computation and should be chosen as uniform as possible.
Often it can be chosen as the initial physical mesh.

A key idea of the MMPDE method is to view any nonuniform mesh as a uniform one
in some Riemannian metric \cite{Huang-2006CiCP, Huang-Russell-2011}.
The metric tensor $\mathbb{M}$,  a symmetric and uniformly positive definite
matrix-valued function defined on $\mathcal{D}$, provides the information
needed for determining the size, shape, and orientation of the mesh elements throughout the domain.
Various metric tensors have been proposed; e.g., see \cite{Huang-Sun-2003JCP, Huang-Russell-2011}.
We use here an optimal metric tensor based on the $L^2$-norm of piece linear interpolation error.
To be specific, we consider a physical variable $q$ and its finite element approximation $q_h$.
Let $H_K$ be a recovered Hessian of $q_h$ on $K\in \mathcal{T}_{h}$
obtained using the least squares fitting Hessian recovery technique \cite{Zhang-Naga-2005Siam}.
Assuming that the eigen-decomposition of $H_K$ is given by
\[
H_K = Q\hbox{diag}(\lambda_1,\cdots,\lambda_d)Q^T,
\]
where $Q$ is an orthogonal matrix, we define
\[
|H_K| = Q\hbox{diag}(|\lambda_1|,...,|\lambda_d|)Q^T.
\]
The metric tensor is defined as
\begin{equation}\label{mer}
\mathbb{M}_{K} =\det \big{(}\alpha_{h}\mathbb{I}+|H_K|\big{)}^{-\frac{1}{d+4}}
\big{(}\alpha_{h}\mathbb{I}+|H_K|\big{)},
\quad \forall K \in \mathcal{T}_h
\end{equation}
where $\mathbb{I}$ is the identity matrix, $\det(\cdot)$ is the determinant of a matrix,
$d$ is the dimension of the spatial domain,
and $\alpha_{h}$ is a regularization parameter defined through the algebraic equation
\[
\sum_{K\in\mathcal{T}_h}|K|\, \hbox{det}(\alpha_{h}\mathbb{I}+|H_K|)^{\frac{2}{d+4}}
=2\sum_{K\in\mathcal{T}_h}|K|\,
\hbox{det}(|H_K|)^{\frac{2}{d+4}}.
\]

It is interesting to note that the entropy has been commonly used for adaptive mesh simulation
of shock waves \cite{Li-Tang-2006JSC,Luo-Huang-Qiu-2019JCP}. For the SWEs, the total energy
$E =\frac{1}{2}(hu^2+hv^2)+\frac{1}{2}gh^2+ghB$ plays the role of the entropy function. However, we have found through numerical experiment
that the metric tensor based on the equilibrium variable $\mathcal{E}=\frac{1}{2}(u^2+v^2)+g(h+B) $ and
the water depth $h$ leads to more desirable mesh concentration than that based on the total energy/entropy alone
(cf. Figs.~\ref{Fig:test3-1d-large-metric}, \ref{Fig:test3-1d-small-metric},
and \ref{Fig:test4-2d-comparison-12}~--~\ref{Fig:test4-2d-comparison-48}).
A motivation for this is that the mesh is adapted to both the perturbations of the lake-at-rest steady state
through $\mathcal{E}$ and the water depth distribution through $h$.
Specifically, we first compute $\mathbb{M}^{\mathcal{E}}_{K} $ and $\mathbb{M}^{h}_{K} $ using
(\ref{mer}) with $q = \mathcal{E}$ and $h$, respectively.
Then, a new metric tensor is obtained through matrix intersection as
\begin{equation}\label{mer-Eh}
\tilde{\mathbb{M}}_{K}
=\frac{{\mathbb{M}}^{\mathcal{E}}_{K} }{|||{\mathbb{M}}^{\mathcal{E}}_{K}|||}
\cap\frac{\delta\cdot{\mathbb{M}}^{h}_{K}}{|||{\mathbb{M}}^{h}_{K}|||} ,
\end{equation}
where $|||\cdot |||$ denotes the maximum absolute value
of the entries of a matrix and ``$\cap$" stands for matrix intersection.
The reader is referred to
\cite{Zhang-Cheng-Huang-Qiu-2020CiCP} for the definition and geometric interpretation of matrix intersection.
We take $\delta=0.1$ in our numerical simulations.

The next step is to make sure that the metric tensor is bounded since it is known \cite{Huang-Kamenski-2018MC}
that the moving mesh generated by the MMPDE method stays nonsingular (free from tangling)
if the metric tensor is bounded and the initial mesh is nonsingular.
We define
\begin{equation}\label{mer-ceil}
  \hat{\mathbb{M}}_K = \frac{\tilde{\mathbb{M}}_K}{\sqrt{1+\Big(\frac{\hbox{tr}(\tilde{\mathbb{M}}_K)}{\beta}\Big)^2}},
  \end{equation}
where $\beta$ is a positive number and $\hbox{tr}(\cdot)$ denotes the trace of a matrix. It is not difficult
to show $\lambda_{\max}(\hat{\mathbb{M}}_K) \leq \beta$.
In our computation, we take $\beta=1000$.

The metric tensor is generally non-smooth since the received Hessian can contain the discontinuous.
A common practice in the moving mesh context is to smooth the metric tensor for smoother meshes.
Here we simply average the metric tensor at a vertex over its neighboring vertices, i.e.,
\begin{equation}
\mathbb{M}_i \longleftarrow \frac{1}{|\mathcal{N}_i|}\sum_{j\in\mathcal{N}_i}\hat{\mathbb{M}}_j, \quad \forall
\bm{x}_i\in \mathcal{T}_h
\end{equation}
where $\hat{\mathbb{M}}_i$ is the nodal value at vertex $\bm{x}_i$ obtained by area-averaging the values
of the metric tensor on the neighboring elements,
the $\mathcal{N}_i$ denotes the set of the immediate neighboring vertices (including itself) of $\bm{x}_i$,
and $|\mathcal{N}_i|$ is the length of the set $\mathcal{N}_i$.
This process can be repeated several times every time in the computation.

Having defined the metric tensor, we are ready to describe the MMPDE method.
Recalling that $\mathcal{T}_c$ and $\mathcal{T}_h$ are a deformation of each other,
for any element $K\in \mathcal{T}_h$, there exits an element $K_c\in \mathcal{T}_c$ corresponding to $K$.
Denote the affine mapping from $K_c$ to $K$ as $F_K$ and its Jacobian matrix as $F'_K$.
It is known \cite{Huang-2006CiCP, Huang-Russell-2011}
that any mesh $\mathcal{T}_h$ which is uniform
in the metric $\mathbb{M}$ in reference to the computational mesh $\mathcal{T}_c$, satisfies
\begin{align}
\label{ec}
& |K|\sqrt{\det(\mathbb{M}_K)}=\frac{\sigma_h|K_c|}{|\mathcal{D}_c|},&\quad \forall K\in \mathcal{T}_h
\\
\label{al}
& \frac{1}{d}\hbox{tr}\big{(} (F'_K)^{-1}\mathbb{M}_K^{-1}(F'_K)^{-T}\big{)}
= \hbox{det}\big{(}
(F'_K)^{-1}\mathbb{M}_K^{-1}(F'_K)^{-T}\big{)}^{\frac{1}{d}},&\quad \forall K \in \mathcal{T}_h
\end{align}
where $\mathbb{M}_K$ is the average of $\mathbb{M}$ over $K$ and
\[
|\mathcal{D}_c|=\sum\limits_{K_c\in\mathcal{T}_c}|K_c|,\quad \sigma_h
=\sum\limits_{K\in\mathcal{T}_h}|K|\hbox{det}(\mathbb{M}_K)^{\frac{1}{2}} .
\]
The equidistribution condition \eqref{ec} determines the size of elements through
the metric tensor $\mathbb{M}$ and implies that all the elements have the same size in the metric tensor $\mathbb{M}$.
The alignment condition \eqref{al} determines the shape and orientation of $K$ through $\mathbb{M}$ and shape of $K_c$ and implies that any element $K$ when measured in the metric $\mathbb{M}_K$ is similar to $K_c$ measured in the Euclidean metric.

A mesh energy function associated with the equidistribution and alignment conditions is given by
\begin{equation}
\label{energy}
\begin{split}
I_h(\mathcal{T}_h;\mathcal{T}_c)
=&\frac{1}{3}\sum_{K\in\mathcal{T}_h}|K|\hbox{det}(\mathbb{M}_K)^{\frac{1}{2}}\big{(}
\hbox{tr}((F'_K)^{-1}\mathbb{M}^{-1}_K(F'_K)^{-T})\big{)}^{\frac{3 d}{4}}
\\&+\frac{1}{3} d^{\frac{3 d}{4}}\sum_{K\in\mathcal{T}_h}|K|\hbox{det}(\mathbb{M}_K)^{\frac{1}{2}}
\left (\hbox{det}(F'_K) \hbox{det}(\mathbb{M}_K)^{\frac{1}{2}} \right )^{-\frac{3}{2}},
\end{split}
\end{equation}
which is a Riemann sum of a continuous functional developed in \cite{Huang-2001JCP}.

For a given mesh $\mathcal{T}_h^{n}$, we want to find a new mesh $\mathcal{T}_h^{n+1}$ by minmimizing the mesh energy function $I_h$.
Note that $I_h(\mathcal{T}_h,\mathcal{T}_c)$ is a function of
the vertices $\bm{\xi}_i,\; i=1,...,N_v$ of $\mathcal{T}_c$
and the vertices $\bm{x}_i,\;i=1,...,N_v$ of $\mathcal{T}_h$.
We take $\mathcal{T}_h$ as $\mathcal{T}_h^{n}$,
minimize $I_h(\mathcal{T}_h^{n},\mathcal{T}_c)$ with respect to $\mathcal{T}_c$
(and denote the final mesh as $\mathcal{T}_c^{n+1}$),
and obtain $\mathcal{T}_h^{n+1}$ through the relation between
$\mathcal{T}_h^{n}$ and $\mathcal{T}_c^{n+1}$.
The minimization of $I_h(\mathcal{T}_h^{n},\mathcal{T}_c)$ is carried out by solving
the mesh equation defined as the gradient system of the energy function (the MMPDE approach), i.e.,
\begin{equation}\label{MM}
\begin{split}
\frac{d \bm{\xi}_i }{dt}
=-\frac{\hbox{det}(\mathbb{M}(\bm{x_i}))^{\frac{1}{2}} }{\tau}
\Big{(}\frac{\partial I_h(\mathcal{T}^{n}_h,\mathcal{T}_c)}{\partial \pmb{\xi}_i}\Big{)}^T,
\quad i=1,...,N_v
\end{split}
\end{equation}
where ${\partial I_h }/{\partial \bm{\xi}_i}$ is considered as a row vector and
$\tau>0$ is a parameter used to adjust the response time of mesh movement to
the changes in $\mathbb{M}$.

Define the function $G$ associated with the energy \eqref{energy} as
\begin{equation}\label{G}
G(\mathbb{J},\hbox{det}(\mathbb{J}))=
\frac{1}{3}\hbox{det}(\mathbb{M}_{K})^{\frac{1}{2}}
(\hbox{tr}(\mathbb{J}\mathbb{M}_{K}^{-1}\mathbb{J}^T))^{\frac{3 d}{4}}
+\frac{1}{3} d^{\frac{3 d}{4}}\hbox{det}(\mathbb{M}_{K})^{\frac{1}{2}}
\left (\frac{\hbox{det}(\mathbb{J})}{\hbox{det}(\mathbb{M}_{K})^{\frac{1}{2}}} \right )^{\frac{3}{2}},
\end{equation}
where $\mathbb{J}=(F'_K)^{-1} = E_{K_c}E_K^{-1}$,
and $E_K=[\bm{x}_1^K-\bm{x}_0^K,\;...,\; \bm{x}_d^K-\bm{x}_0^K]$ and
$E_{K_c}=[\bm{\xi}_1^K - \bm{\xi}_0^K,\;...,\; \bm{\xi}_d^K -\bm{\xi}_0^K ]$ are the edge matrices of $K$ and $K_c$, respectively.
Using the notion of scalar-by-matrix differentiation, the derivatives of $G$ with respect
to $\mathbb{J}$ and $\hbox{det}(\mathbb{J})$ can be obtained \cite{Huang-Kamenski-2015JCP} as
\begin{align}
&\frac{\partial G}{\partial\mathbb{J}}=
\frac{d}{2}\hbox{det}(\mathbb{M}_{K})^{\frac{1}{2}}
(\hbox{tr}(\mathbb{J}\mathbb{M}_K^{-1}\mathbb{J}^T))^{\frac{3 d}{4}-1}\mathbb{M}_K^{-1}\mathbb{J}^T,
\label{partial-J}
\\
&\frac{\partial G}{\partial \hbox{det}(\mathbb{J})}=
\frac{1}{2} d^\frac{3 d}{4}\hbox{det}(\mathbb{M}_{K})^{-\frac{1}{4}}\hbox{det}(\mathbb{J})^{\frac{1}{2}}
\label{partial-detJ}.
\end{align}
With these formulas, we can rewrite \eqref{MM} as (cf. \cite{Huang-Kamenski-2015JCP})
\begin{equation}\label{xim}
\begin{split}
\frac{d\bm{\xi}_i}{dt}=
\frac{\hbox{det}(\mathbb{M}(\bm{x_i}))^{\frac{1}{2}} }{\tau}
\sum_{K\in\omega_i}|K|\bm{v}^K_{i_K},
\quad i=1,...,N_v
\end{split}
\end{equation}
where $\omega_i$ is the element patch associated with the vertex $\bm{x}_i$,
$i_K$ is the local index of $\bm{x}_i$ on $K$,
and $\bm{v}^K_{i_K}$ is the local velocity contributed by the element $K$ to the vertex $i_K$.
The local velocities $\bm{v}^K_{i_K},\; i_K=0,...,d$ are given by
\begin{equation}\label{vjk}
\begin{split}
\left[
  \begin{array}{c}
    ( \bm{v}_1^K  )^T  \\
    ( \bm{v}_2^K  )^T   \\
    \vdots\\
    ( \bm{v}_d^K  )^T   \\
   \end{array}
 \right]
=
- E_K^{-1}\frac{\partial G }{\partial \mathbb{J} }
- \frac{\partial G}{\partial \hbox{det}(\mathbb{J})}\frac{ \hbox{det}( E_{K_c} )}
{\hbox{det}(E_K)} E_{K_c}^{-1},
\quad
\bm{v}^K_0=-\sum_{i_K=1}^d\bm{v}^K_{i_K} .
\end{split}
\end{equation}
Note that the velocities for the boundary nodes need to be modified properly.
For example, they should be set to be zero for the corner vertices.
For other boundary vertices, the velocities should be modified such that their normal component along the domain boundary are zeros so they slide only along the boundary and do not move out of the domain.

Starting with the reference computational mesh $\hat{\mathcal{T}}_c$ as the initial mesh,
the mesh equation \eqref{xim} is integrated over a physical time step.
The obtained new mesh is denoted by $\mathcal{T}_c^{n+1}$.
Note that $\mathcal{T}_h^{n}$ is kept fixed during the integration
and forms a correspondence with $\mathcal{T}_c^{n+1}$,
i.e., $\mathcal{T}_h^{n}=\Phi_h(\mathcal{T}_c^{n+1})$.
Then the new physical mesh $\mathcal{T}_h^{n+1}$ is defined
 as $\mathcal{T}_h^{n+1}=\Phi_h(\hat{\mathcal{T}}_c)$,
which can be computed using linear interpolation.

\section{Numerical results}
\label{sec:numerical-results}

In this section we present numerical results obtained with the MM-DG method described in the previous sections for the one- and two-dimensional SWEs.
We take the CFL number $C_{cfl}$ as $0.3$ for $P^1$-DG and $0.18$ for $P^2$-DG in one dimension,  and $0.2$ for $P^1$-DG and $0.1$ for $P^2$-DG in two dimensions, unless otherwise stated.
For the TVB limiter implemented in the RKDG scheme, the constant $M_{tvb}$ is taken as zero except for accuracy test Example \ref{test5-1d} (to avoid the accuracy order reduction near the extrema).
For mesh movement, we take $\tau=0.1 N^{-1/d}$.
The gravitation constant is taken as $g = 9.812$ in the computation.
For examples where the analytical exact solution is unavailable, we take the numerical solution obtained by the $P^2$-DG method with a fixed mesh of $N=10000$ as a reference solution, unless otherwise stated.
The error is computed with $21$ points used on each element.
Except for the accuracy test (Example \ref{test5-1d}) and the lake-at-rest steady-state flow tests
(Example \ref{test1+2-1d} and Example \ref{test1-2d}), to save space we omit the results for the $P^1$-DG method
since they are similar to those for the $P^2$-DG method.

\begin{example}\label{test5-1d}
(The accuracy test for the 1D SWEs over a sinusoidal hump.)
\end{example}
This example is used to verify the high order accuracy of the MM-DG method.
The bottom topography is
\begin{equation*}
B(x)=\sin^2(\pi x), \quad x \in (0,1).
\end{equation*}
We use periodic boundary conditions for all unknown variables. The initial conditions are given as
\begin{equation*}
h(x,0)=5+e^{\cos(2\pi x)}, \quad
hu(x,0)=\sin\big(\cos(2\pi x)\big) .
\end{equation*}
This example has been used as an accuracy-test in literature, e.g. see
\cite{Li-etal-2018JCAM,Xing-Shu-2006JCP,Xing-Shu-2006CiCP}.
The final simulation time is $T=0.1$ when the solutions remain smooth.
A reference solution is obtained using the $P^2$-DG method with a fixed mesh of $N=20000$.
The minmod constant $M_{tvb}$ in the TVB limiter is taken as $40$ in this example to avoid the accuracy order reduction near the extrema.
The $L^1$ and $L^\infty$ norm of the error with moving and fixed meshes for the water depth $h$ and water discharge $hu$ are plotted in Fig.~\ref{Fig:test5-1d-h-hu}.
It can be seen that the MM-DG method has the expected second-order convergence for $P^1$-DG and
third-order for $P^2$-DG in both $L^1$ and $L^\infty$ norm. Moreover, the figures show that a moving mesh produces a slightly smaller but otherwise comparable error than a fixed mesh of the same number of elements.
The error is much smaller for $P^2$-DG method than the $P^1$-DG method, as expected for smooth problems.

\begin{figure}[h]
\centering
\subfigure[$L^1$-error: $h$]{
\includegraphics[width=0.4\textwidth,trim=40 0 40 10,clip]{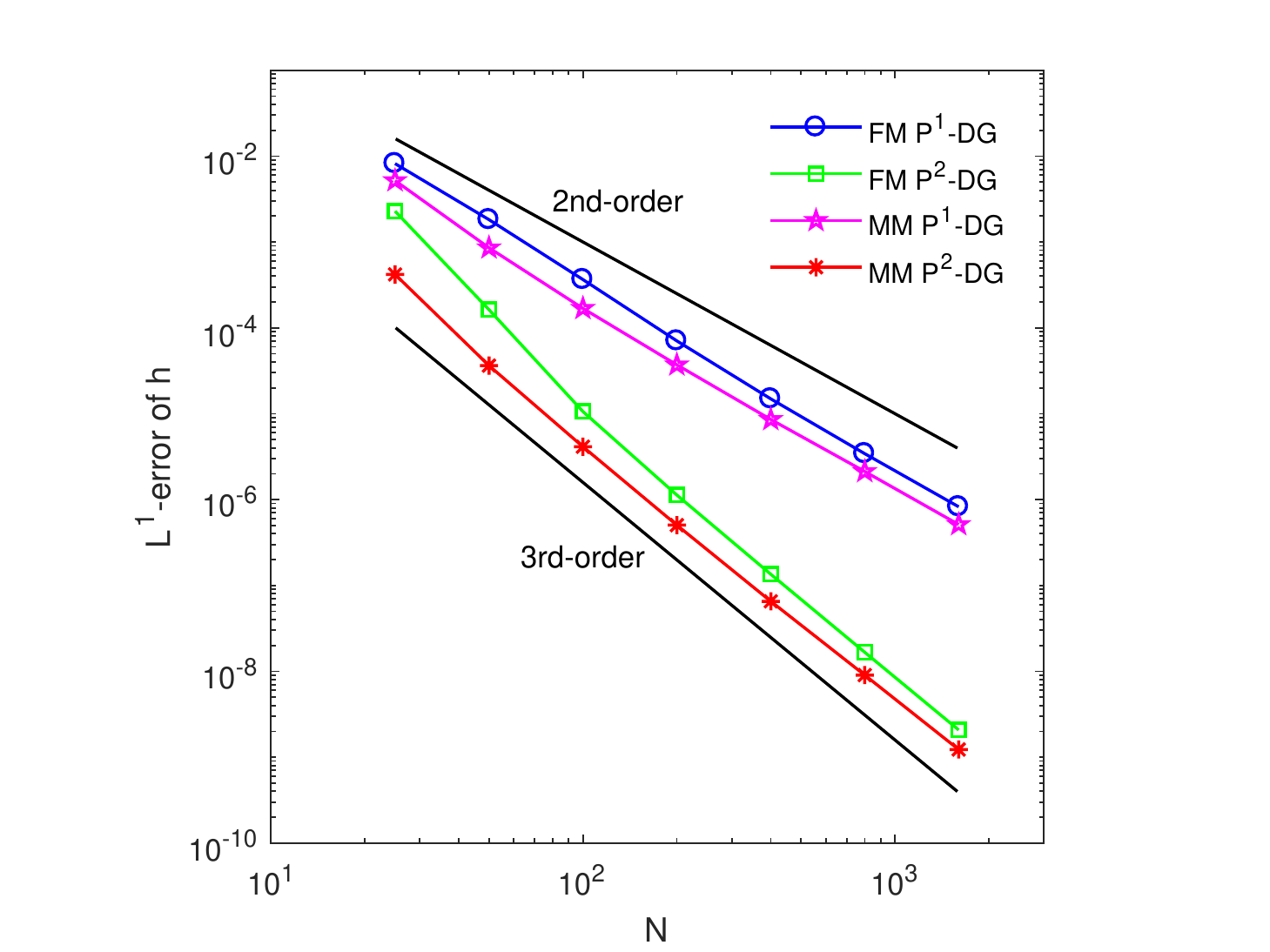}}
\subfigure[$L^\infty$-error: $h$]{
\includegraphics[width=0.4\textwidth,trim=40 0 40 10,clip]{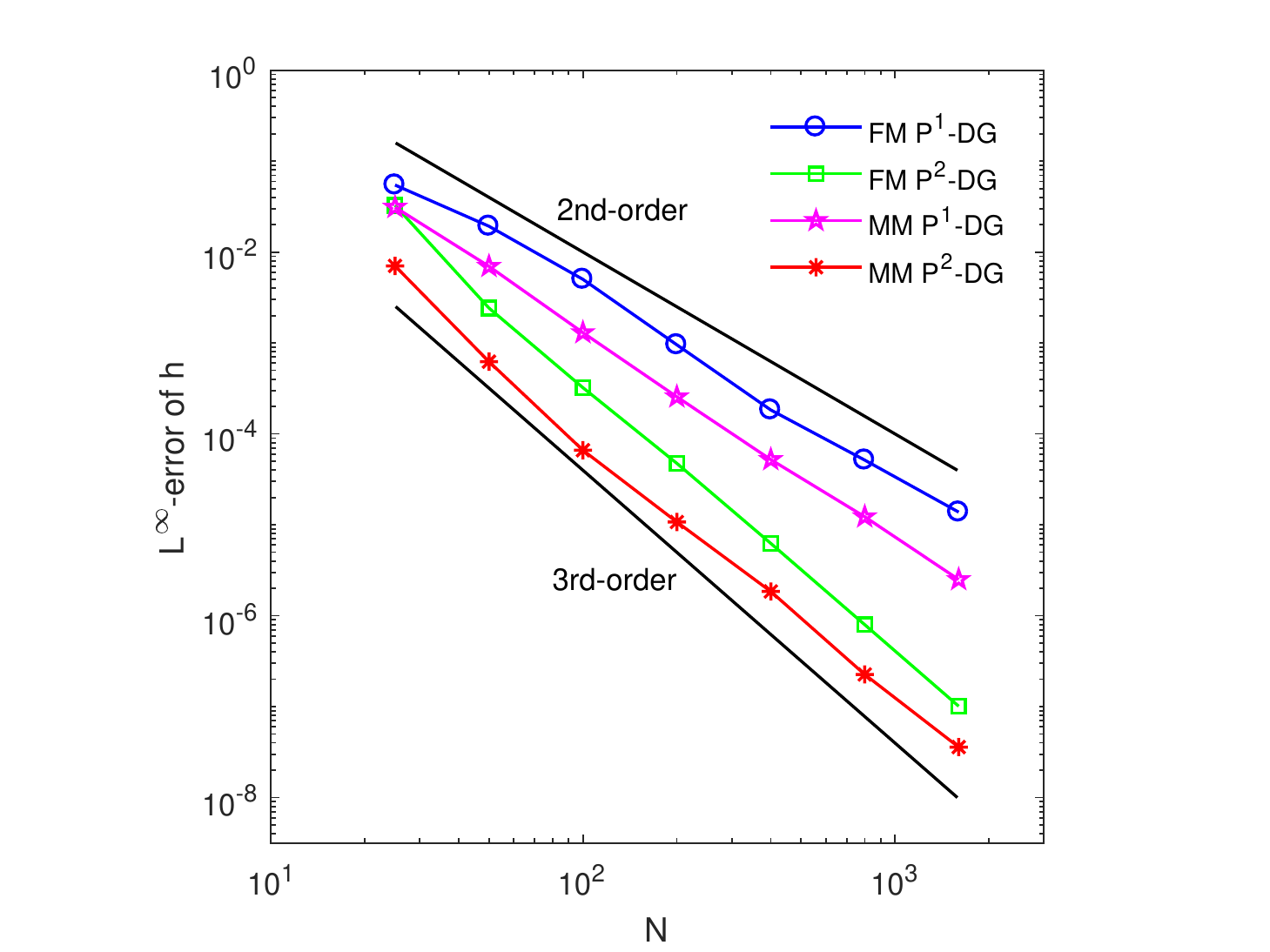}}
\subfigure[$L^1$-error: $hu$]{
\includegraphics[width=0.4\textwidth,trim=40 0 40 10,clip]{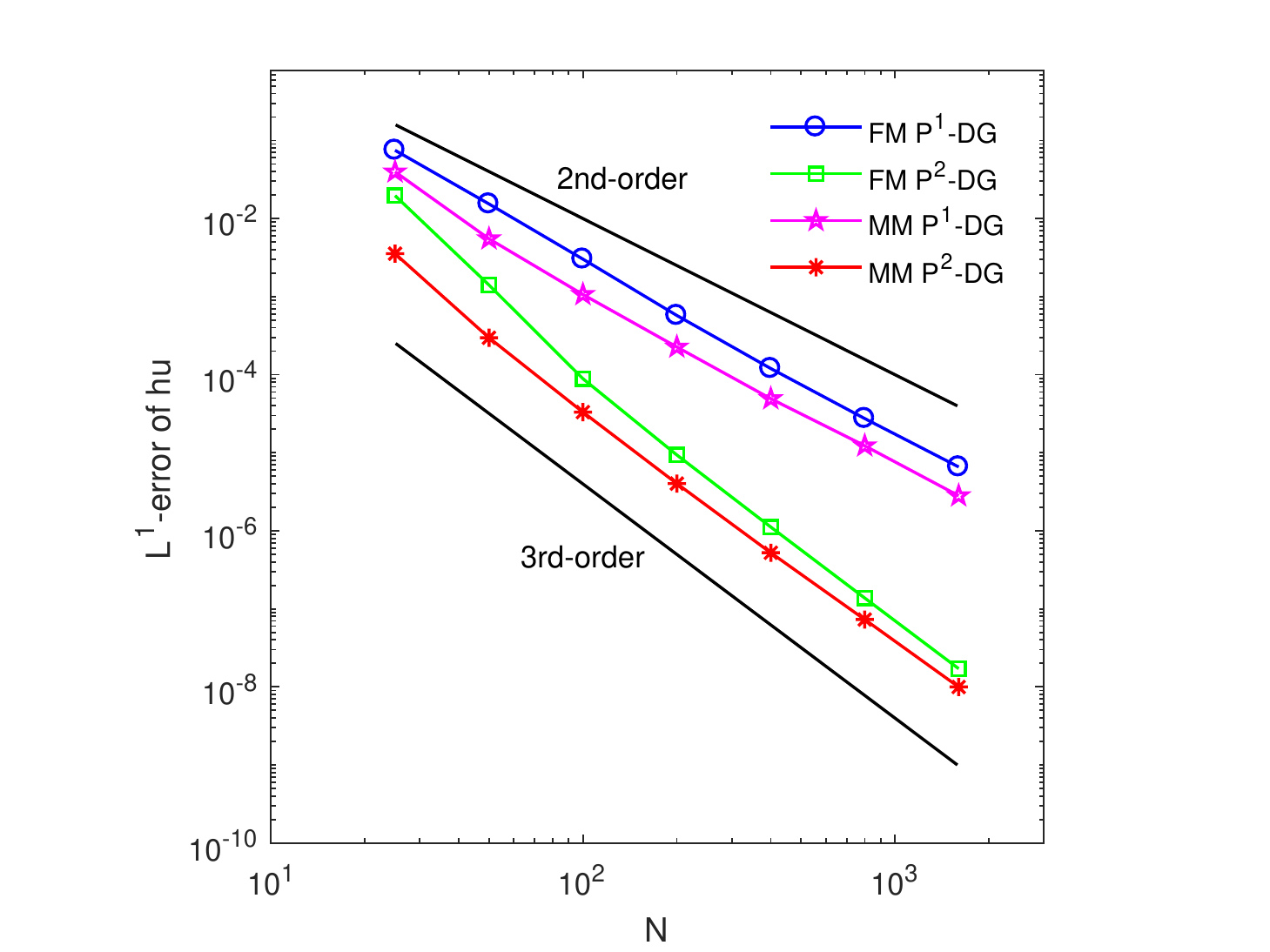}}
\subfigure[$L^\infty$-error: $hu$]{
\includegraphics[width=0.4\textwidth,trim=40 0 40 10,clip]{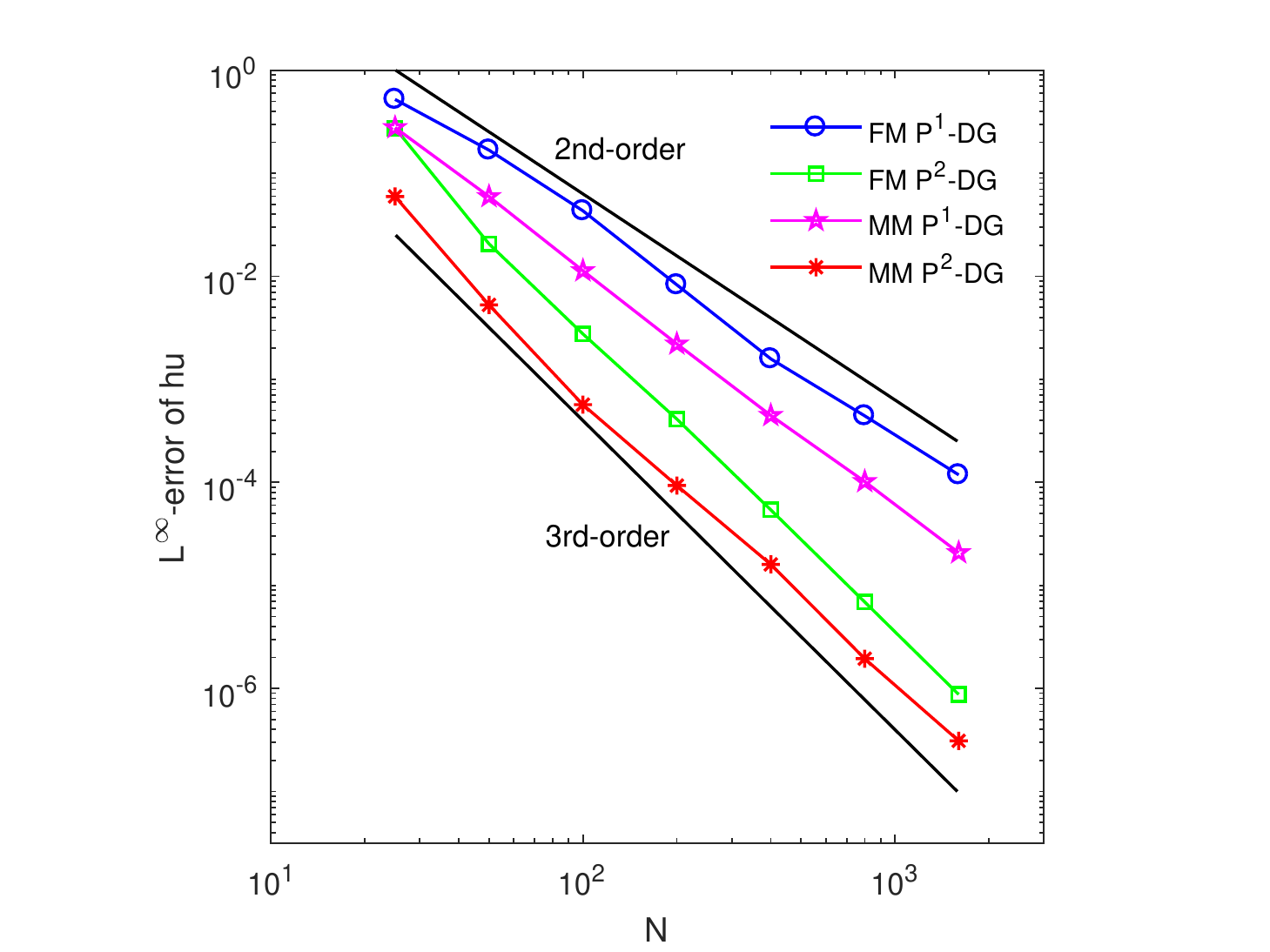}}
\caption{Example \ref{test5-1d}. The $L^1$ and $L^\infty$ norm of the error with moving and fixed meshes for the  water depth $h$ and the water discharge $hu$.}
\label{Fig:test5-1d-h-hu}
\end{figure}

\begin{example}\label{test1+2-1d}
(The lake-at-rest steady-state flow test for the 1D SWEs.)
\end{example}
In this example we consider a lake-at-rest steady-state flow over three different bottom topographies
to verify the well-balance property of the MM-DG method. Two of these topographies are smooth and
the other is discontinuous, i.e.,
\begin{align}
& B(x)=5 e^{-\frac25(x-5)^2},  \quad x \in (0,10)
\label{B-1}
\\
& B(x)=
\begin{cases}
4,& \text{for } x \in (4, 8) \\
0,& \text{for } x \in (0,4) \cup (8, 10)
\end{cases}
\label{B-2}
\\
& B(x)=10 e^{-\frac25(x-5)^2},  \quad x \in (0,10) .
\label{B-1-1}
\end{align}
The initial data is at the lake-at-rest steady state,
\begin{equation*}
u=0, \quad h+B=10.
\end{equation*}
This still water state solution should be preserved exactly if the MM-DG method is well-balanced.

We computed the solution up to $t=0.5$ on moving meshes.
To show that the well-balance property is attained up to the level of round-off error (double precision in MATLAB), we present the $L^1$ and $L^\infty$ error for $h+B$ and $hu$ in Tables~\ref{tab:test1-1d-p1-error} and~\ref{tab:test1-1d-p2-error} for $P^1$-DG and $P^2$-DG on moving meshes, respectively, for the smooth bottom topography.
Similar results for the discontinuous bottom topography are shown in Tables~\ref{tab:test2-1d-p1-error} and~\ref{tab:test2-1d-p2-error}. For comparison purpose, we also list the results obtained
with the $L^2$-projection of $B$.
From the results one can see that the MM-DG method with $B$ updated with DG-interpolation
preserves the lake-at-rest solution. On the other hand, when $B$ is updated with $L^2$-projection,
the lake-at-rest solution is not preserved by the MM-DG method.
The results show that the error, more precisely, the deviation from the lake-at-rest steady state,
is about 2nd-order for the smooth
topography for both $P^1$-DG and $P^2$-DG.
For the discontinuous topography, the error is about 2nd-order
in $L^1$ and 1st-order in $L^\infty$.
To explain this, we notice that in the current situation the deviation from the lake-at-rest steady state
comes from the error of the scheme which in turn is dominated by the error of the TVB limiter.
For this example, we have not made effort in optimizing the choice
of the TVB minmod constant $M_{tvb}$. Instead, we have simply chosen $M_{tvb} = 0$, for which
the TVB limiter becomes the TVD limiter that is second-order at best in one dimension \cite{DG-series2,Goodman-LeVeque-1985MC}.

The bottom topography (\ref{B-1-1}) (with a dry region) is used to verify the well-balance and PP properties
of the MM-DG method. This topography has a similar shape as (\ref{B-1}) but its height touches
the water level at $x=5$ where $h=0$ initially. The computed water depth can have negative values during
the computation and the application of the PP limiter is necessary.
We computed the solution up to $t=0.5$ on moving meshes.
To ensure positivity preservation (cf. \cite{Xing-Zhang-Shu-2010}), we take
smaller CFL numbers as $0.3$ and $0.15$ for $P^1$-DG and $P^2$-DG, respectively, for this test.
The $L^1$ and $L^\infty$ error for $h+B$ and $hu$ is listed in Table~\ref{tab:test-1d-wb-pp-error} for $P^1$-DG and $P^2$-DG. The results clearly show that the MM-DG method is well-balanced.

\begin{table}[htb]
\caption{Example \ref{test1+2-1d}. Well-balance test for the moving mesh $P^1$-DG method over the smooth bottom topography (\ref{B-1}).}
\vspace{3pt}
\centering
\label{tab:test1-1d-p1-error}
\begin{tabular}{ccccc}
 \toprule
  & \multicolumn{2}{c}{$h+B$}&\multicolumn{2}{c}{$hu$}\\
$N$&$L^1$-error  &$L^{\infty}$-error  &$L^1$-error  &$L^{\infty}$-error \\
\midrule
~   &  \multicolumn{4}{c}{\em{$B$ updated with DG-interpolation} }\\
\midrule
50	&	1.970E-14	&	2.551E-14	&	2.810E-14	&	7.858E-14	\\
100	&	4.066E-14	&	4.958E-14	&	4.955E-14	&	1.164E-13	\\
200	&	7.874E-14	&	9.286E-14	&	7.765E-14	&	1.911E-13	\\
\midrule
~   &  \multicolumn{4}{c}{\em{B updated with $L^2$-projection} }\\
\midrule
50	&3.393E-04&	9.016E-04&	2.827E-03&	7.166E-03\\
100	&9.550E-05&	2.852E-04&	8.202E-04&	2.338E-03\\
200	&2.381E-05&	8.143E-05&	2.039E-04&	6.823E-04\\
 \bottomrule	
\end{tabular}
\end{table}

\begin{table}[htb]
\caption{Example \ref{test1+2-1d}. Well-balance test for the moving mesh $P^2$-DG method over the smooth bottom topography (\ref{B-1}).}
\vspace{3pt}
\centering
\label{tab:test1-1d-p2-error}
\begin{tabular}{ccccc}
 \toprule
  & \multicolumn{2}{c}{$h+B$}&\multicolumn{2}{c}{$hu$}\\
$N$&$L^1$-error  &$L^{\infty}$-error  &$L^1$-error  &$L^{\infty}$-error \\
\midrule
~   &  \multicolumn{4}{c}{\em{$B$ updated with DG-interpolation} }\\
\midrule
50	&3.051E-14	&	3.813E-14&	3.951E-14	&	9.699E-14\\
100	&6.282E-14	&	7.583E-14&	7.464E-14	&	1.769E-13\\
200	&1.276E-13	&	1.503E-13&	1.449E-13	&	3.481E-13\\
\midrule
~   &  \multicolumn{4}{c}{\em{$B$ updated with $L^2$-projection} }\\
\midrule
50	&4.922E-06&	1.251E-05&	4.995E-05&	1.393E-04\\
100	&1.160E-06&	3.775E-06&	1.139E-05&	3.573E-05\\
200	&2.395E-07&	8.878E-07&	2.299E-06&	8.875E-06\\
 \bottomrule	
\end{tabular}
\end{table}

\begin{table}[htb]
\caption{Example \ref{test1+2-1d}. Well-balance test for the moving mesh $P^1$-DG method over the discontinuous bottom topography (\ref{B-2}).}
\vspace{3pt}
\centering
\label{tab:test2-1d-p1-error}
\begin{tabular}{ccccc}
 \toprule
  & \multicolumn{2}{c}{$h+B$}&\multicolumn{2}{c}{$hu$}\\
$N$&$L^1$-error  &$L^{\infty}$-error  &$L^1$-error  &$L^{\infty}$-error \\
\midrule
~   &  \multicolumn{4}{c}{\em{$B$ updated with DG-interpolation} }\\
\midrule
50	&1.175E-14	&	2.014E-14&	4.114E-14	&	1.001E-13\\
100	&1.912E-14	&	3.143E-14&	6.747E-14	&	1.501E-13\\
200	&3.108E-14	&	5.032E-14&	1.136E-13	&	2.576E-13\\
\midrule
~   &  \multicolumn{4}{c}{\em{$B$ updated with $L^2$-projection} }\\
\midrule
50	&6.001E-03	&	5.651E-02&	4.542E-02	&	3.688E-01\\
100	&1.823E-03	&	2.629E-02&	1.637E-02	&	1.713E-01\\
200	&5.499E-04	&	1.237E-02&	5.166E-03	&	9.294E-02\\
 \bottomrule	
\end{tabular}
\end{table}

\begin{table}[htb]
\caption{Example \ref{test1+2-1d}. Well-balance test for the moving mesh $P^2$-DG method over the discontinuous bottom topography (\ref{B-2}).}
\vspace{3pt}
\centering
\label{tab:test2-1d-p2-error}
\begin{tabular}{ccccc}
 \toprule
  & \multicolumn{2}{c}{$h+B$}&\multicolumn{2}{c}{$hu$}\\
$N$&$L^1$-error  &$L^{\infty}$-error  &$L^1$-error  &$L^{\infty}$-error \\
\midrule
~   &  \multicolumn{4}{c}{\em{$B$ updated with DG-interpolation} }\\
\midrule
50	&1.518E-14	&	2.635E-14&	5.499E-14	&	1.562E-13\\
100	&2.166E-14	&	4.096E-14&	1.003E-13	&	2.367E-13\\
200	&3.419E-14	&	6.230E-14&	1.679E-13	&	3.781E-13\\
\midrule
~   &  \multicolumn{4}{c}{\em{$B$ updated with $L^2$-projection} }\\
\midrule
50	&3.548E-03	&	2.796E-02&	3.090E-02	&	2.522E-01\\
100	&8.669E-04	&	1.073E-02&	7.585E-03	&	9.252E-02\\
200	&1.903E-04	&	3.710E-03&	1.624E-03	&	2.714E-02\\						
 \bottomrule	
\end{tabular}
\end{table}

\begin{table}[H]
\caption{Example \ref{test1+2-1d}. Well-balance test for the MM-DG method over the bottom topography (\ref{B-1-1})
(with a dry region).}
\vspace{3pt}
\centering
\label{tab:test-1d-wb-pp-error}
\begin{tabular}{ccccc}
 \toprule
  & \multicolumn{2}{c}{$h+B$}&\multicolumn{2}{c}{$hu$}\\
$N$&$L^1$-error  &$L^{\infty}$-error  &$L^1$-error  &$L^{\infty}$-error \\
\midrule
~   &  \multicolumn{4}{c}{\em{$P^1$ MM-DG method with PP limiter} }\\
\midrule
50	&	1.362E-14	&	2.131E-14	&	3.802E-14	&	8.519E-14	\\
100	&	2.426E-14	&	3.709E-14	&	6.166E-14	&	1.356E-13	\\
200	&	5.079E-14	&	7.733E-14	&	1.225E-13	&	2.779E-13	\\
\midrule
~   &  \multicolumn{4}{c}{\em{$P^2$ MM-DG method with PP limiter} }\\
\midrule
50	&	2.404E-14	&	4.065E-14	&	7.451E-14	&	1.781E-13	\\
100	&	4.197E-14	&	6.853E-14	&	1.355E-13	&	2.919E-13	\\
200	&	8.531E-14	&	1.322E-13	&	2.496E-13	&	5.637E-13	\\		
 \bottomrule	
\end{tabular}
\end{table}

\begin{example}\label{test3-1d}
(The perturbed lake-at-rest steady-state flow test for the 1D SWEs.)
\end{example}
This example has been used by a number of researchers
\cite{Donat-etal-2014JCP,LeVeque-1998JCP,Li-etal-2018JCAM,Xing-Shu-2006JCP,Xing-Shu-2006CiCP}
to demonstrate the capability of a numerical method to accurately compute small perturbations of a lake-at-rest steady-state flow over
non-flat bottom topographies.
The bottom topography in this example is taken as
\begin{equation}
\label{B-5}
B(x)=
\begin{cases}
0.25(\cos(10\pi(x-1.5))+1),& \text{for } x \in (1.4, 1.6) \\
0,& \text{for } x \in (0, 1.4) \cup (1.6, 2)
\end{cases}
\end{equation}
which has a bump in the middle of the physical interval.
The initial conditions are given by
\begin{equation*}
h(x,0)=
\begin{cases}
1-B(x)+\varepsilon,& \text{for}~1.1\leq x\leq 1.2\\
1-B(x),& \text{otherwise}
\end{cases}
\quad \hbox{and}\quad u(x,0)=0,
\end{equation*}
where $\varepsilon$ is a constant for the perturbation magnitude.
We consider two cases, $\varepsilon=0.2$ (big pulse) and $\varepsilon=10^{-5}$ (small pulse).
The initial conditions for both cases are plotted in Fig.~\ref{Fig:test3-1d-initial}.
The initial wave splits into two waves propagating at the characteristic speeds $\pm \sqrt{gh}$.
We use the transmissive boundary conditions and compute the solution up to $t=0.2$ when the right wave has
already passed the bottom bump.

The mesh trajectories, water surface $B+h$, and discharge $hu$ at $t=0.2$ obtained with $P^2$-DG and moving and fixed meshes of $N=160$ are shown in Figs.~\ref{Fig:test3-1d-large-metric} (for $\varepsilon=0.2$)
and \ref{Fig:test3-1d-small-metric} (for $\varepsilon=10^{-5}$).
Recall that these results are obtained with the metric tensor \eqref{mer-ceil} based on
the equilibrium variable $\mathcal{E}=\frac{1}{2}u^2+g(h+B)$ and the water depth $h$.
For comparison purpose, we also include the results obtained with the metric tensor based on
the entropy/total energy $E =\frac{1}{2}hu^2+\frac{1}{2}gh^2+ghB$.
Interestingly, for the large pulse $\varepsilon=0.2$ (Fig.~\ref{Fig:test3-1d-large-metric}),
both ways of computing the metric tensor lead to almost indistinguishable solutions although
they have slightly different but correct mesh concentration: the former concentrates more points
in the sharp wave regions whereas the latter leads to more points in the non-flat bottom region.
For the case with a small pulse $\varepsilon=10^{-5}$ (Fig.~\ref{Fig:test3-1d-small-metric}),
the situation is very different. The metric tensor based on $E$ misses the wave regions and concentrates
points only in the non-flat bottom region. On the other hand, the metric tensor based on
$\mathcal{E}$ and $h$ works well for this case too and concentrates mesh points
in both the wave regions and the non-flat bottom region.
The advantages of using $\mathcal{E}$ and $h$ over $E$ in mesh adaptation are clear
when the perturbation is smaller.

It is interesting to see that the small perturbation splits into two waves at about $t=0.0165$
and the mesh elements are concentrated properly near the waves before and after the split.

The water surface $B+h$ and discharge $hu$ obtained with $P^2$-DG and a moving mesh of $N=160$ and fixed meshes of $N=160$ and $N=480$ are plotted in Figs.~\ref{Fig:test3-1d-large-Bph} -- \ref{Fig:test3-1d-large-hu} (for $\varepsilon=0.2$) and Figs.~\ref{Fig:test3-1d-small-Bph} -- \ref{Fig:test3-1d-small-hu} (for $\varepsilon=10^{-5}$).
The results show that the moving mesh solutions with $N=160$ are more accurate than those with the fixed mesh of $N=160$ and $N=480$ and contain no visible spurious numerical oscillations.
They also show that the MM-DG method with moving or fixed meshes is able to capture the waves of
large or small pulses.

Comparison results are shown in Fig.~\ref{Fig:test3-1d-Bjob-large} ($\varepsilon=0.2$) and
Fig.~\ref{Fig:test3-1d-Bjob-small} ($\varepsilon=10^{-5}$) by using DG-interpolation or $L^2$-projection
in updating the bottom topography $B$. The computed solutions are almost indistinguishable for the large pulse
case with $\varepsilon=0.2$. On the other hand, for the small pulse case with $\varepsilon=10^{-5}$
(Fig.~\ref{Fig:test3-1d-Bjob-small}), the MM-DG method with DG-interpolation captures the waves
well whereas the method with $L^2$-projection misses them badly. Recall that the former is well-balanced
but the latter is not.
For the latter, the deviation from the lake-at-rest steady state comes from the error of the scheme
that is second-order at best since the TVB limiter with $M_{tvb}=0$ is used.
When this deviation is smaller than the size of the perturbation, the scheme is able to capture the waves
as shown in Fig.~\ref{Fig:test3-1d-Bjob-large}.
However, when the deviation is larger than the size of the perturbation,
the scheme will miss the waves, as seen in Fig.~\ref{Fig:test3-1d-Bjob-small}, although the situation can be improved with finer meshes. In contrast, well-balanced schemes are able to capture small perturbations
even on a reasonably coarse mesh.

We now comment on the cost of the DG-interpolation scheme.
Recall from Remark~\ref{rem-DG-interp} that for a given example, the integration of (\ref{pde})
requires only a constant number of time steps, which implies that each interpolation costs
$\mathcal{O}(N_v)$ operations. This has been observed
in all the numerical examples we have tested.
For example, the average number of time steps
is around $3$ ($\varepsilon=0.2$) and $4$ ($\varepsilon=10^{-5}$) for $P^2$-DG applied to the current example.
Since the cost of the DG-interpolation scheme has been studied extensively
in \cite{Zhang-Huang-Qiu-2019arXiv}, we will not discuss this further in this work.

To verify the well-balance and positivity-preserving properties of the MM-DG method
we modify the bottom topography (\ref{B-5}) to contain a dry region (near $x = 1.5$),
\begin{equation}
\label{B-5-1}
\begin{split}
&B(x)=
\begin{cases}
0.5(\cos(10\pi(x-1.5))+1),& \text{for } x \in (1.4, 1.6) \\
0,& \text{for } x \in (0, 1.4) \cup (1.6, 2)
\end{cases}
\end{split}
\end{equation}
We repeat the computation with $\varepsilon = 10^{-5}$.
The bottom topography, the initial water level, and the mesh trajectories of $N=160$
obtained with $P^2$ MM-DG method are plotted in Fig.~\ref{Fig:test-1d-wb-pp-preturb-initial}.
The mesh points are concentrated around the shock waves and the non-flat topography region.
Interestingly, the mesh trajectories show that the right moving shock stops after it hits the dry region.

The water surface $B+h$ and discharge $hu$ obtained with $P^2$-DG and a moving mesh of $N=160$
and fixed meshes of $N=160$ and $N=640$ are plotted in Figs.~\ref{Fig:test-1d-perturb-Bph} and
\ref{Fig:test-1d-preturb-hu}. The results show that the MM-DG method with moving or fixed meshes
is able to capture the waves of small perturbation for the situation with dry regions.
Moreover, the moving mesh solutions with $N=160$ are more accurate than those with
fixed mesh of $N=160$ and $N=640$ and contain no visible spurious numerical oscillations.

\begin{figure}[H]
\centering
\subfigure[Big pulse $\varepsilon=0.2$]{
\includegraphics[width=0.4\textwidth,trim=20 0 40 10,clip]{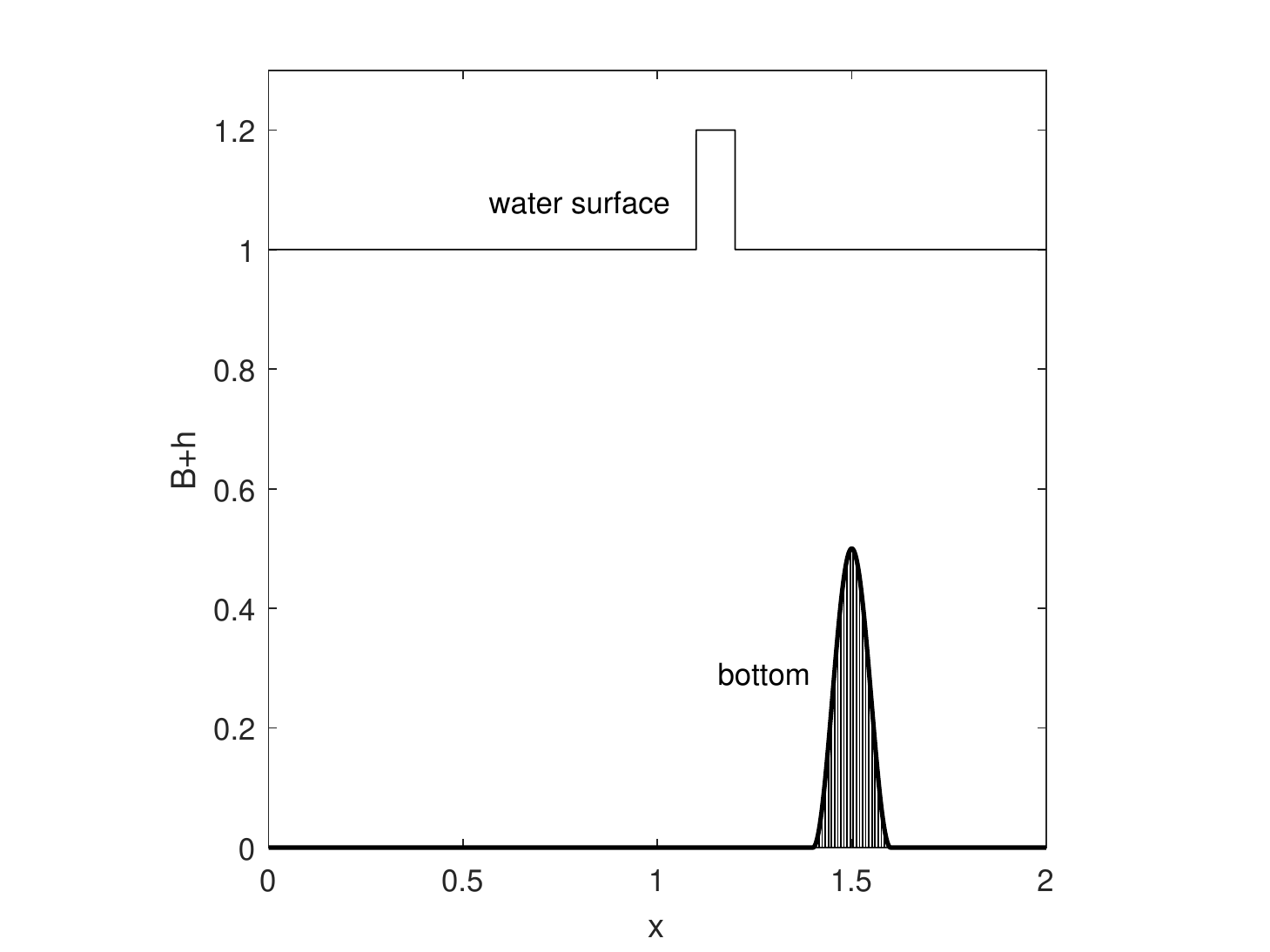}}
\subfigure[Small pulse $\varepsilon=10^{-5}$]{
\includegraphics[width=0.4\textwidth,trim=20 0 40 10,clip]{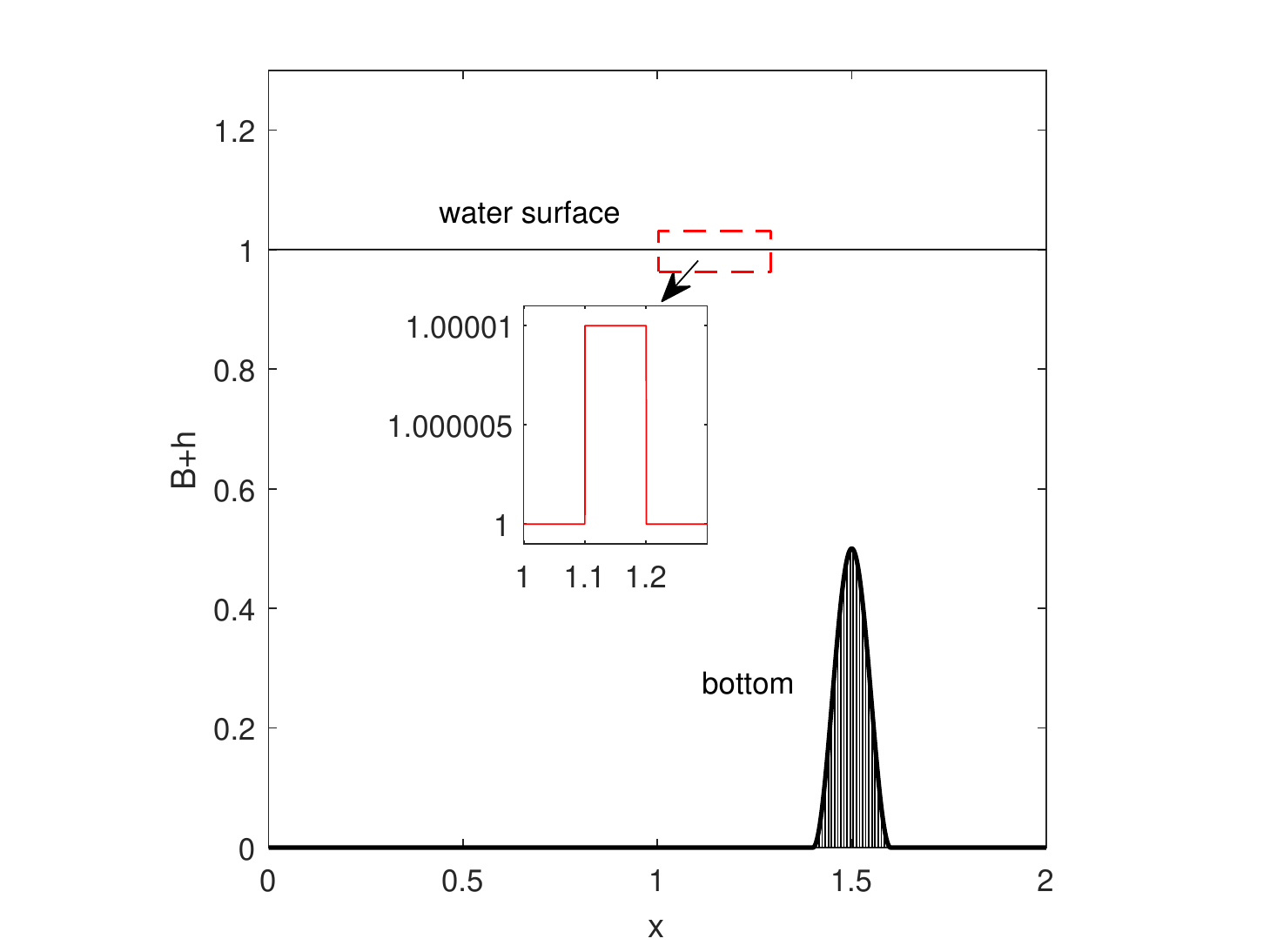}}
\caption{Example \ref{test3-1d}. The initial water surface level $B+h$ and bottom topography $B$ for the pulse of $\varepsilon=0.2$ and $\varepsilon=10^{-5}$.}
\label{Fig:test3-1d-initial}
\end{figure}

\begin{figure}[H]
\centering
\subfigure[$\mathcal{E}$ and $h$: Mesh trajectories]{
\includegraphics[width=0.4\textwidth,trim=20 0 40 10,clip]{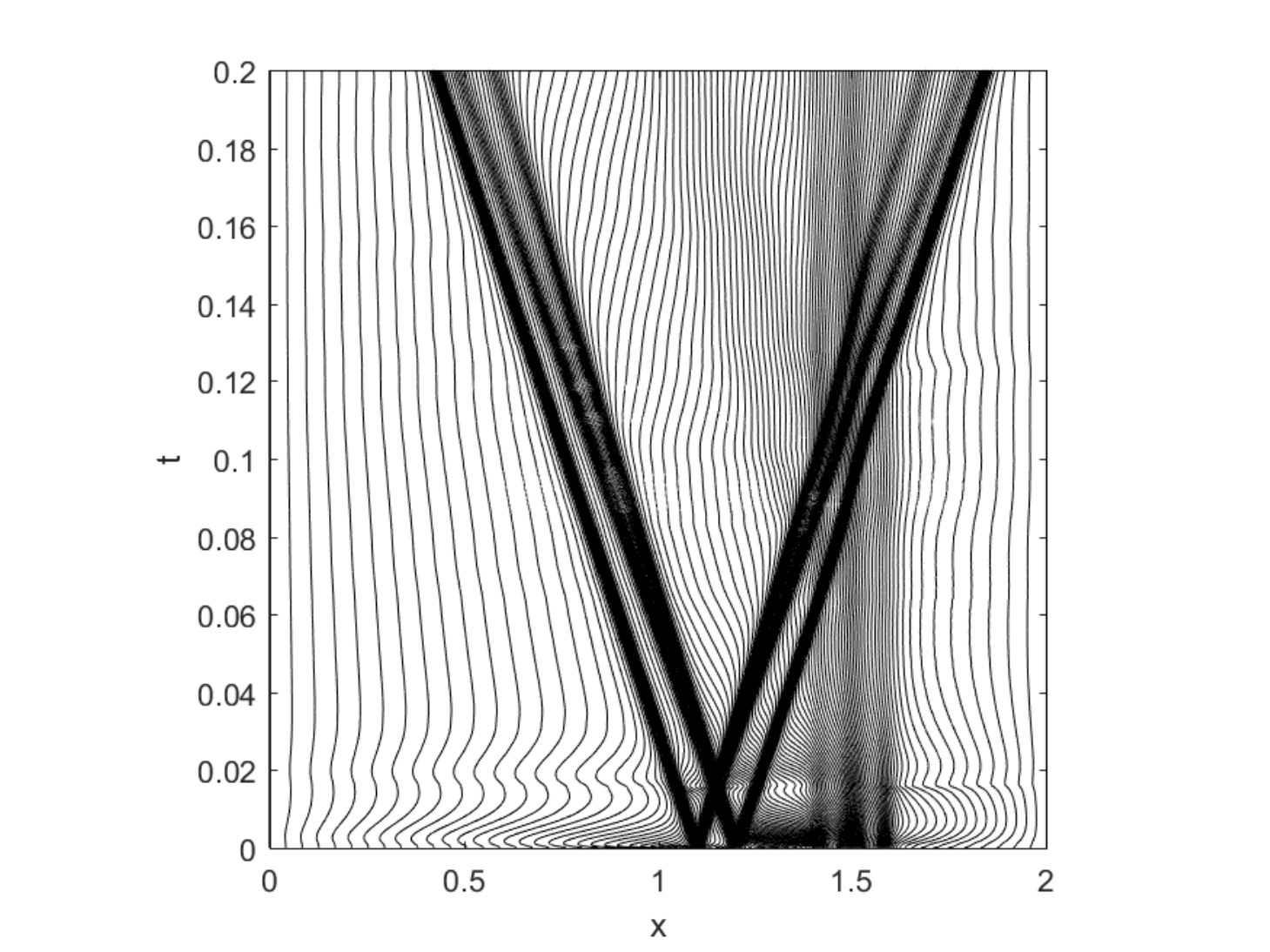}}
\subfigure[Entropy: Mesh trajectories]{
\includegraphics[width=0.4\textwidth,trim=20 0 40 10,clip]{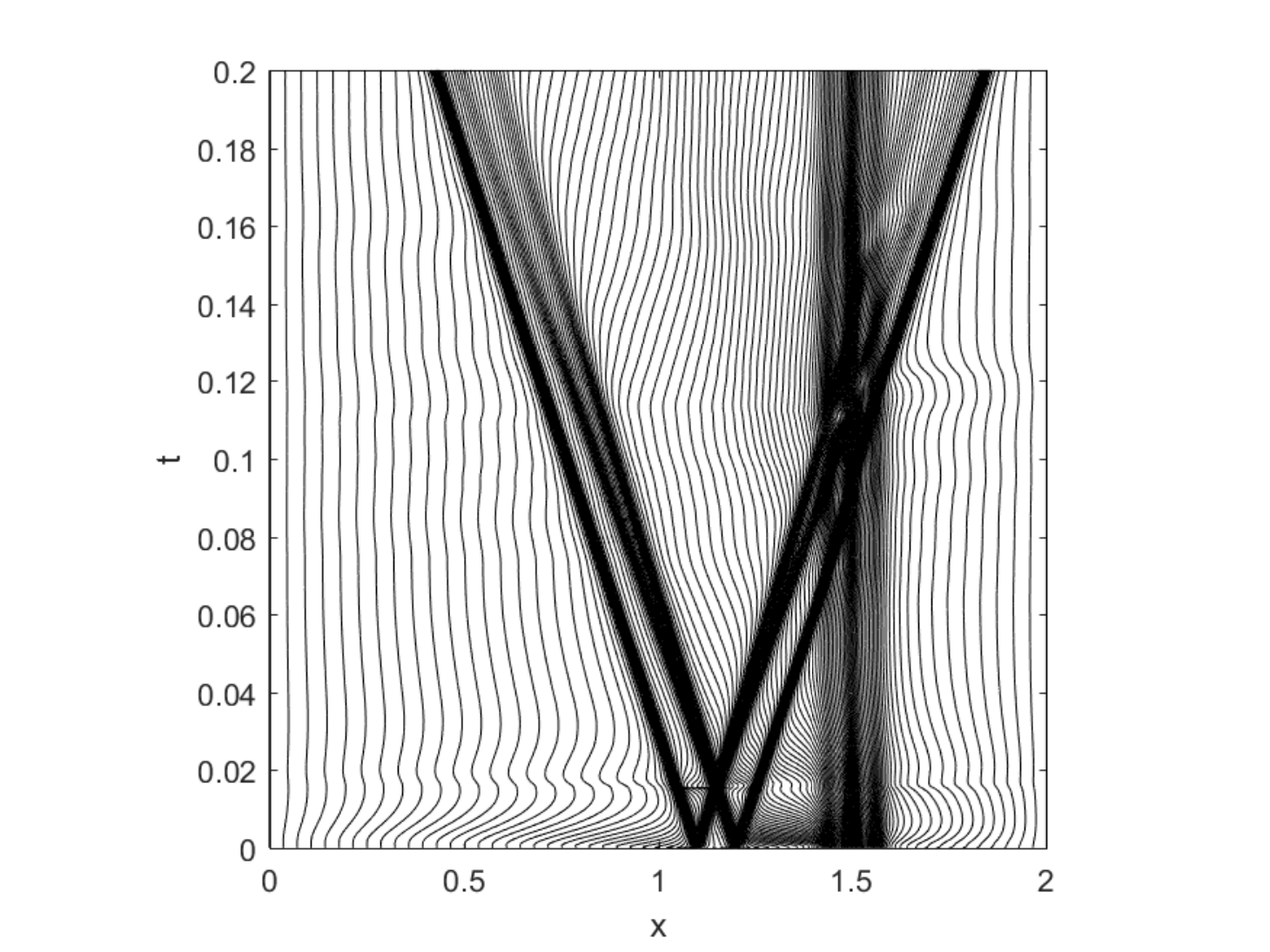}}
\subfigure[$\mathcal{E}$ and $h$: $B+h$]{
\includegraphics[width=0.4\textwidth,trim=20 0 40 10,clip]{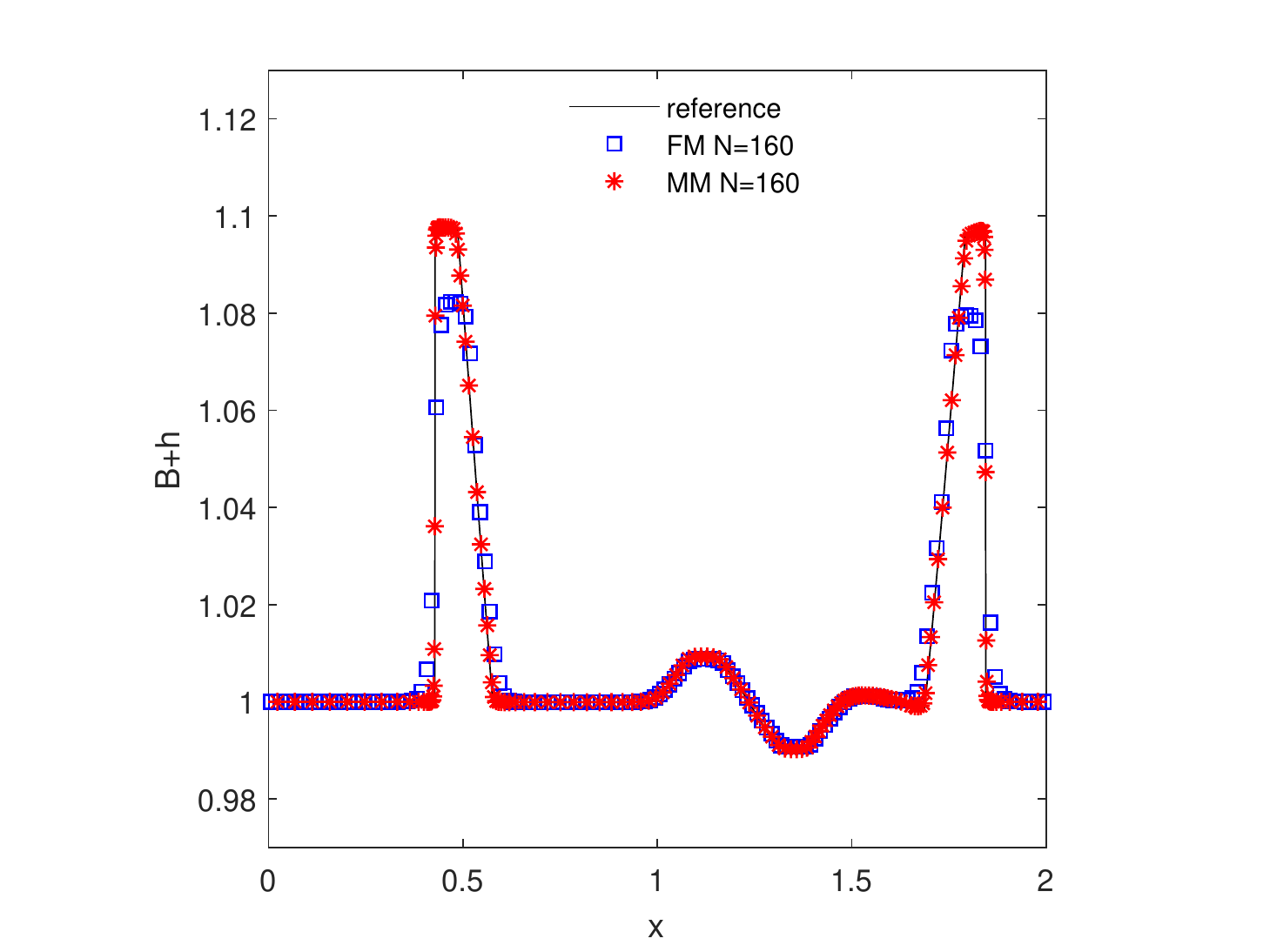}}
\subfigure[Entropy: $B+h$]{
\includegraphics[width=0.4\textwidth,trim=20 0 40 10,clip]{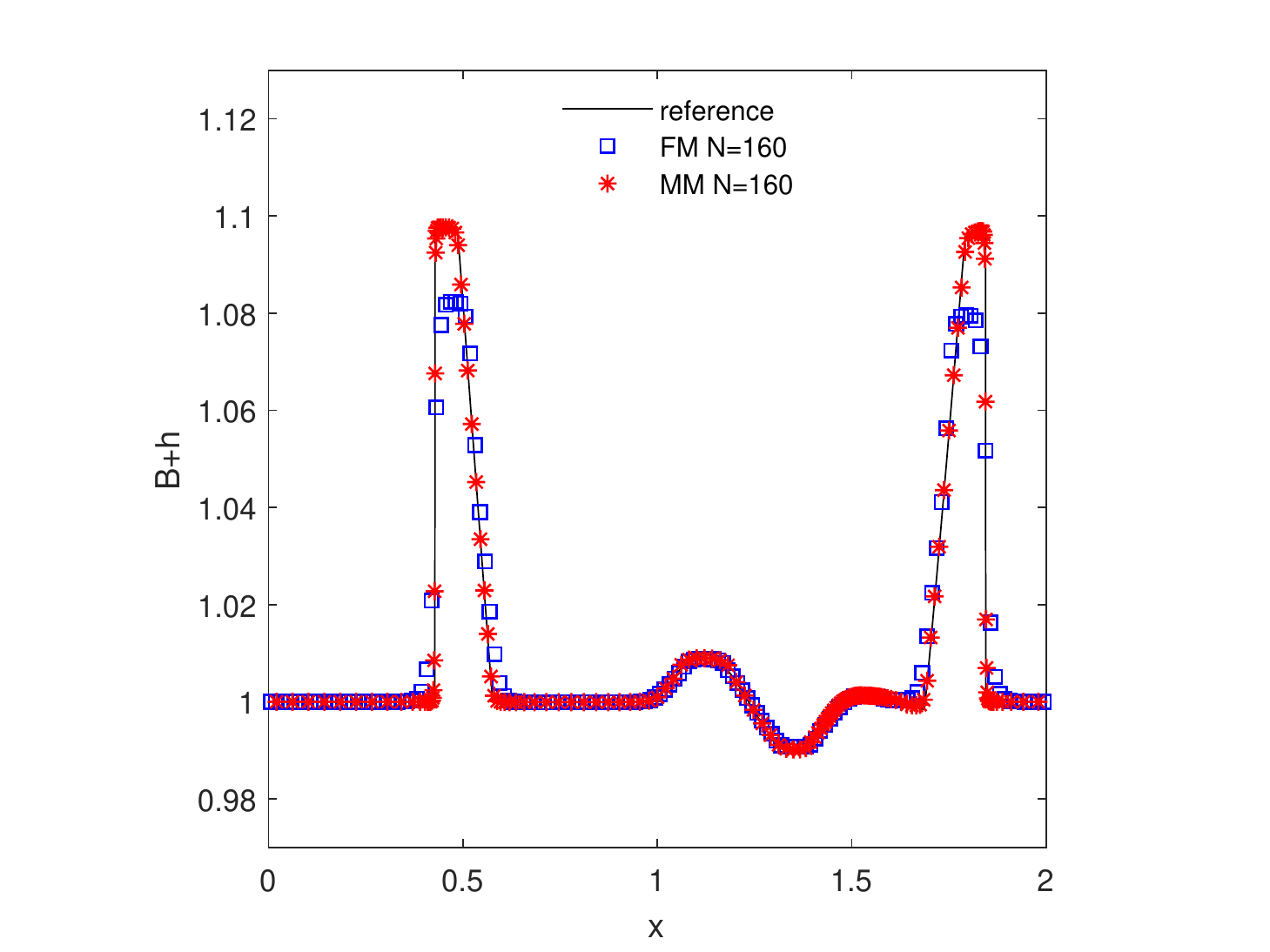}}
\subfigure[$\mathcal{E}$ and $h$: $hu$]{
\includegraphics[width=0.4\textwidth,trim=20 0 40 10,clip]{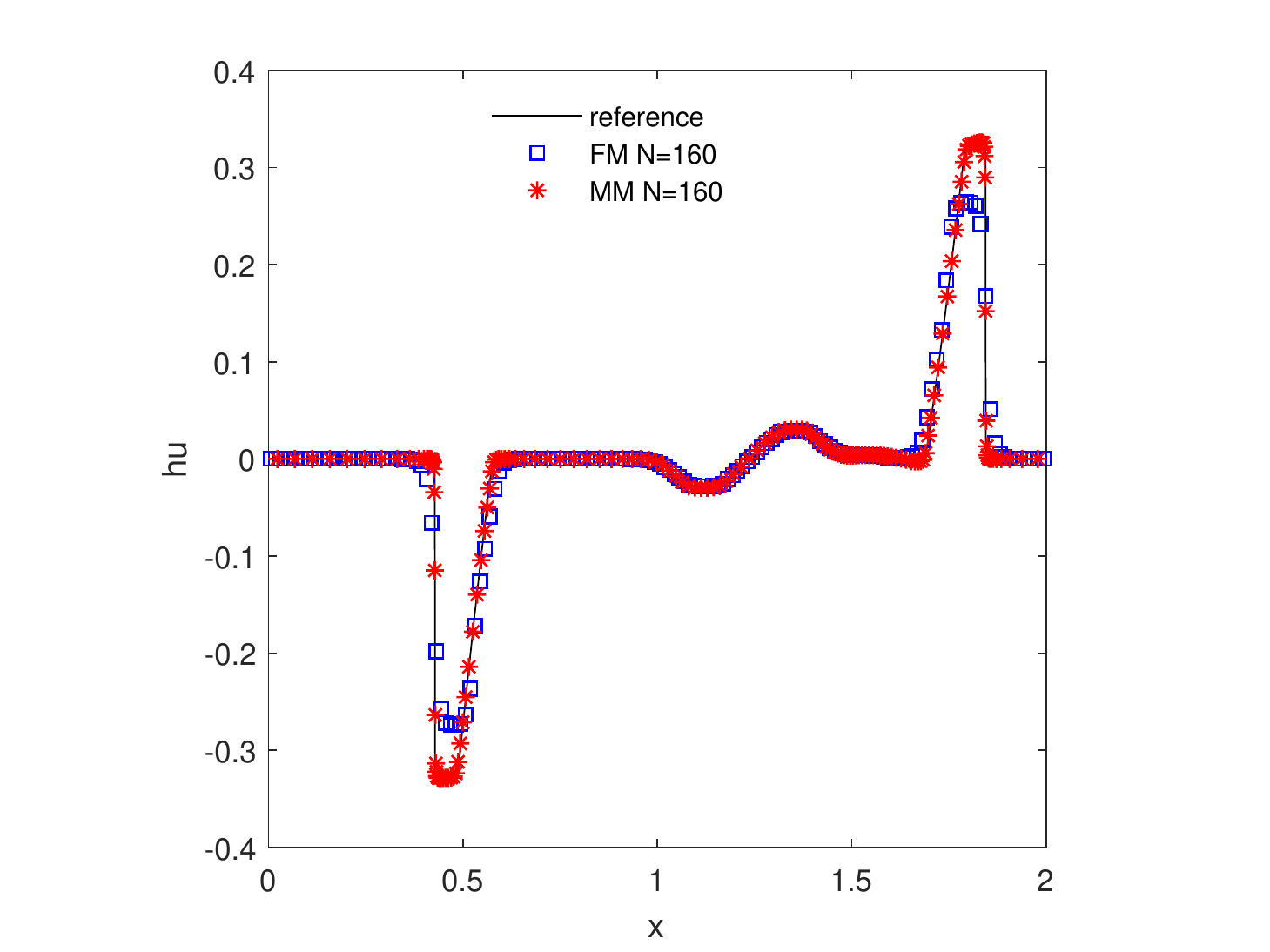}}
\subfigure[Entropy: $hu$]{
\includegraphics[width=0.4\textwidth,trim=20 0 40 10,clip]{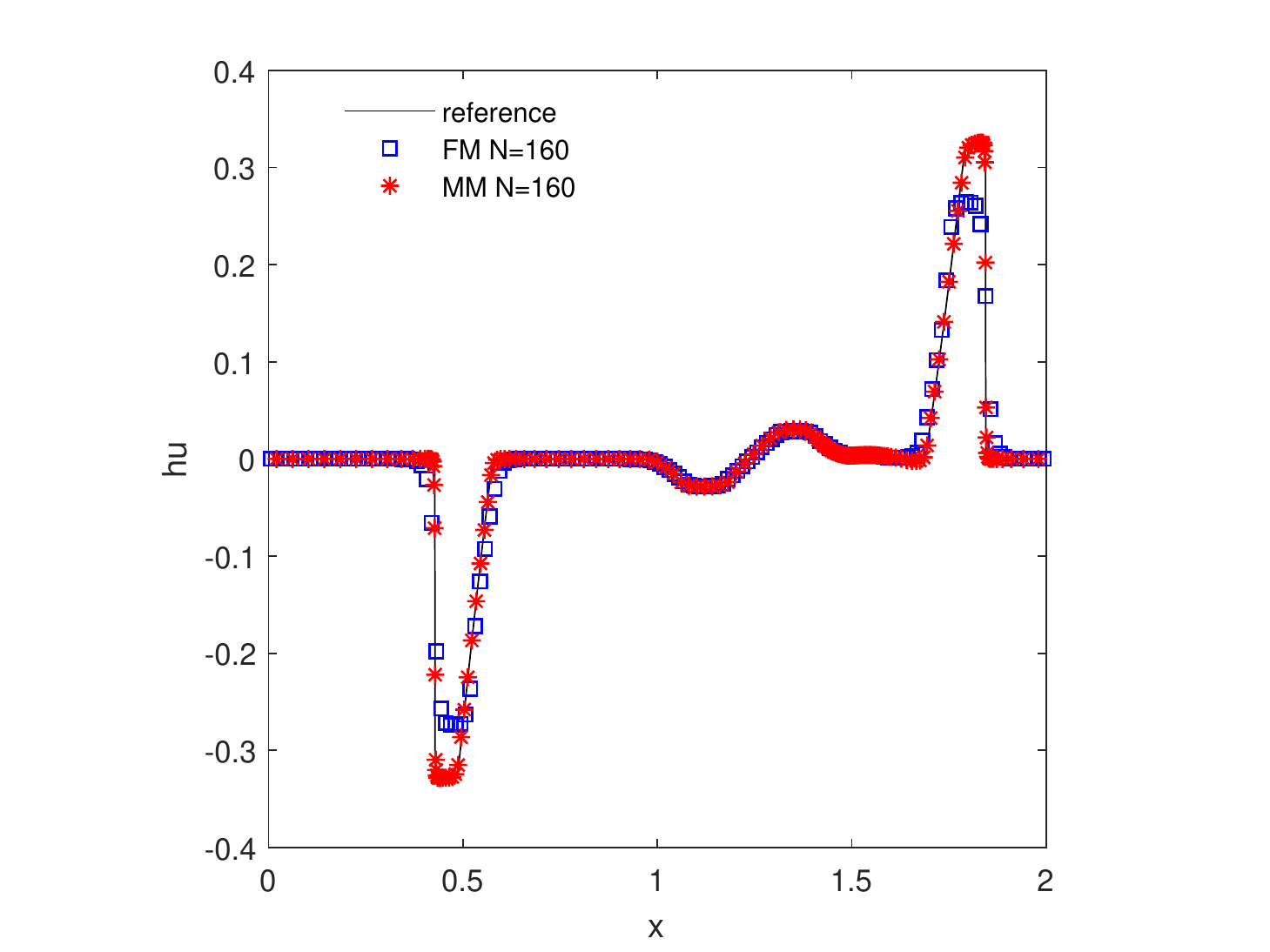}}
\caption{Example \ref{test3-1d}. The mesh trajectories and the water surface $B+h$ and discharge $hu$ at $t=0.2$ obtained with $P^2$-DG of a moving mesh of $N=160$ and a fixed mesh of $N=160$ for a small pulse $\varepsilon=0.2$.
The metric tensor is computed based on the equilibrium variable $\mathcal{E}$ and the water depth $h$ (left column)
or the entropy/total energy (right column).}
\label{Fig:test3-1d-large-metric}
\end{figure}

\begin{figure}[H]
\centering
\subfigure[$\mathcal{E}$ and $h$: Mesh trajectories]{
\includegraphics[width=0.4\textwidth,trim=20 0 40 10,clip]{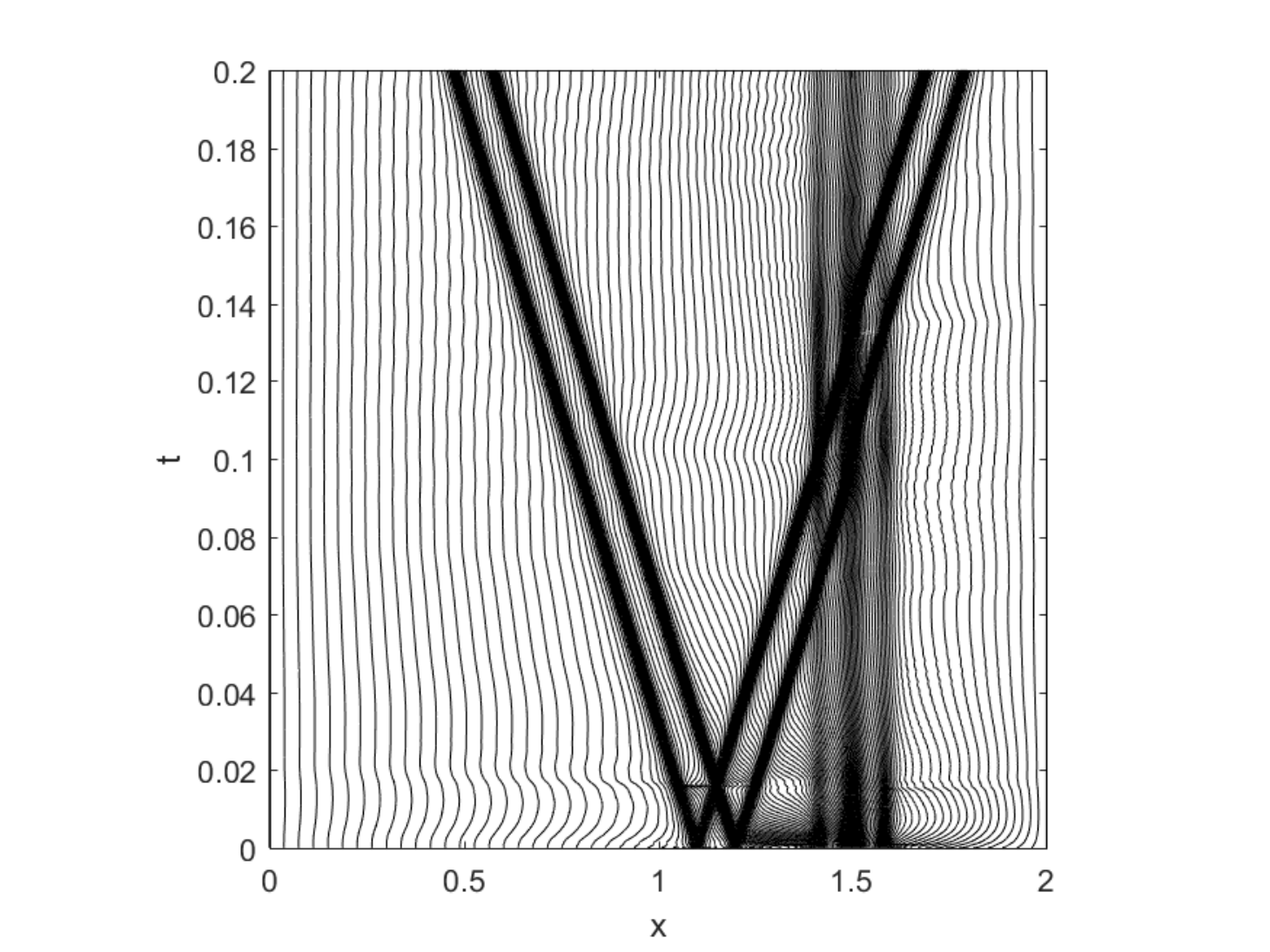}}
\subfigure[Entropy: Mesh trajectories]{
\includegraphics[width=0.4\textwidth,trim=20 0 40 10,clip]{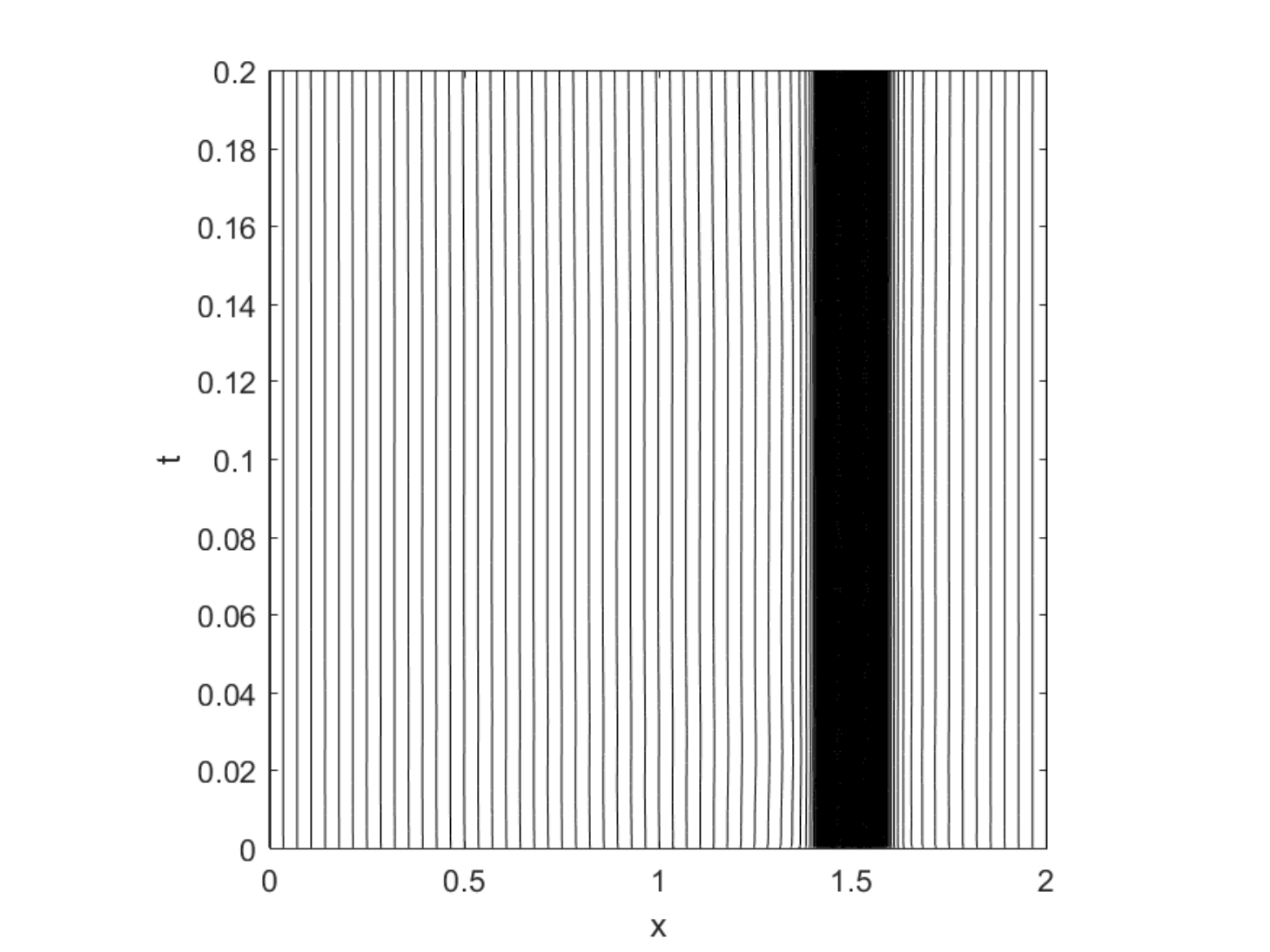}}
\subfigure[$\mathcal{E}$ and $h$: $B+h$]{
\includegraphics[width=0.4\textwidth,trim=20 0 40 10,clip]{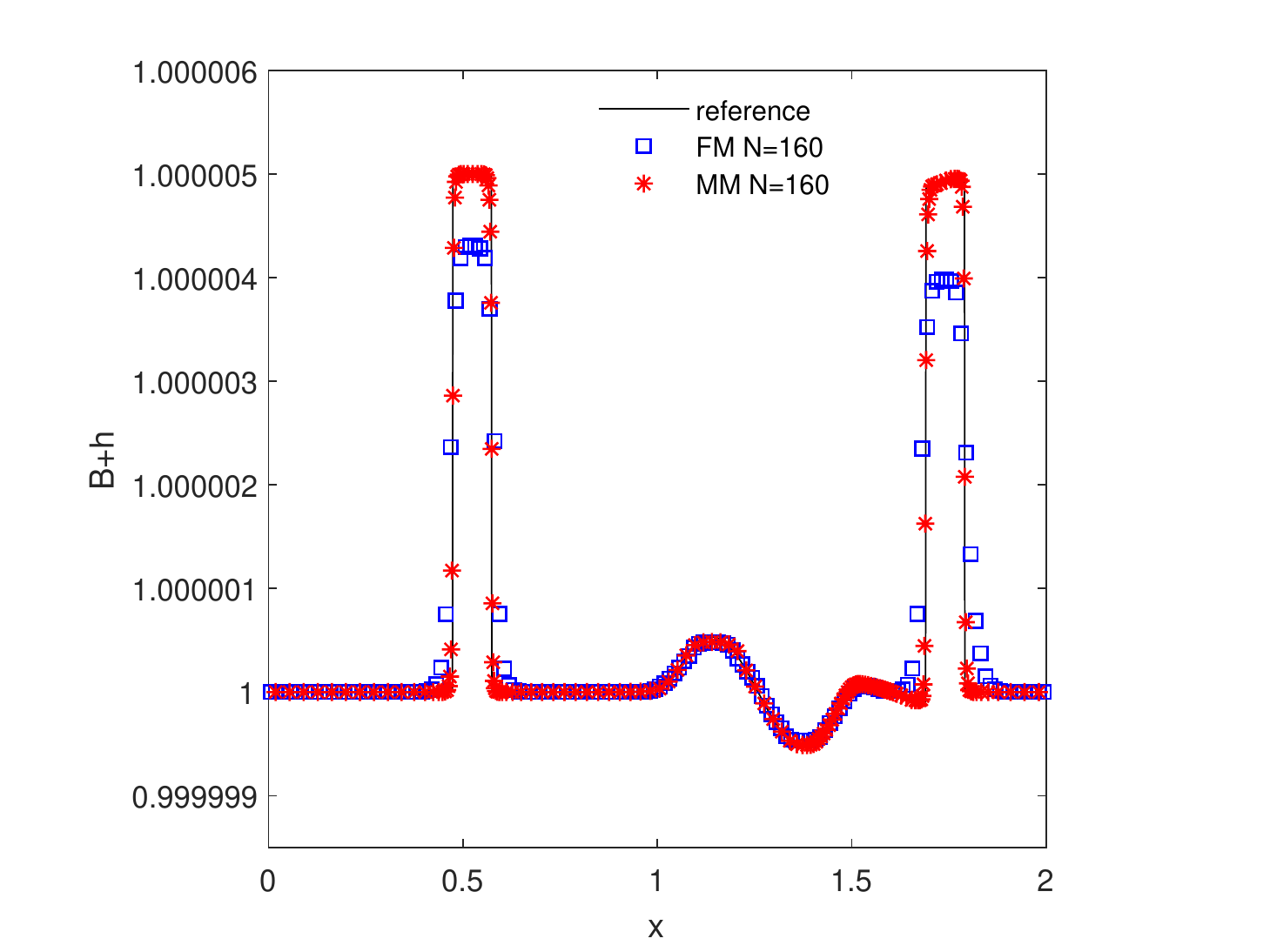}}
\subfigure[Entropy: $B+h$]{
\includegraphics[width=0.4\textwidth,trim=20 0 40 10,clip]{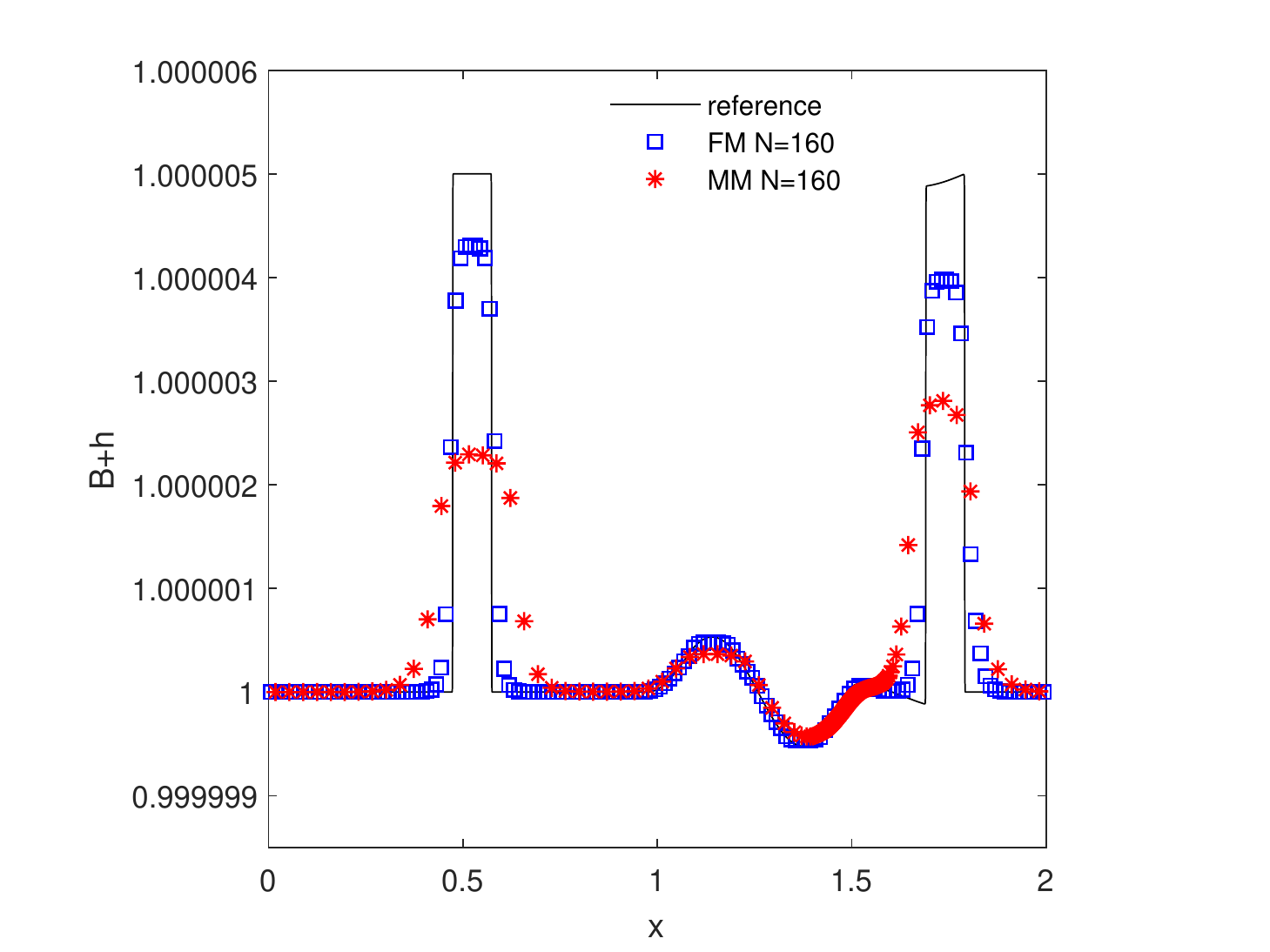}}
\subfigure[$\mathcal{E}$ and $h$: $hu$]{
\includegraphics[width=0.4\textwidth,trim=20 0 40 10,clip]{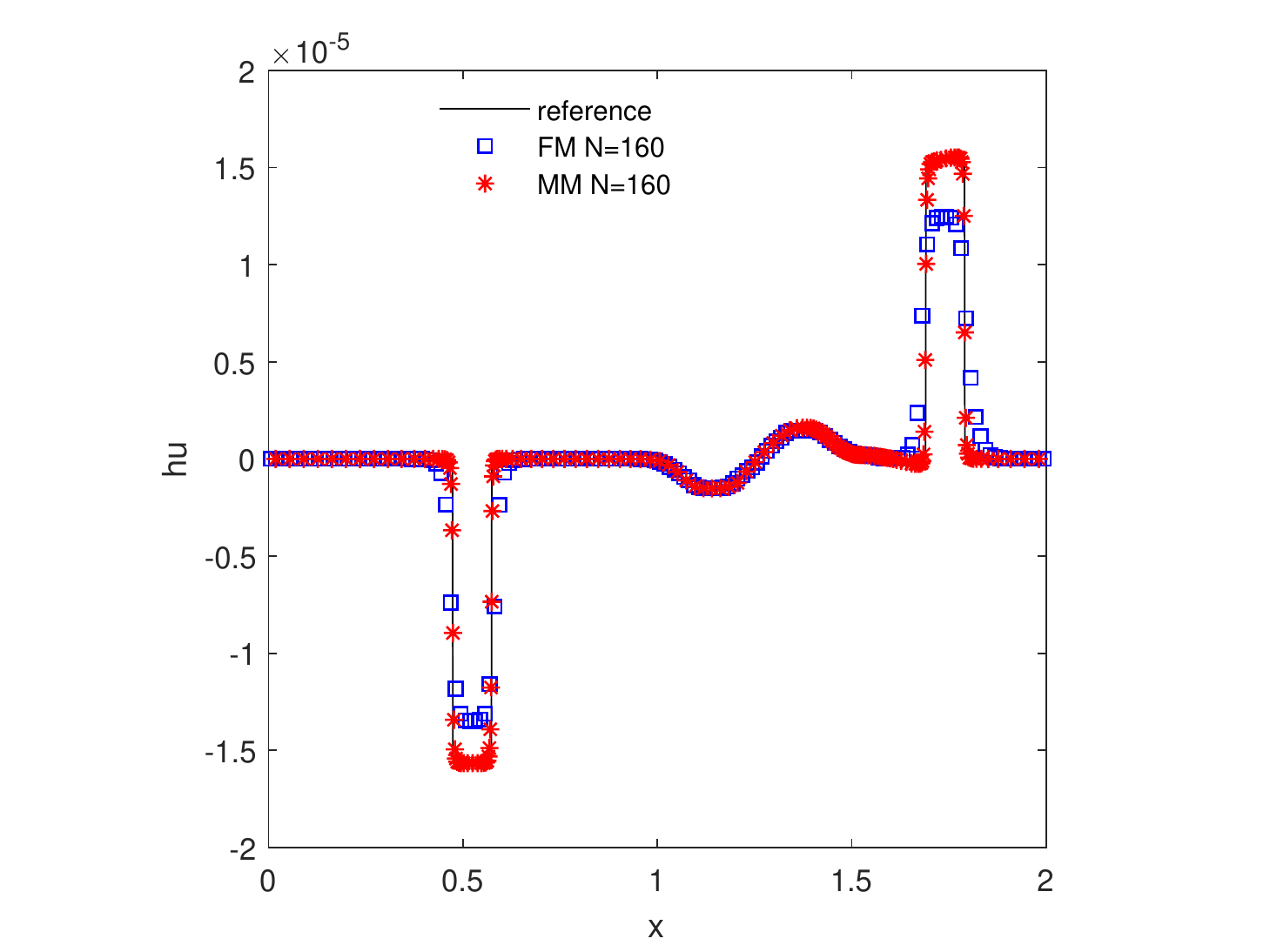}}
\subfigure[Entropy: $hu$]{
\includegraphics[width=0.4\textwidth,trim=20 0 40 10,clip]{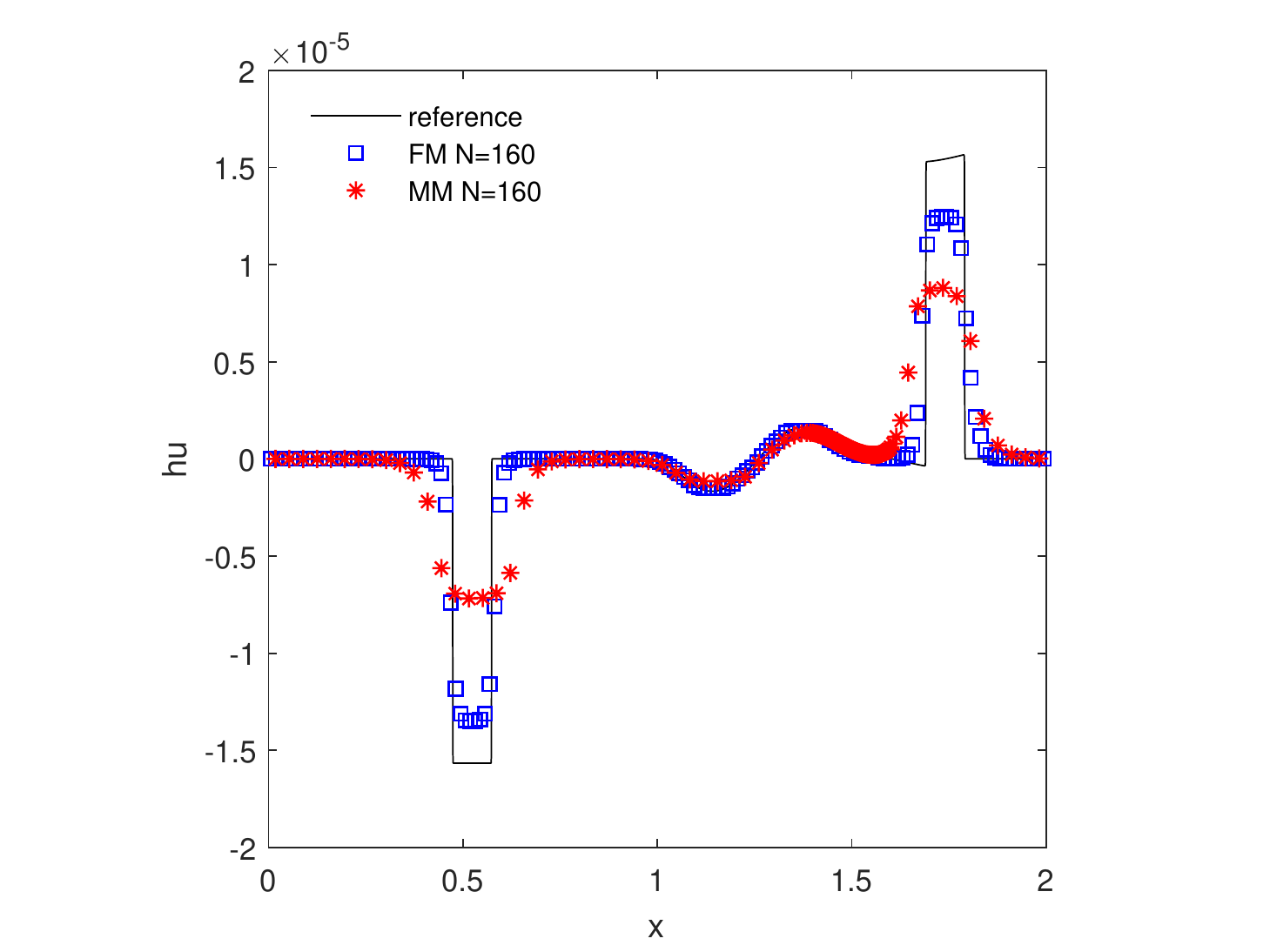}}
\caption{Example \ref{test3-1d}. The mesh trajectories and the water surface $B+h$ and discharge $hu$ at $t=0.2$ obtained with $P^2$-DG of a moving mesh of $N=160$ and a fixed mesh of $N=160$ for a small pulse $\varepsilon=10^{-5}$.
The metric tensor is computed based on the equilibrium variable $\mathcal{E}$ and the water depth $h$ (left column) or the entropy/total energy (right column).}
\label{Fig:test3-1d-small-metric}
\end{figure}

\begin{figure}[H]
\centering
\subfigure[$B+h$: FM 160 vs MM 160]{
\includegraphics[width=0.4\textwidth,trim=20 0 40 10,clip]{R_test3_large_Bph_1d_P2F160M160-eps-converted-to.pdf}}
\subfigure[Close view of (a)]{
\includegraphics[width=0.4\textwidth,trim=20 0 39 10,clip]{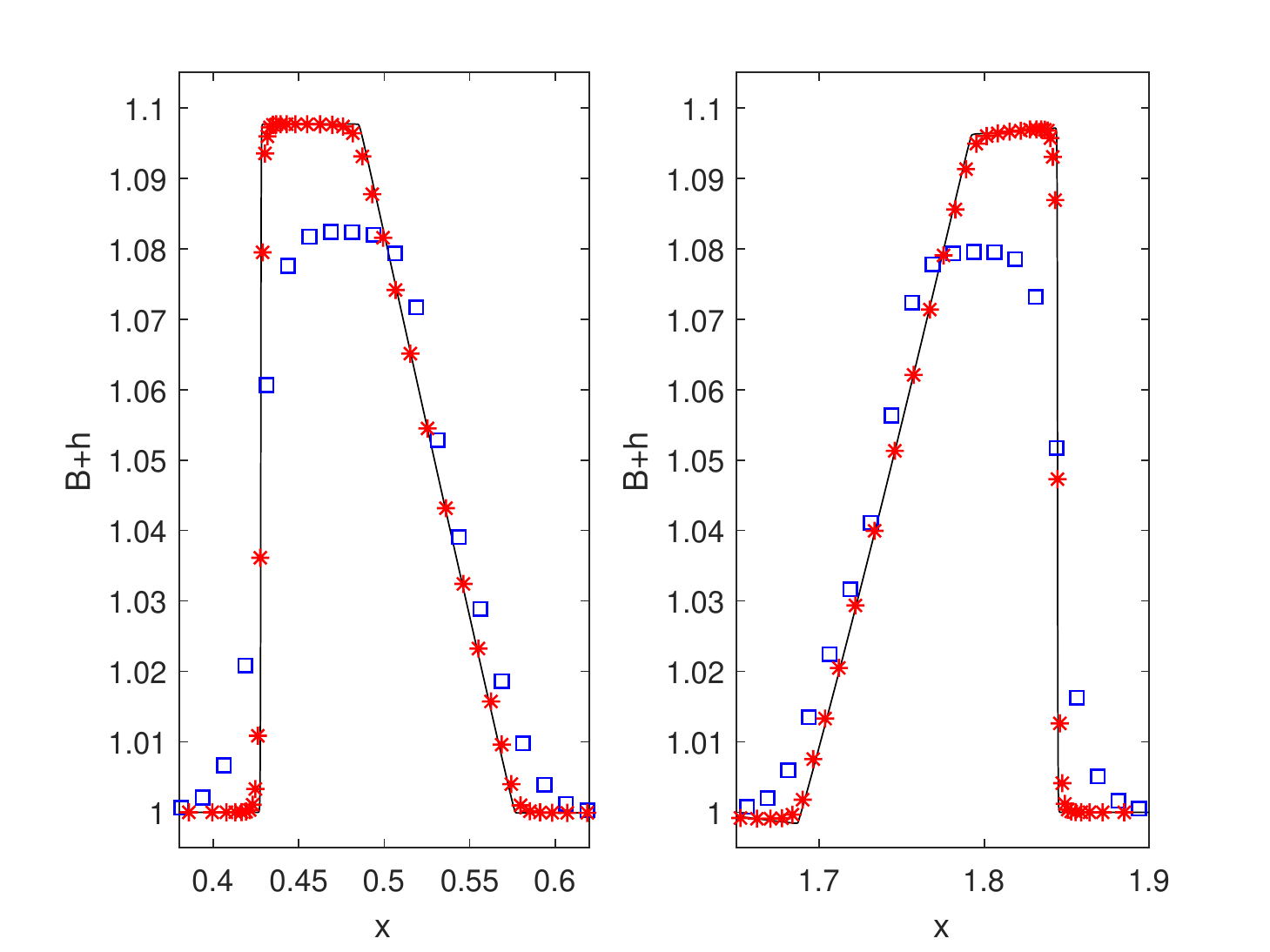}}
\subfigure[$B+h$: FM 480 vs MM 160]{
\includegraphics[width=0.4\textwidth,trim=20 0 40 10,clip]{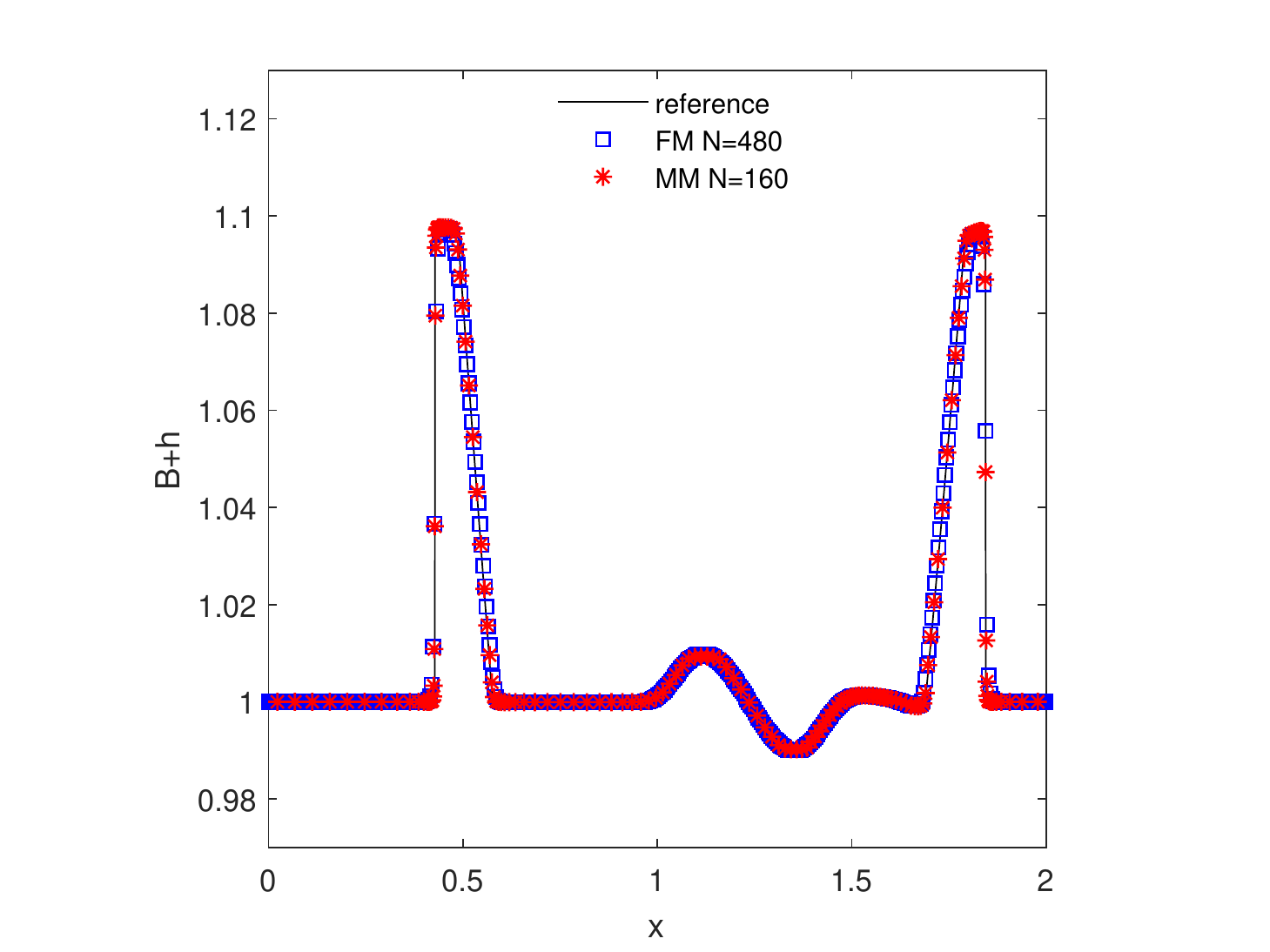}}
\subfigure[Close view of (c)]{
\includegraphics[width=0.4\textwidth,trim=20 0 39 10,clip]{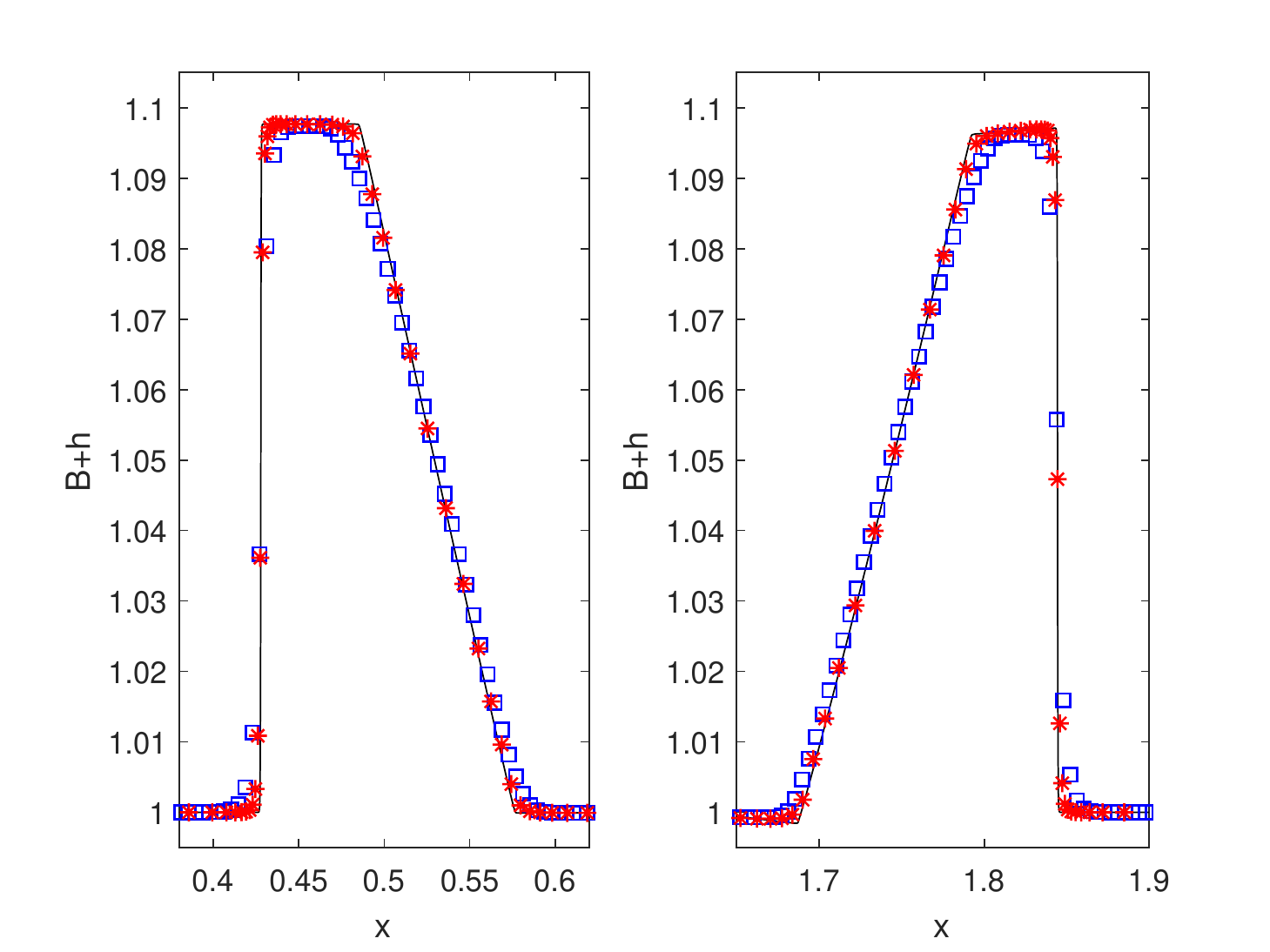}}
\caption{Example \ref{test3-1d}. The water surface $B+h$ at $t=0.2$ obtained with $P^2$-DG and a moving mesh of $N=160$ are compared with those obtained with a fixed mesh of $N=160$ and $N=480$ for a large pulse $\varepsilon=0.2$.}
\label{Fig:test3-1d-large-Bph}
\end{figure}

\begin{figure}[H]
\centering
\subfigure[$hu$: FM 160 vs MM 160]{
\includegraphics[width=0.4\textwidth,trim=20 0 40 10,clip]{R_test3_large_hu_1d_P2F160M160-eps-converted-to.pdf}}
\subfigure[Close view of (a)]{
\includegraphics[width=0.4\textwidth,trim=20 0 39 10,clip]{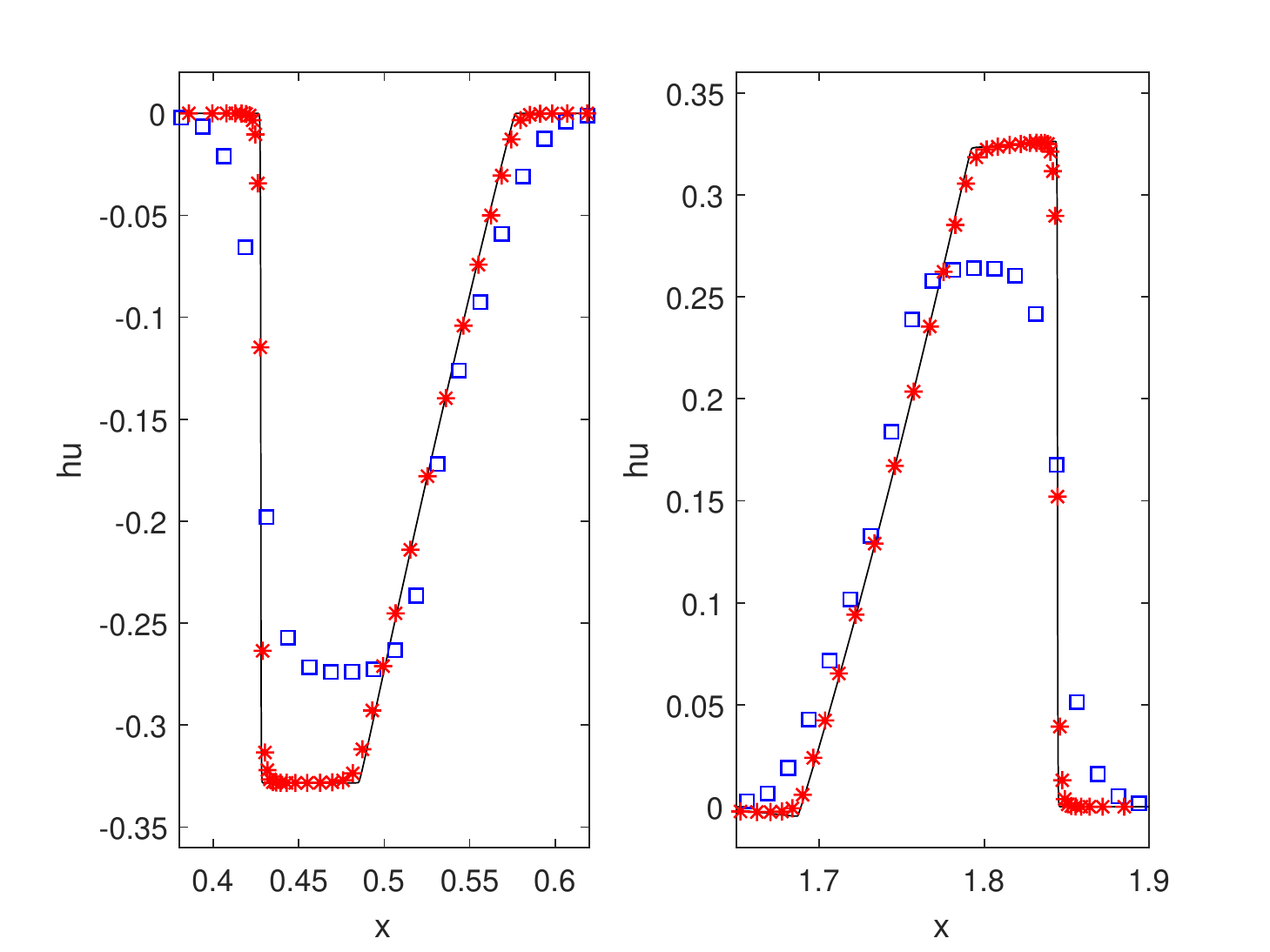}}
\subfigure[$hu$: FM 480 vs MM 160]{
\includegraphics[width=0.4\textwidth,trim=20 0 40 10,clip]{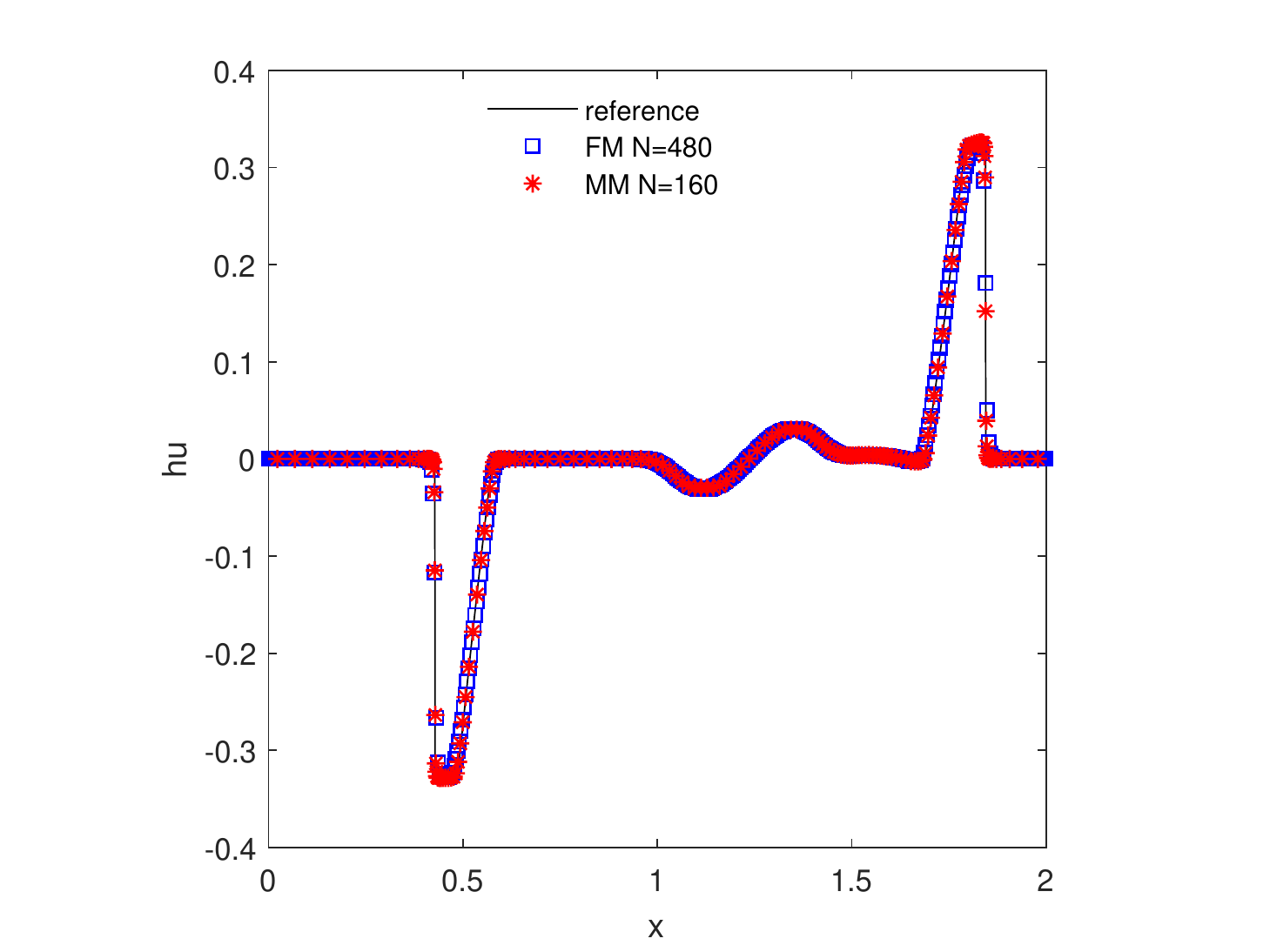}}
\subfigure[Close view of (c)]{
\includegraphics[width=0.4\textwidth,trim=20 0 39 10,clip]{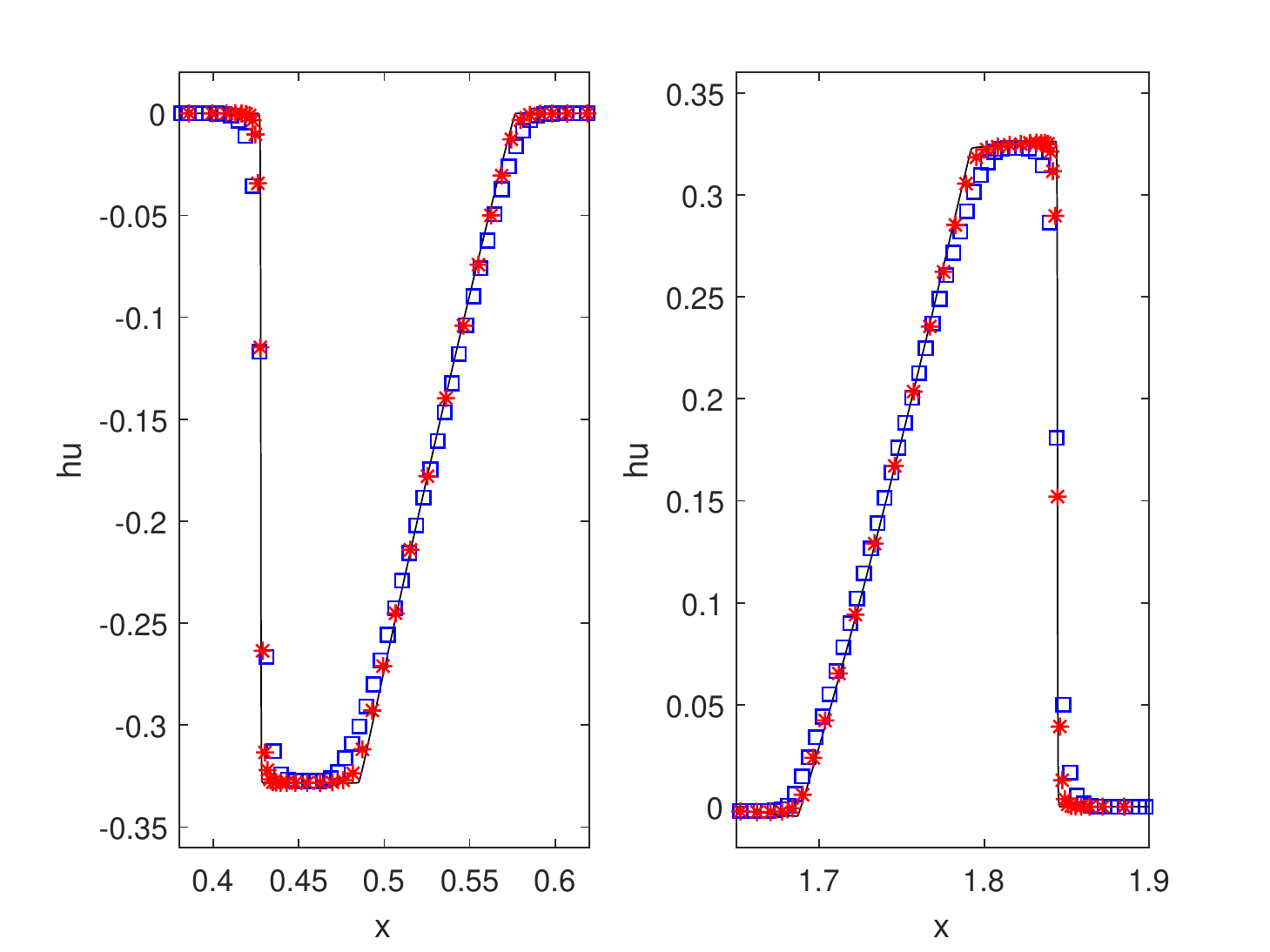}}
\caption{Example \ref{test3-1d}. The water discharge $hu$ at $t=0.2$ obtained with $P^2$-DG and a moving mesh of $N=160$ are compared with those obtained with a fixed mesh of $N=160$ and $N=480$ for a large pulse $\varepsilon=0.2$.}
\label{Fig:test3-1d-large-hu}
\end{figure}

\begin{figure}[H]
\centering
\subfigure[$B+h$: FM 160 vs MM 160]{
\includegraphics[width=0.4\textwidth,trim=20 0 40 10,clip]{R_test3_small_Bph_1d_P2F160M160-eps-converted-to.pdf}}
\subfigure[Close view of (a)]{
\includegraphics[width=0.4\textwidth,trim=20 0 39 10,clip]{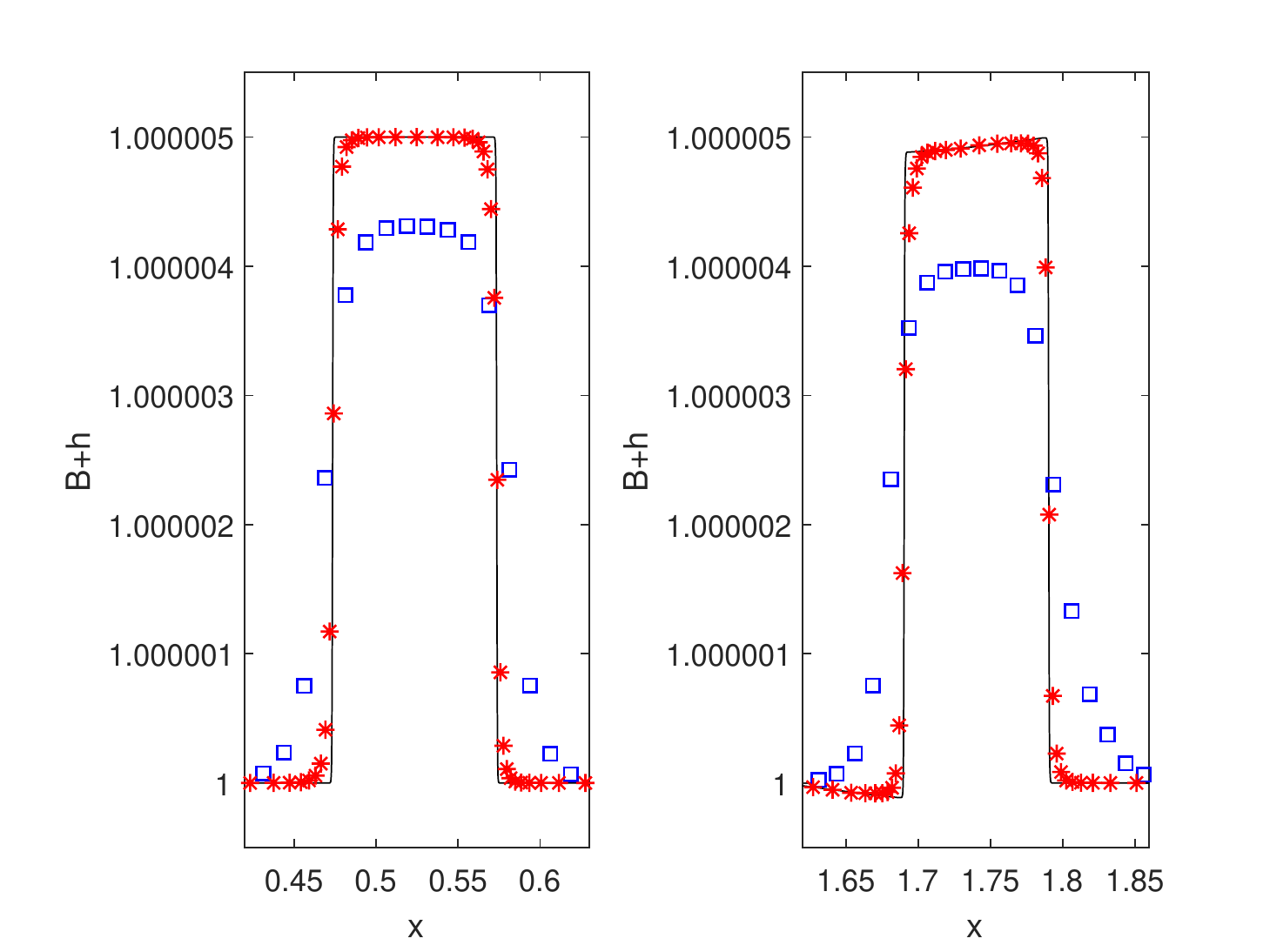}}
\subfigure[$B+h$: FM 480 vs MM 160]{
\includegraphics[width=0.4\textwidth,trim=20 0 40 10,clip]{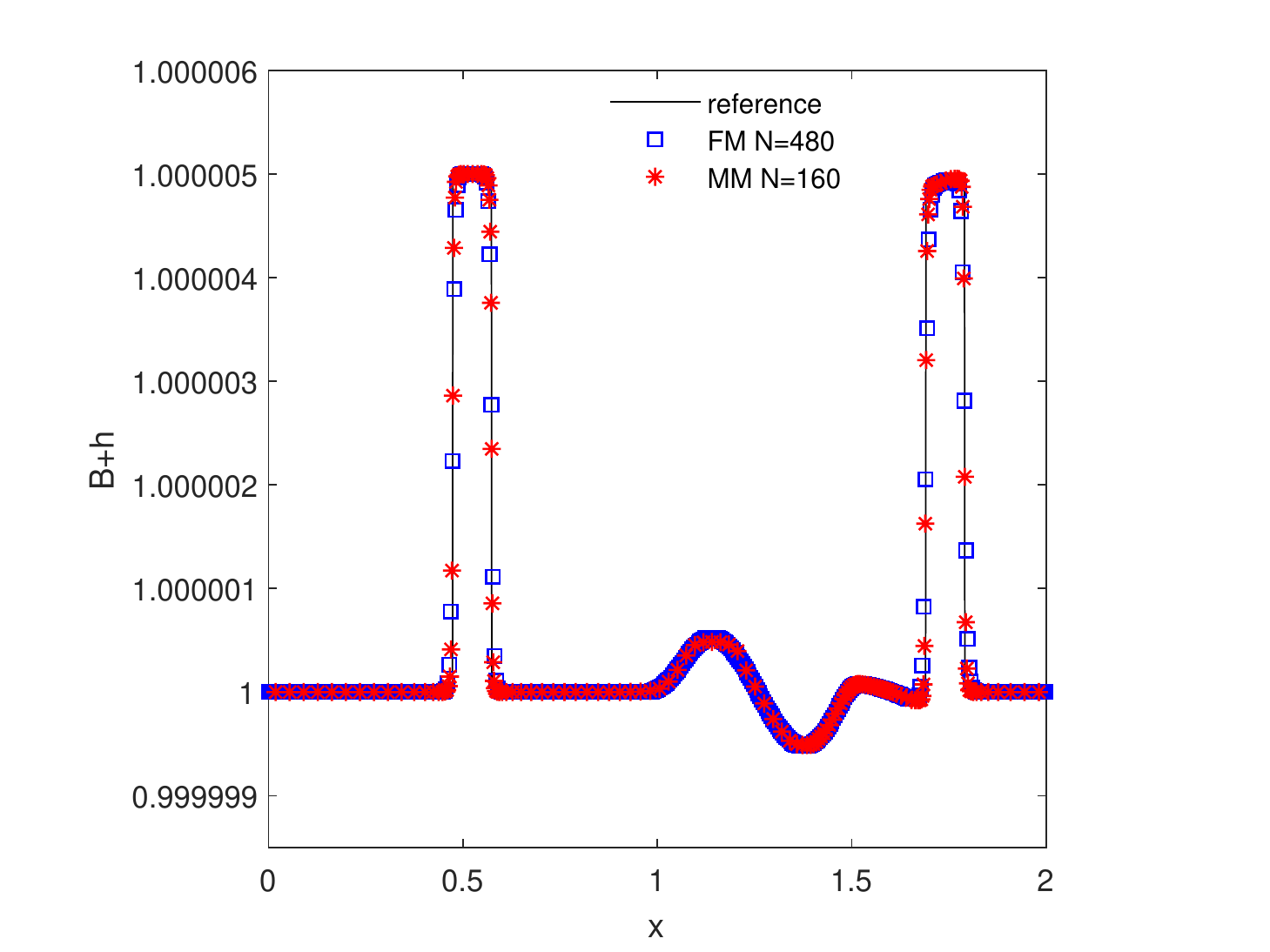}}
\subfigure[Close view of (c)]{
\includegraphics[width=0.4\textwidth,trim=20 0 39 10,clip]{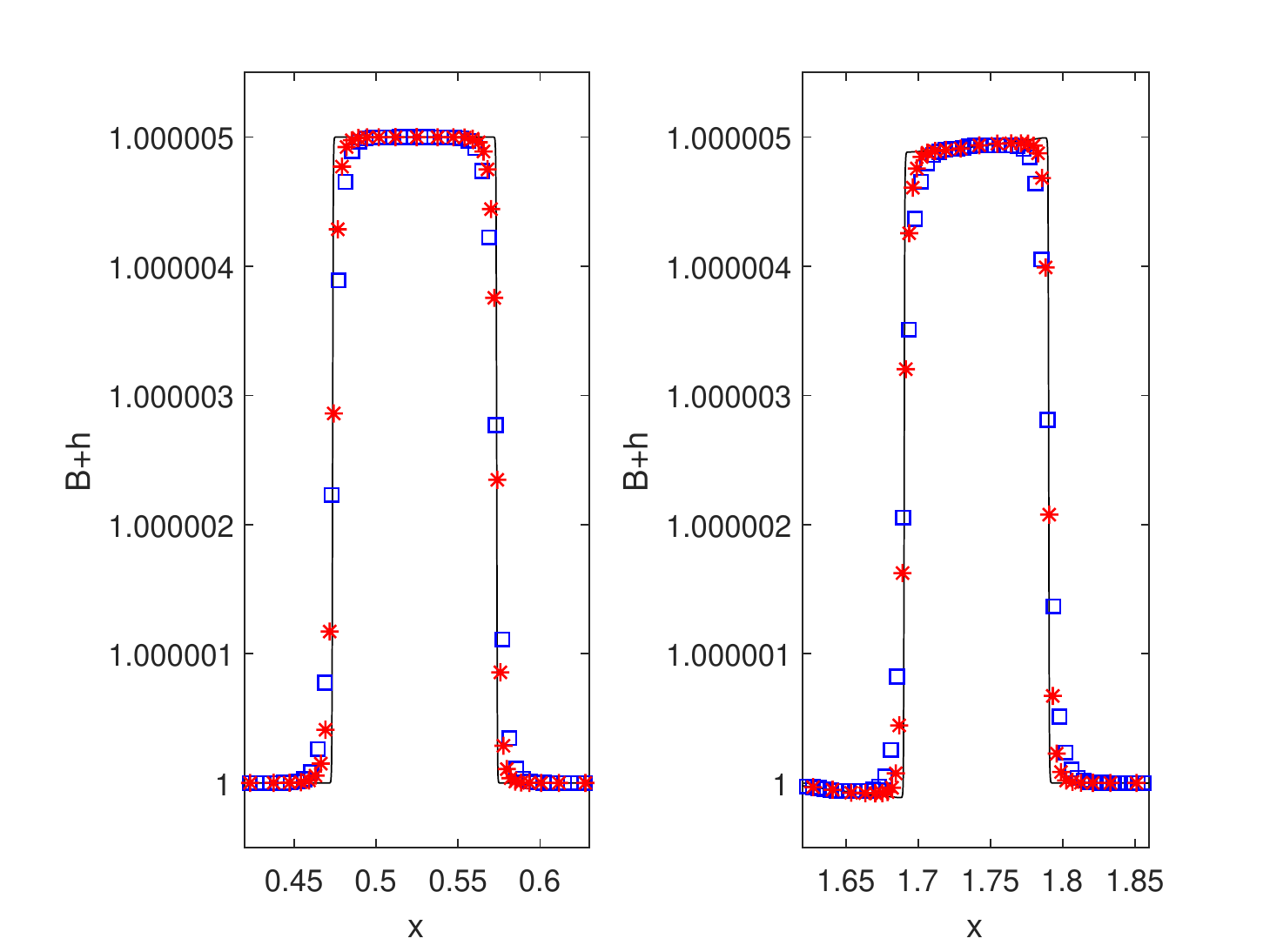}}
\caption{Example \ref{test3-1d}. The water surface $B+h$ at $t=0.2$ obtained with $P^2$-DG and a moving mesh of $N=160$ are compared with those obtained with a fixed mesh of $N=160$ and $N=480$ for a small pulse $\varepsilon=10^{-5}$.}
\label{Fig:test3-1d-small-Bph}
\end{figure}

\begin{figure}[H]
\centering
\subfigure[$hu$: FM 160 vs MM 160]{
\includegraphics[width=0.4\textwidth,trim=20 0 40 10,clip]{R_test3_small_hu_1d_P2F160M160-eps-converted-to.pdf}}
\subfigure[Close view of (a)]{
\includegraphics[width=0.4\textwidth,trim=20 0 39 10,clip]{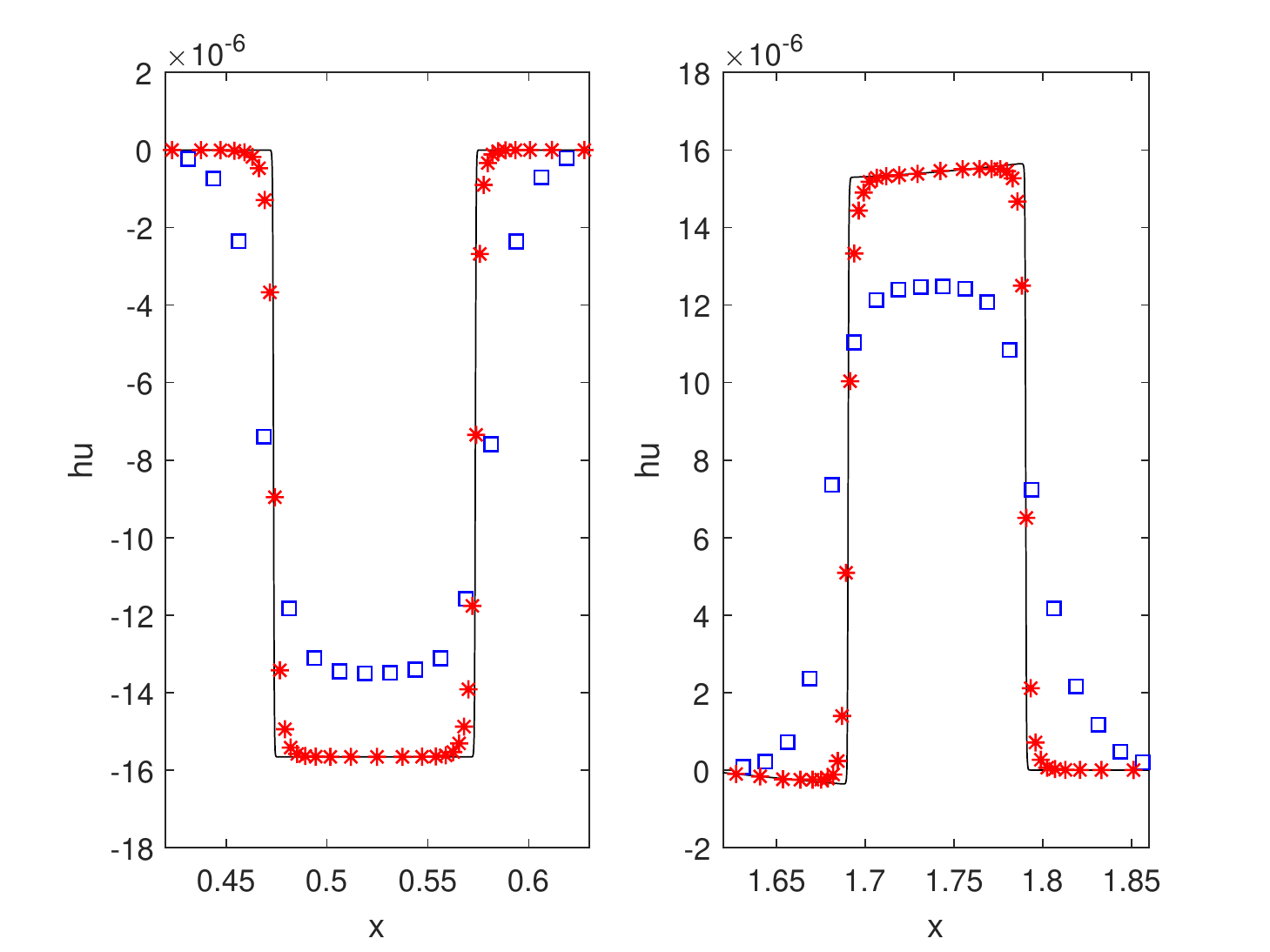}}
\subfigure[$hu$: FM 480 vs MM 160]{
\includegraphics[width=0.4\textwidth,trim=20 0 40 10,clip]{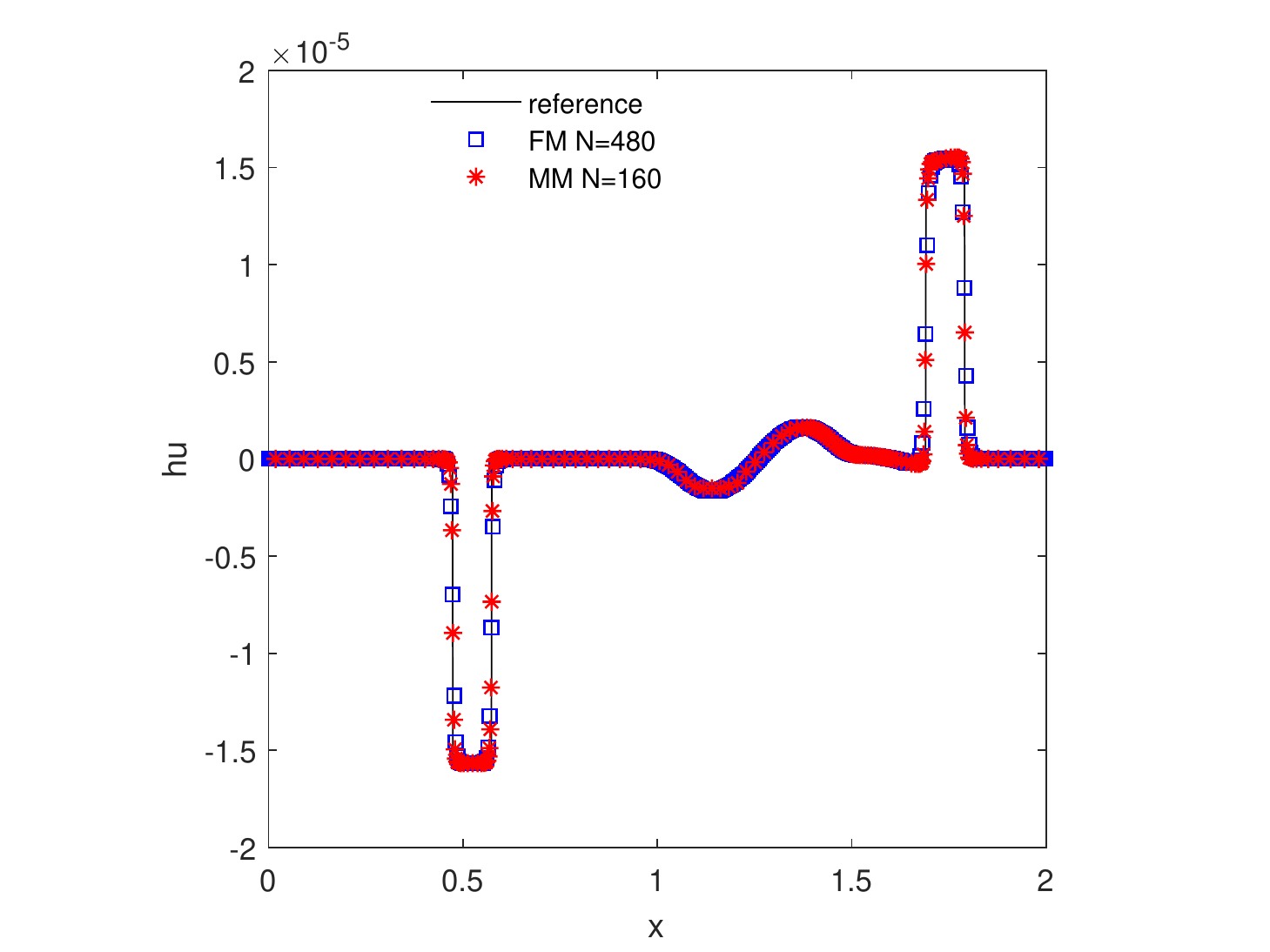}}
\subfigure[Close view of (c)]{
\includegraphics[width=0.4\textwidth,trim=20 0 39 10,clip]{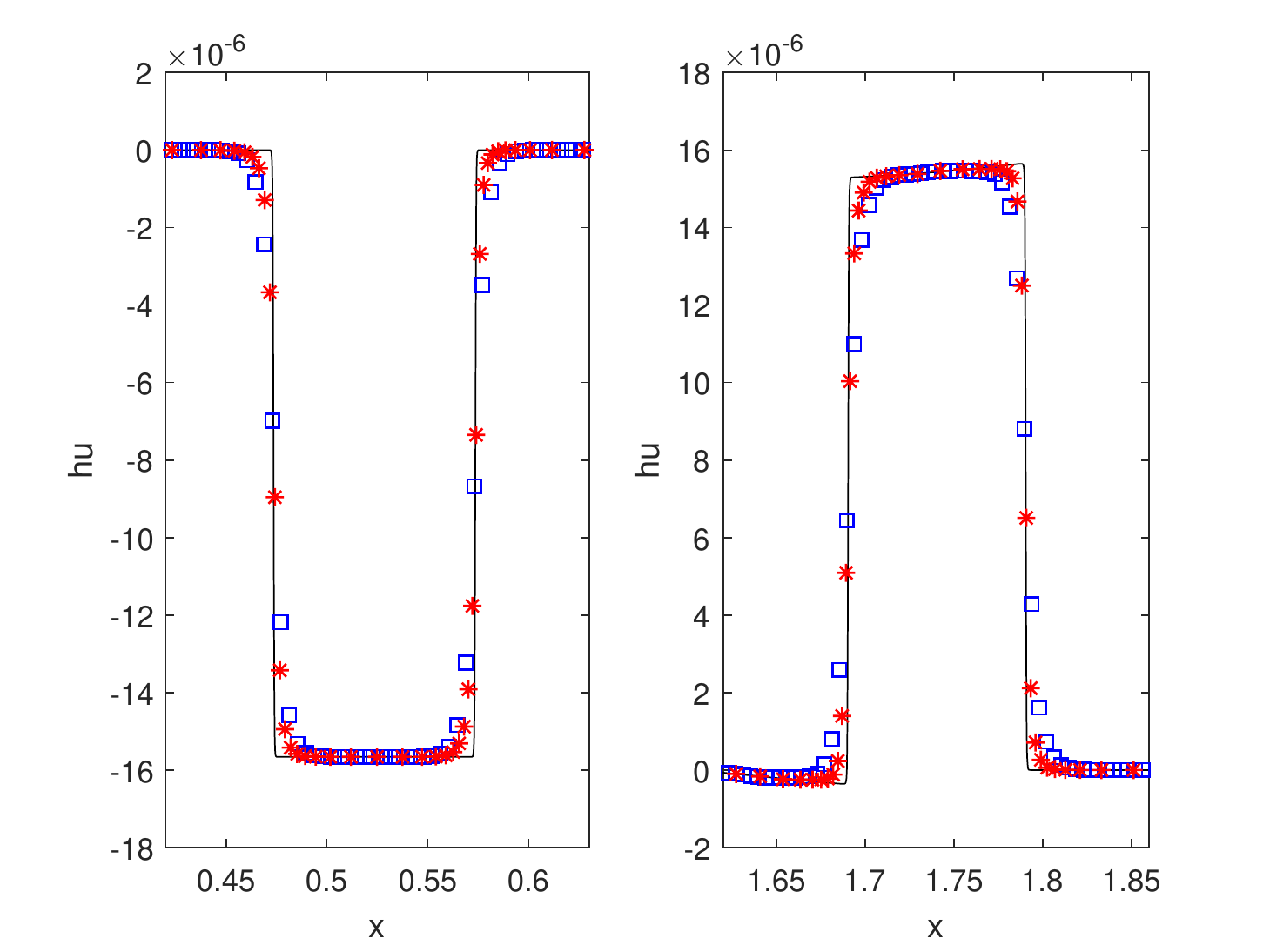}}
\caption{Example \ref{test3-1d}. The water discharge $hu$ at $t=0.2$ obtained with $P^2$-DG and a moving mesh of $N=160$ are compared with those obtained with a fixed mesh of $N=160$ and $N=480$ for a small pulse $\varepsilon=10^{-5}$.}
\label{Fig:test3-1d-small-hu}
\end{figure}
\begin{figure}[H]
\centering
\subfigure[$B+h$]{
\includegraphics[width=0.4\textwidth,trim=20 0 40 10,clip]{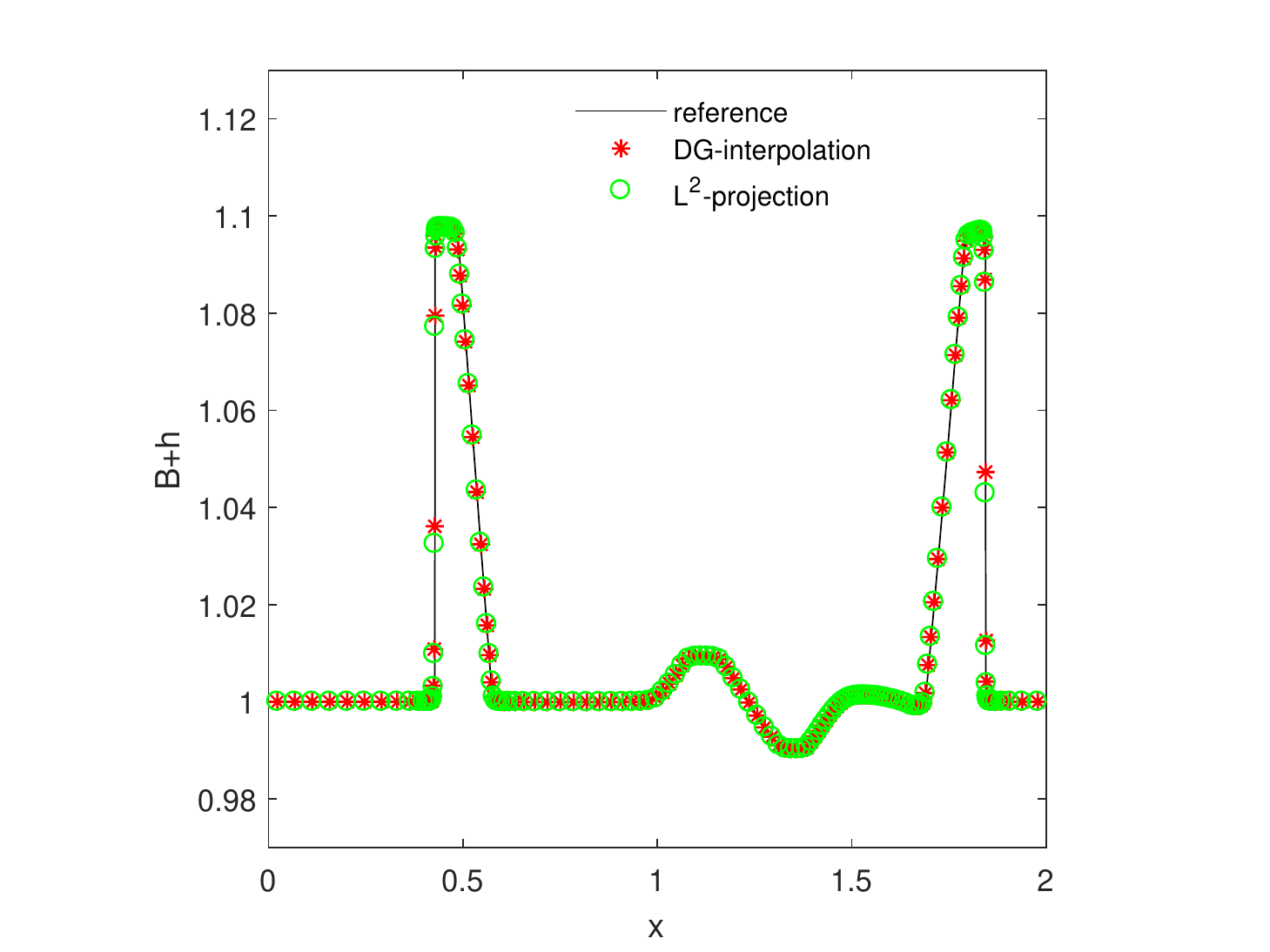}}
\subfigure[$hu$]{
\includegraphics[width=0.4\textwidth,trim=20 0 40 10,clip]{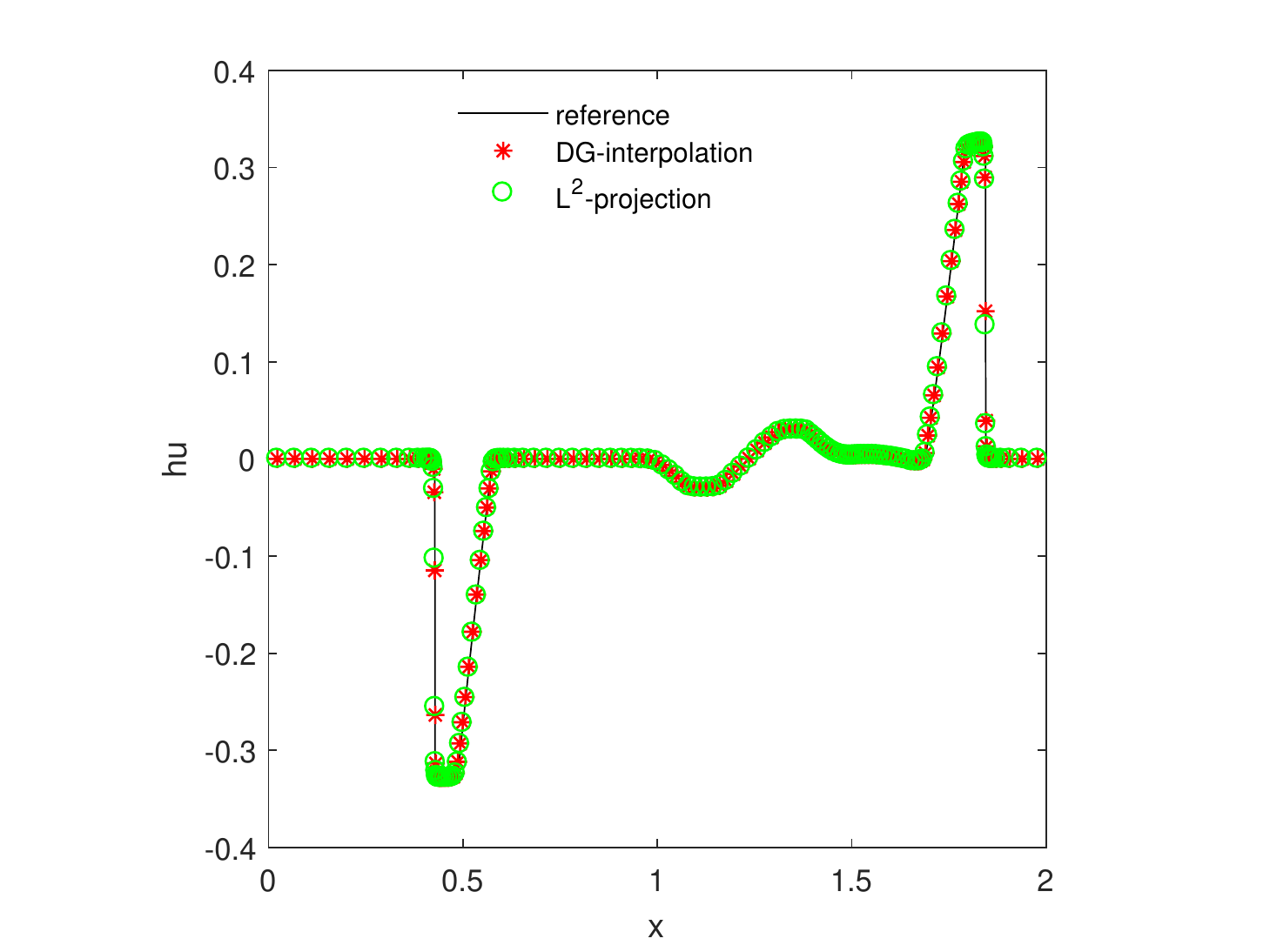}}
\caption{Example \ref{test3-1d}. The water surface $B+h$ and discharge $hu$ at $t=0.2$ are obtained
with $P^2$-DG and a moving mesh of $N=160$ for a large pulse $\varepsilon=0.2$.
Comparison is made for the use of DG-interpolation or $L^2$-projection in updating $B$.}
\label{Fig:test3-1d-Bjob-large}
\end{figure}
\begin{figure}[H]
\centering
\subfigure[$B+h$]{
\includegraphics[width=0.4\textwidth,trim=20 0 40 10,clip]{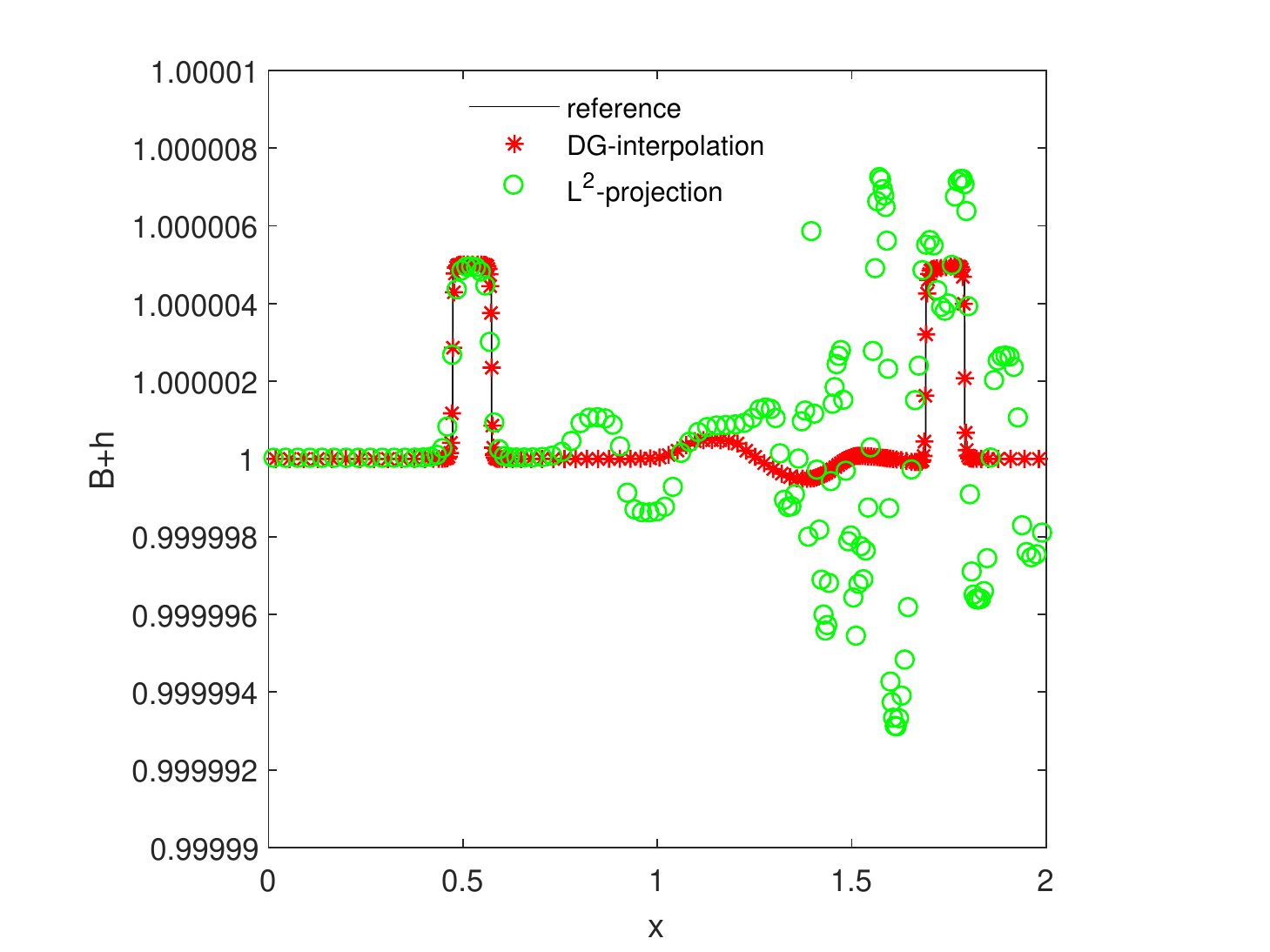}}
\subfigure[$hu$]{
\includegraphics[width=0.4\textwidth,trim=20 0 40 10,clip]{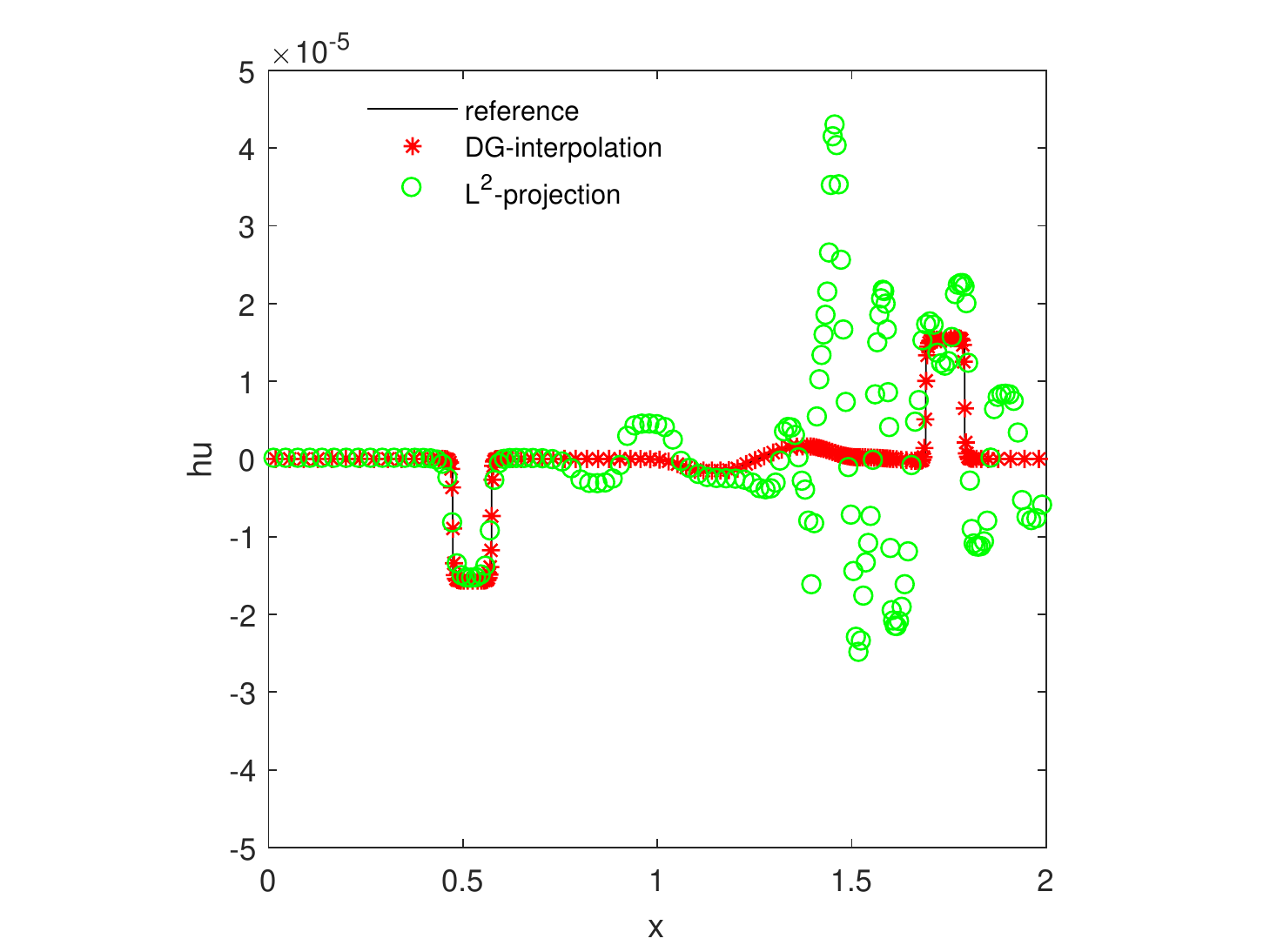}}
\caption{Example \ref{test3-1d}. The water surface $B+h$ and discharge $hu$ at $t=0.2$ are obtained
with $P^2$-DG and a moving mesh of $N=160$ for a large pulse $\varepsilon=10^{-5}$.
Comparison is made for the use of DG-interpolation or $L^2$-projection in updating $B$.}
\label{Fig:test3-1d-Bjob-small}
\end{figure}

\begin{figure}[H]
\centering
\subfigure[Initial surface and bottom]{
\includegraphics[width=0.4\textwidth,trim=20 0 40 10,clip]{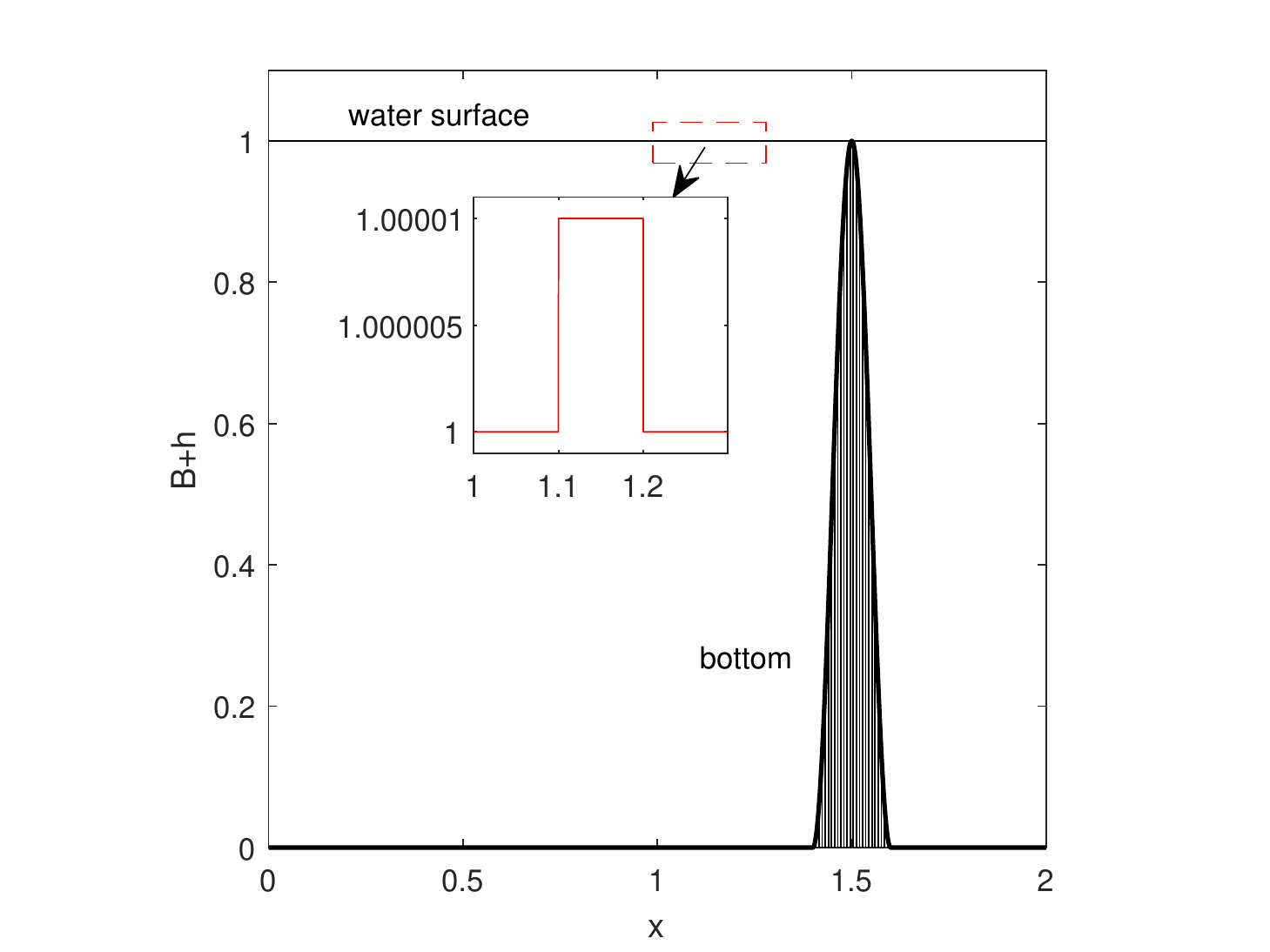}}
\subfigure[Mesh trajectories]{
\includegraphics[width=0.4\textwidth,trim=20 0 40 10,clip]{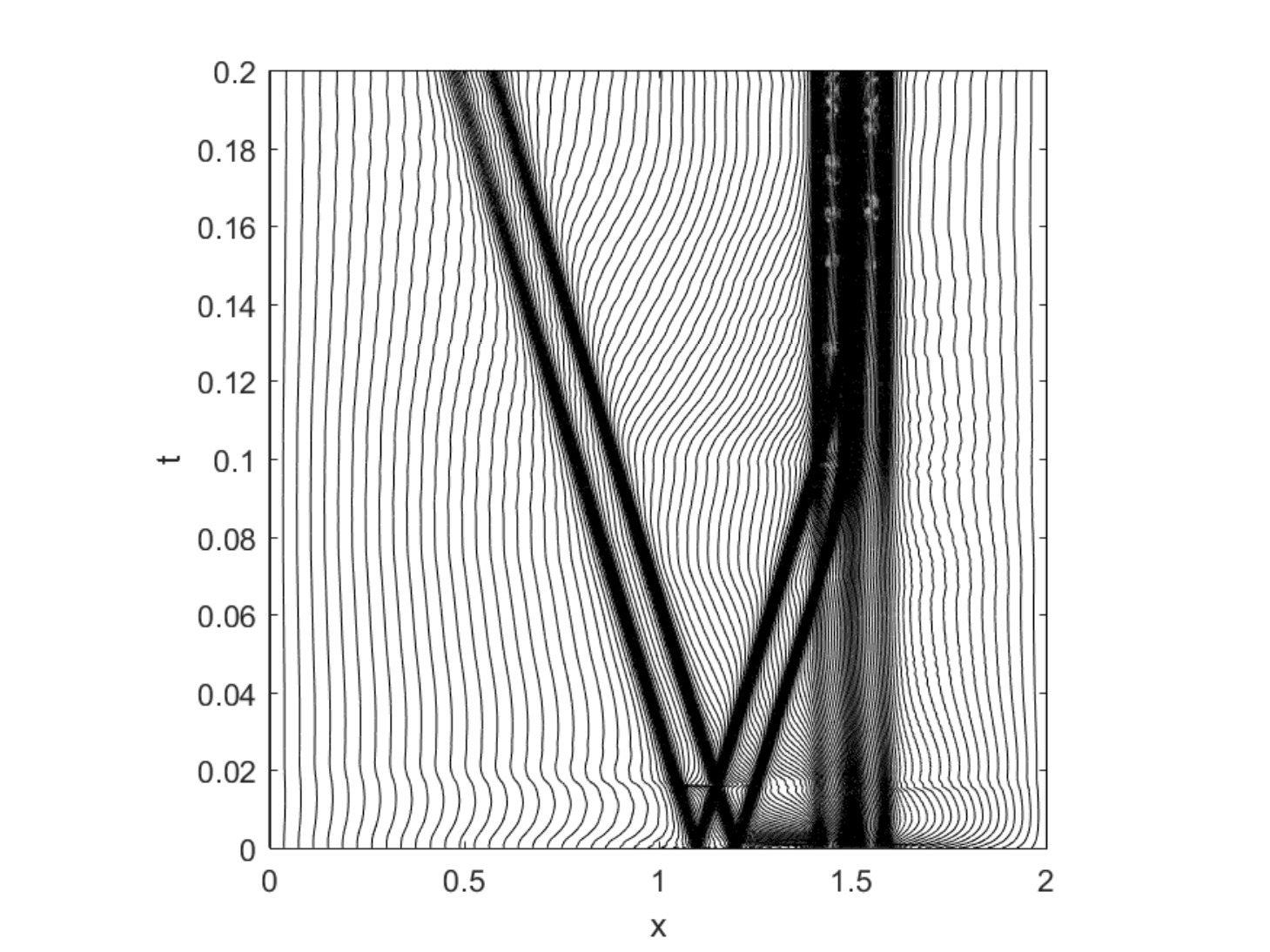}}
\caption{Example \ref{test3-1d}.
(a) The initial water surface $B+h$ and the bottom $B$ (\ref{B-5-1}) for the small perturbation test with a dry region.
(b) The mesh trajectories obtained with $P^2$ MM-DG method of $N=160$.}
\label{Fig:test-1d-wb-pp-preturb-initial}
\end{figure}

\begin{figure}[H]
\centering
\subfigure[$B+h$: FM 160 vs MM 160]{
\includegraphics[width=0.4\textwidth,trim=20 0 40 10,clip]
{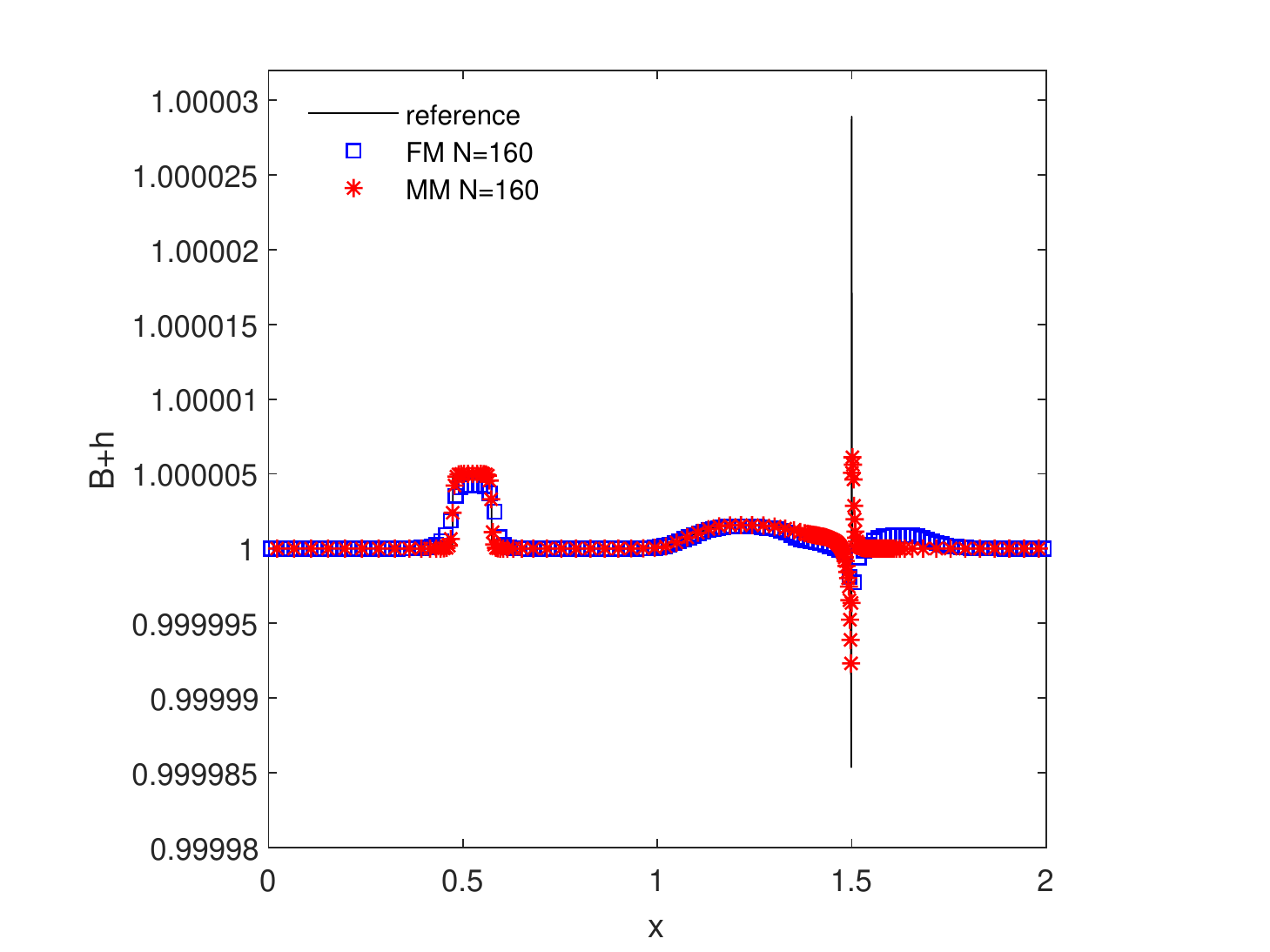}}
\subfigure[Close view of (a)]{
\includegraphics[width=0.4\textwidth,trim=20 0 39 10,clip]{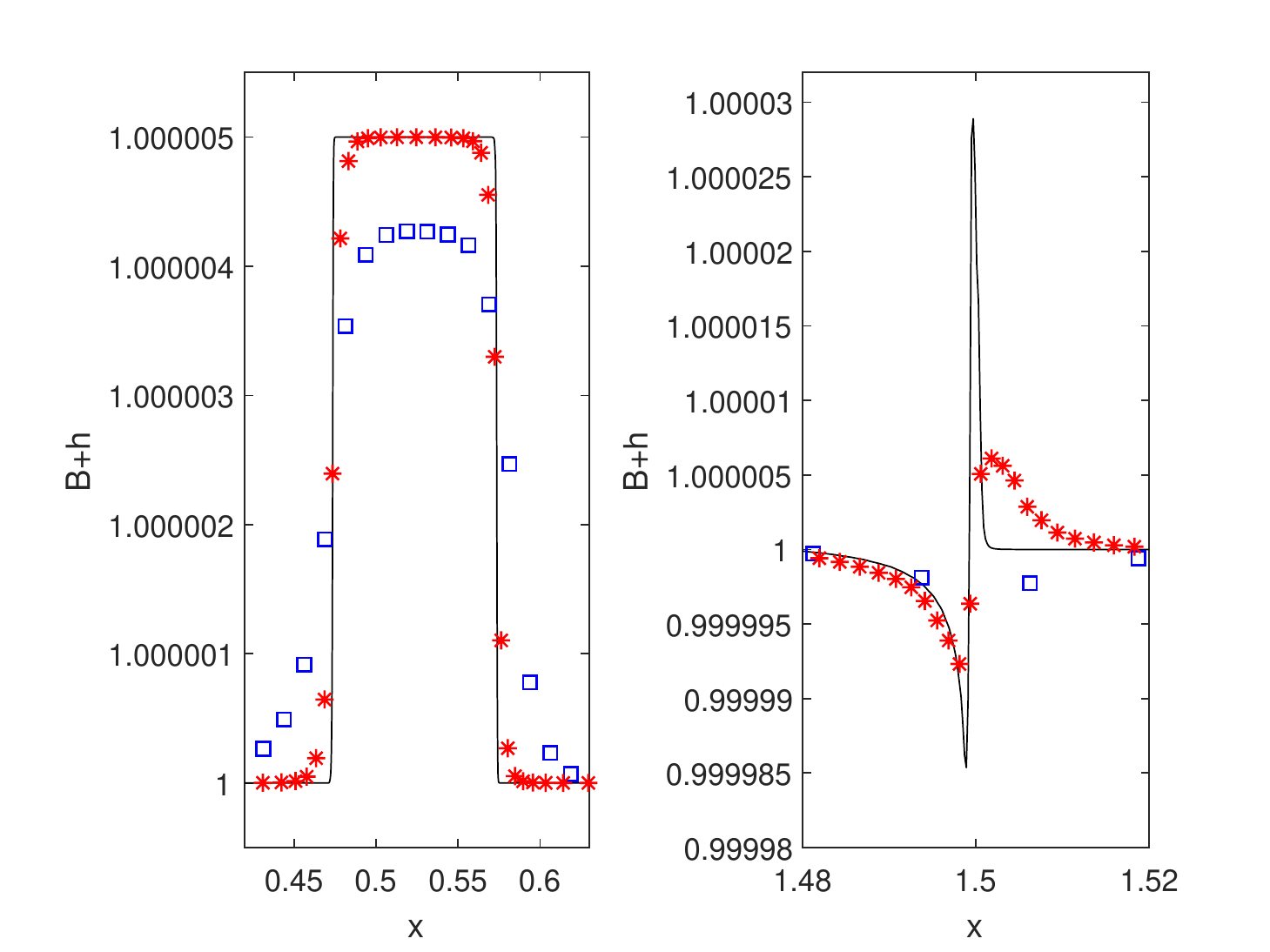}}
\subfigure[$B+h$: FM 640 vs MM 160]{
\includegraphics[width=0.4\textwidth,trim=20 0 40 10,clip]
{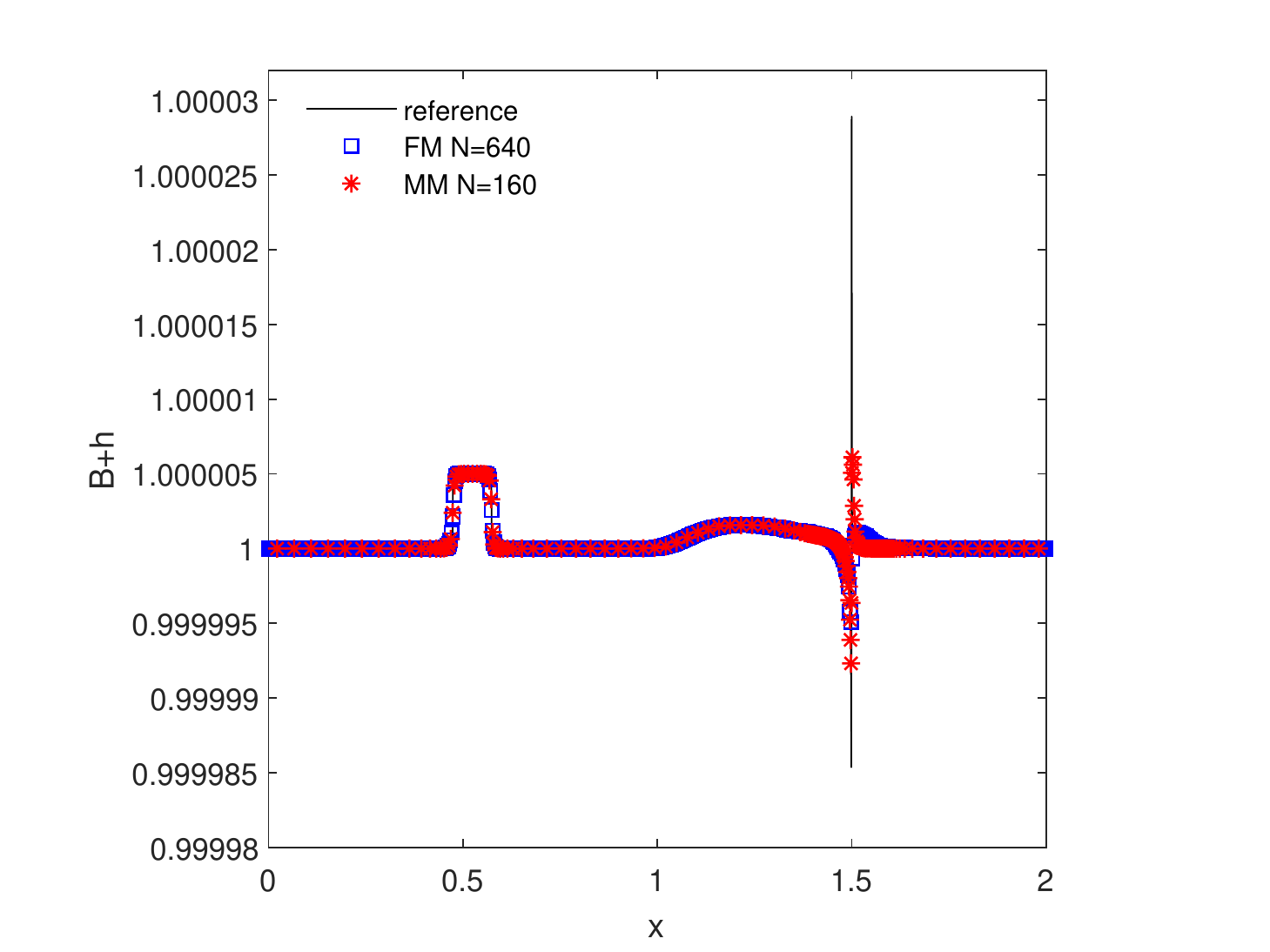}}
\subfigure[Close view of (c)]{
\includegraphics[width=0.4\textwidth,trim=20 0 39 10,clip]{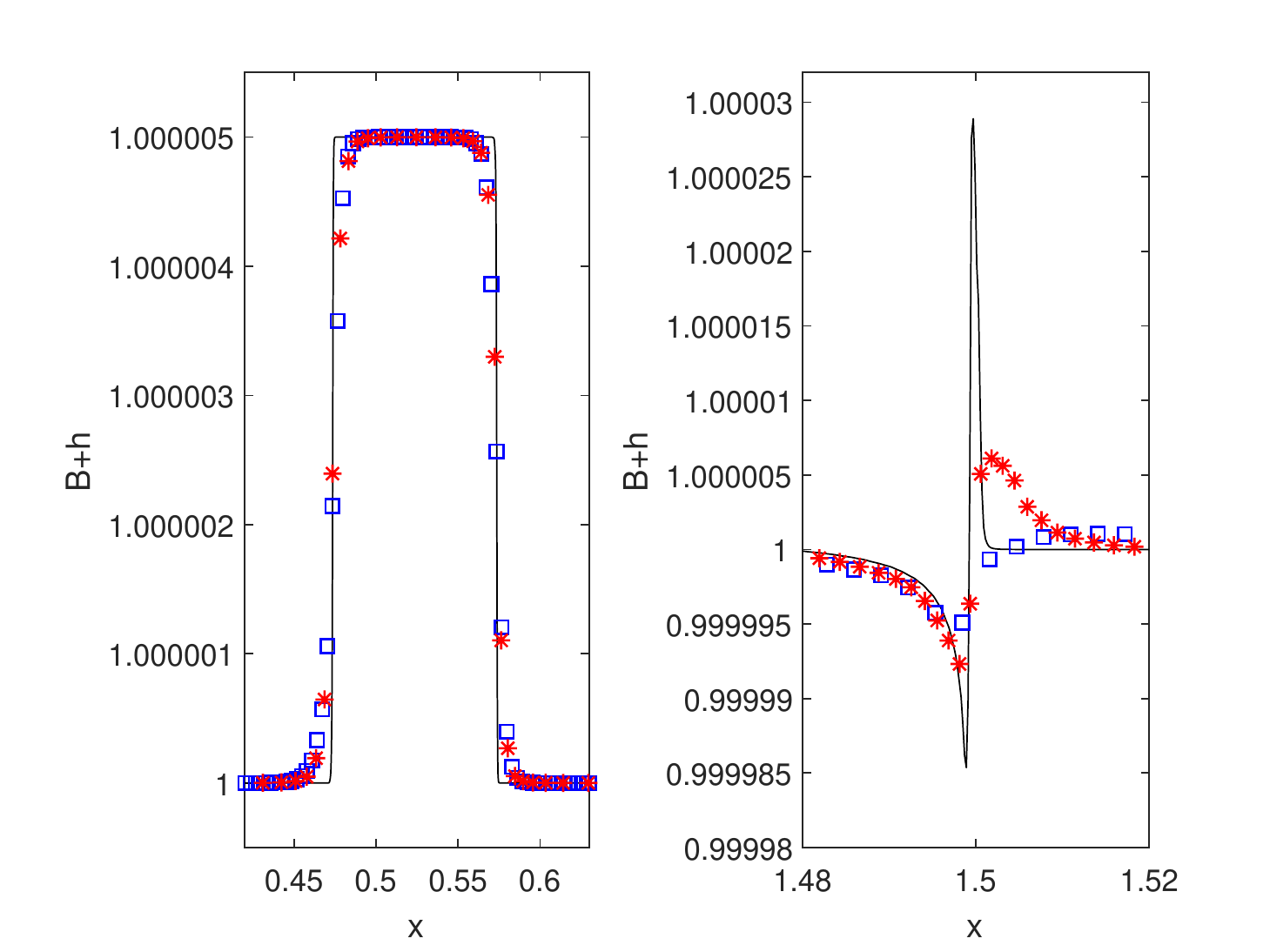}}
\caption{Example \ref{test3-1d} with the bottom topography (\ref{B-5-1}) with a dry region.
The water surface $B+h$ at $t=0.2$ obtained with $P^2$-DG and a moving mesh of $N=160$ are compared with those obtained with a fixed mesh of $N=160$ and $N=640$.}
\label{Fig:test-1d-perturb-Bph}
\end{figure}

\begin{figure}[H]
\centering
\subfigure[$hu$: FM 160 vs MM 160]{
\includegraphics[width=0.4\textwidth,trim=20 0 40 10,clip]
{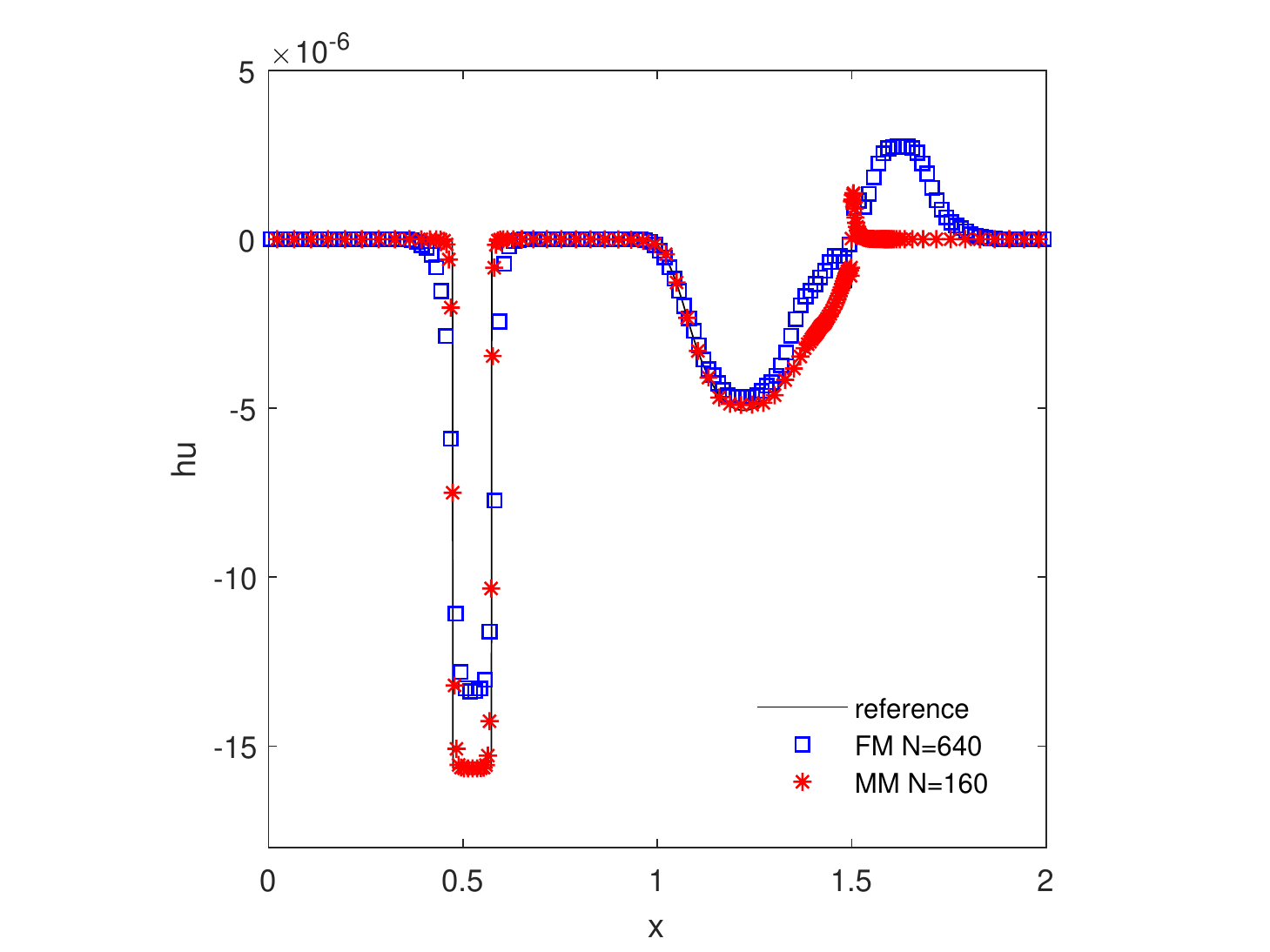}}
\subfigure[Close view of (a)]{
\includegraphics[width=0.4\textwidth,trim=20 0 39 10,clip]{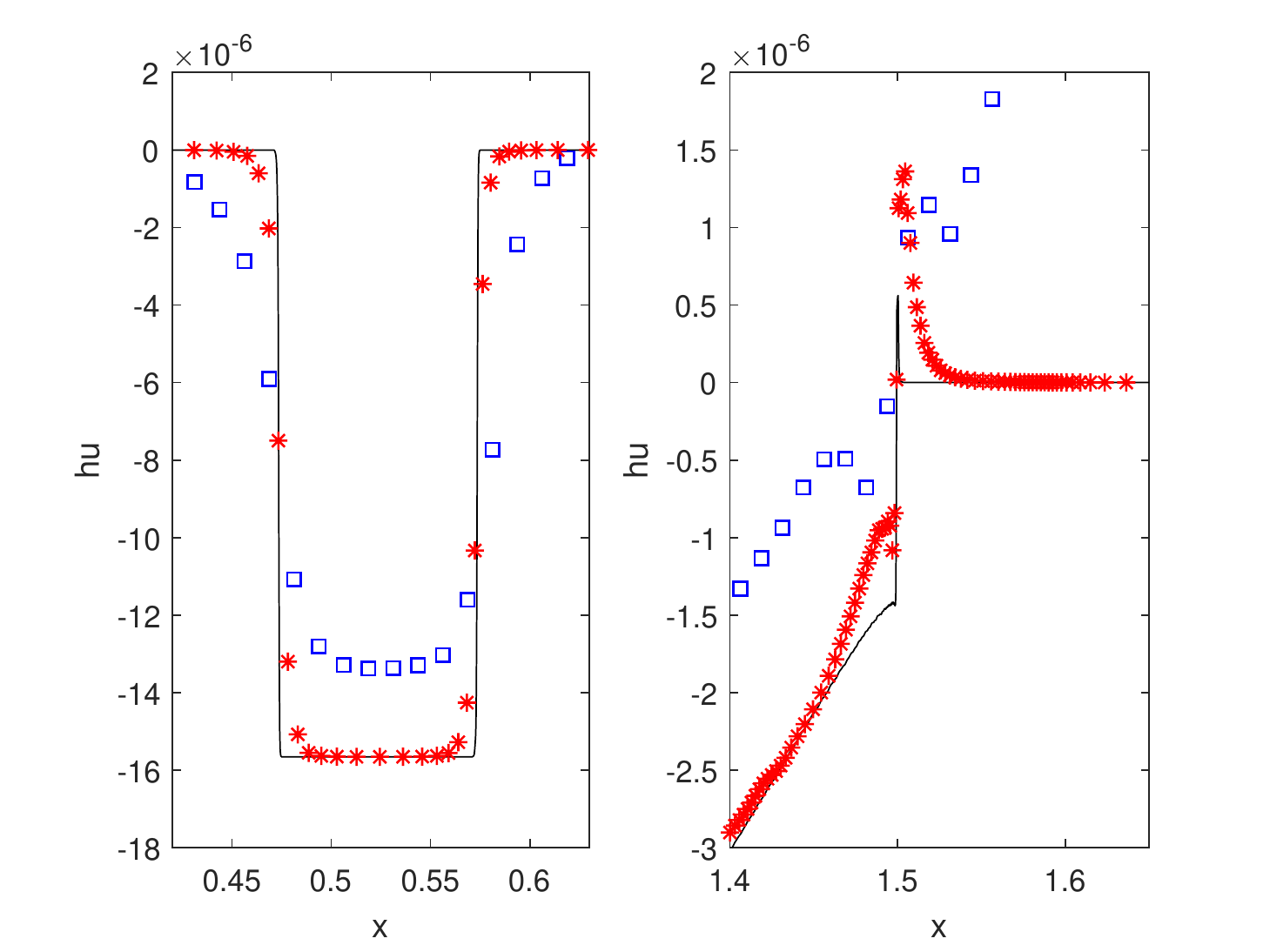}}
\subfigure[$hu$: FM 640 vs MM 160]{
\includegraphics[width=0.4\textwidth,trim=20 0 40 10,clip]
{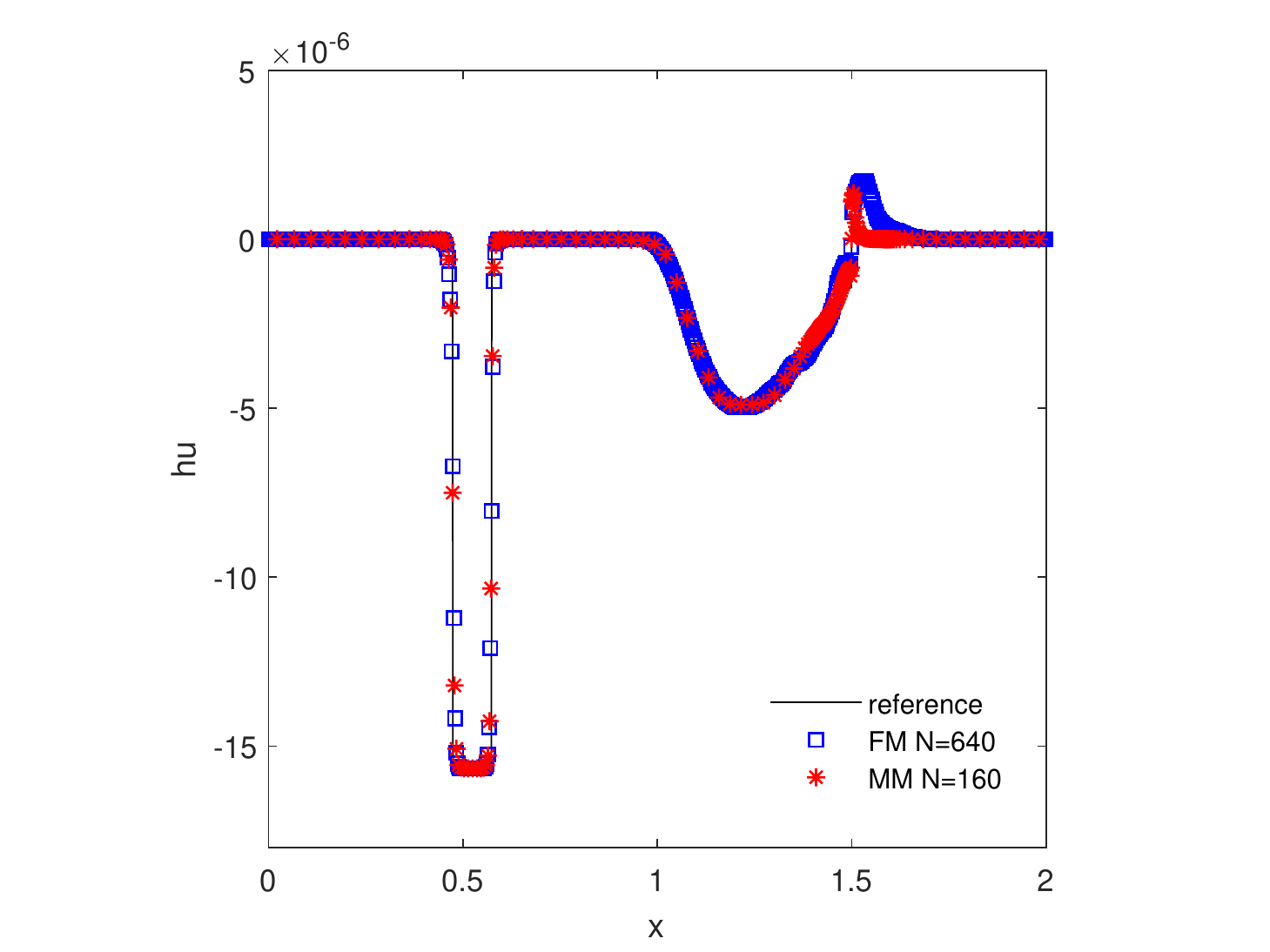}}
\subfigure[Close view of (c)]{
\includegraphics[width=0.4\textwidth,trim=20 0 39 10,clip]{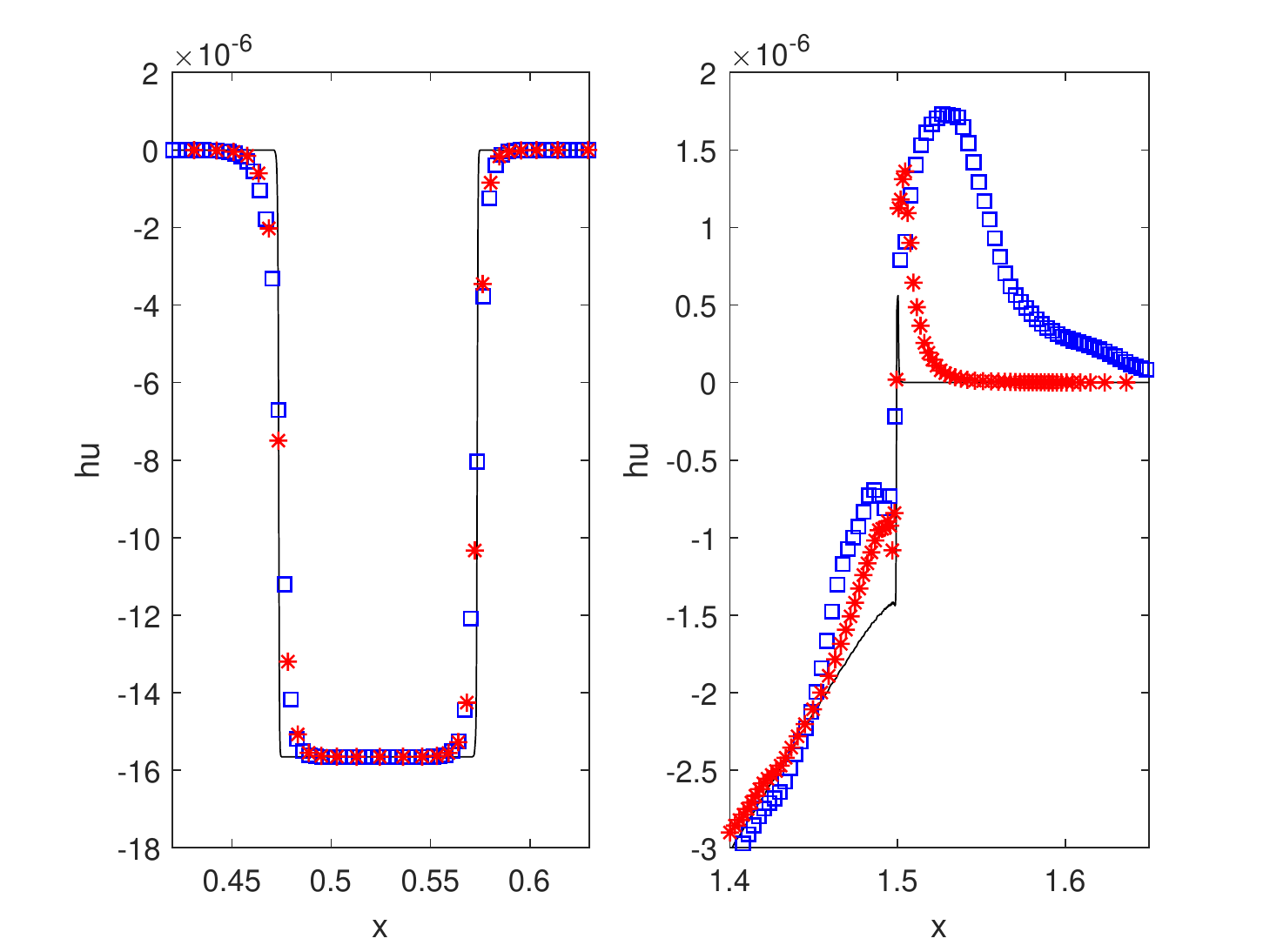}}
\caption{Example \ref{test3-1d} with the bottom topography (\ref{B-5-1}) with a dry region.
The water discharge $hu$ at $t=0.2$ obtained with $P^2$-DG and a moving mesh of $N=160$ are compared with those obtained with a fixed mesh of $N=160$ and $N=640$.}
\label{Fig:test-1d-preturb-hu}
\end{figure}

\begin{example}\label{test8-1d}
(The Riemann problem for the 1D SWEs with a step bottom topography.)
\end{example}
In this test we solve the 1D SWEs with a step bottom topography \cite{Alcrudo-Benkhaldoun-2001,Li-Lu-Qiu-2012JSC},
\begin{equation*}
B(x)=
\begin{cases}
0,& \text{for}~x \in (-10,0)\\
1,& \text{for}~x \in (0, 10)
\end{cases}
\end{equation*}
The initial conditions are given by
\begin{equation*}
h(x,0)=
\begin{cases}
4,& \text{for}~x\leq0\\
1,& \text{otherwise}\\
\end{cases}
\quad \quad
u(x,0)=\begin{cases}
5,& \text{for}~x\leq0\\
-0.9,& \text{otherwise.}\\
\end{cases}
\end{equation*}
The solution contains two hydraulic jumps/shocks and a stationary step subcritical transition.
The first shock moves to the left while the second one to the right.
The stationary step transition of the flow occurs at $x = 0$.

The mesh trajectories for the $P^2$-DG method with a moving mesh of $N=100$ are shown in Fig.~\ref{Fig:test8-1d-mesh}. We can see that the mesh has high element concentration at the location of two hydraulic jumps/shocks and the stationary step transition.
The moving mesh solutions of $N=100$ at $t=1$ and the fixed mesh solutions obtained with $N=100$ and $N=800$ are shown in Fig.~\ref{Fig:test8-1d-Bph} and Fig.~\ref{Fig:test8-1d-hu} for water surface level and water discharge, respectively.
The results show that the moving mesh solution ($N=100$)  provides a
better resolution of the shocks than that with the fixed mesh of $N=100$ and $N=800$.

\begin{figure}[H]
\centering
\includegraphics[width=0.4\textwidth,trim=40 0 40 10,clip]{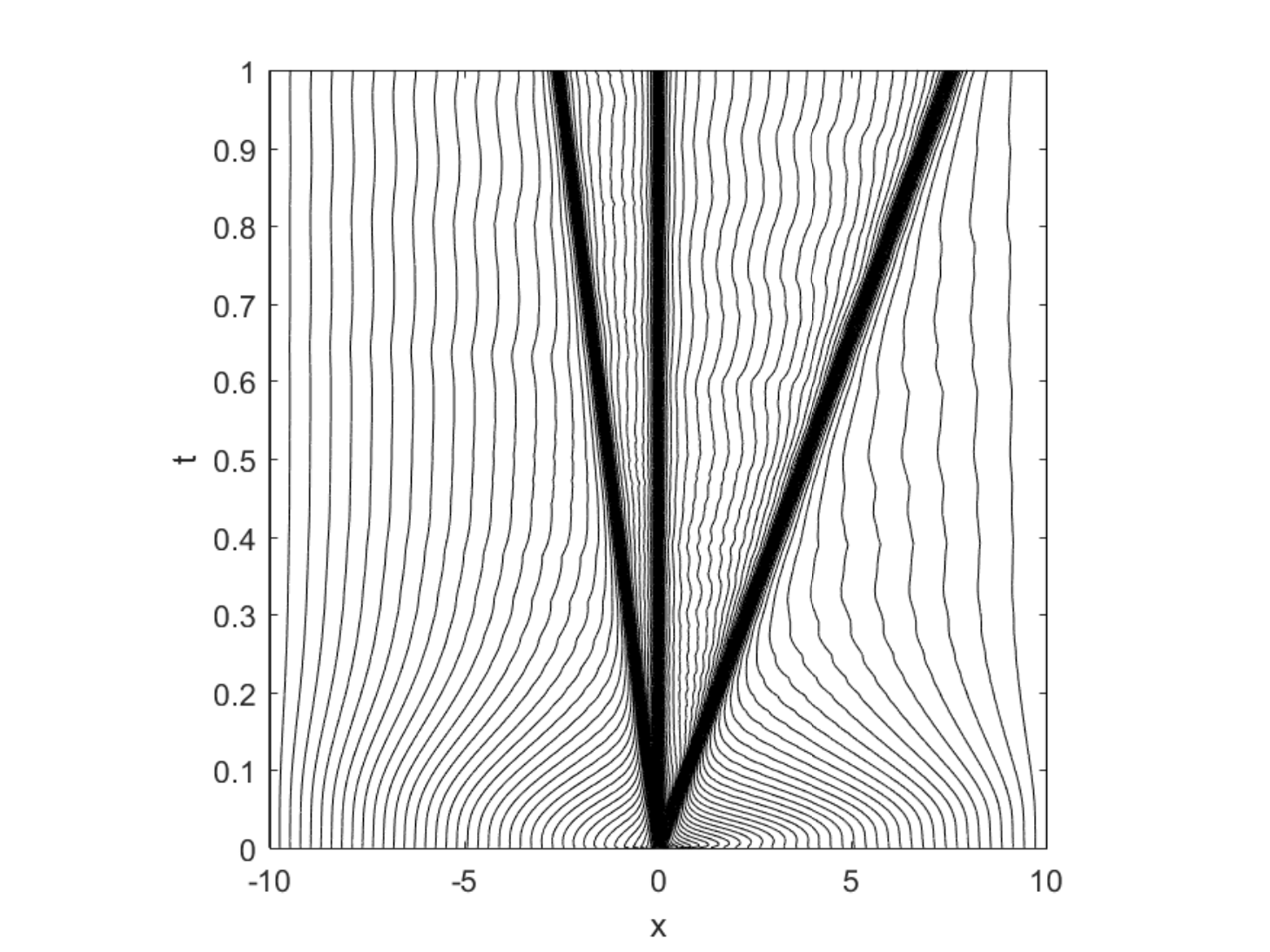}
\caption{Example \ref{test8-1d}.
 The mesh trajectories are obtained with $P^2$-DG method with a moving mesh of $N=100$.}
\label{Fig:test8-1d-mesh}
\end{figure}

\begin{figure}[H]
\centering
\subfigure[FM 100 vs MM 100]{
\includegraphics[width=0.4\textwidth,trim=10 0 30 10,clip]{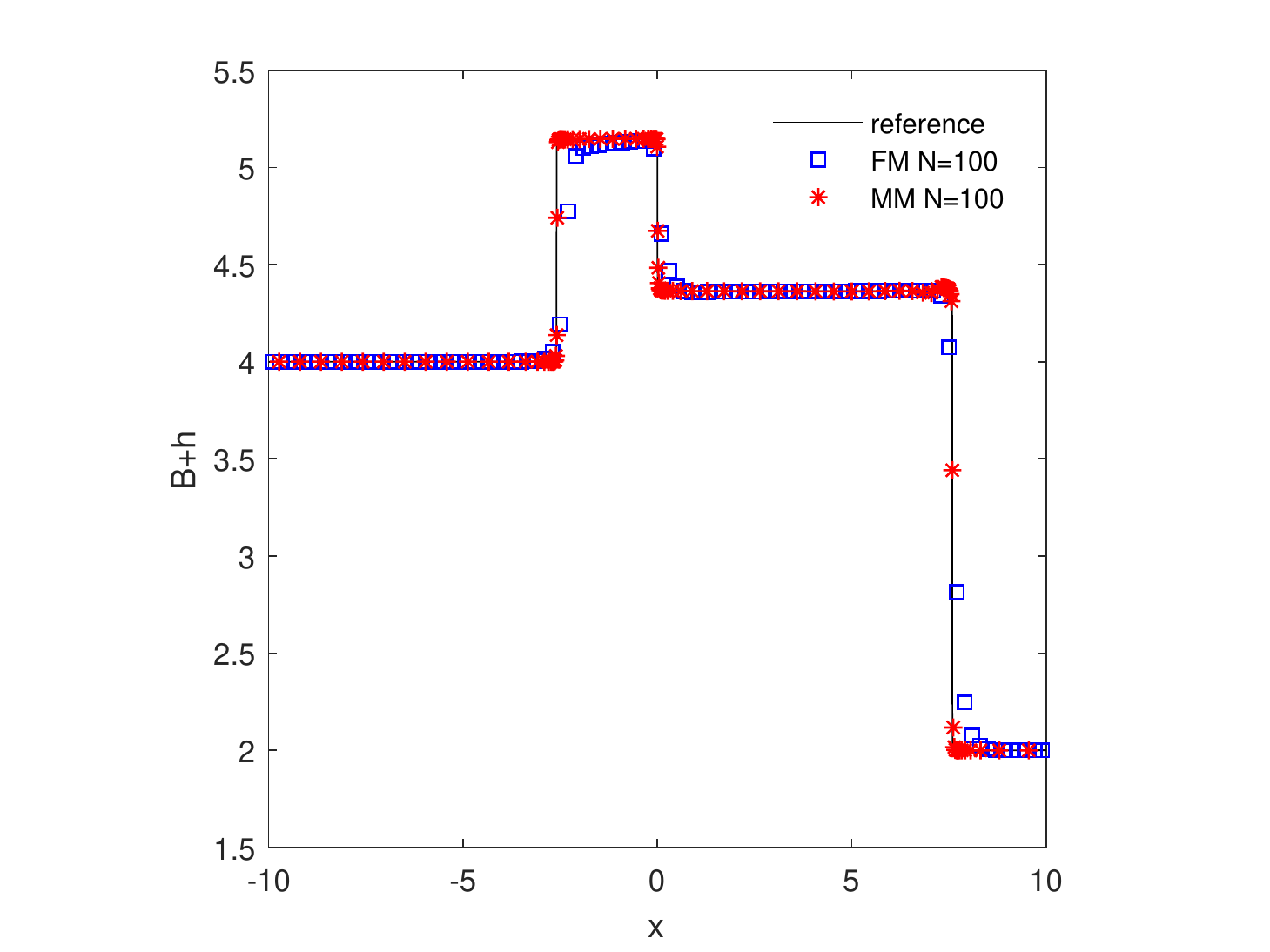}}
\subfigure[Close view of (a) near shocks]{
\includegraphics[width=0.4\textwidth,trim=10 0 30 10,clip]{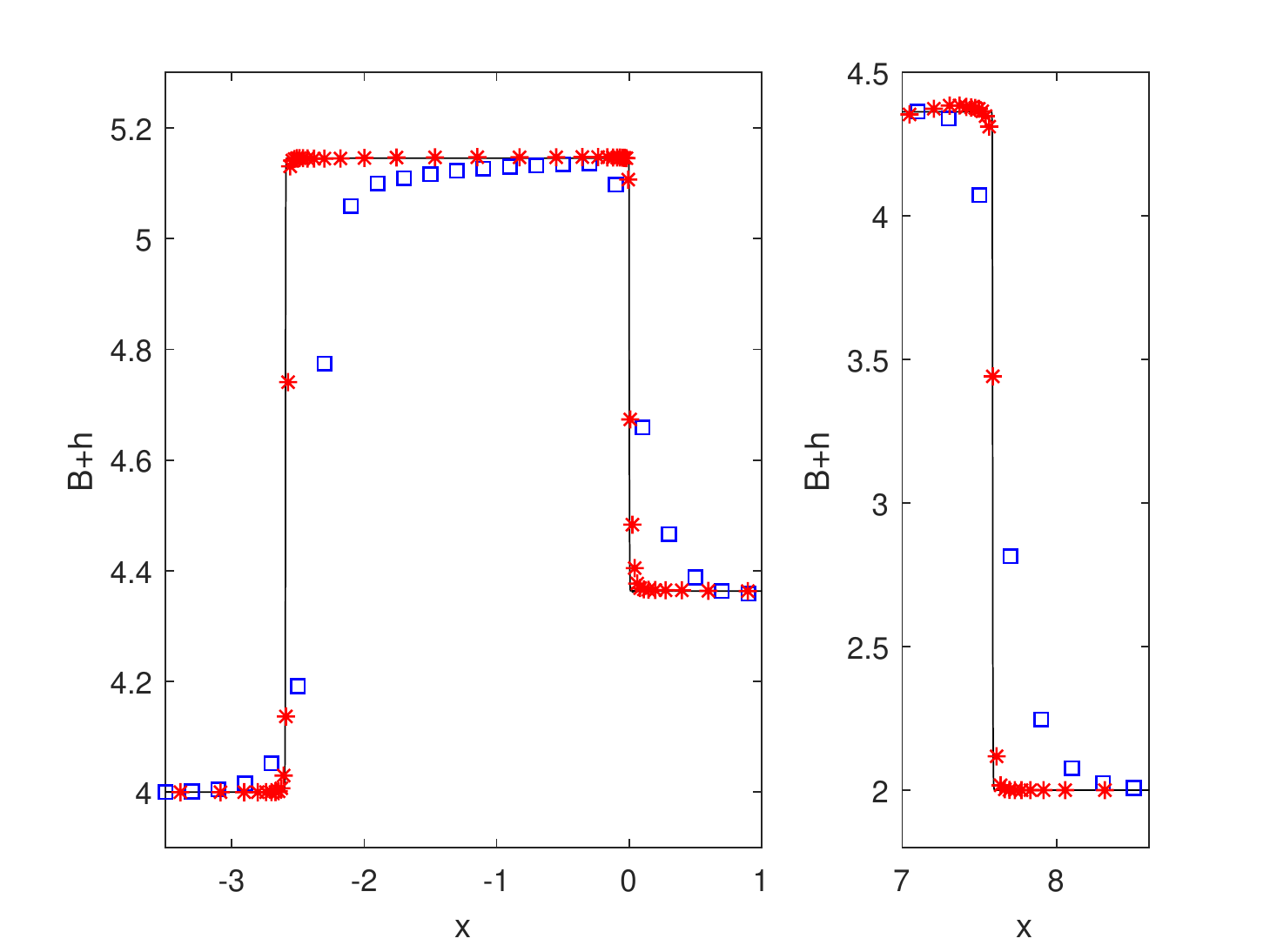}}
\subfigure[FM 800 vs MM 100]{
\includegraphics[width=0.4\textwidth,trim=10 0 30 10,clip]{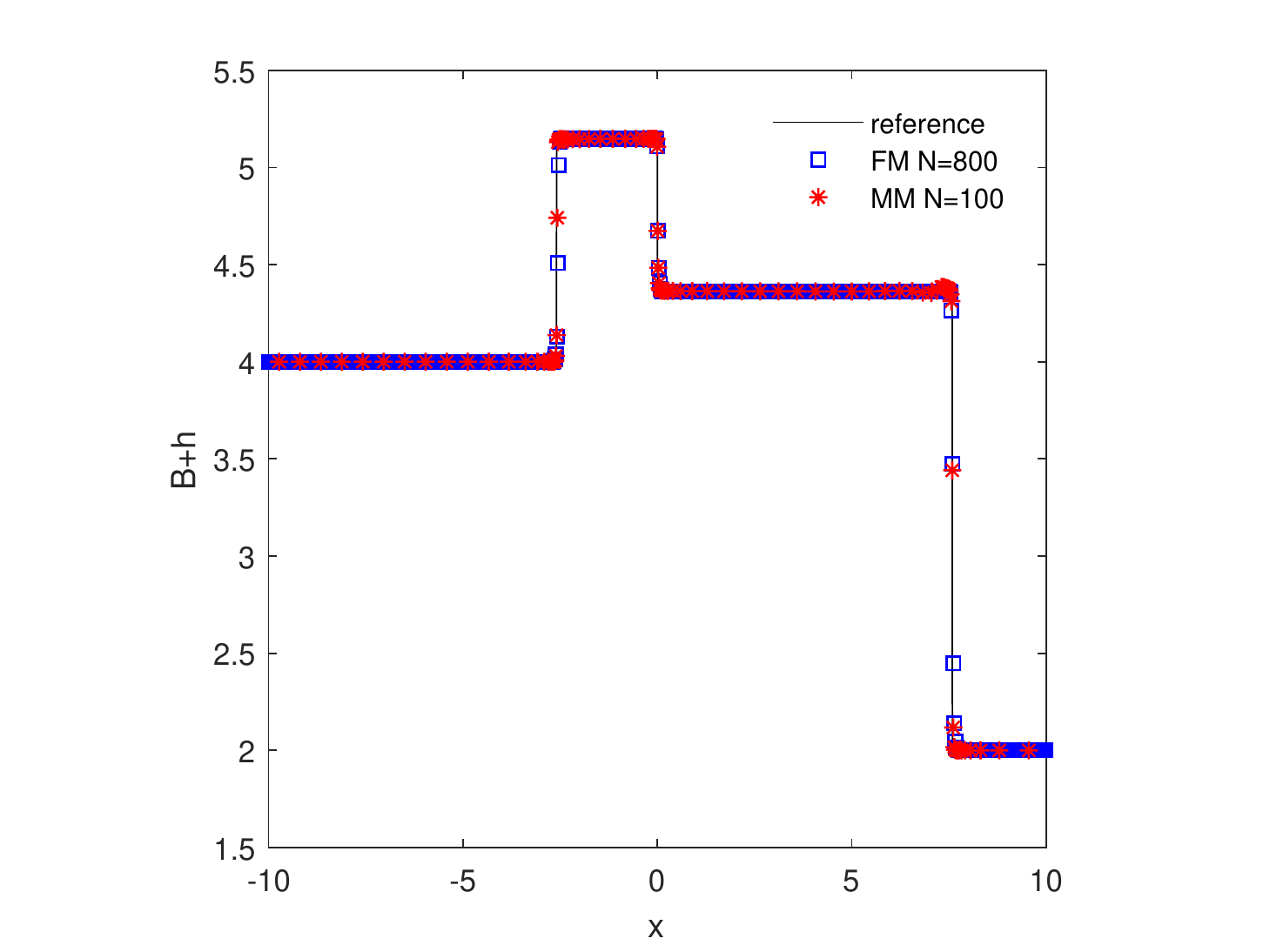}}
\subfigure[Close view of (c) near shocks]{
\includegraphics[width=0.4\textwidth,trim=10 0 30 10,clip]{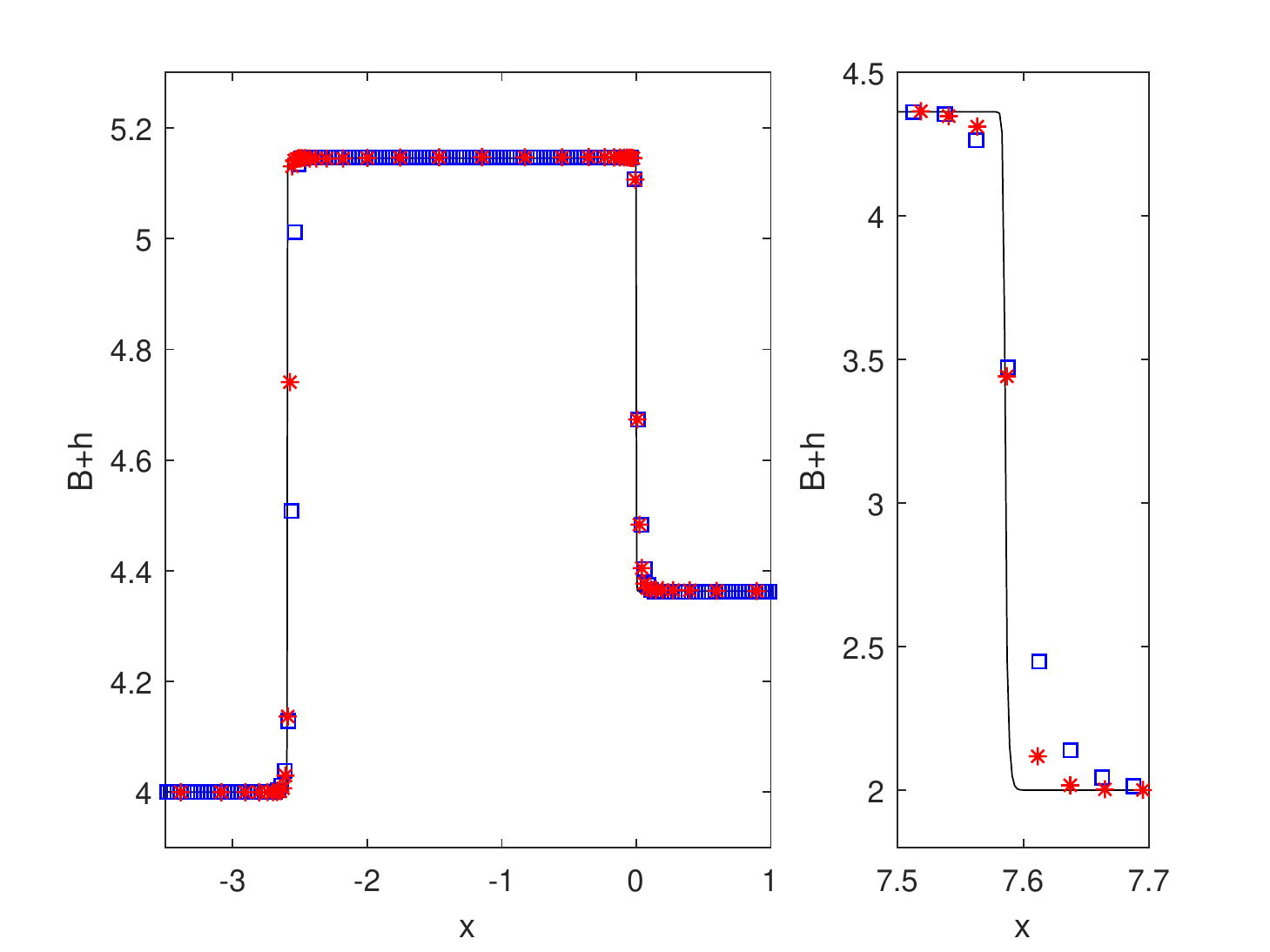}}
\caption{Example \ref{test8-1d}.
The water surface $B+h$ at $t=1$ obtained with $P^2$-DG method with a moving mesh of $N=100$ is compared with those obtained with the fixed meshes of $N=100$ and $N=800$.}
\label{Fig:test8-1d-Bph}
\end{figure}

\begin{figure}[H]
\centering
\subfigure[FM 100 vs MM 100]{
\includegraphics[width=0.4\textwidth,trim=10 0 30 10,clip]{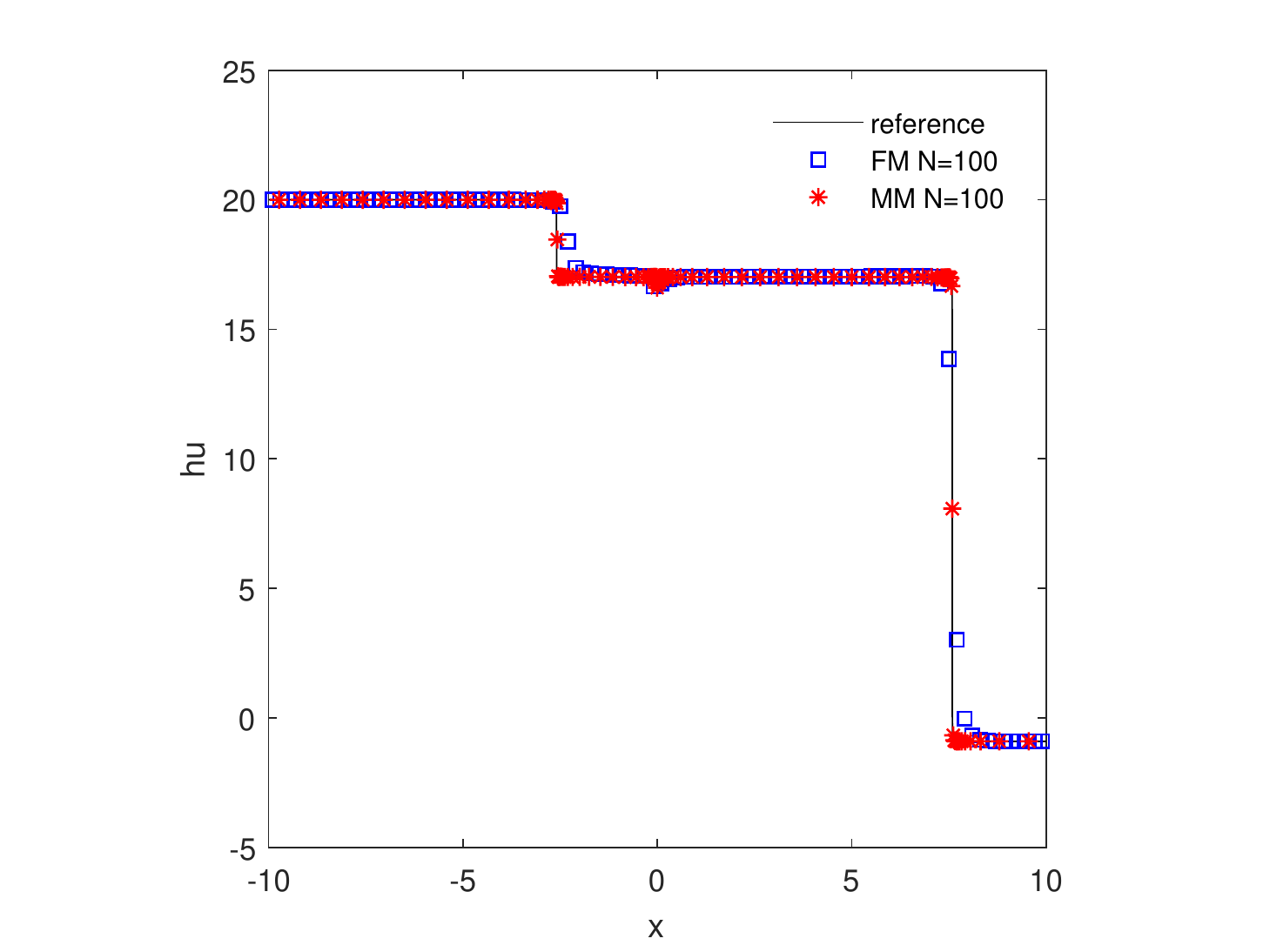}}
\subfigure[Close view of (a) near shocks]{
\includegraphics[width=0.4\textwidth,trim=10 0 30 10,clip]{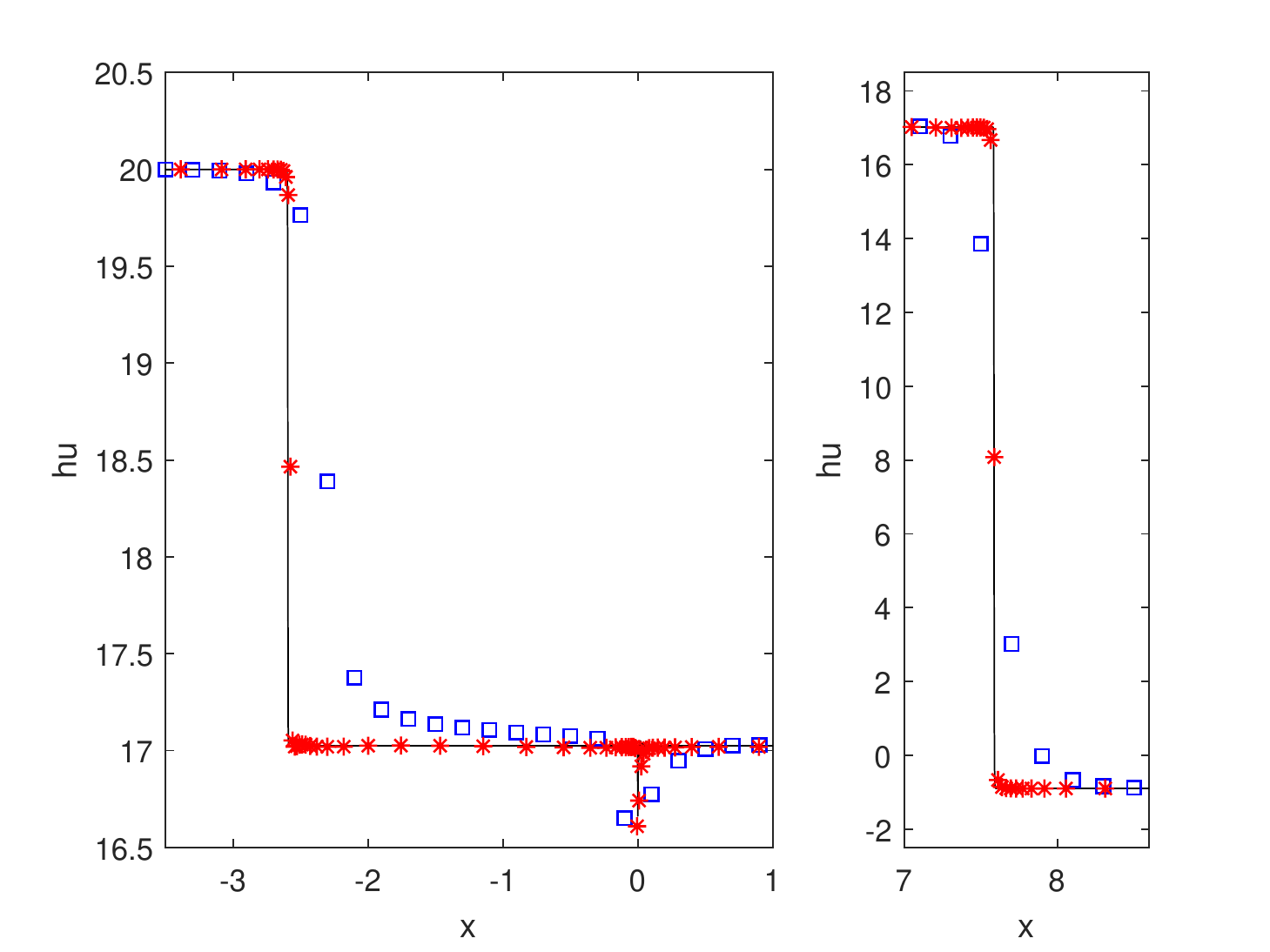}}
\subfigure[FM 800 vs MM 100]{
\includegraphics[width=0.4\textwidth,trim=10 0 30 10,clip]{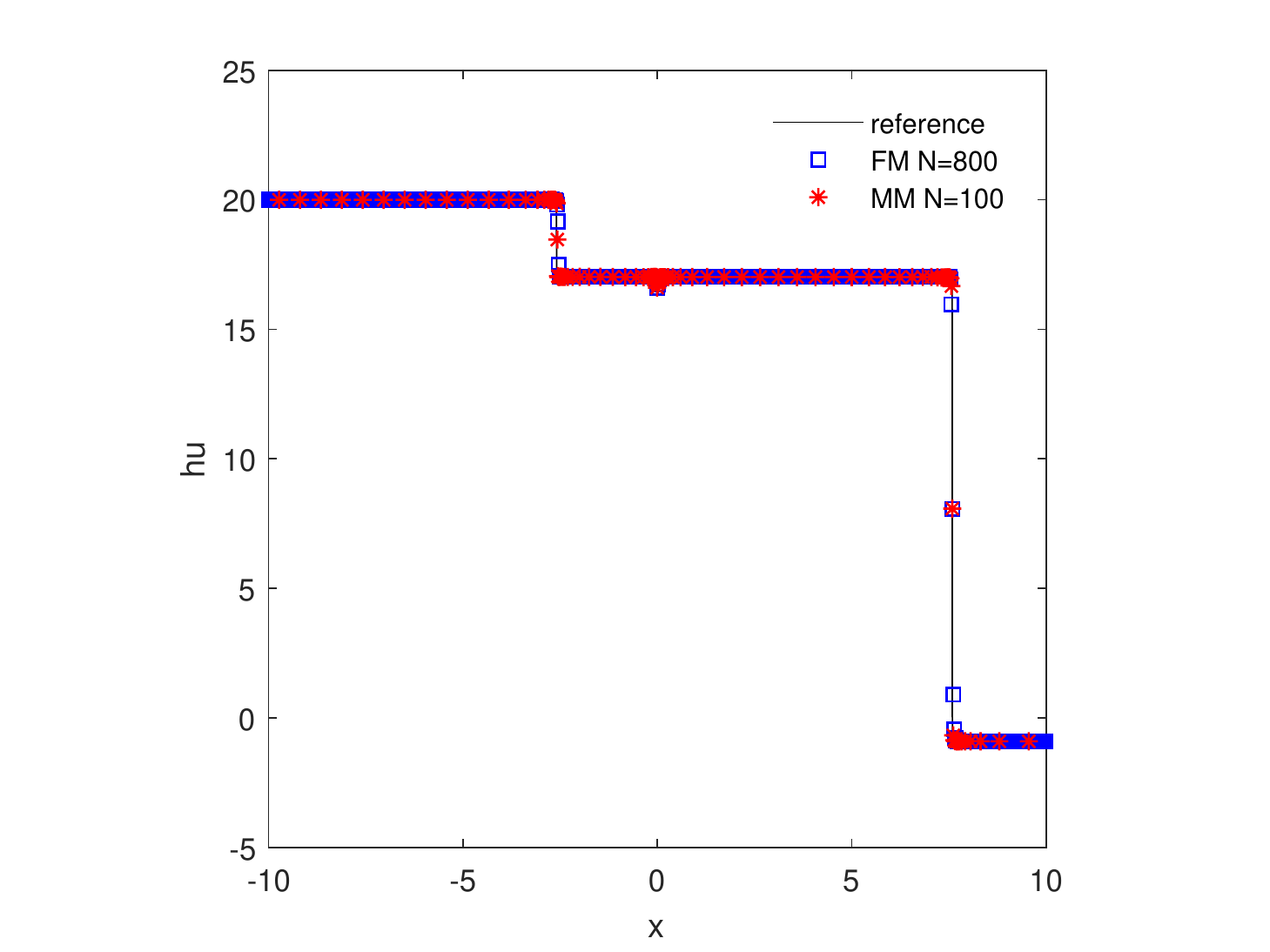}}
\subfigure[Close view of (c) near shocks]{
\includegraphics[width=0.4\textwidth,trim=10 0 30 10,clip]{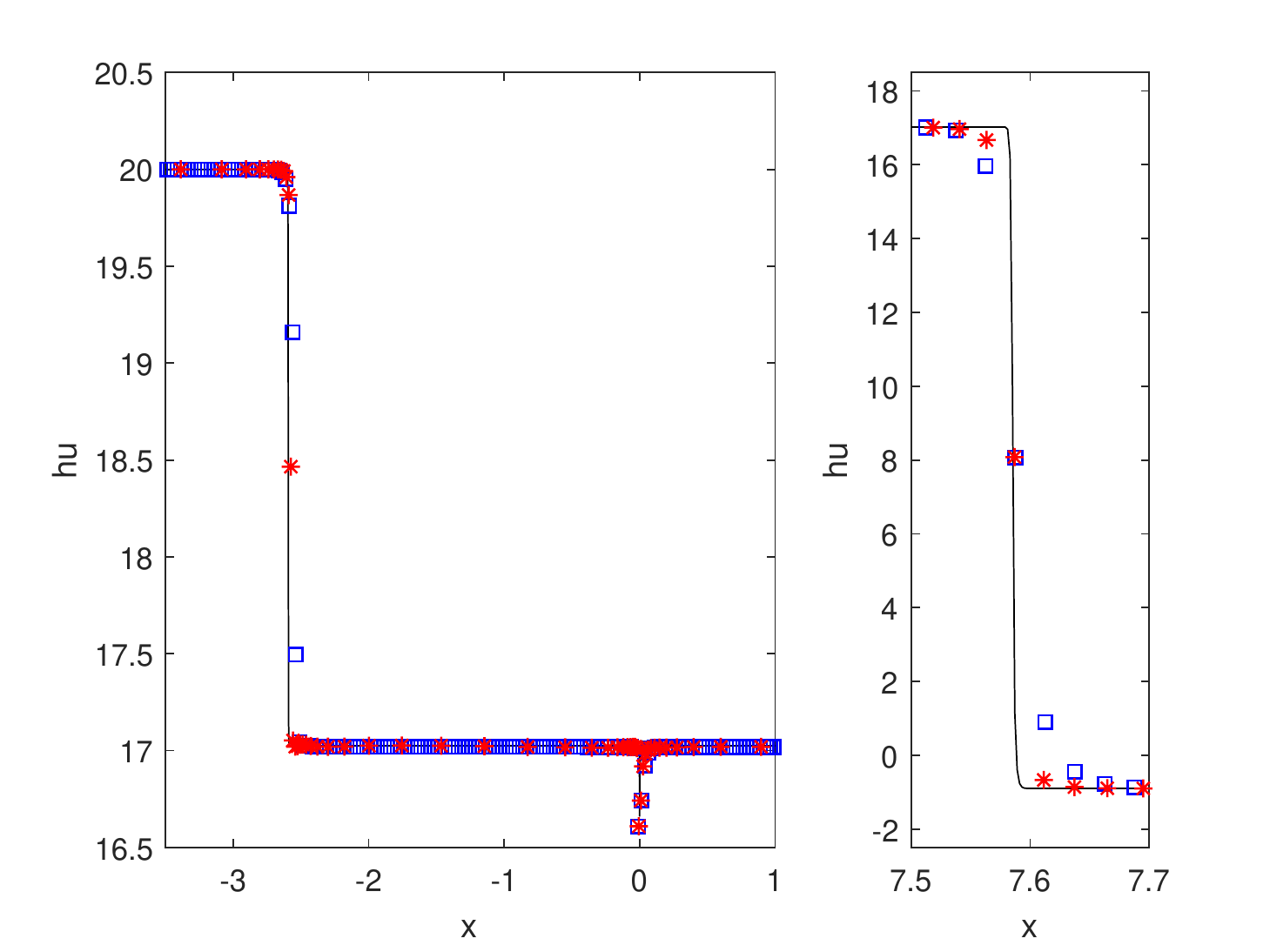}}
\caption{Example \ref{test8-1d}.
The water discharge $hu$ at $t=1$ obtained with $P^2$-DG method with a moving mesh of $N=100$ is compared with those obtained with the fixed meshes of $N=100$ and $N=800$.}
\label{Fig:test8-1d-hu}
\end{figure}

\begin{example}\label{test6-1d}
(The rarefaction and shock waves test for the 1D SWEs with flat and wavy bottom topographies.)
\end{example}
In this test we compute the 1D SWEs with two sets of bottom topographies $B(x)$.
The initial conditions are given by
\begin{equation*}
h(x,0)=
\begin{cases}
2-B(x),& \text{for}~x\in (-10,1)\\
0.35-B(x),& \text{for}~x\in (1, 10) \\
\end{cases}
\quad \quad
u(x,0)=
\begin{cases}
1,& \text{for}~x\in (-10,1)\\
0,& \text{for}~x\in (1, 10).\\
\end{cases}
\end{equation*}
We choose the transmissive boundary conditions and compute the solution up to $t=1$.

We first consider $B(x)=0$.
The solution contains a rarefaction wave moving to the left and a shock wave (hydraulic jump) traveling to the right.
The mesh trajectories with a moving mesh of $N=100$ are shown in Fig.~\ref{Fig:test6-1d-mesh-B0}.
One can see that the mesh has high element concentration around the rarefaction and shock.
In Figs.~\ref{Fig:test6-1d-Bph-B0} and \ref{Fig:test6-1d-hu-B0}, we plot the water surface $B+h$
and water discharge $hu$ at $t=1$ obtained with $P^2$-DG and a moving mesh of $N=100$
and fixed meshes of $N=100$ and $N=1280$.
It can be seen that the moving mesh solutions of $N=100$ are more accurate than those
with a fixed mesh of $N=100$ and comparable with those with a fixed mesh of $ N =1280$.

To show the effects of non-flat bottom topographies on water flow, we repeat this example
with a wavy bottom topography \cite{Tang-2004},
\begin{equation}
\label{B-3}
B(x)=
\begin{cases}
0.3 \cos^{30}(\frac{\pi}{2}(x-1)),& \text{for}~0\leq x\leq 2\\
0,& \text{otherwise}.
\end{cases}
\end{equation}
The solution shows more complex features, made up of a left rarefaction wave and two hydraulic jumps/shocks.
One shock near the location of $x=2$ occurs in the flow over the non-flat bed topography.

The moving mesh trajectories ($N=160$) are plotted in Fig.~\ref{Fig:test6-1d-mesh}. One can see that the mesh elements concentrate around the rarefaction and hydraulic jumps/shocks.
In Figs.~\ref{Fig:test6-1d-Bph} and \ref{Fig:test6-1d-hu}, we plot the water surface level $B+h$ and water discharge $hu$ at $t=1$ obtained with $P^2$-DG and a moving mesh of $N=160$ and fixed meshes of $N=160$ and $N=1280$.
These results show that the moving mesh solutions of $N=160$ are more accurate than those
with a fixed mesh of $N=160$ and comparable with that with the fixed mesh of $N =1280$.
Moreover, the MM-DG method does a good job in resolving the shock near $x=2$ which seems
to be a difficult structure for a fixed mesh to resolve.

\begin{figure}[H]
\centering
\includegraphics[width=0.4\textwidth,trim=40 0 40 10,clip]{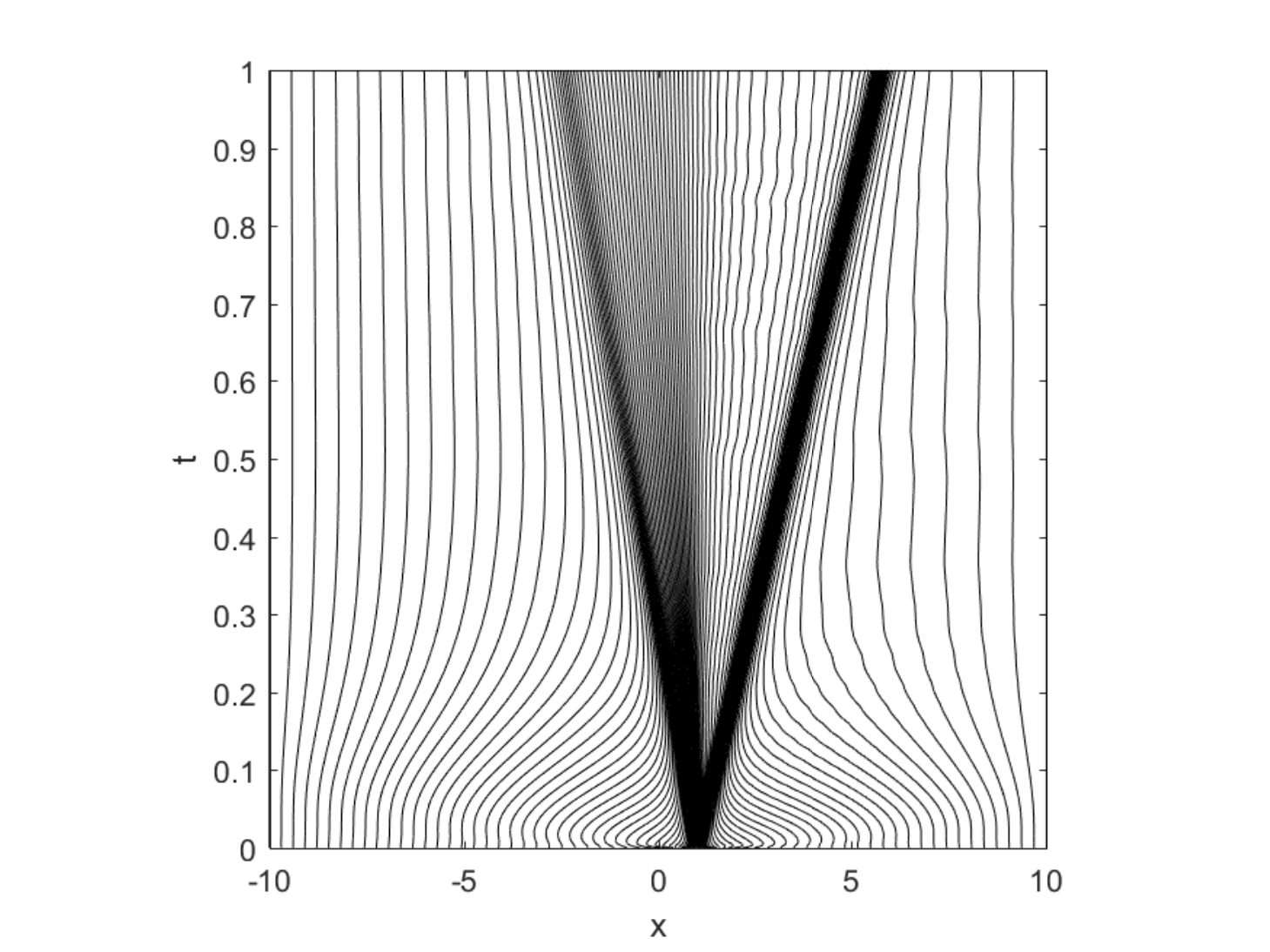}
\caption{Example \ref{test6-1d}. The mesh trajectories for $ B=0$ are obtained with $P^2$-DG and a moving mesh of $N=100$.}
\label{Fig:test6-1d-mesh-B0}
\end{figure}

\begin{figure}[H]
\centering
\subfigure[FM 100 vs MM 100]{
\includegraphics[width=0.4\textwidth,trim=10 0 30 10,clip]{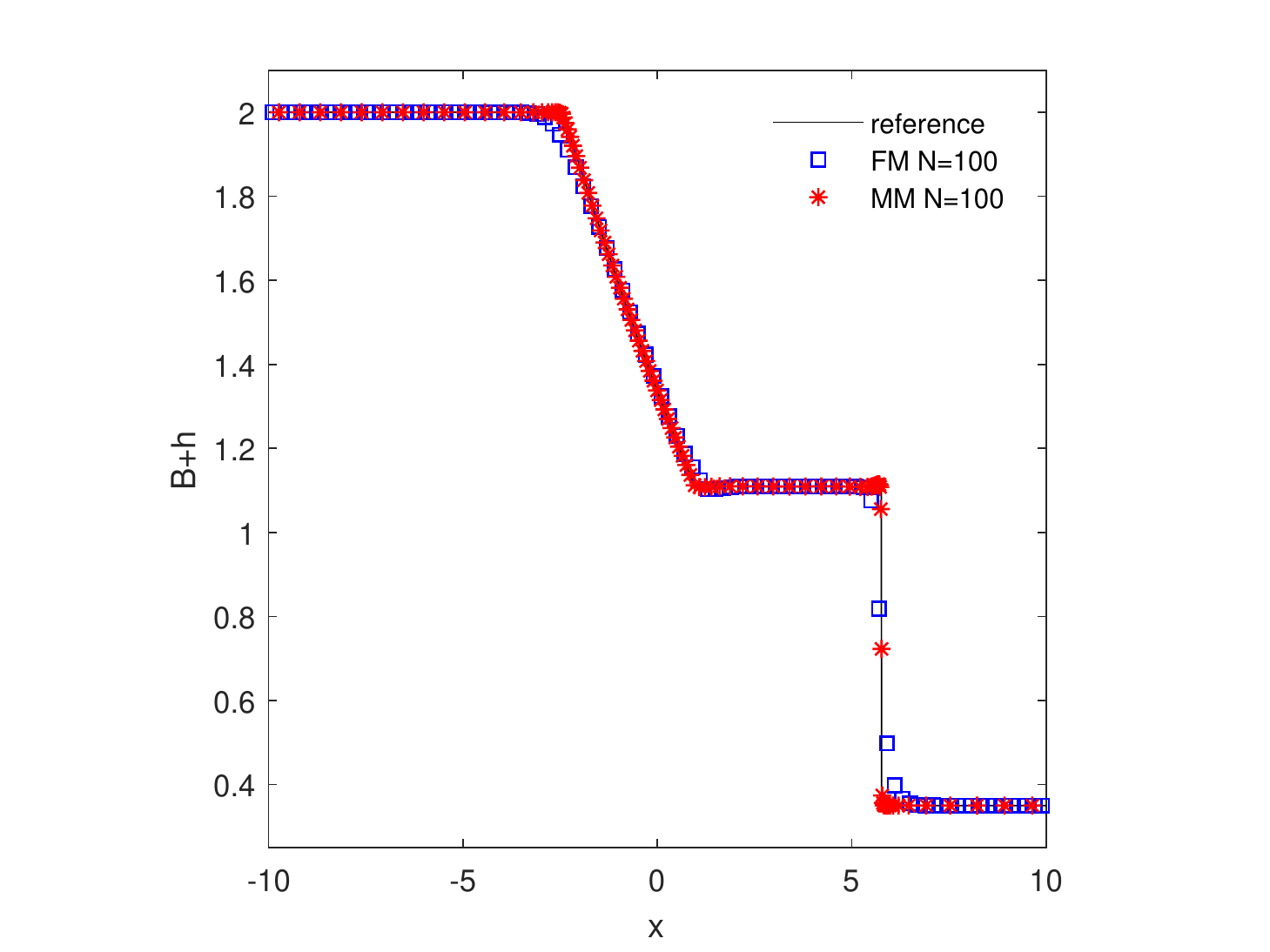}}
\subfigure[Close view of (a)]{
\includegraphics[width=0.4\textwidth,trim=10 0 30 10,clip]{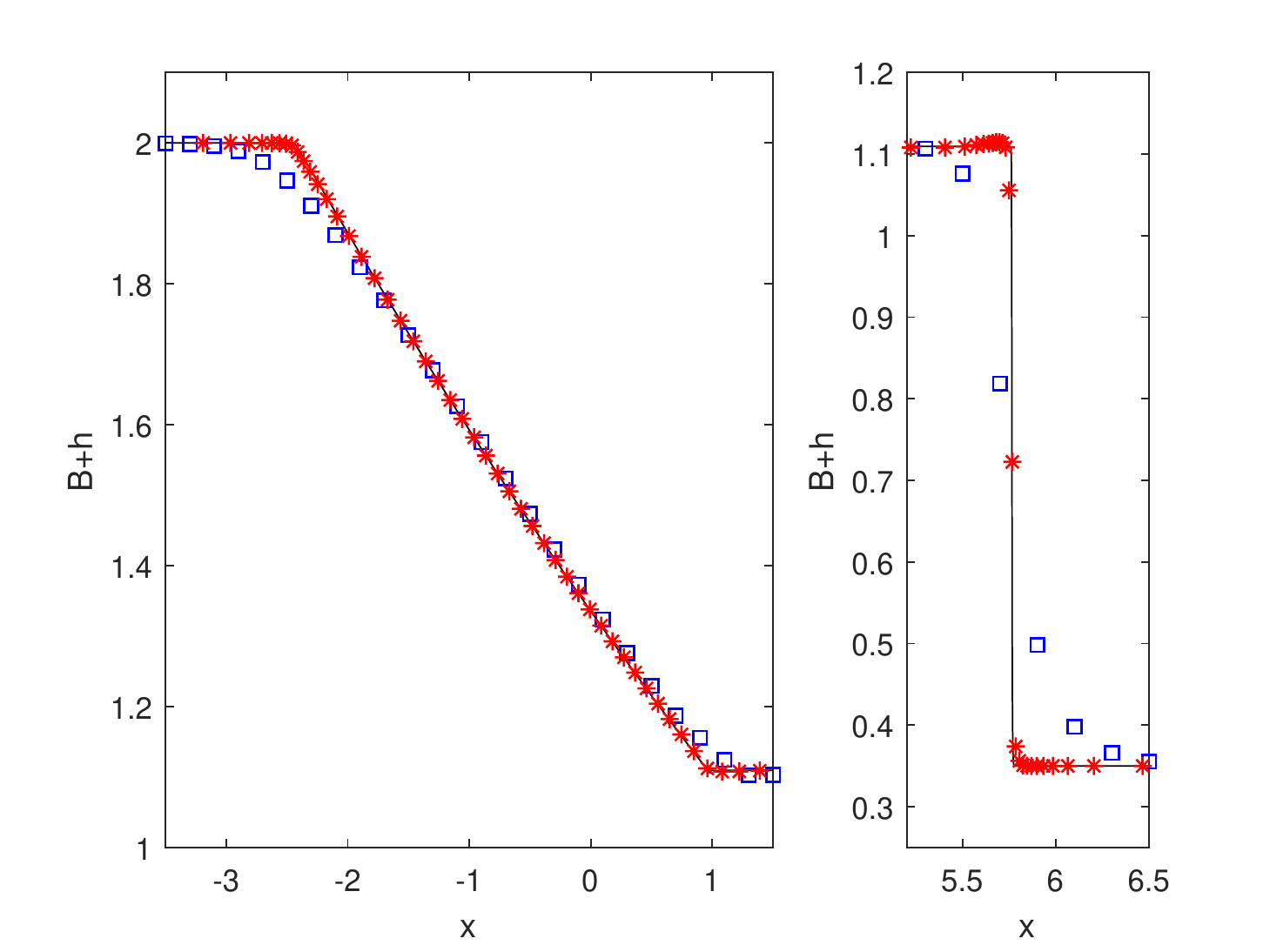}}
\subfigure[FM 1280 vs MM 100]{
\includegraphics[width=0.4\textwidth,trim=10 0 30 10,clip]{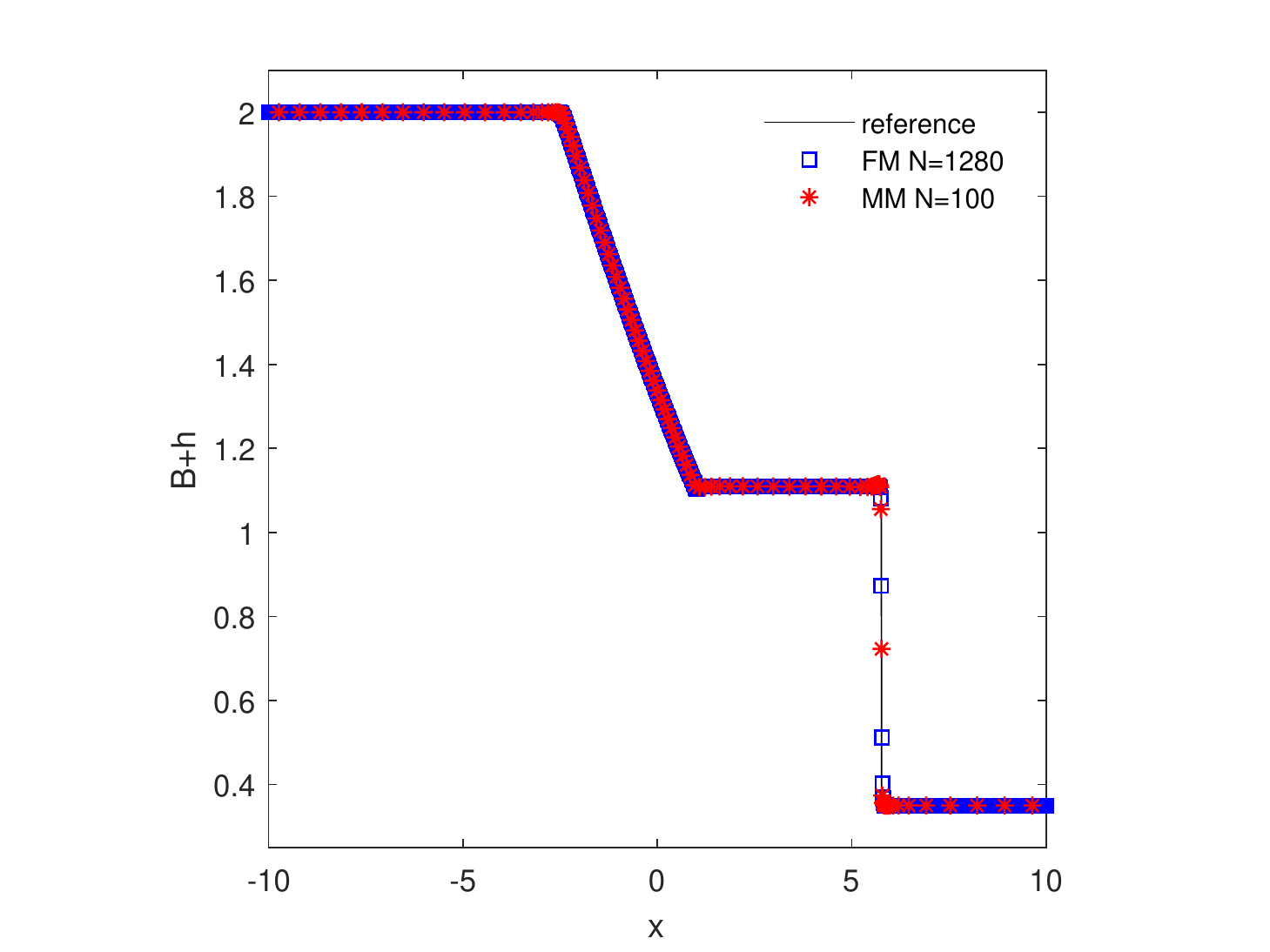}}
\subfigure[Close view of (c)]{
\includegraphics[width=0.4\textwidth,trim=10 0 30 10,clip]{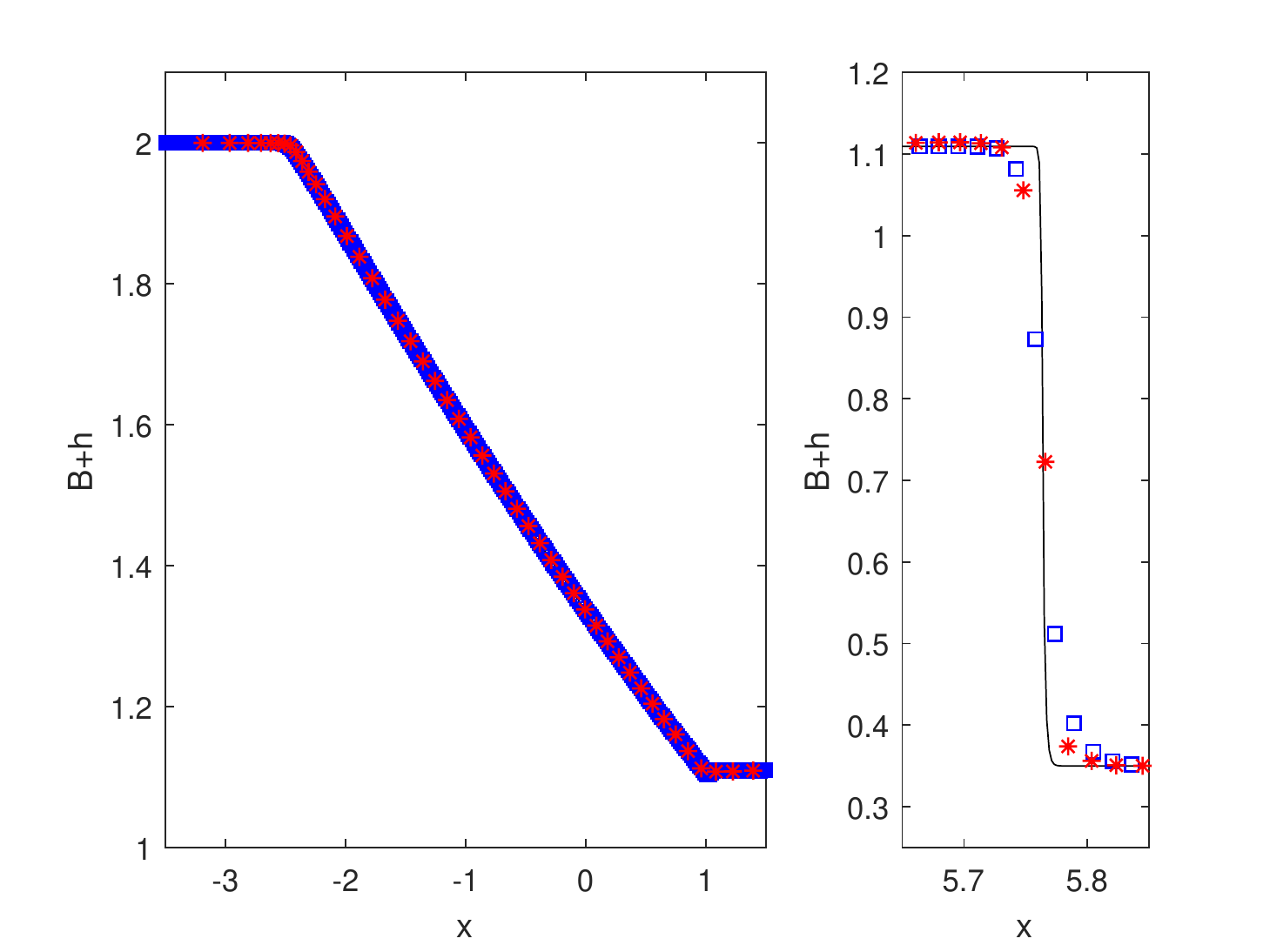}}
\caption{Example \ref{test6-1d}. The water depth $h$ (for $ B=0$) at $t=1$ obtained with $P^2$-DG and a moving mesh of $N=100$ is compared with those obtained with fixed meshes of $N=100$ and $N=1280$.}
\label{Fig:test6-1d-Bph-B0}
\end{figure}

\begin{figure}[H]
\centering
\subfigure[FM 100 vs MM 100]{
\includegraphics[width=0.4\textwidth,trim=10 0 30 10,clip]{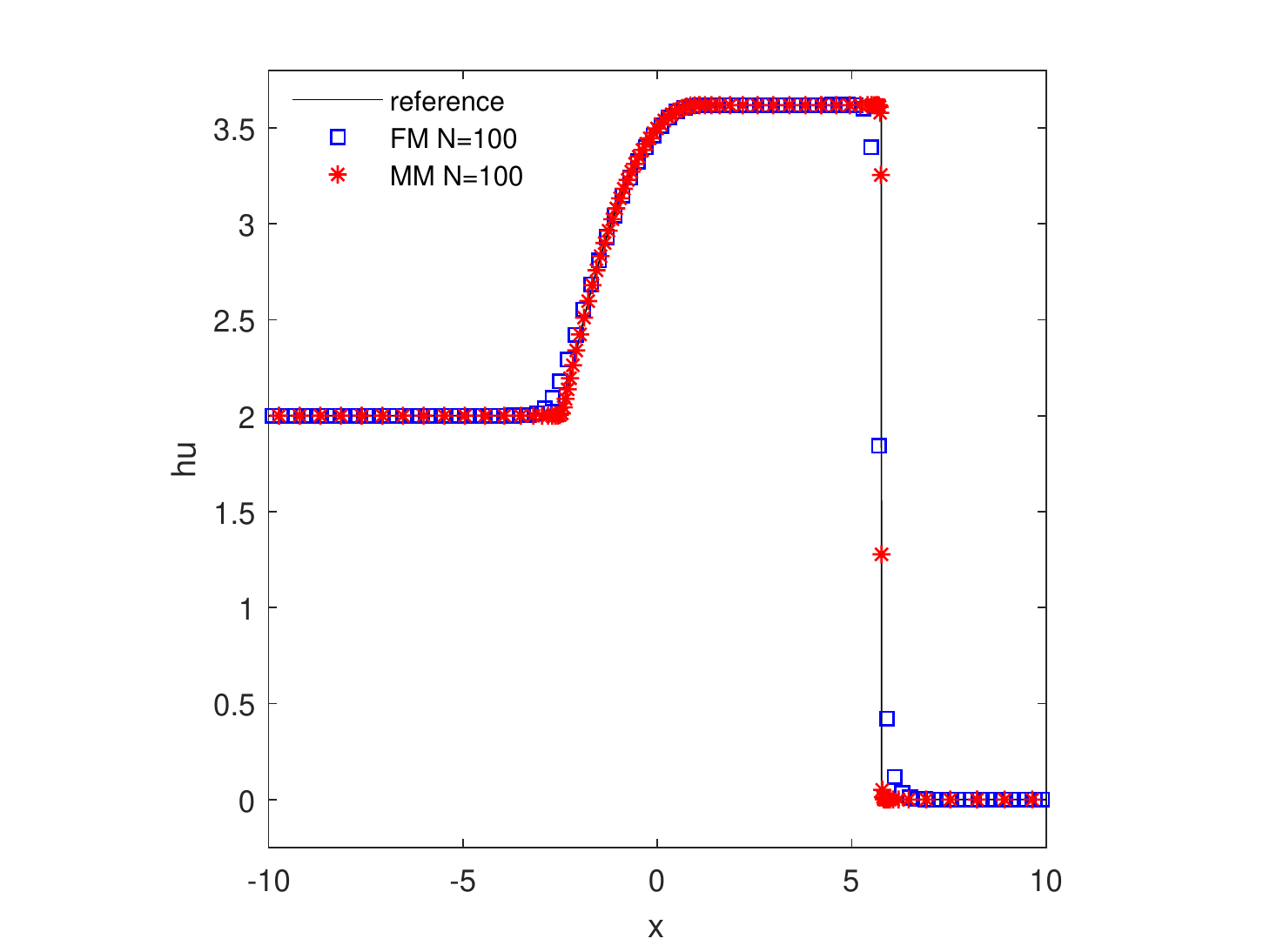}}
\subfigure[Close view of (a)]{
\includegraphics[width=0.4\textwidth,trim=10 0 30 10,clip]{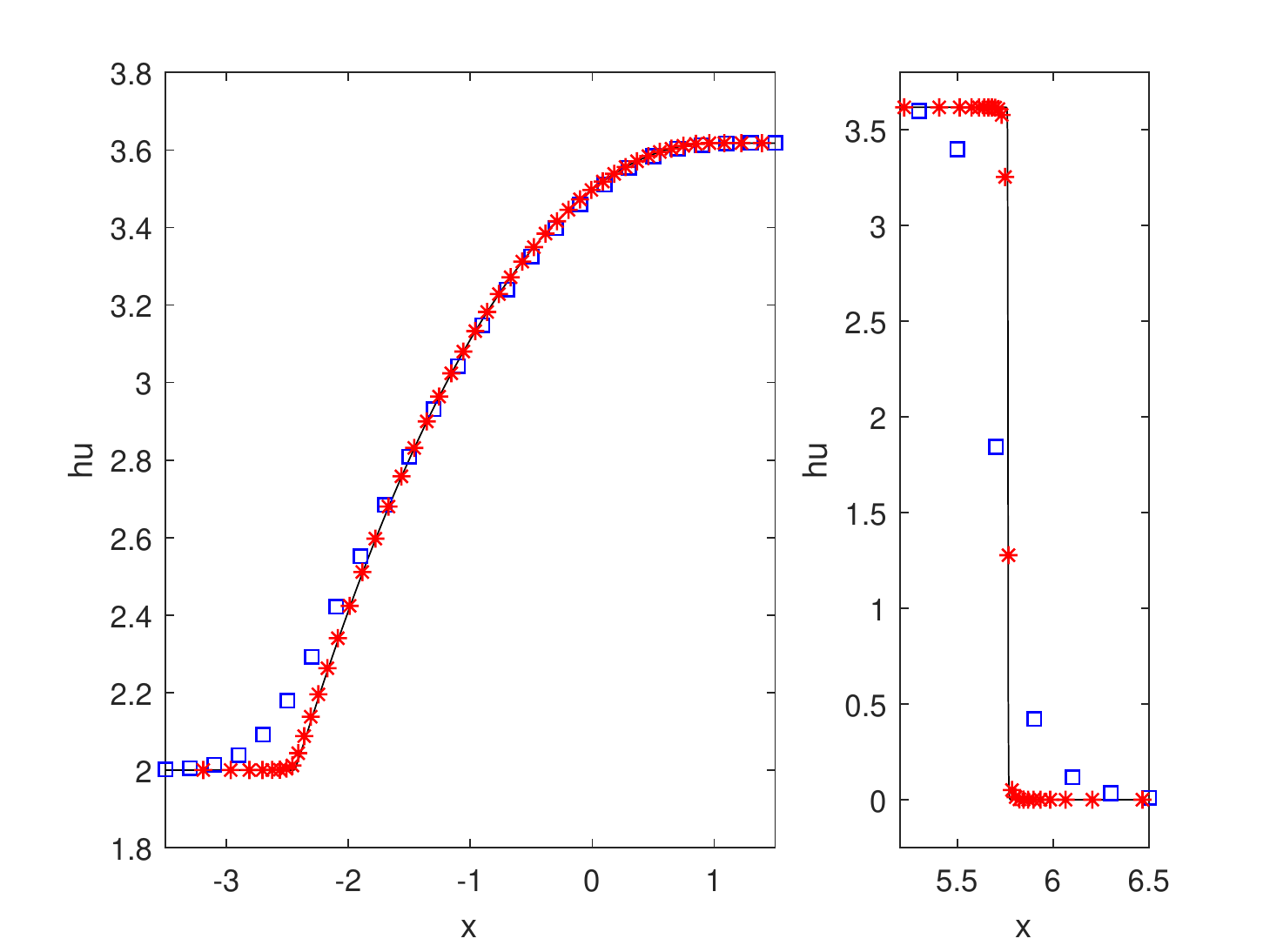}}
\subfigure[FM 1280 vs MM 100]{
\includegraphics[width=0.4\textwidth,trim=10 0 30 10,clip]{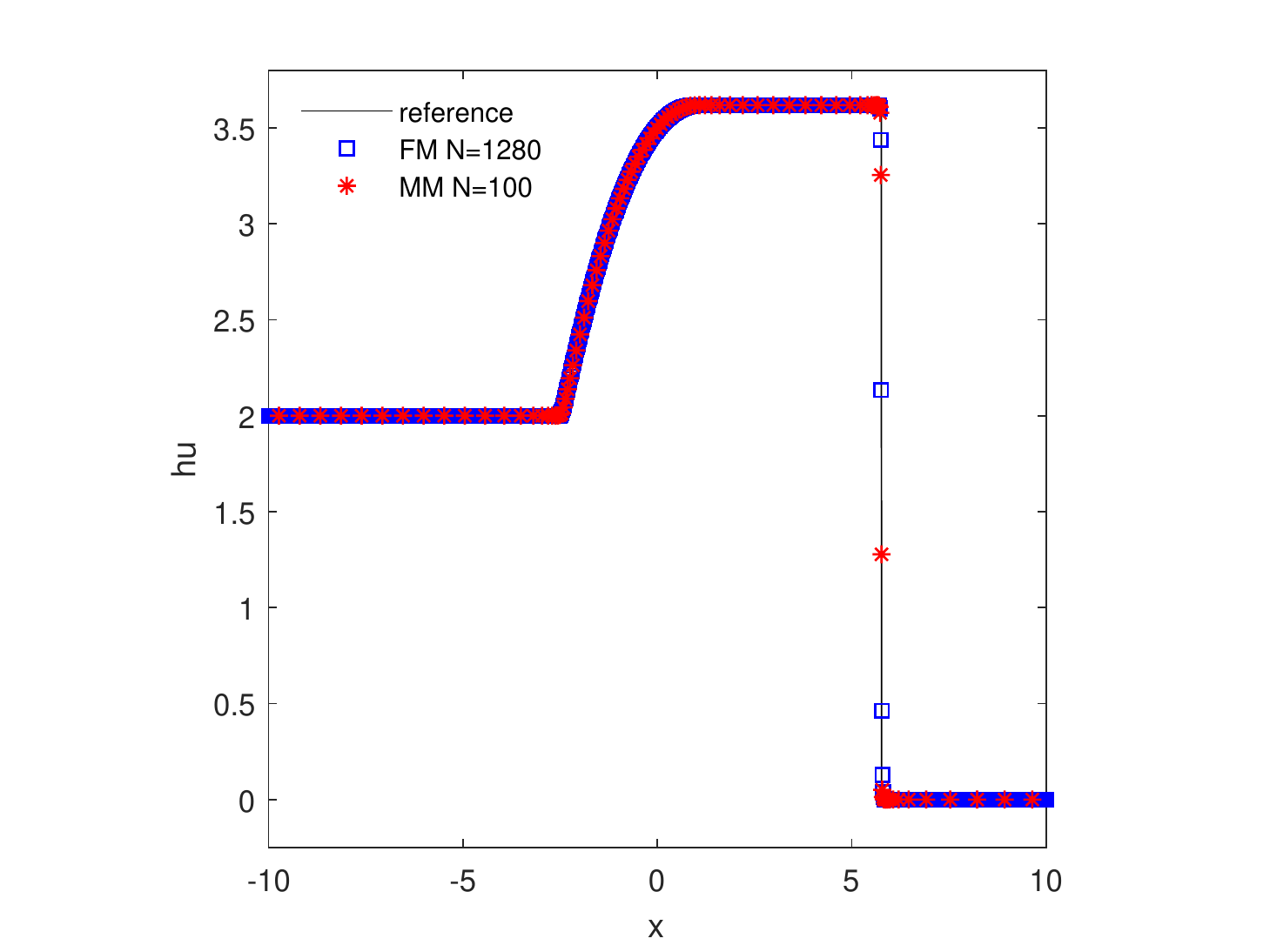}}
\subfigure[Close view of (c)]{
\includegraphics[width=0.4\textwidth,trim=10 0 30 10,clip]{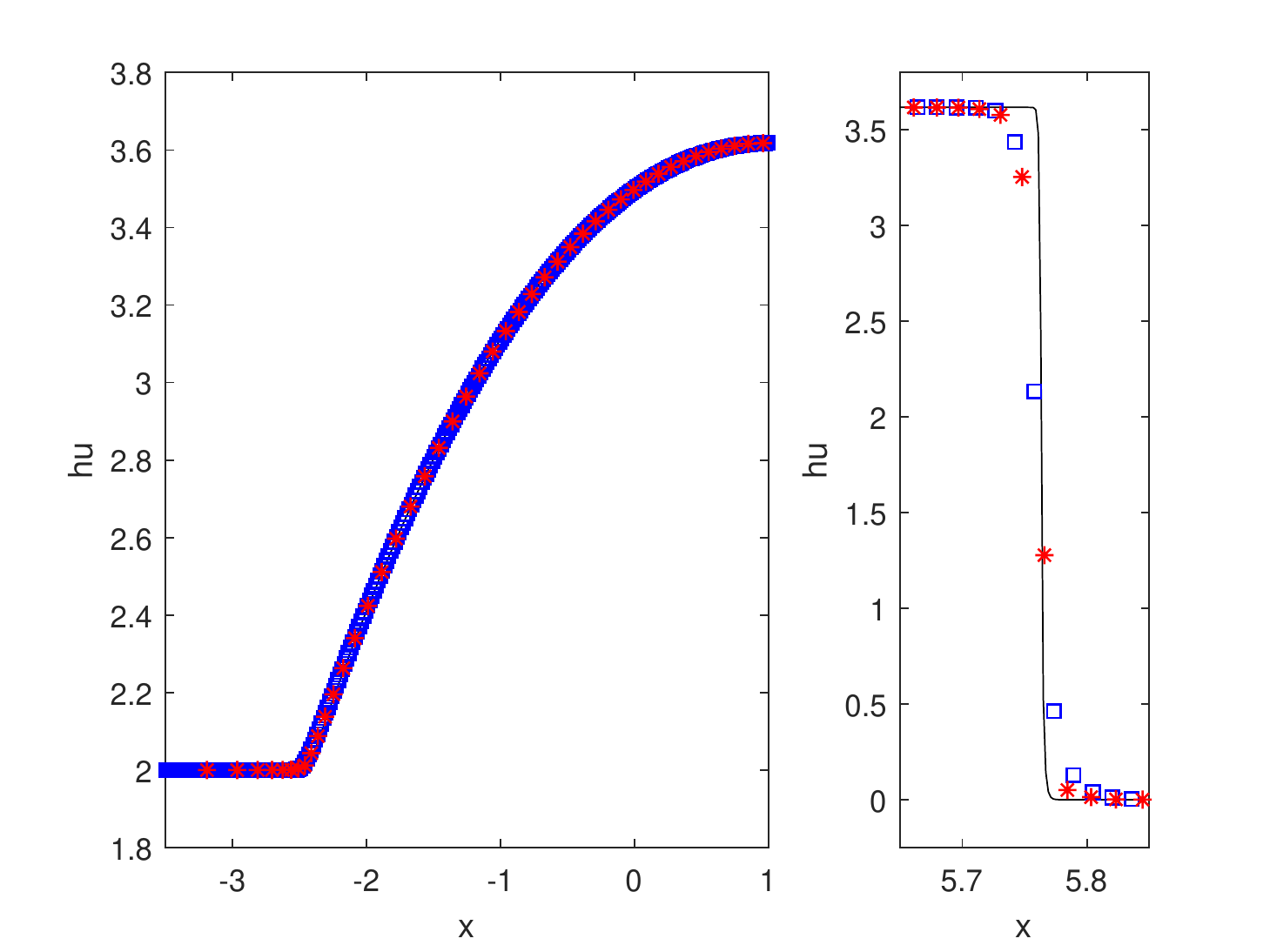}}
\caption{Example \ref{test6-1d}. The water discharge  $hu$ (for $ B=0$) at $t=1$ obtained with $P^2$-DG and a moving mesh of $N=100$ is compared with those obtained with fixed meshes of $N=100$ and $N=1280$.}
\label{Fig:test6-1d-hu-B0}
\end{figure}

\begin{figure}[H]
\centering
\includegraphics[width=0.4\textwidth,trim=40 0 40 10,clip]{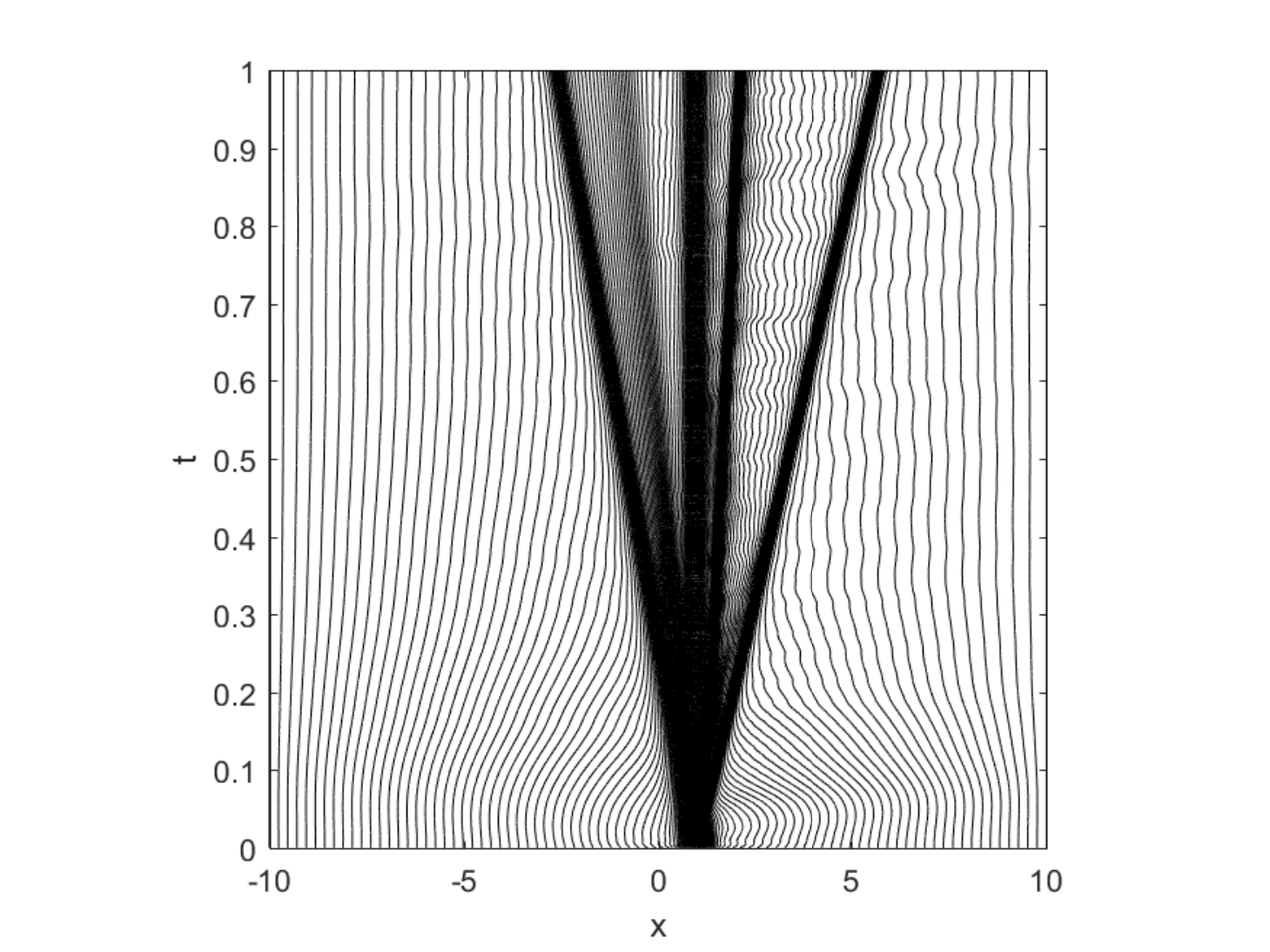}
\caption{Example \ref{test6-1d}.
The mesh trajectories (for $B$ defined in (\ref{B-3})) are obtained with $P^2$-DG and a moving mesh of $N=160$.}
\label{Fig:test6-1d-mesh}
\end{figure}

\begin{figure}[H]
\centering
\subfigure[FM 160 vs MM 160]{
\includegraphics[width=0.4\textwidth,trim=10 0 30 10,clip]{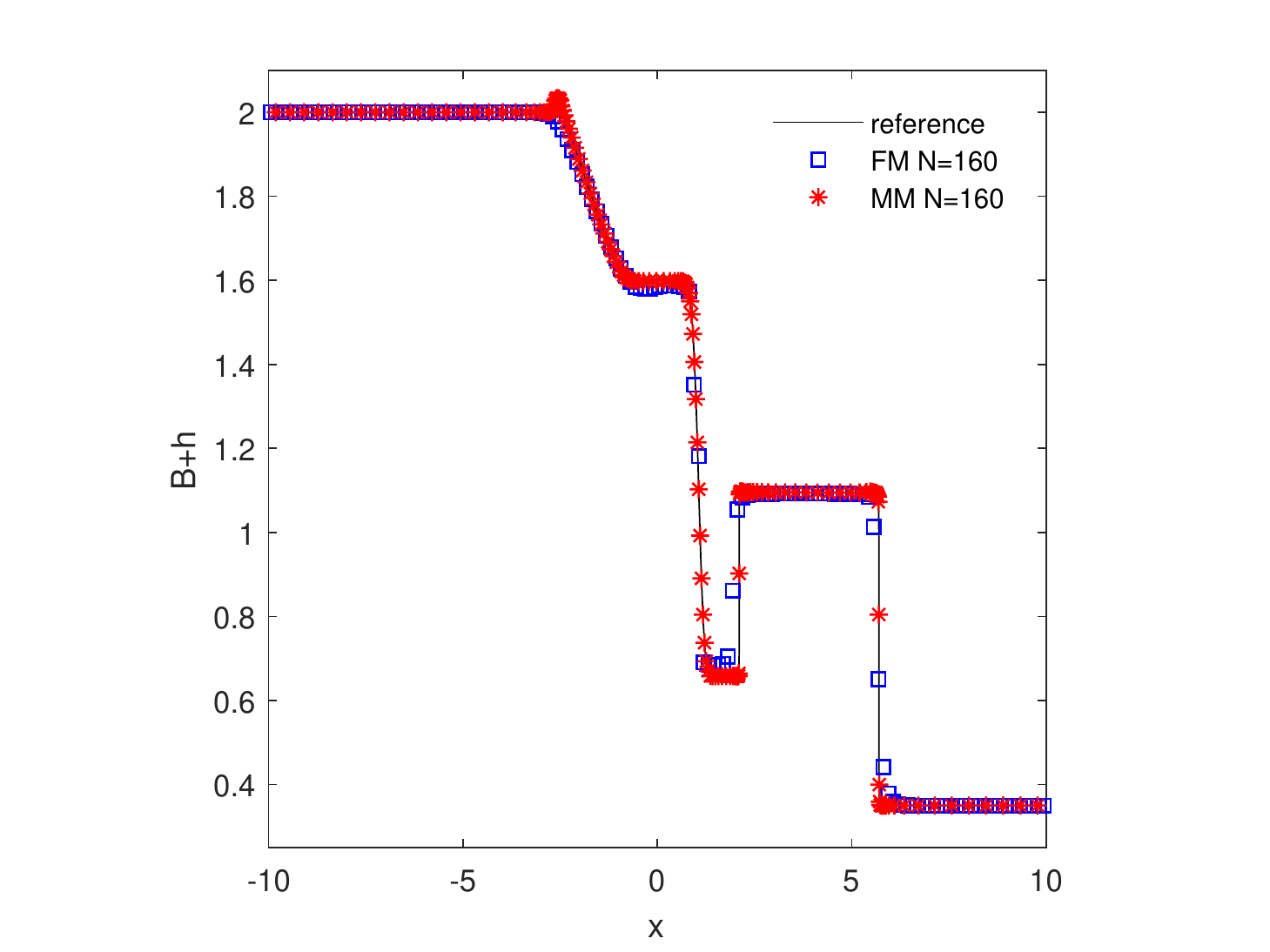}}
\subfigure[Close view of (a)]{
\includegraphics[width=0.4\textwidth,trim=10 0 30 10,clip]{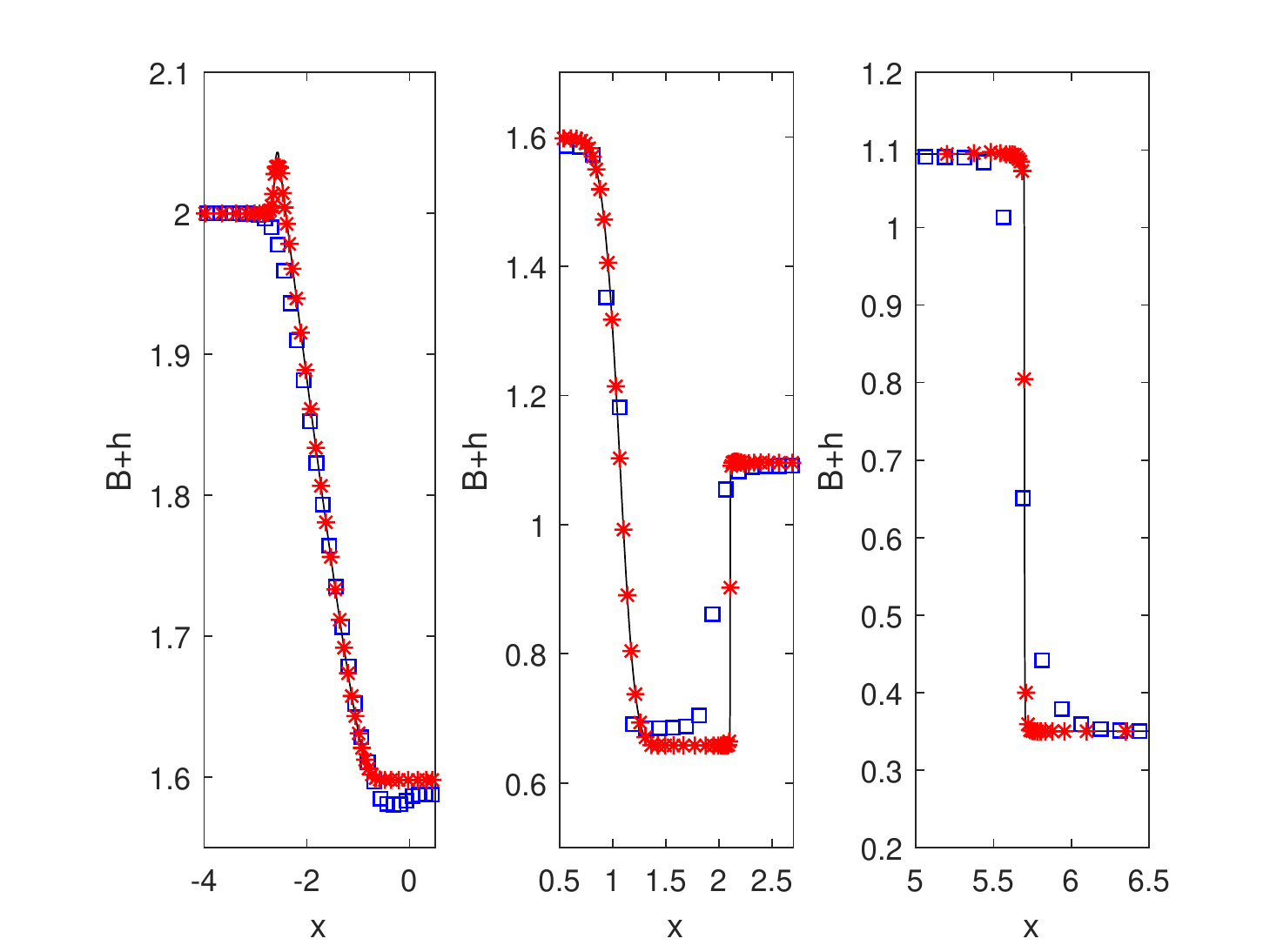}}
\subfigure[FM 1280 vs MM 160]{
\includegraphics[width=0.4\textwidth,trim=10 0 30 10,clip]{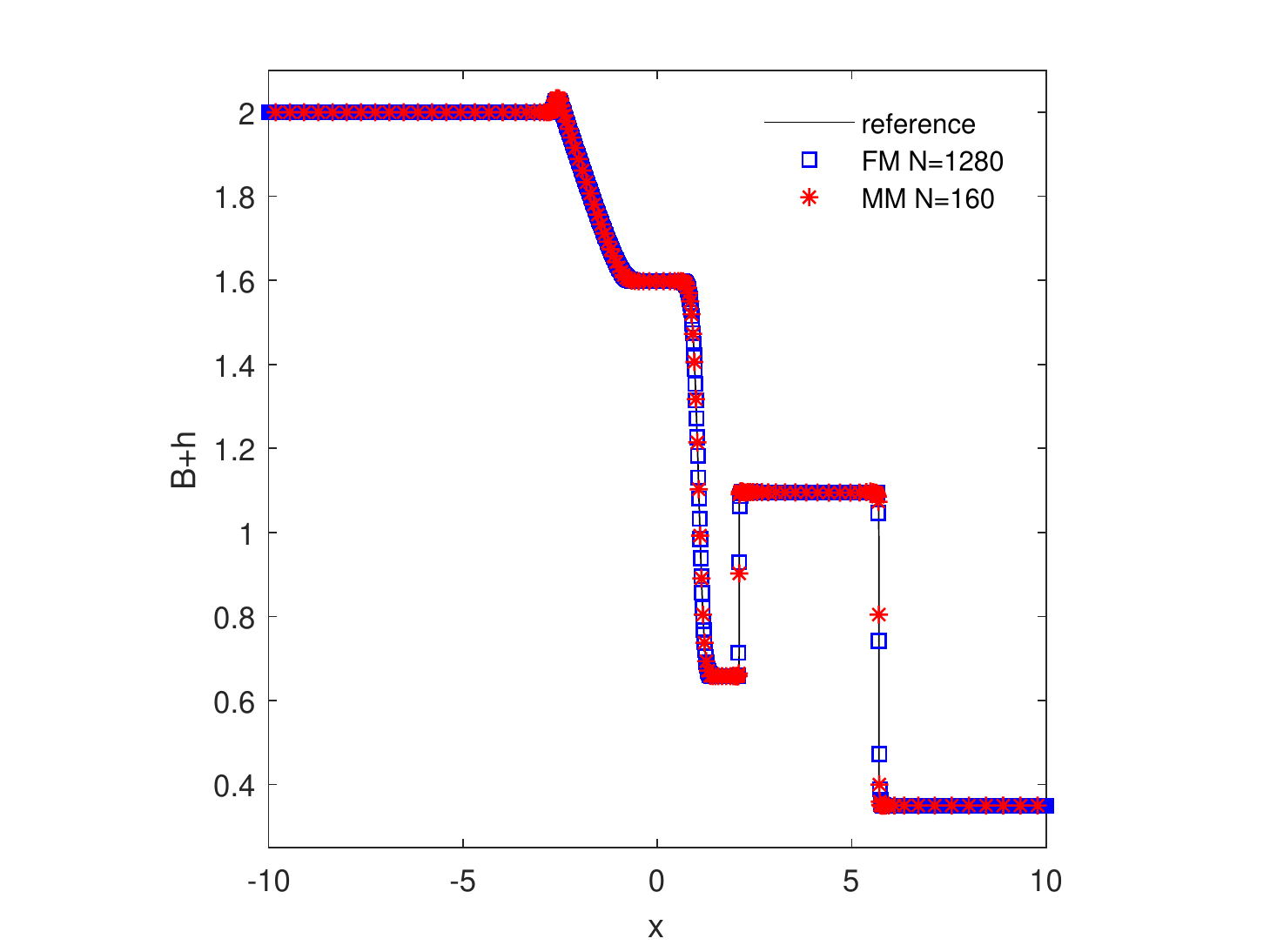}}
\subfigure[Close view of (c)]{
\includegraphics[width=0.4\textwidth,trim=10 0 30 10,clip]{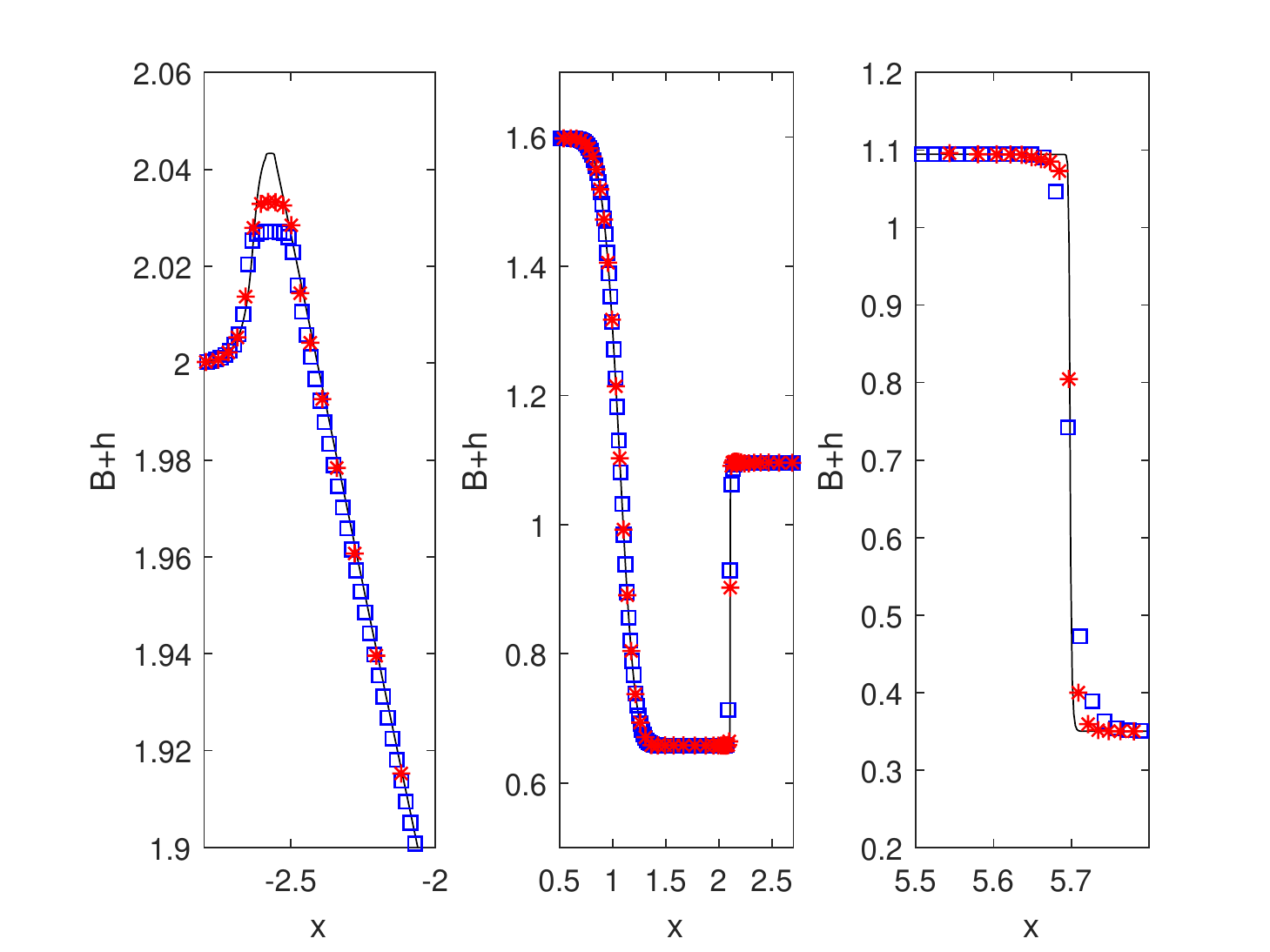}}
\caption{Example \ref{test6-1d}. The water surface level $B+h$ (for $B$ defined in (\ref{B-3})) at $t=1$ obtained with $P^2$-DG and a moving mesh of $N=160$ is compared with those obtained with fixed meshes of $N=160$ and $N=1280$.}
\label{Fig:test6-1d-Bph}
\end{figure}

\begin{figure}[H]
\centering
\subfigure[FM 160 vs MM 160]{
\includegraphics[width=0.4\textwidth,trim=10 0 30 10,clip]{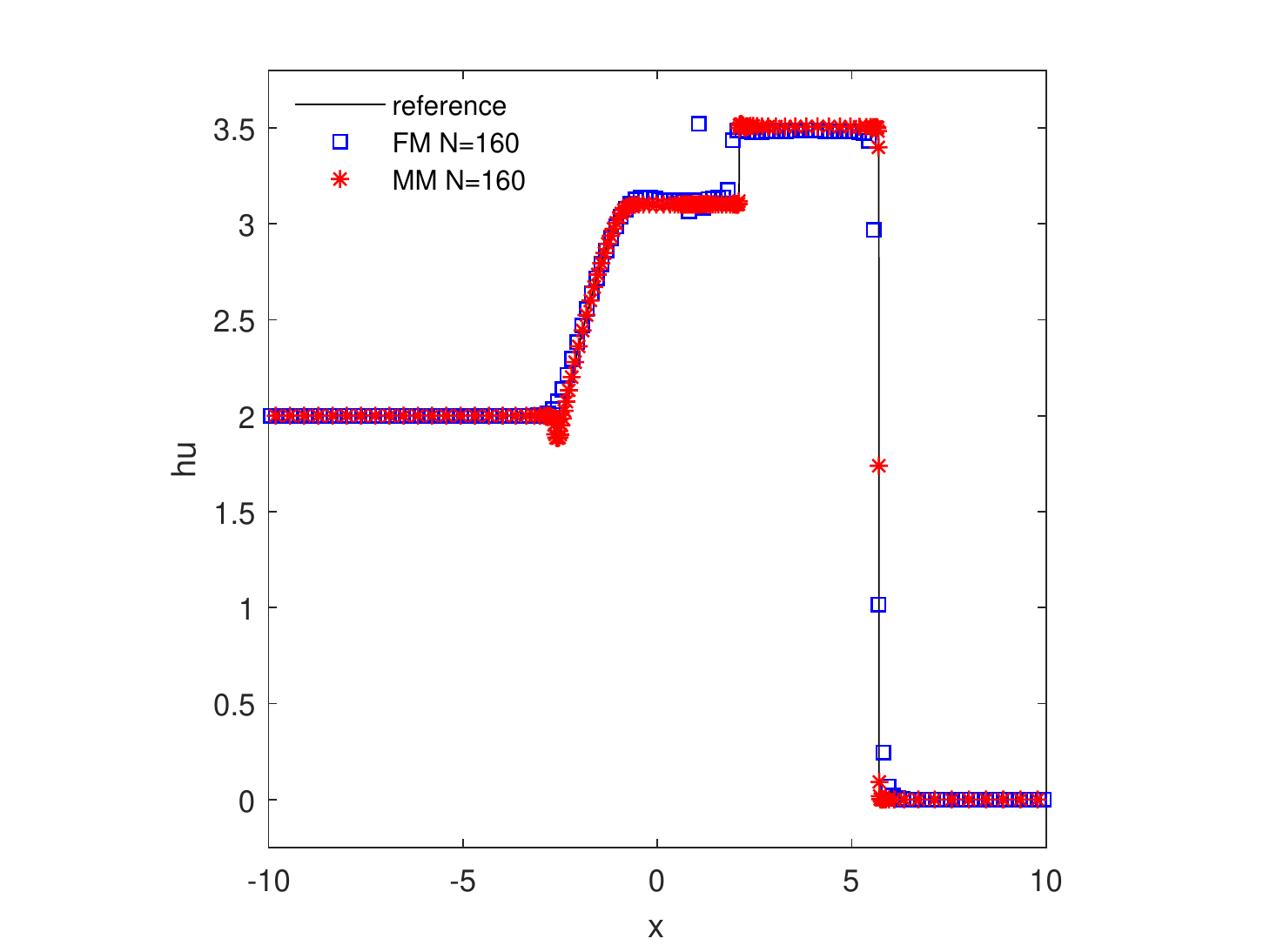}}
\subfigure[Close view of (a)]{
\includegraphics[width=0.4\textwidth,trim=10 0 30 10,clip]{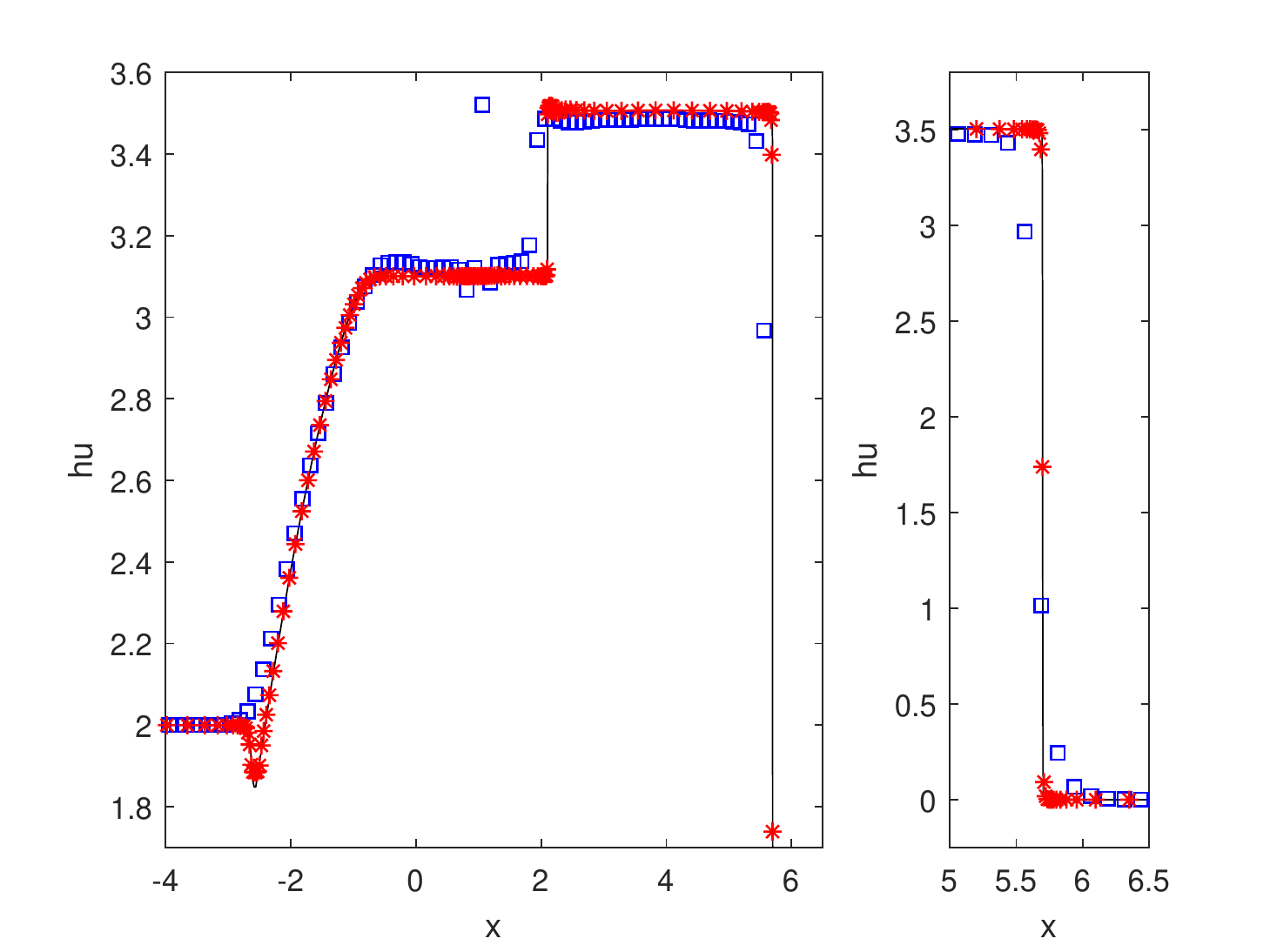}}
\subfigure[FM 1280 vs MM 160]{
\includegraphics[width=0.4\textwidth,trim=10 0 30 10,clip]{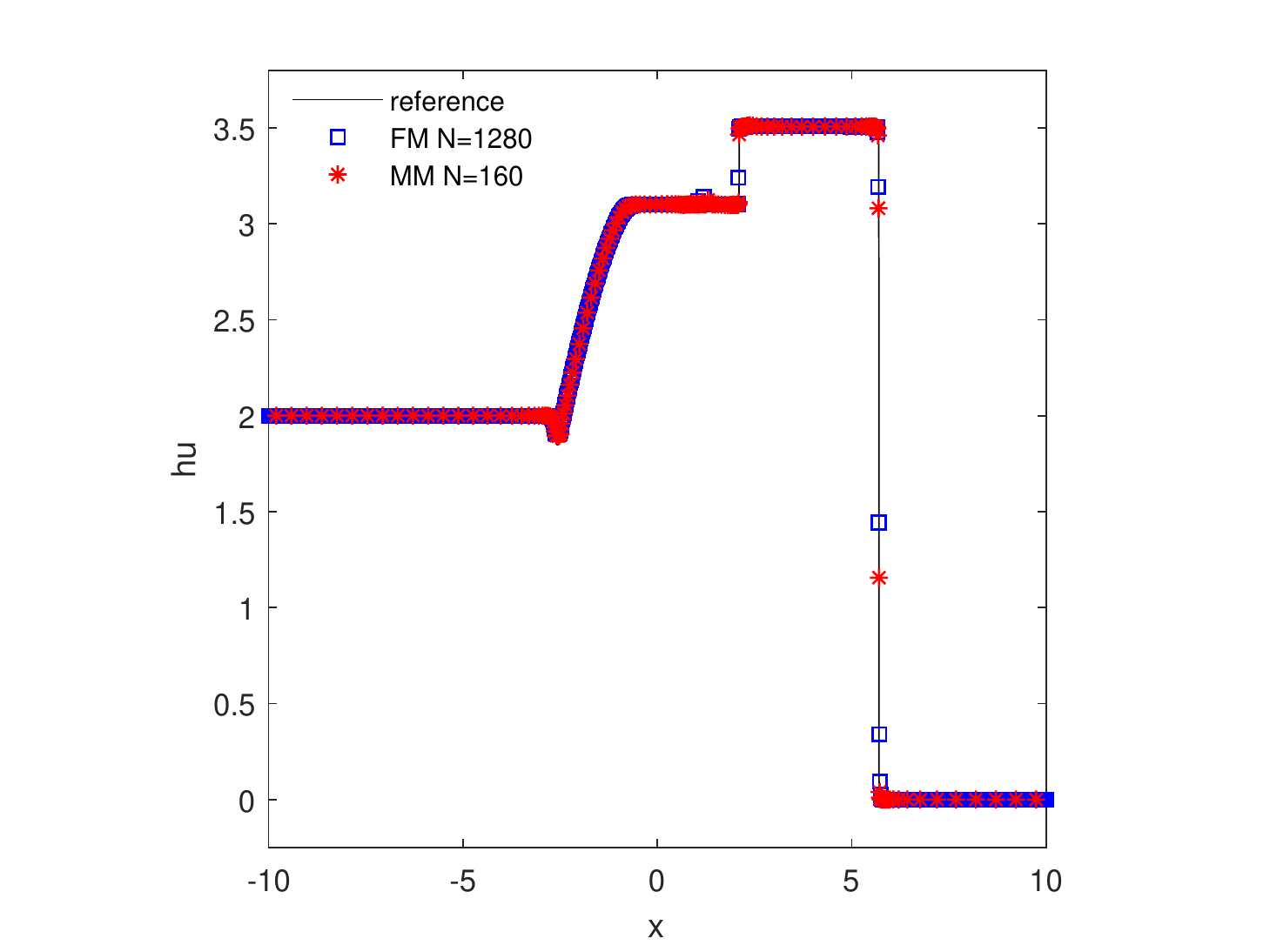}}
\subfigure[Close view of (c)]{
\includegraphics[width=0.4\textwidth,trim=10 0 30 10,clip]{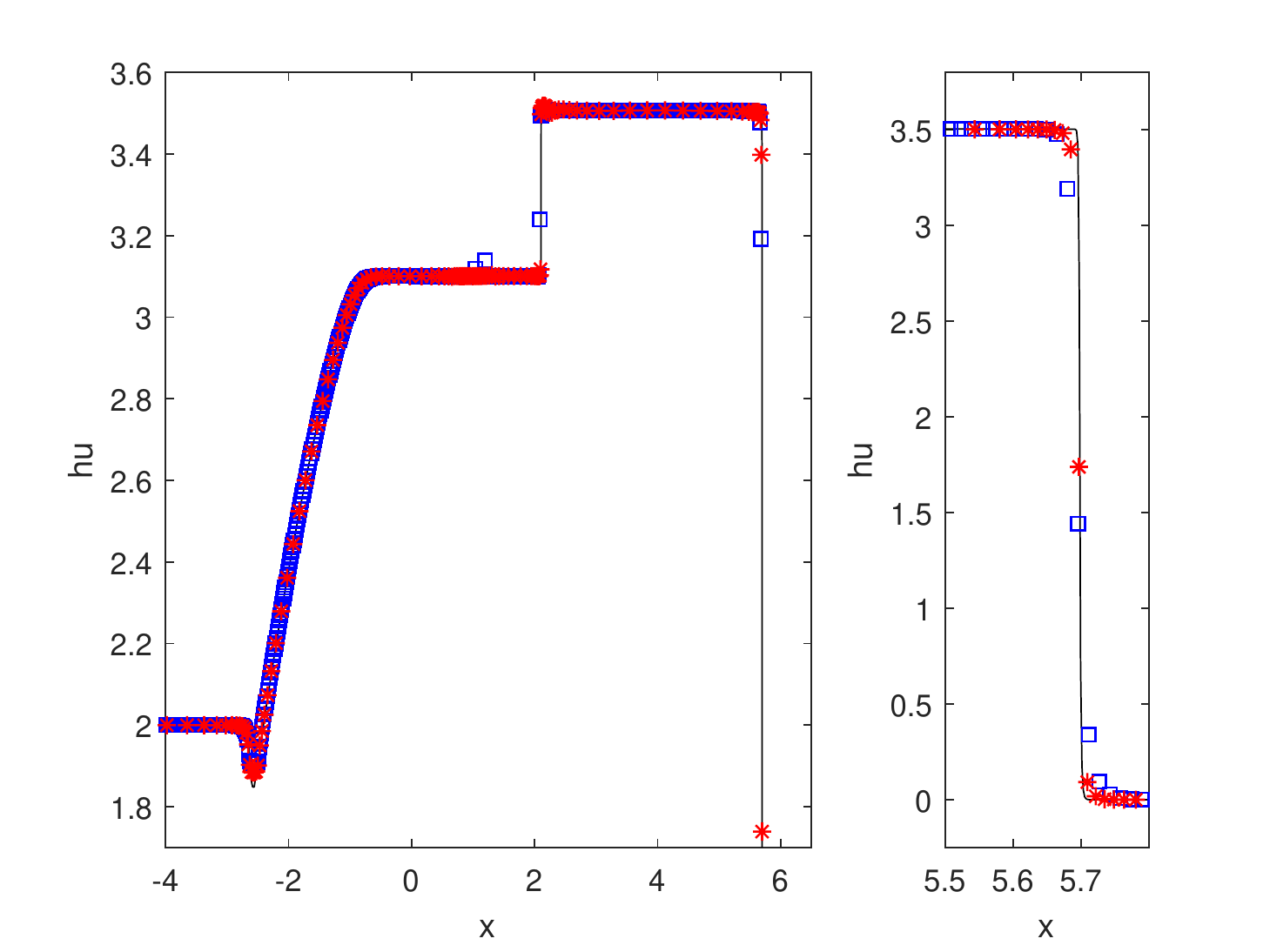}}
\caption{Example \ref{test6-1d}. The water discharge $hu$ (for $B$ defined in (\ref{B-3})) at $t=1$ obtained with $P^2$-DG and a moving mesh of $N=160$ is compared with those obtained with fixed meshes of $N=160$ and $N=1280$.}
\label{Fig:test6-1d-hu}
\end{figure}


\begin{example}\label{test1-2d}
(The lake-at-rest steady-state flow test for the 2D SWEs.)
\end{example}
We choose this example to verify the well-balance property of the MM-DG scheme
over non-flat bottom topographies. We take a bottom topography with an isolated elliptical-shaped
bump \cite{LeVeque-1998JCP} as
\begin{equation}
\label{B-4}
B(x,y)=0.8e^{-50\big((x-0.5)^2+(y-0.5)^2\big)}, \quad (x, y) \in (0,1)\times(0,1).
\end{equation}
The initial depth of water and velocities are given by
\begin{equation*}
h(x,y,0)=1-B(x,y),
\quad u(x,y,0)=0,\quad  v(x,y,0)=0.
\end{equation*}
We use periodic boundary conditions and compute the solution up to $t=0.1$ on moving meshes.
The flow surface should remain steady.
An initial triangular mesh, shown in Fig.~\ref{Fig:test-2d-tri}, is formed by dividing each cell of a rectangular mesh into four triangular elements. The $L^1$ and $L^\infty$ error for $h+B$ and discharges $hu$ and $hv$
is listed in Tables~\ref{tab:test1-2d-p1-error} and ~\ref{tab:test1-2d-p2-error} for $P^1$-DG and $P^2$-DG,
respectively. They show that the MM-DG method with $B$ updated with DG-interpolation
maintains the lake-at-rest steady state to the level of round-off error in both $L^1$ and $L^\infty$ norm.
On the other hand, the MM-DG method with $B$ updated with $L^2$-projection is not well-balanced.
The deviation from the lake-at-rest steady state is about first-order, which attributes to the use
of the TVB limiter with $M_{tvb} = 0$ that is about first-order in two dimensions \cite{DG-series5,Goodman-LeVeque-1985MC}.

\begin{figure}[H]
\centering
\includegraphics[width=0.4\textwidth,trim=40 0 40 10,clip]{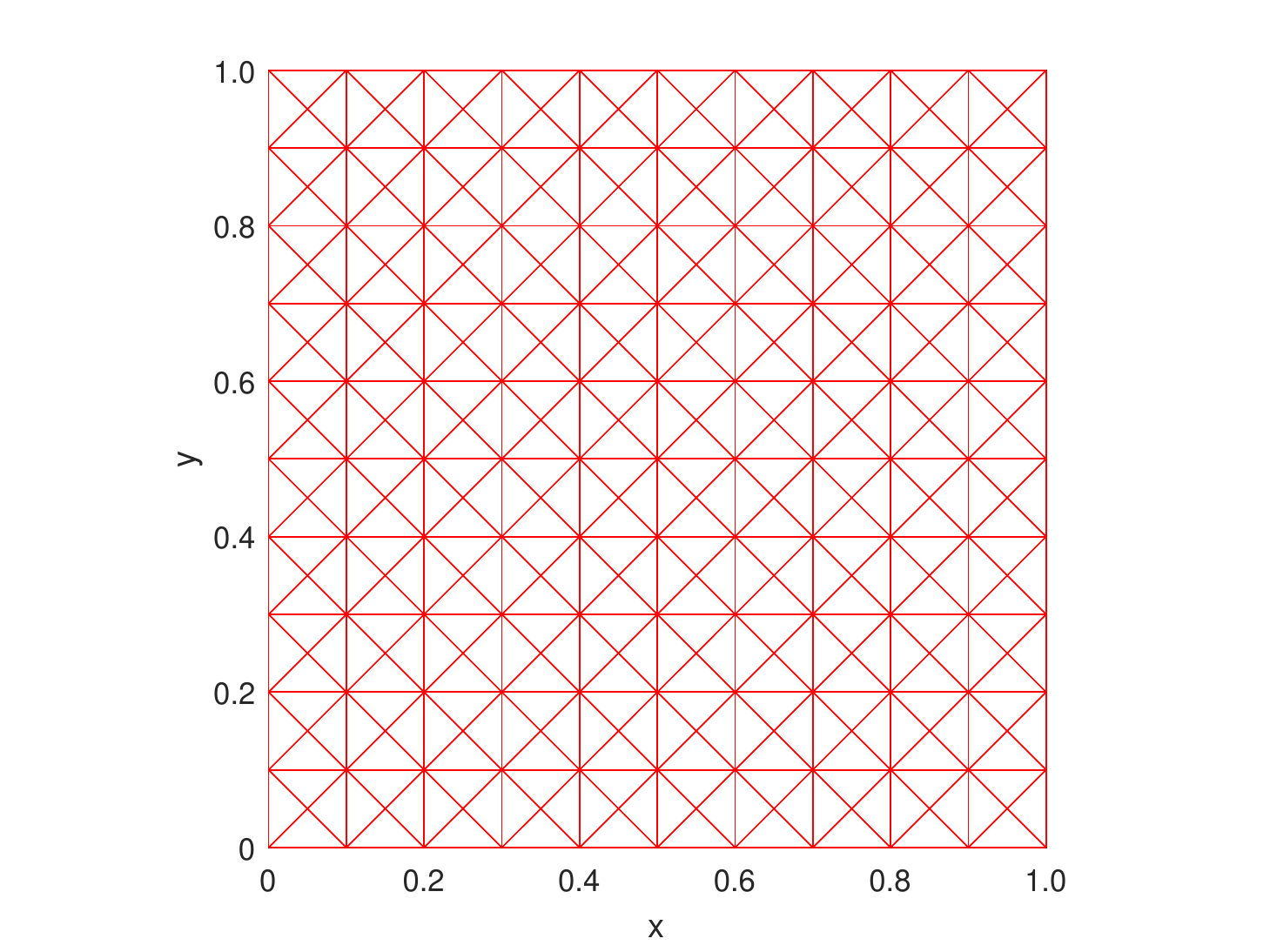}
\caption{Example \ref{test1-2d}.  An initial triangular mesh used in the computation is formed by dividing each cell of a rectangular mesh into 4 triangular elements. }
\label{Fig:test-2d-tri}
\end{figure}

\begin{table}[H]
\caption{Example \ref{test1-2d}. Well-balance test for the moving mesh $P^1$-DG method over an isolated elliptical shaped hump bottom topography \eqref{B-4}.}
\vspace{3pt}
\centering
\label{tab:test1-2d-p1-error}
\begin{tabular}{ccccccc}
 \toprule
  & \multicolumn{2}{c}{$h+B$}&\multicolumn{2}{c}{$hu$}&\multicolumn{2}{c}{$hv$}\\
$N$&$L^1$-error  &$L^{\infty}$-error  &$L^1$-error  &$L^{\infty}$-error &$L^1$-error  &$L^{\infty}$-error \\
\midrule
~   &  \multicolumn{6}{c}{\em{$B$ updated with DG-interpolation} }\\
\midrule
$10\times10\times4$	&1.589E-16&	3.077E-16	&1.359E-16&	8.941E-16	&1.361E-16& 8.588E-16	\\
$20\times20\times4$	&2.179E-16&	4.641E-16	&1.775E-16&	1.658E-15	&1.787E-16& 1.627E-15	\\
$40\times40\times4$	&3.344E-16&	8.298E-16	&2.793E-16&	3.069E-15	&2.813E-16& 3.054E-15	\\
\midrule
~   &  \multicolumn{6}{c}{\em{$B$ updated with $L^2$-projection} }\\
\midrule
$10\times10\times4$	&8.658E-06&	8.903E-05	&1.567E-05&	1.510E-04	&1.562E-05&	1.527E-04	\\
$20\times20\times4$	&4.444E-06&	5.940E-05	&8.208E-06&	7.290E-05	&8.291E-06&	7.005E-05	\\
$40\times40\times4$	&1.717E-06&	2.262E-05	&3.054E-06&	2.975E-05	&1.120E-06&	3.063E-05	\\
 \bottomrule	
\end{tabular}
\end{table}

\begin{table}[H]
\caption{Example \ref{test1-2d}. Well-balance test for the moving mesh $P^2$-DG method over an isolated elliptical shaped hump bottom topography (\ref{B-4}).}
\vspace{3pt}
\centering
\label{tab:test1-2d-p2-error}
\begin{tabular}{ccccccc}
 \toprule
  & \multicolumn{2}{c}{$h+B$}&\multicolumn{2}{c}{$hu$}&\multicolumn{2}{c}{$hv$}\\
$N$&$L^1$-error  &$L^{\infty}$-error  &$L^1$-error  &$L^{\infty}$-error &$L^1$-error  &$L^{\infty}$-error \\
\midrule
~   &  \multicolumn{6}{c}{\em{$B$ updated with DG-interpolation} }\\
\midrule
$10\times10\times4$	&4.259E-16&	1.855E-15	&8.005E-16&	4.425E-15	&8.105E-16&	4.288E-15	\\
$20\times20\times4$	&3.870E-16&	1.944E-15	&9.724E-16&	6.313E-15	&9.104E-16&	5.361E-15	\\
$40\times40\times4$	&4.257E-16&	2.158E-15	&1.467E-15&	1.129E-14	&1.204E-15&	8.449E-15	\\
\midrule
~   &  \multicolumn{6}{c}{\em{$B$ updated with $L^2$-projection} }\\
\midrule
$10\times10\times4$	&6.587E-07&	8.775E-06	&1.635E-06&	1.181E-05	&1.641E-06&	1.144E-05	\\
$20\times20\times4$	&1.114E-08&	1.555E-06	&2.827E-07&	2.472E-06	&2.827E-07&	2.449E-06	\\
$40\times40\times4$	&1.996E-08&	3.701E-07	&5.212E-08&	6.350E-07	&5.213E-08&	6.400E-07	\\
 \bottomrule	
\end{tabular}
\end{table}

\begin{example}\label{test4-2d}
(The perturbed lake-at-rest steady-state flow test for the 2D SWEs.)
\end{example}
We use this example, first proposed by LeVeque \cite{LeVeque-1998JCP}, to demonstrate the ability
of our well-balanced MM-DG scheme to capture small perturbations over the lake-at-rest water surface.
The bottom topography is an isolated elliptical shaped hump,
\begin{equation*}
B(x,y)=0.8e^{-5(x-0.9)^2-50(y-0.5)^2}, \quad (x,y) \in (-1,2)\times(0,1).
\end{equation*}
The initial depth of water and velocities are given by
\begin{equation*}
\begin{split}
&h(x,y,0)=\begin{cases}
1-B(x,y)+0.01,& \hbox{for~$x\in(0.05 , 0.15)$}\\
1-B(x,y),&\hbox{otherwise}\\
\end{cases}
\\&u(x,y,0)=0,\quad \quad v(x,y,0)=0.
\end{split}
\end{equation*}
As time being, the initial perturbation splits into two waves, propagating left and right
at the characteristic speeds $\pm \sqrt{gh}$.
The reflection boundary conditions \cite{Tumolo-2013JCP} are used for all domain boundary.

The mesh and contours of $B+h$, $hu$, and $hv$ at $t= 0.12$, $0.24$, $0.36$, and $0.48$
obtained with the $P^2$ MM-DG method and a moving mesh of $N=150\times 50\times4$ are
shown in Figs.~\ref{Fig:test4-2d-comparison-12}~--~\ref{Fig:test4-2d-comparison-48}.
Recall that these results are obtained with the metric tensor (\ref{mer-ceil}) based on
the equilibrium variable $\mathcal{E}=\frac{1}{2}(u^2+v^2)+g(h+B) $ and the water depth $h$.
For comparison purpose, we also include the results obtained the metric tensor based on
the entropy/total energy $E =\frac{1}{2}(hu^2+hv^2)+\frac{1}{2}gh^2+ghB$.
One can see that the distribution of the mesh concentration based on $\mathcal{E}$ and $h$
is consistent with the contours of $B+h$, $hu$, and $hv$, i.e., the mesh elements have higher concentration
in the regions where those variables have larger variations. This demonstrates the ability
of the MM-DG method to capture complex features in small perturbations.
On the other hand, the concentration
of the mesh based on $E$ does not reflect fully the variations of $B+h$, $hu$, and $hv$
especially at $t = 0.24$, $0.36$, and $0.48$.
The advantages of using $\mathcal{E}$ and $h$ over $E$ in mesh adaptation are clear.

The contours of $h+B$, $hu$, and $hv$ at $t=0.12$, 0.24, 0.36, and 0.48 obtained
with the $P^2$ MM-DG method of a moving mesh of $N=150\times 50\times4$ and fixed meshes of $N=150\times 50\times4$ and $N=600\times 200\times4$ are shown in
Figs.~\ref{Fig:test4-2d-h-hu-hv-t12} -- \ref{Fig:test4-2d-h-hu-hv-t48}.
We can see that the moving mesh solution with $N=150\times 50\times4$ is more accurate than that with
a fixed mesh of $N=150\times 50\times4$ and comparable with that with a fixed mesh of  $N=600\times 200\times4$.


\begin{figure}[H]
\centering
\subfigure[$\mathcal{E}$ and $h$: Mesh at $t=0.12$]{
\includegraphics[width=0.45\textwidth, trim=20 60 15 60, clip]
{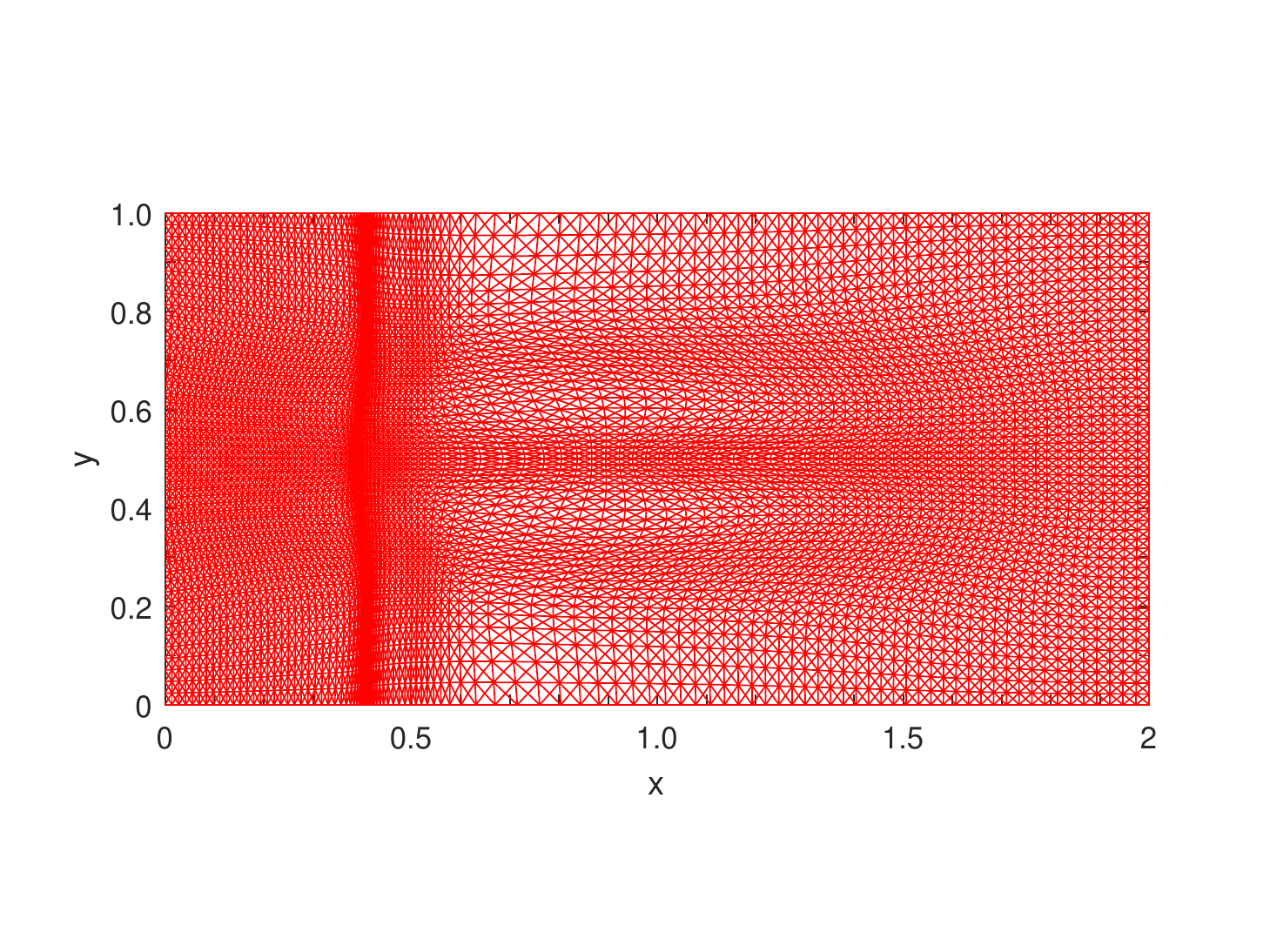}}
\subfigure[Entropy: Mesh at $t=0.12$]{
\includegraphics[width=0.45\textwidth, trim=20 60 15 60, clip]
{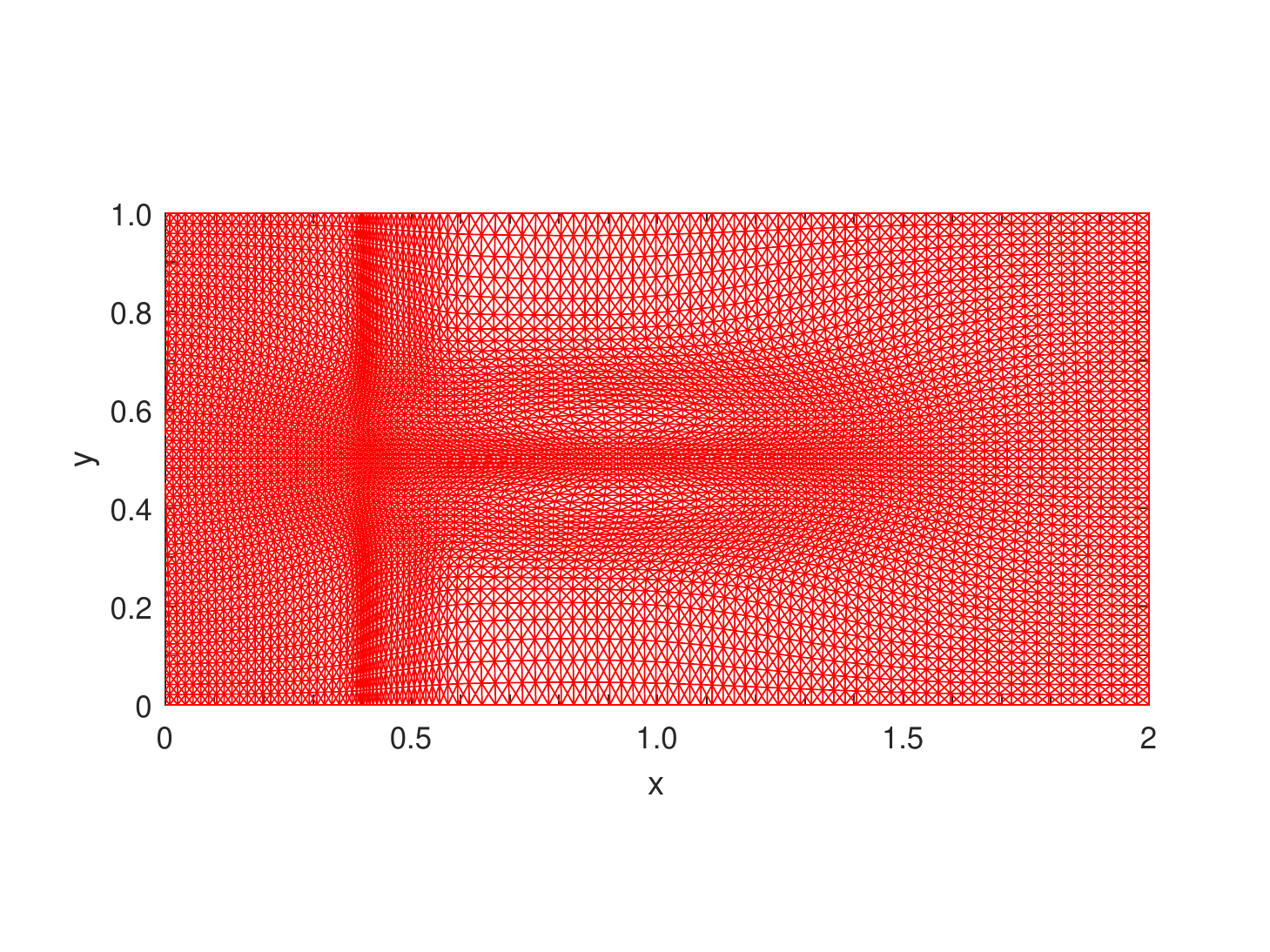}}
\subfigure[$\mathcal{E}$ and $h$: $h+B$ at $t=0.12$]{
\includegraphics[width=0.45\textwidth, trim=15 60 15 60, clip]
{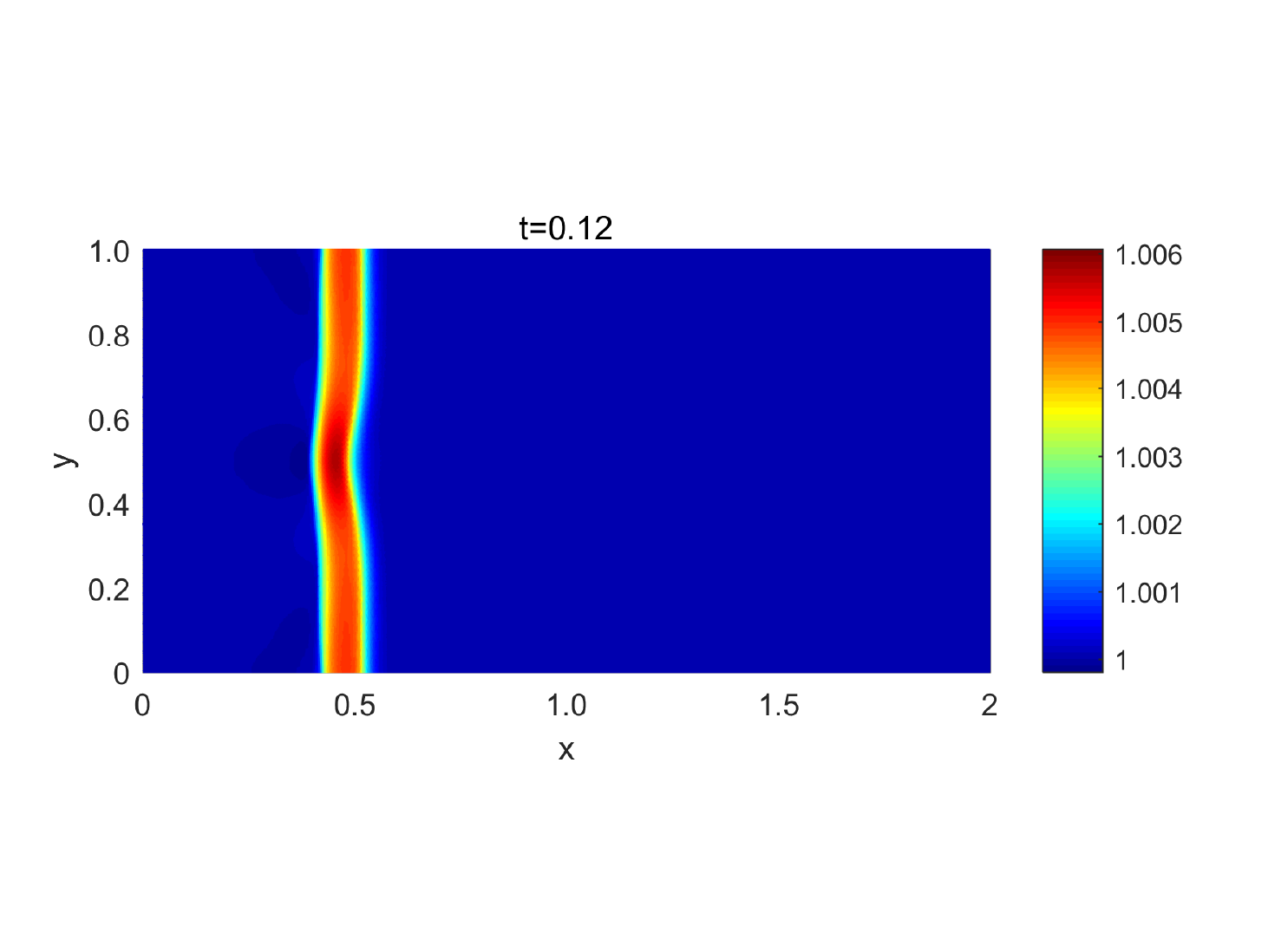}}
\subfigure[Entropy: $h+B$ at $t=0.12$]{
\includegraphics[width=0.45\textwidth, trim=15 60 15 60, clip]
{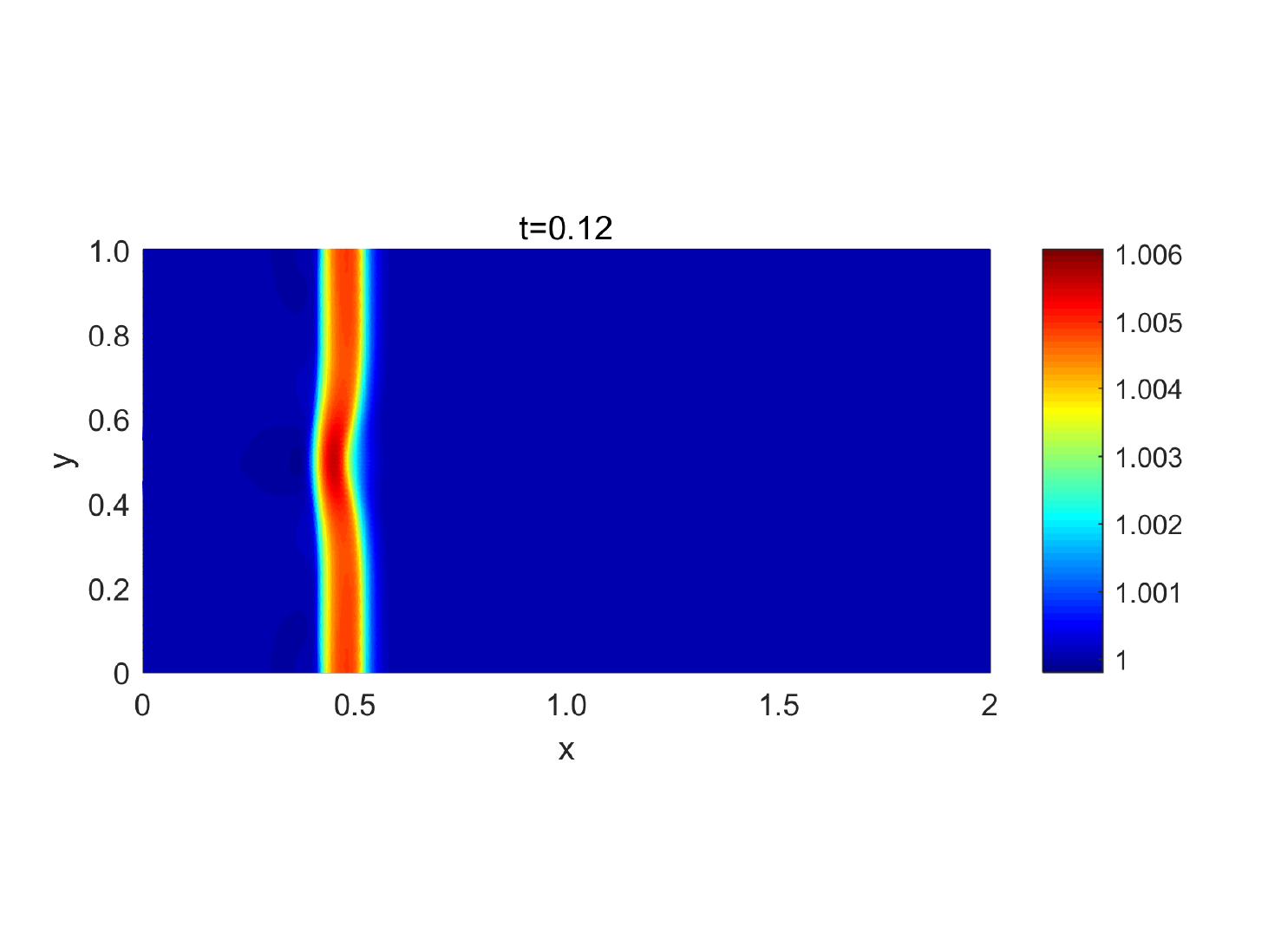}}
\subfigure[$\mathcal{E}$ and $h$: $hu$ at $t=0.12$]{
\includegraphics[width=0.45\textwidth, trim=20 50 0 60, clip]
{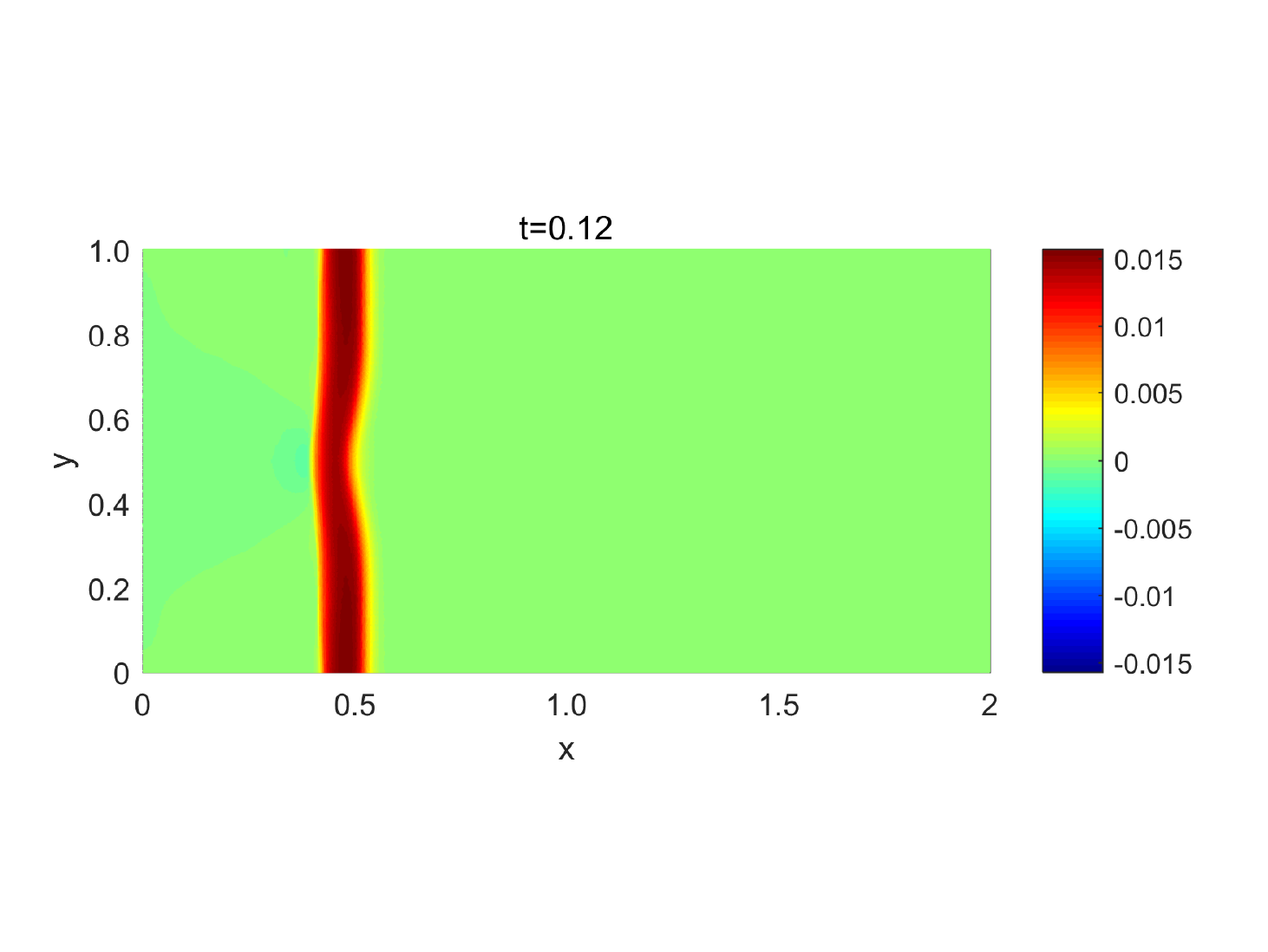}}
\subfigure[Entropy: $hu$ at $t=0.12$]{
\includegraphics[width=0.45\textwidth, trim=15 60 15 60, clip]
{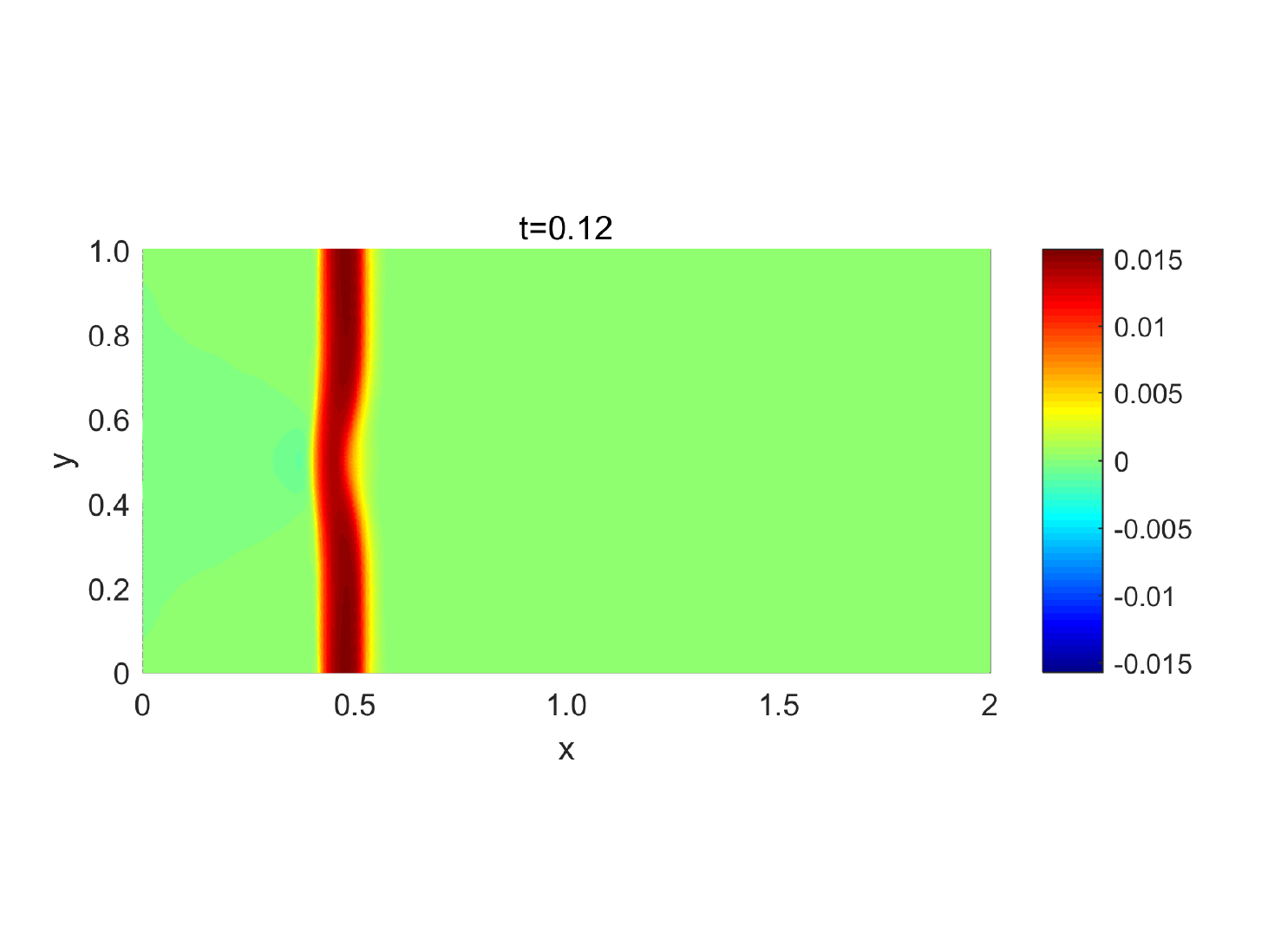}}
\subfigure[$\mathcal{E}$ and $h$: $hv$ at $t=0.12$]{
\includegraphics[width=0.45\textwidth, trim=15 60 15 60, clip]
{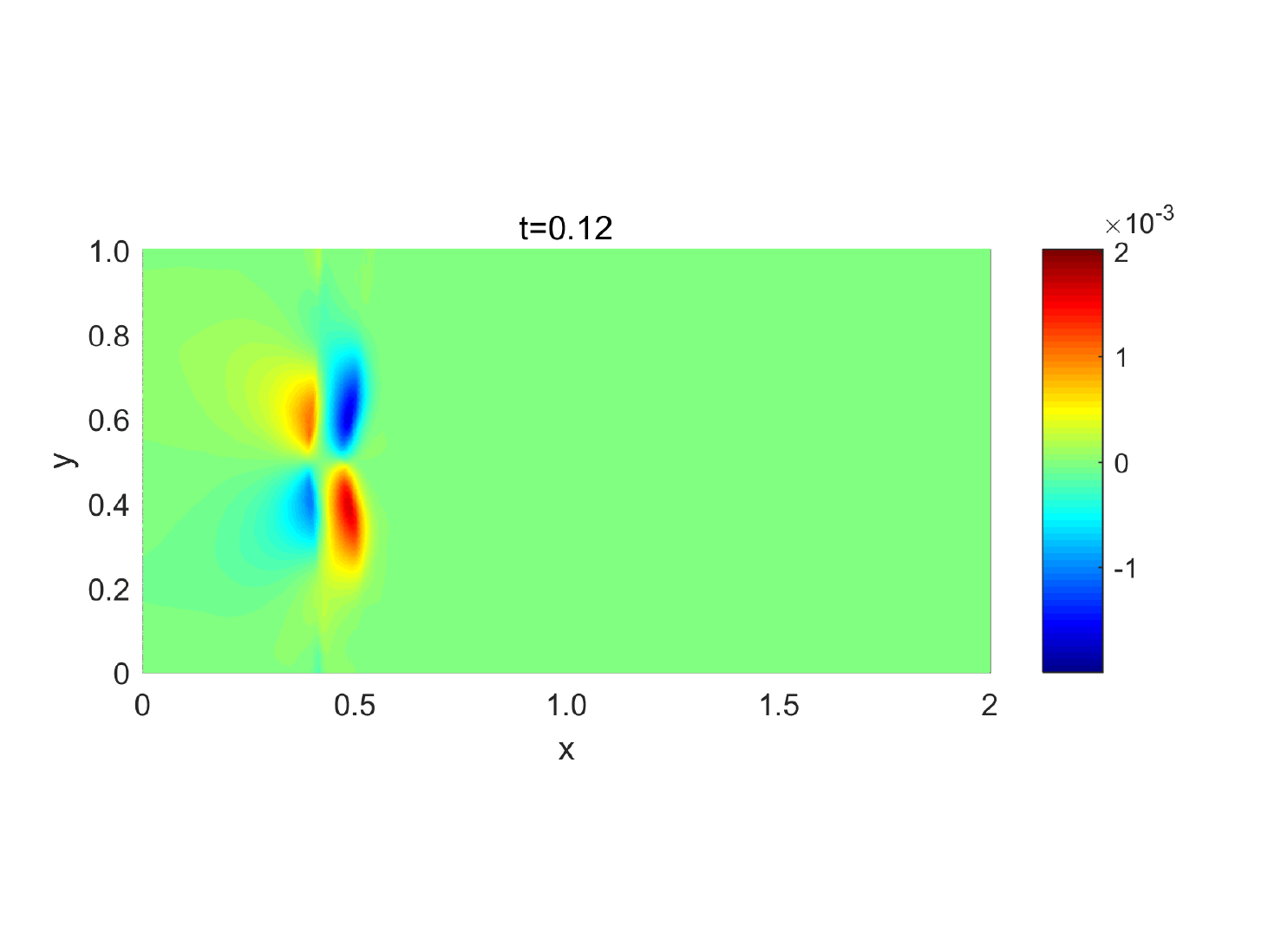}}
\subfigure[Entropy: $hv$ at $t=0.12$]{
\includegraphics[width=0.45\textwidth, trim=15 60 15 60, clip]
{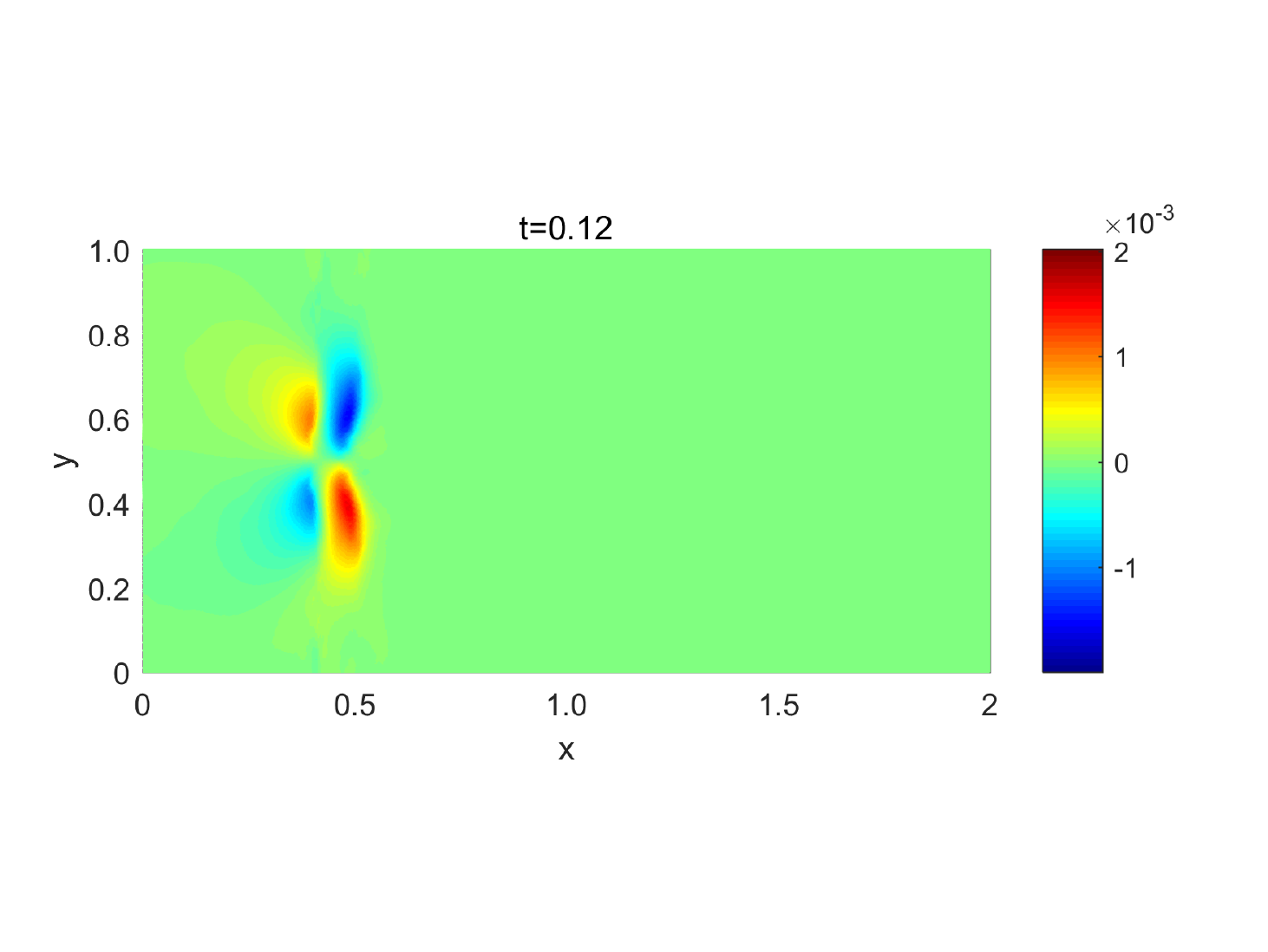}}
\caption{Example \ref{test4-2d}. The mesh and contours of $B+h$, $hu$, and $hv$ at $t=0.12$ are obtained
with the $P^2$ MM-DG method and a moving mesh of $N=150\times 50\times4$. The metric tensor
is computed based on the equilibrium variable $\mathcal{E}$ and the water depth $h$ (left column)
or the entropy/total energy (right column).}
\label{Fig:test4-2d-comparison-12}
\end{figure}

\begin{figure}[H]
\centering
\subfigure[$\mathcal{E}$ and $h$: Mesh at $t=0.24$]{
\includegraphics[width=0.45\textwidth, trim=20 50 0 60, clip]
{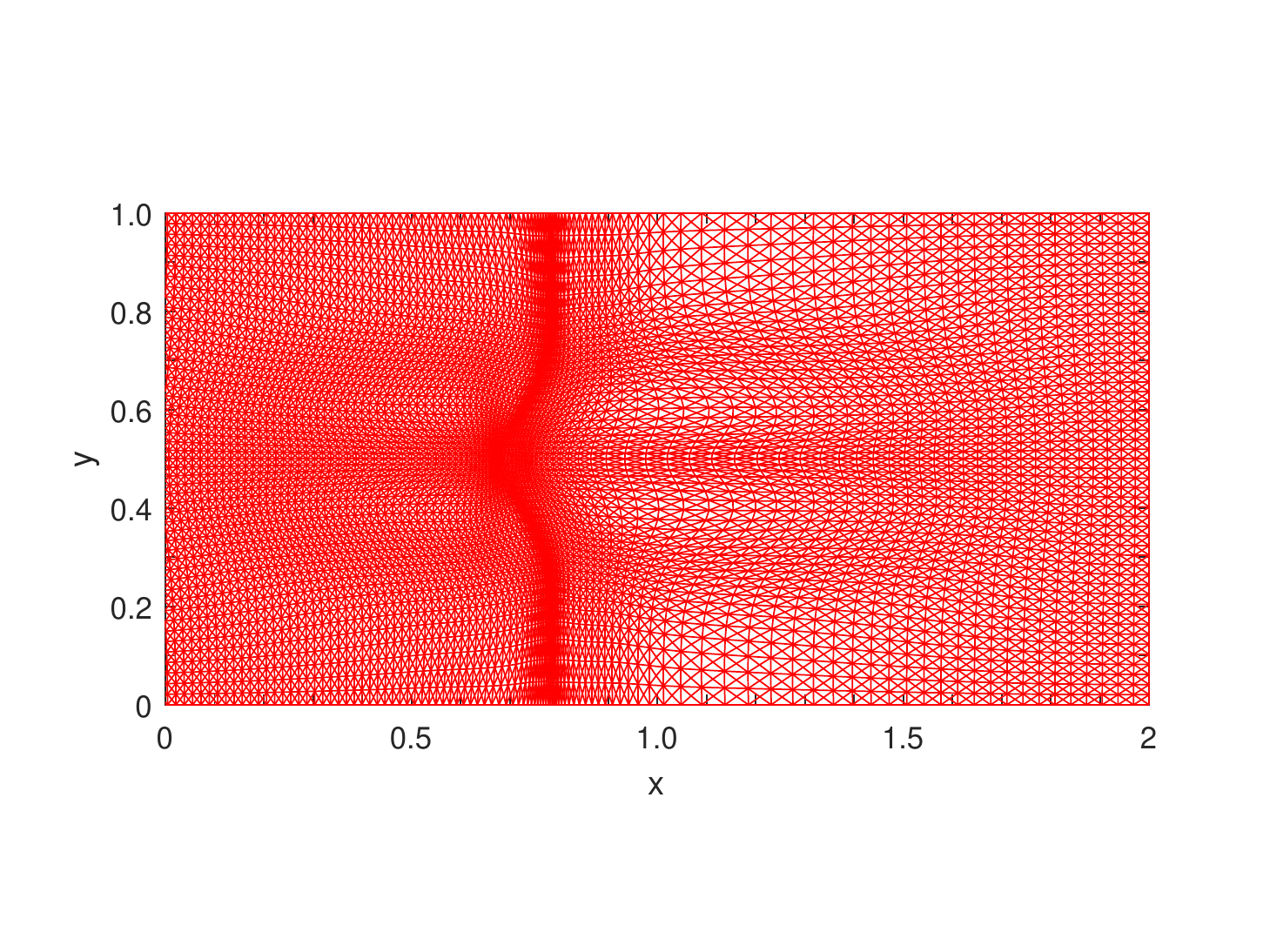}}
\subfigure[Entropy: Mesh at $t=0.24$]{
\includegraphics[width=0.45\textwidth, trim=15 60 15 60, clip]
{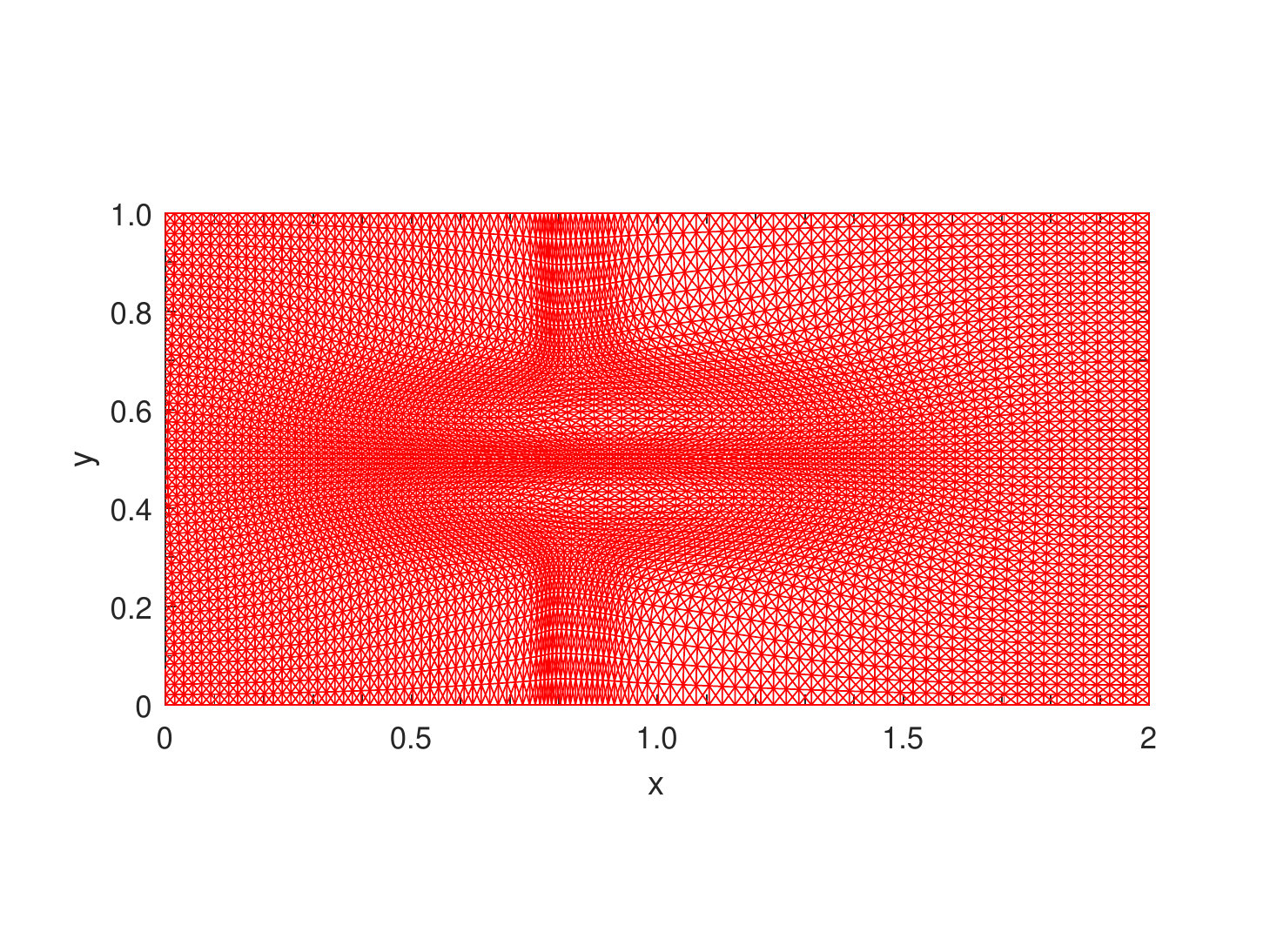}}
\subfigure[$\mathcal{E}$ and $h$: $h+B$ at $t=0.24$]{
\includegraphics[width=0.45\textwidth, trim=15 60 15 60, clip]
{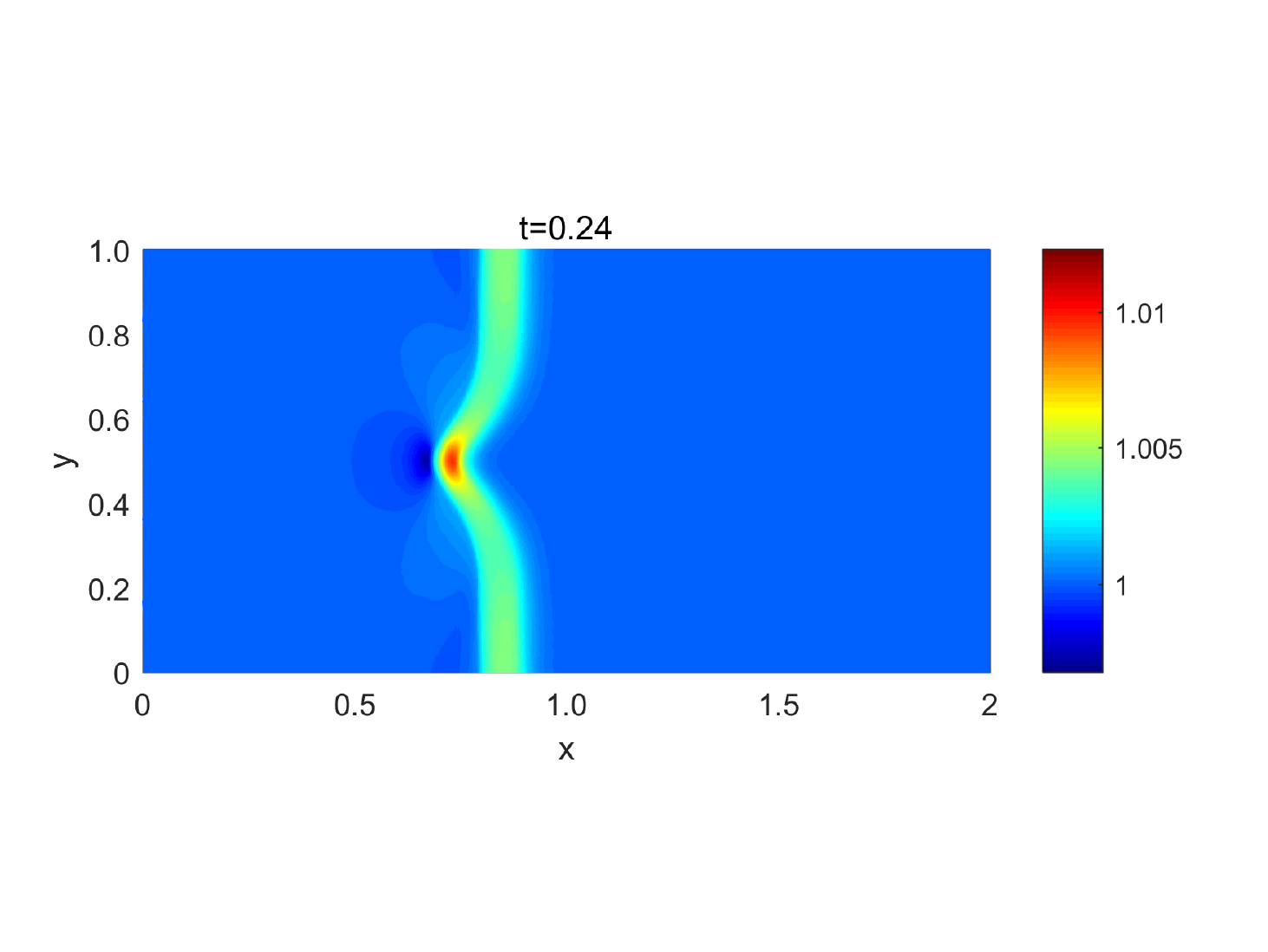}}
\subfigure[Entropy: $h+B$ at $t=0.24$]{
\includegraphics[width=0.45\textwidth, trim=15 60 15 60, clip]
{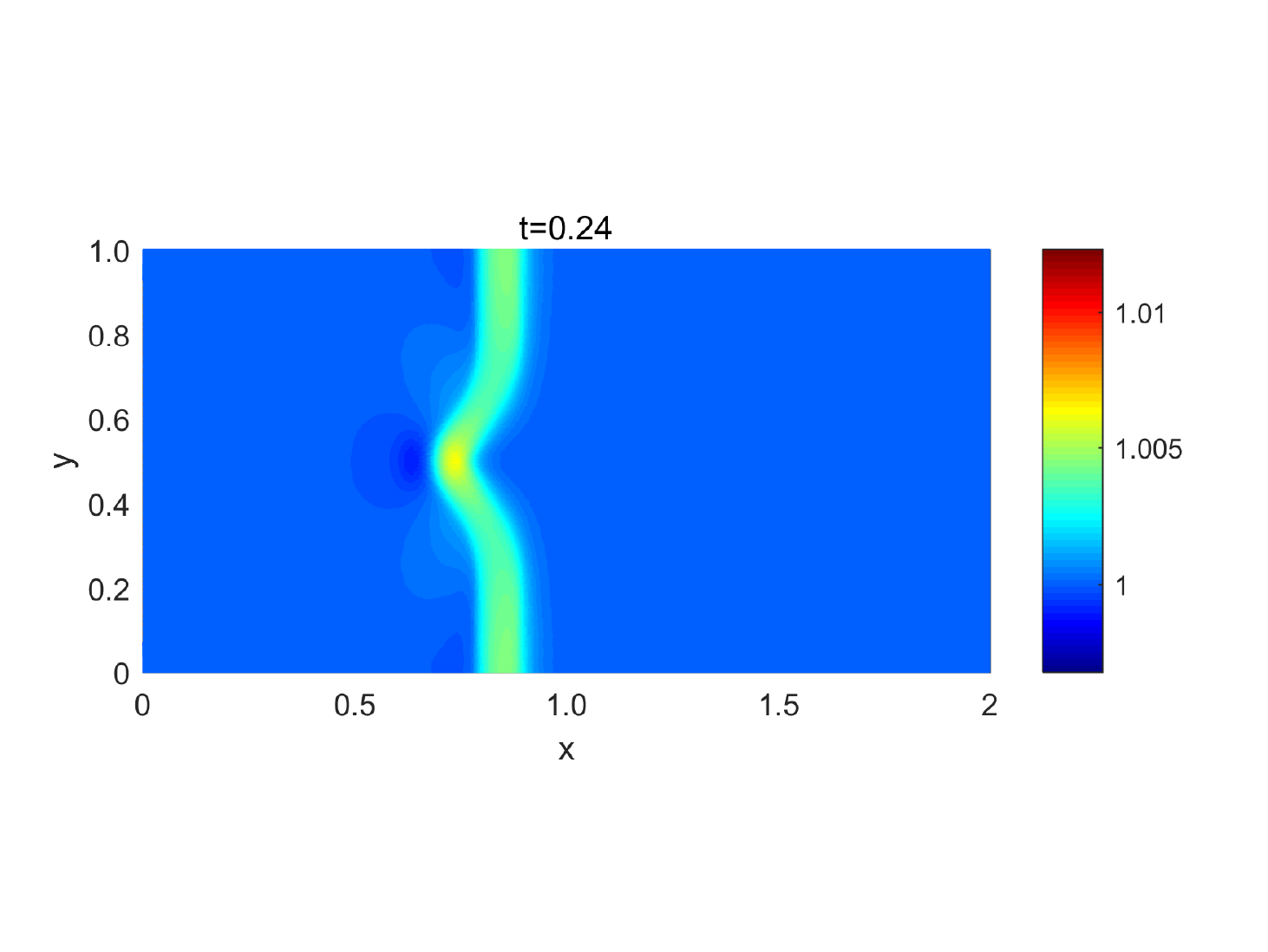}}
\subfigure[$\mathcal{E}$ and $h$: $hu$ at $t=0.24$]{
\includegraphics[width=0.45\textwidth, trim=20 50 0 60, clip]
{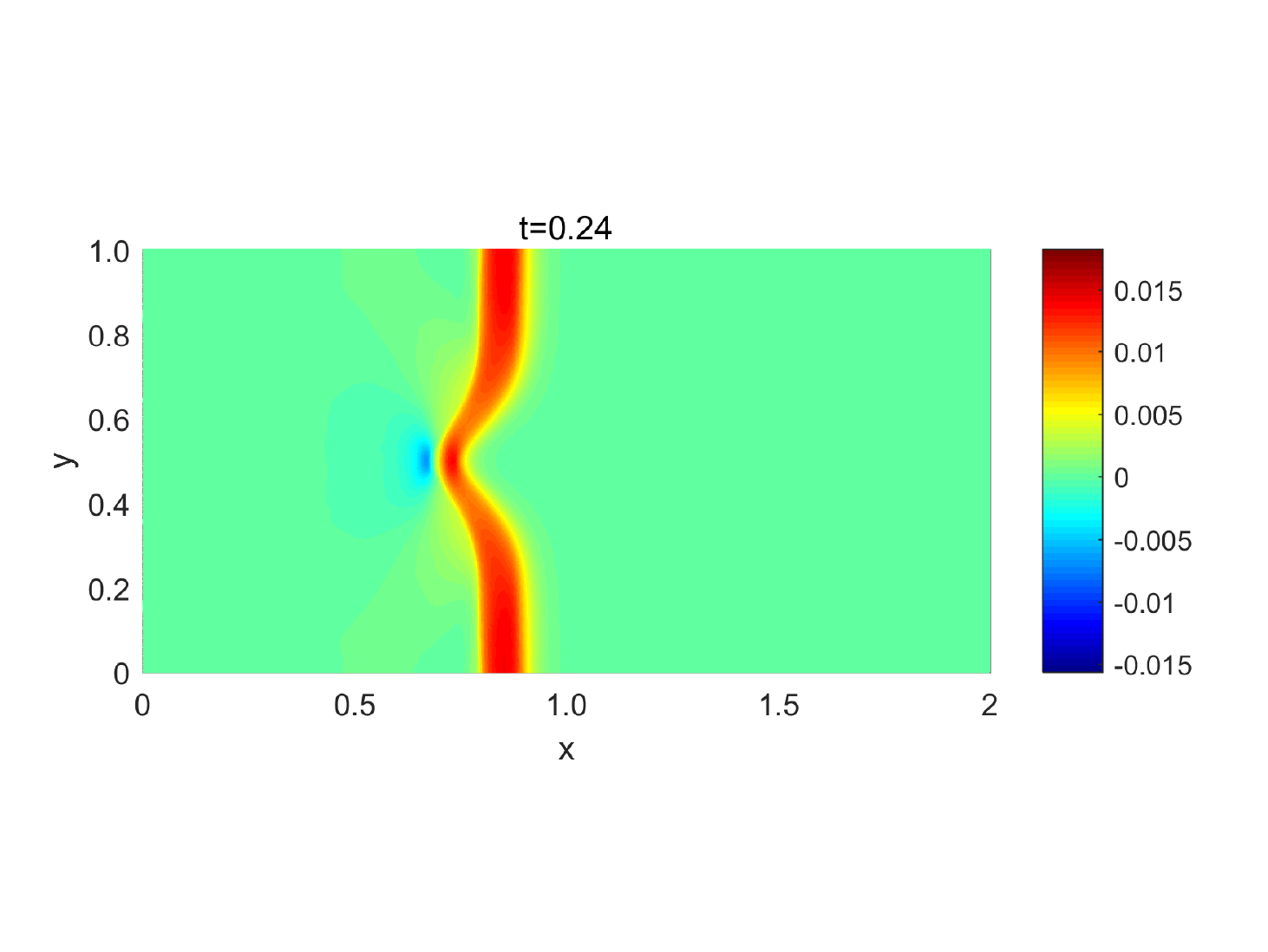}}
\subfigure[Entropy: $hu$ at $t=0.24$]{
\includegraphics[width=0.45\textwidth, trim=15 60 15 60, clip]
{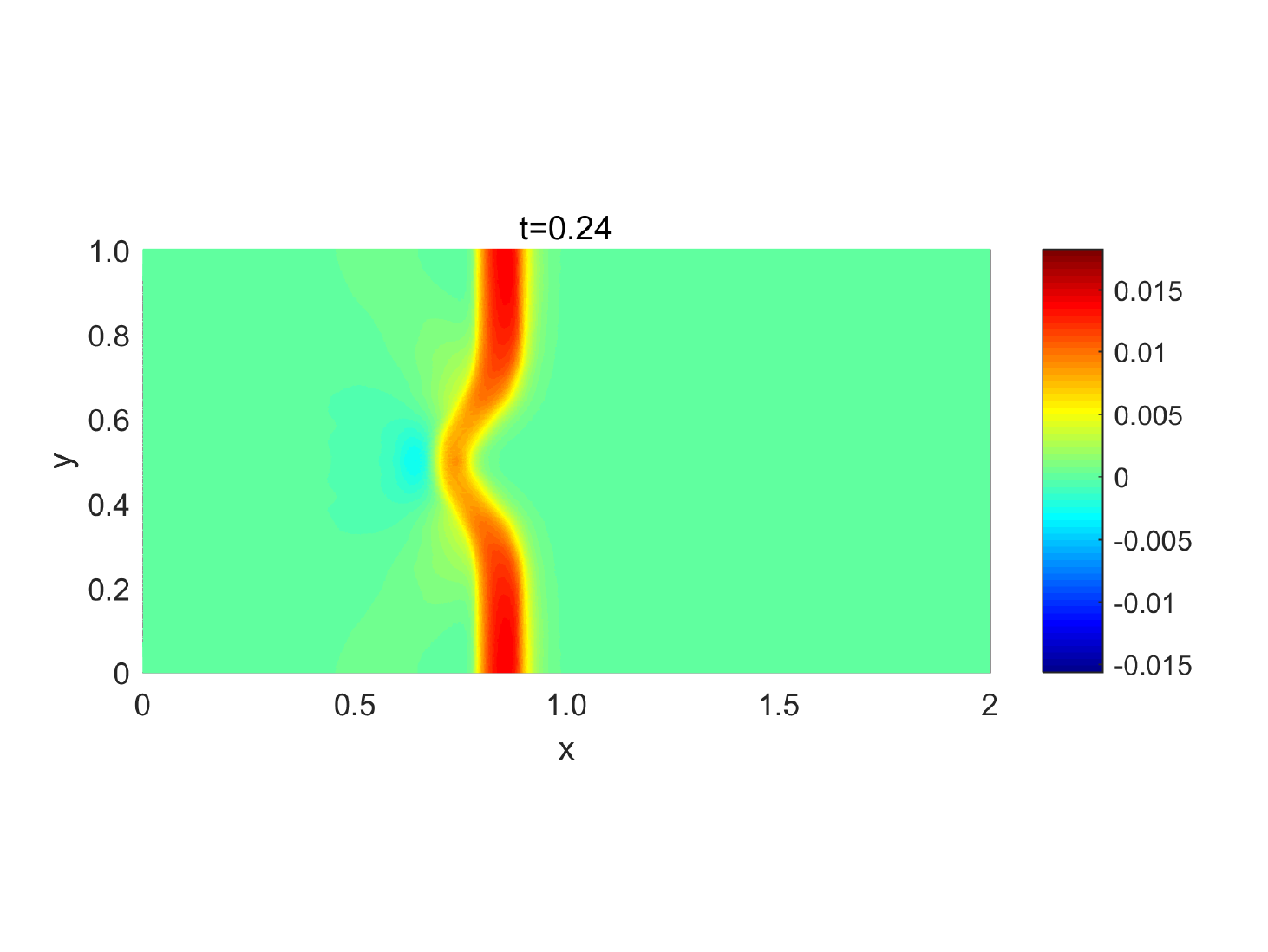}}
\subfigure[$\mathcal{E}$ and $h$: $hv$ at $t=0.24$]{
\includegraphics[width=0.45\textwidth, trim=15 60 15 60, clip]
{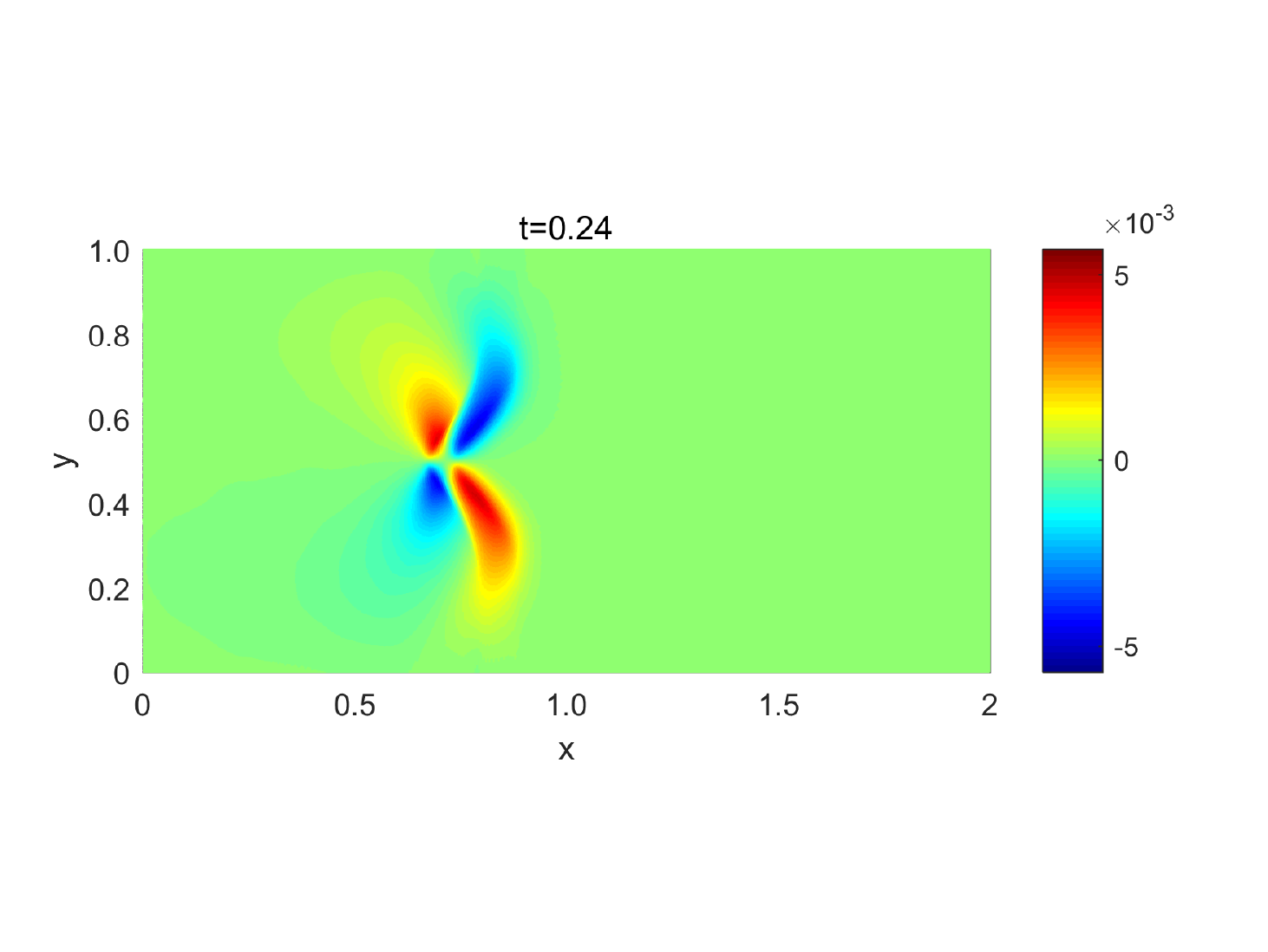}}
\subfigure[Entropy: $hv$ at $t=0.24$]{
\includegraphics[width=0.45\textwidth, trim=15 60 15 60, clip]
{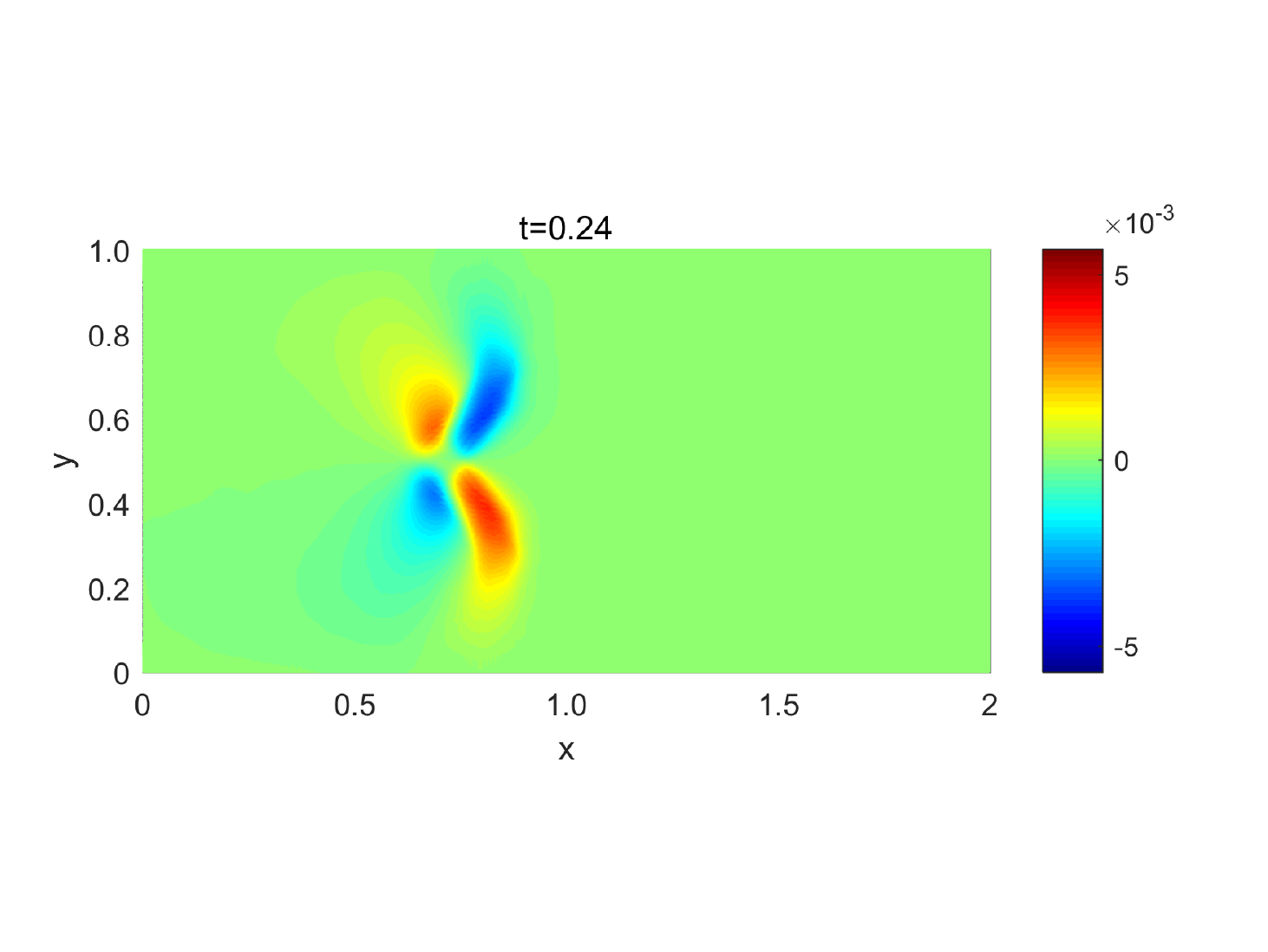}}
\caption{Continuation of Fig.~\ref{Fig:test4-2d-comparison-12}: $t = 0.24$.}
\label{Fig:test4-2d-comparison-24}
\end{figure}

\begin{figure}[H]
\centering
\subfigure[$\mathcal{E}$ and $h$: Mesh at $t=0.36$]{
\includegraphics[width=0.45\textwidth, trim=20 50 0 60, clip]
{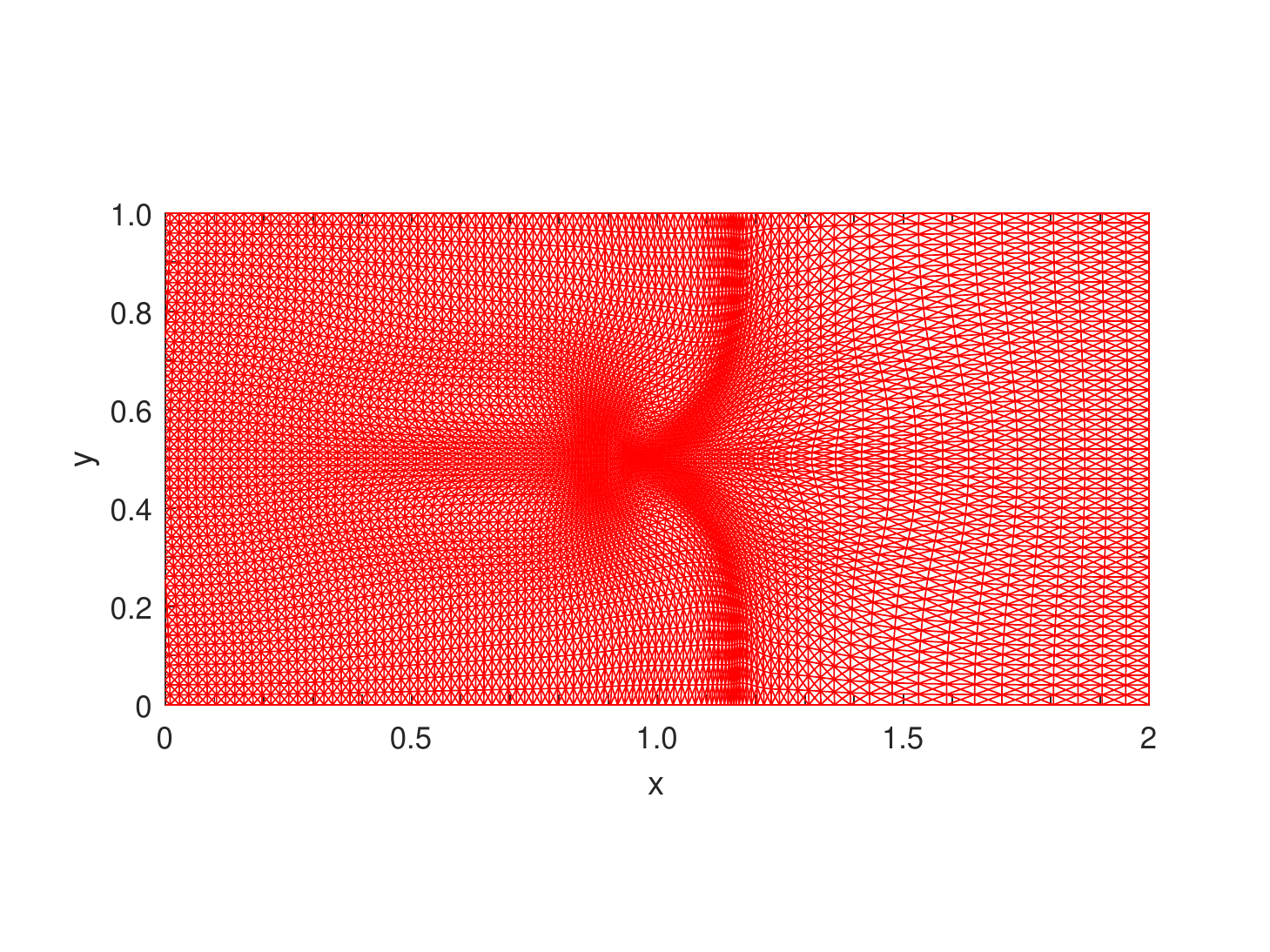}}
\subfigure[Entropy: Mesh at $t=0.36$]{
\includegraphics[width=0.45\textwidth, trim=15 60 15 60, clip]
{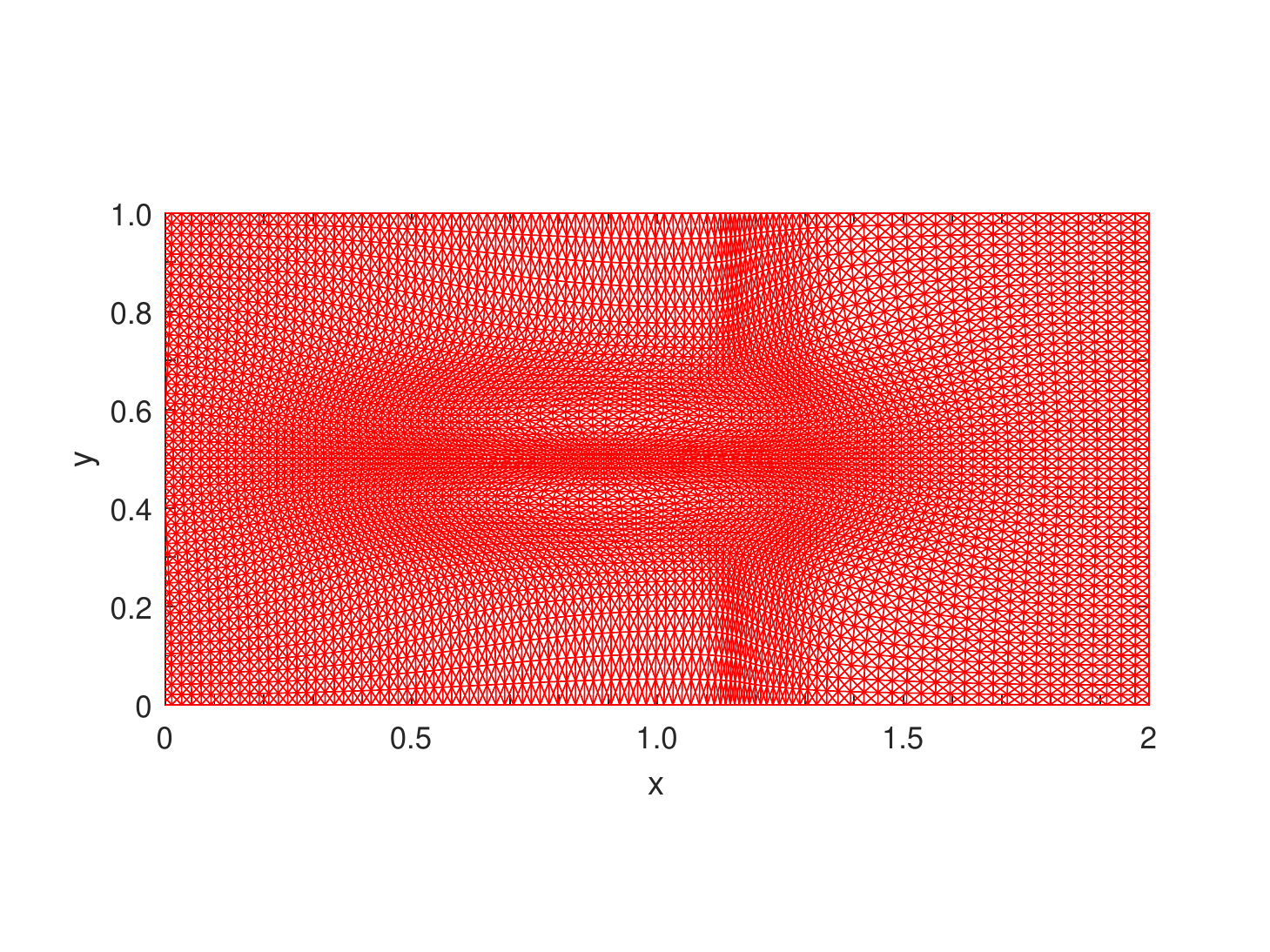}}
\subfigure[$\mathcal{E}$ and $h$: $h+B$ at $t=0.36$]{
\includegraphics[width=0.45\textwidth, trim=15 60 15 60, clip]
{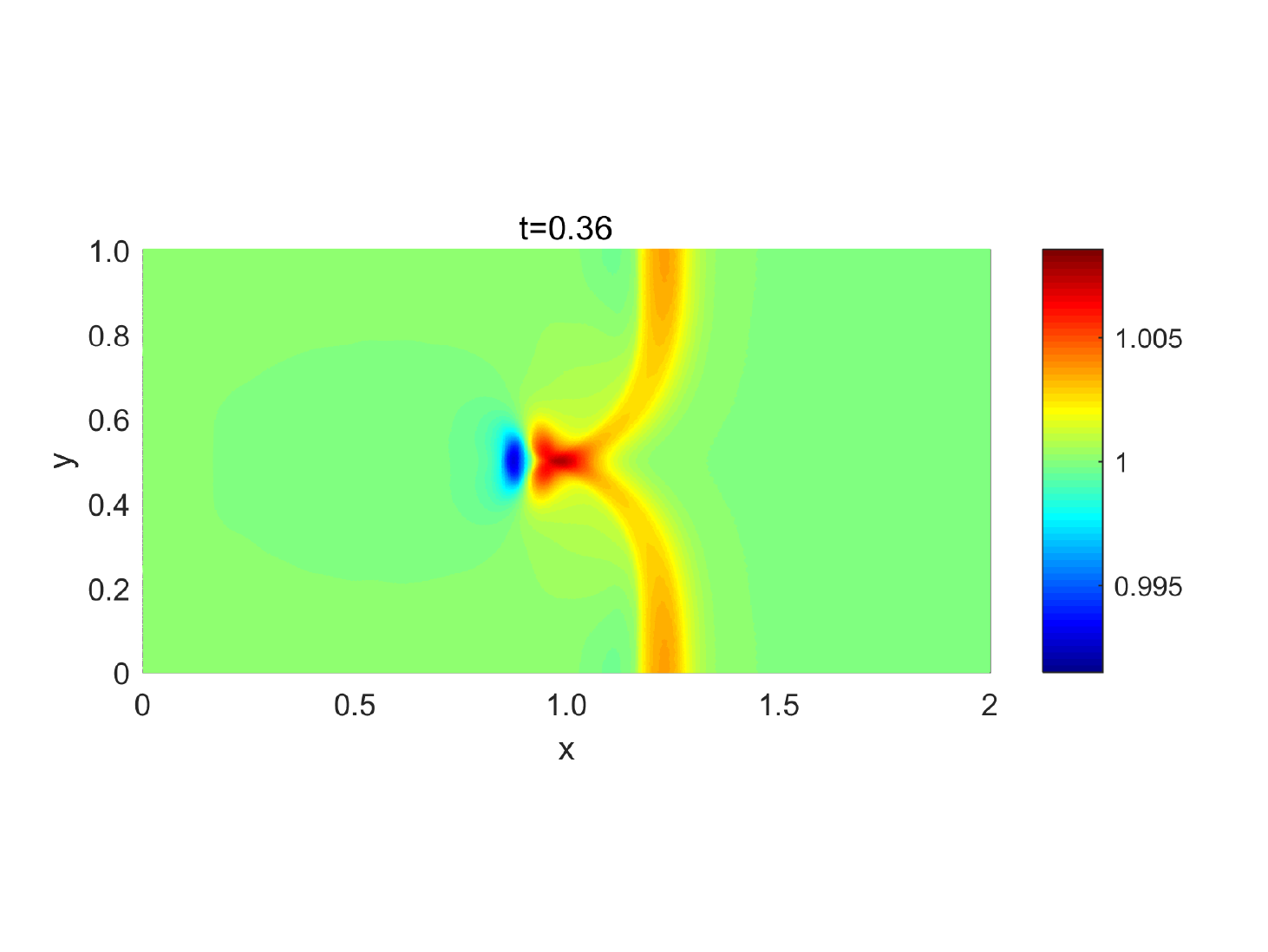}}
\subfigure[Entropy: $h+B$ at $t=0.36$]{
\includegraphics[width=0.45\textwidth, trim=15 60 15 60, clip]
{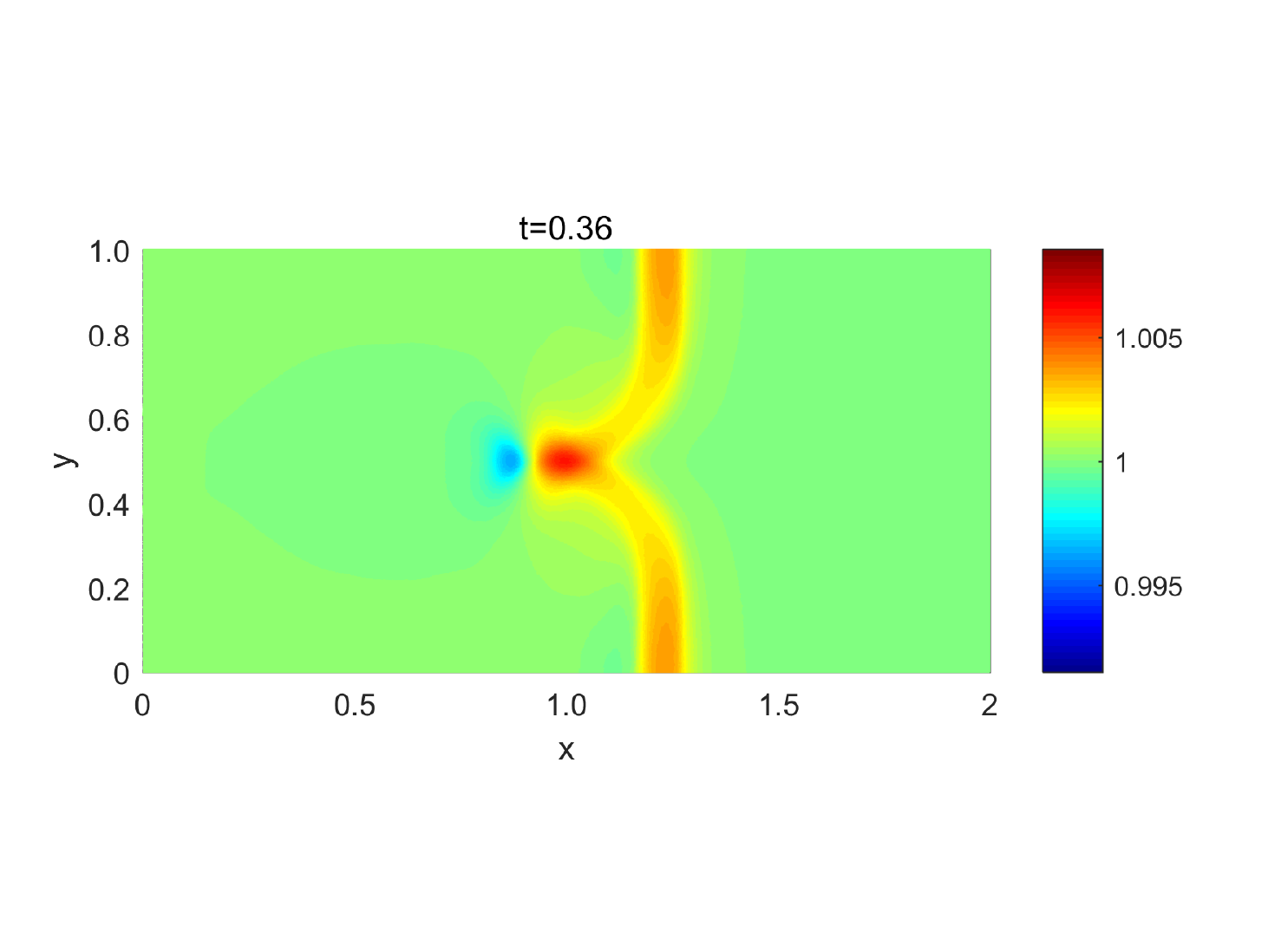}}
\subfigure[$\mathcal{E}$ and $h$: $hu$ at $t=0.36$]{
\includegraphics[width=0.45\textwidth, trim=20 50 0 60, clip]
{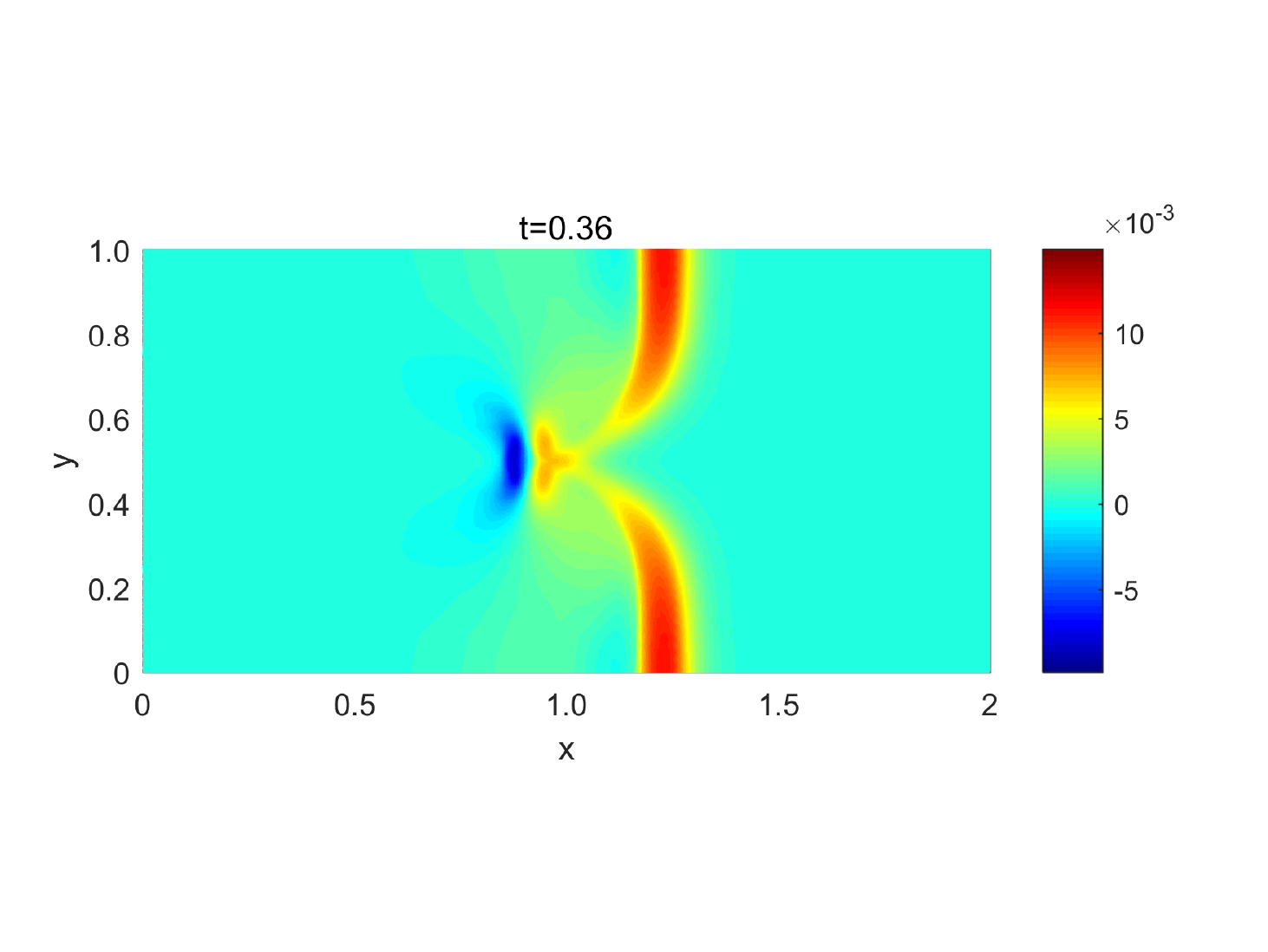}}
\subfigure[Entropy: $hu$ at $t=0.36$]{
\includegraphics[width=0.45\textwidth, trim=15 60 15 60, clip]
{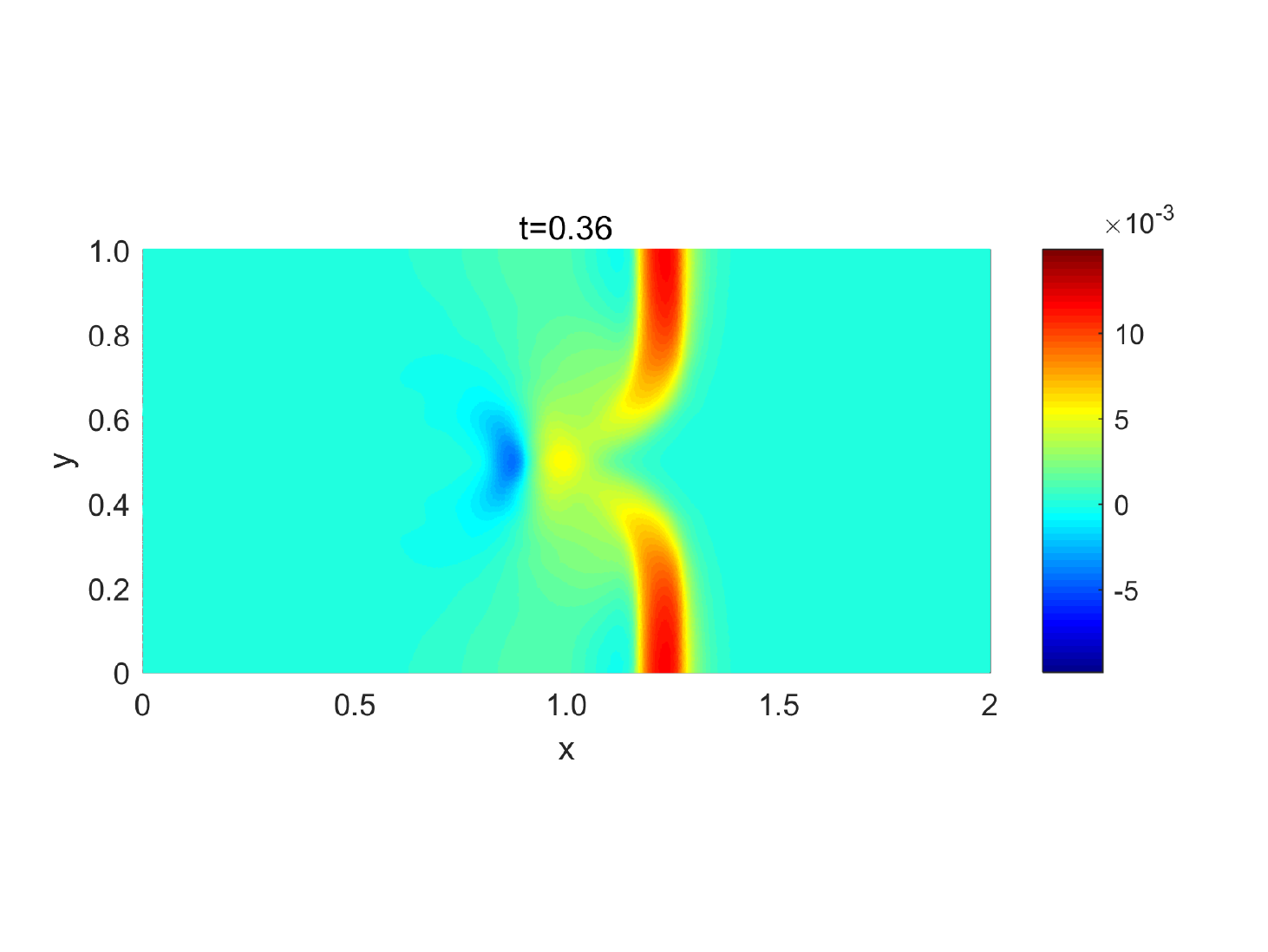}}
\subfigure[$\mathcal{E}$ and $h$: $hv$ at $t=0.36$]{
\includegraphics[width=0.45\textwidth, trim=15 60 15 60, clip]
{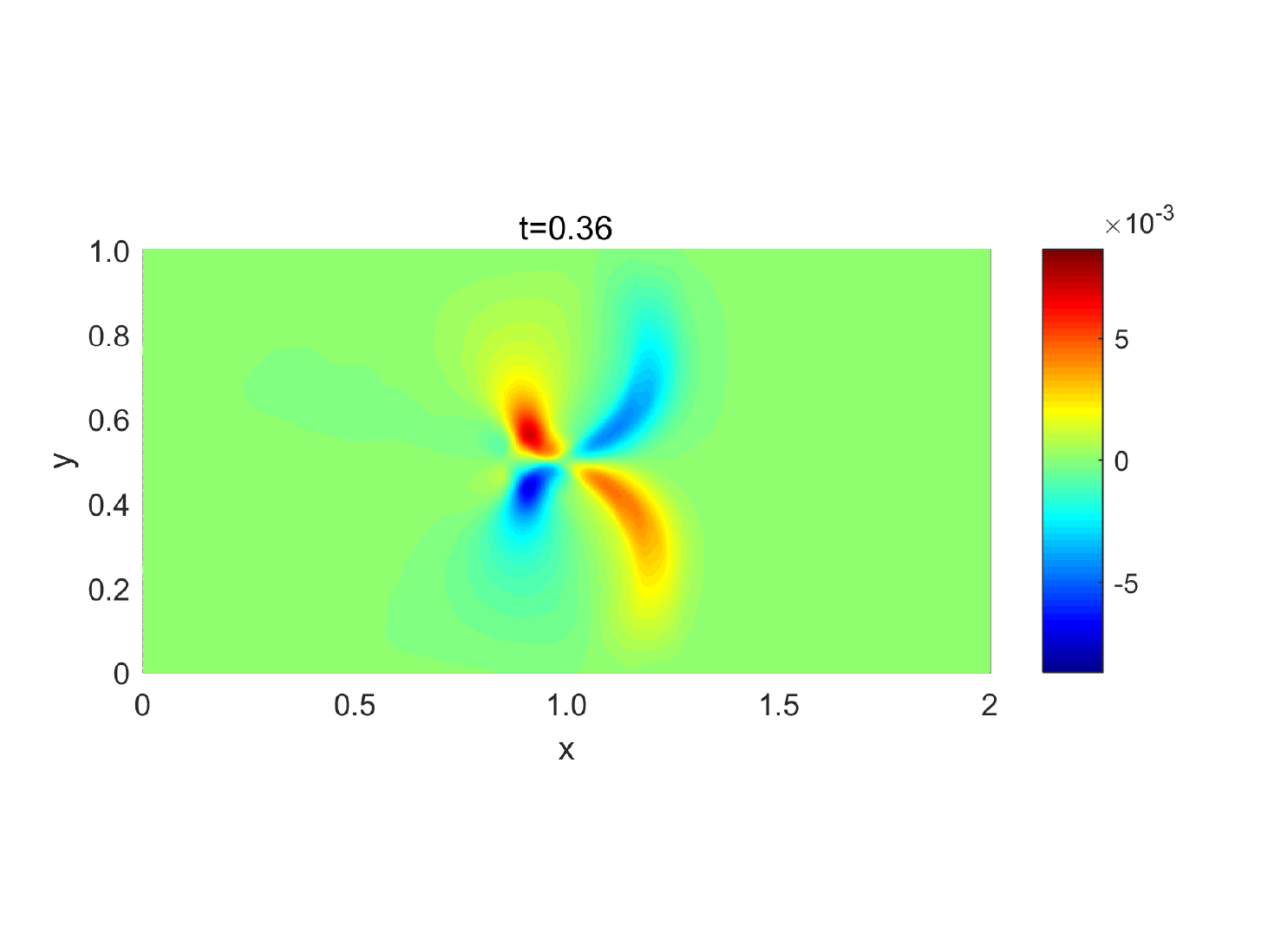}}
\subfigure[Entropy: $hv$ at $t=0.36$]{
\includegraphics[width=0.45\textwidth, trim=15 60 15 60, clip]
{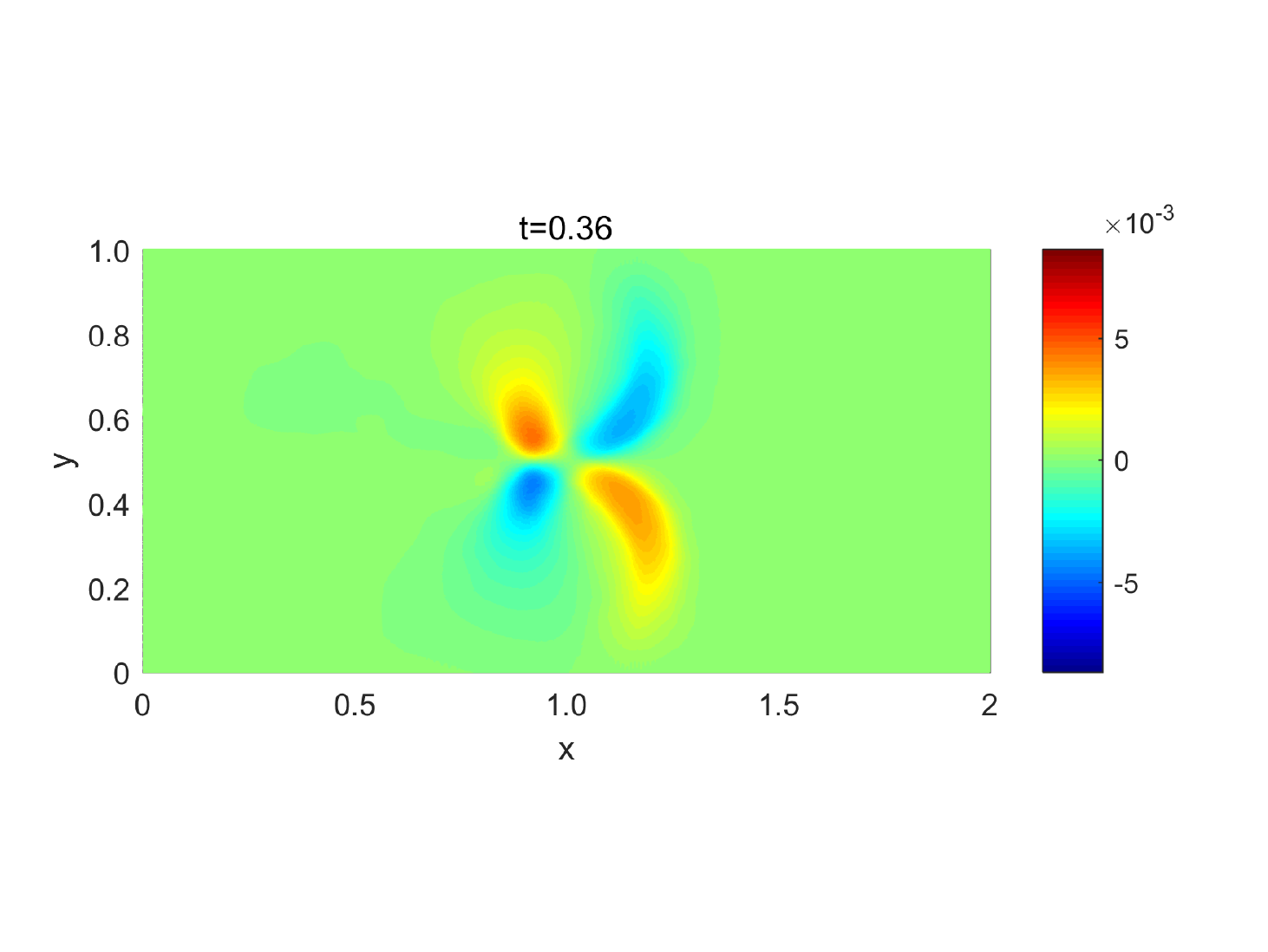}}
\caption{Continuation of Fig.~\ref{Fig:test4-2d-comparison-12}: $t = 0.36$.}
\label{Fig:test4-2d-comparison-36}
\end{figure}
\begin{figure}[H]
\centering
\subfigure[$\mathcal{E}$ and $h$: Mesh at $t=0.48$]{
\includegraphics[width=0.45\textwidth, trim=20 50 0 60, clip]
{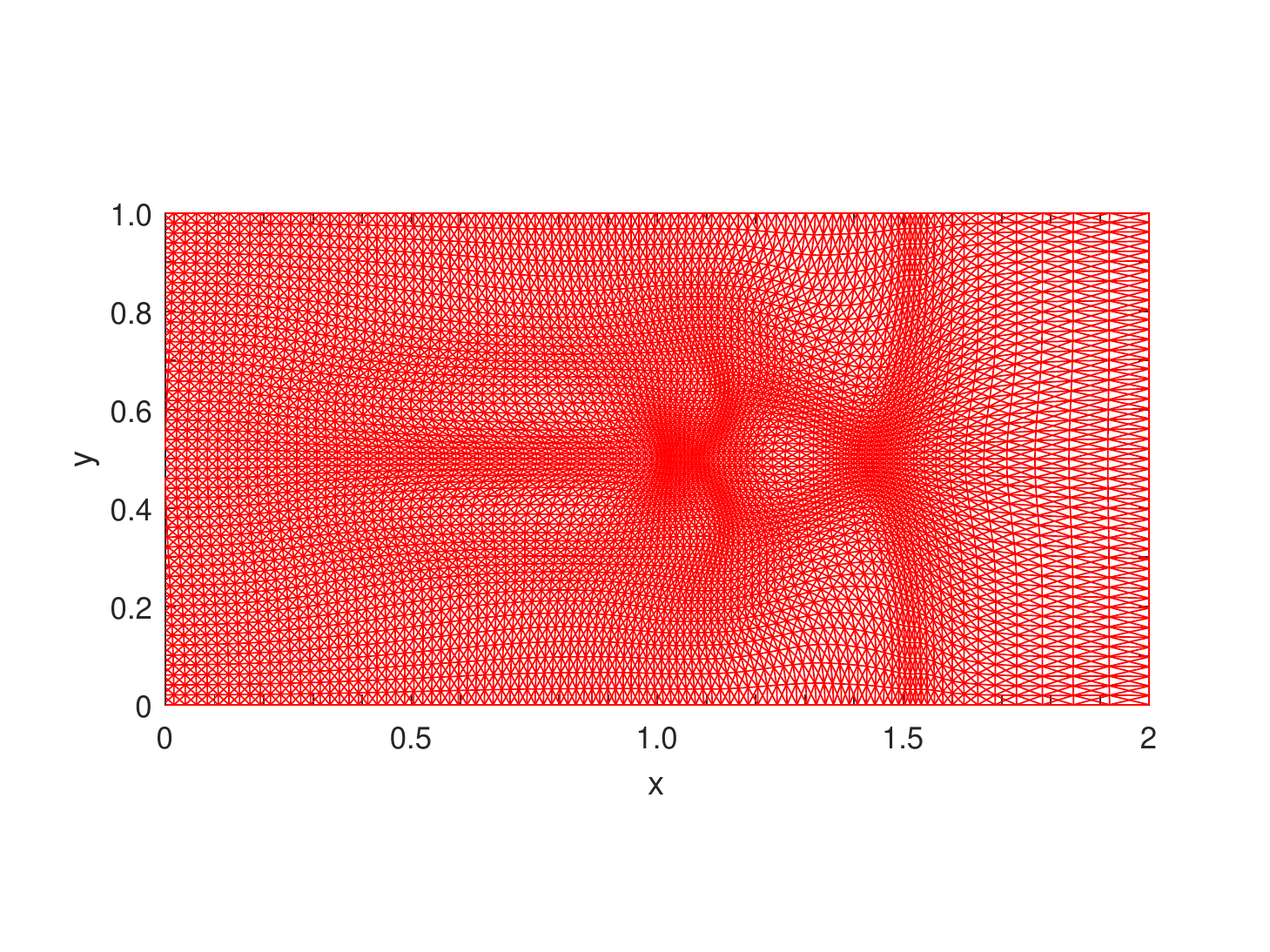}}
\subfigure[Entropy: Mesh at $t=0.48$]{
\includegraphics[width=0.45\textwidth, trim=15 60 15 60, clip]
{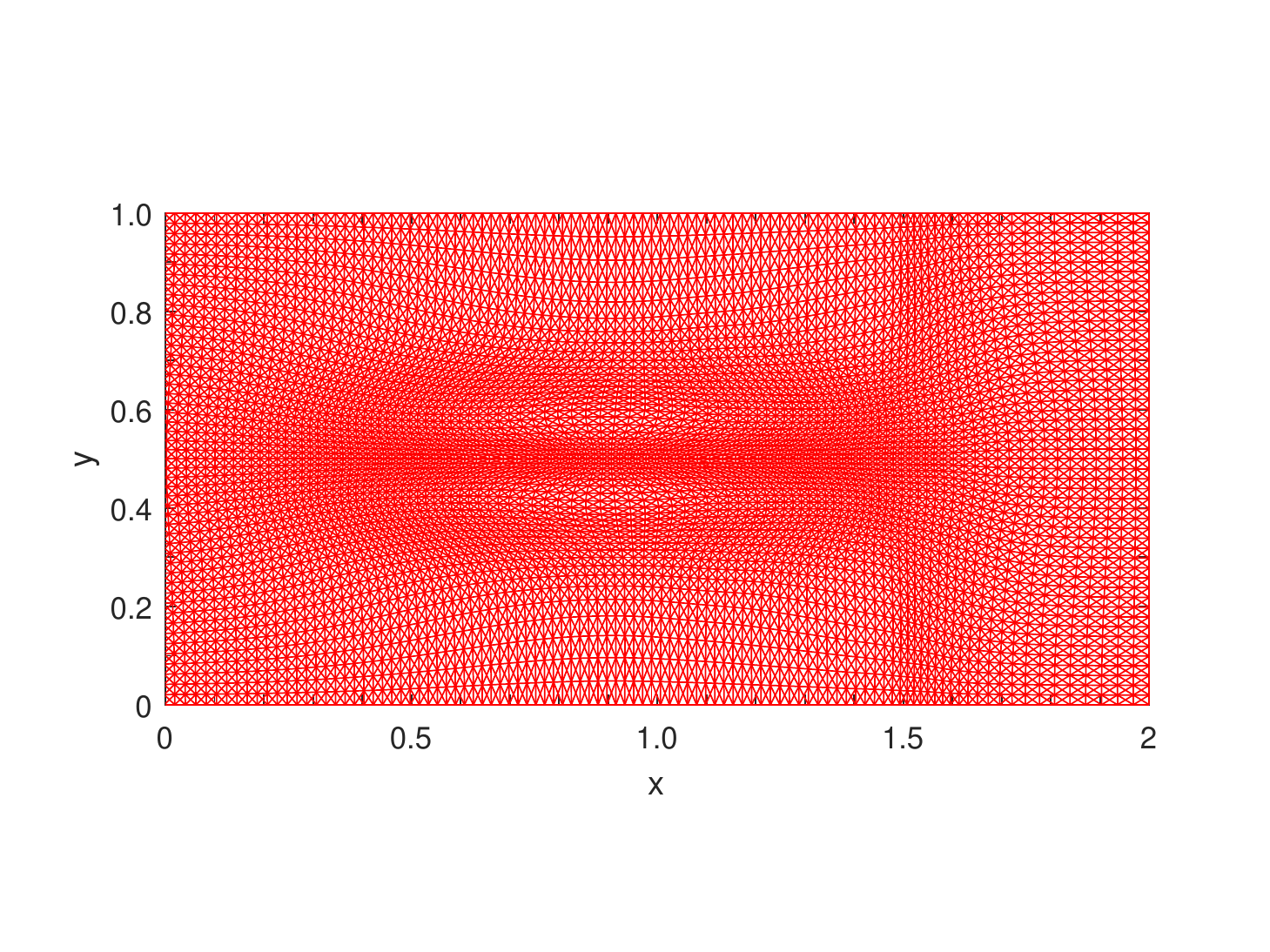}}
\subfigure[$\mathcal{E}$ and $h$: $h+B$ at $t=0.48$]{
\includegraphics[width=0.45\textwidth, trim=15 60 15 60, clip]
{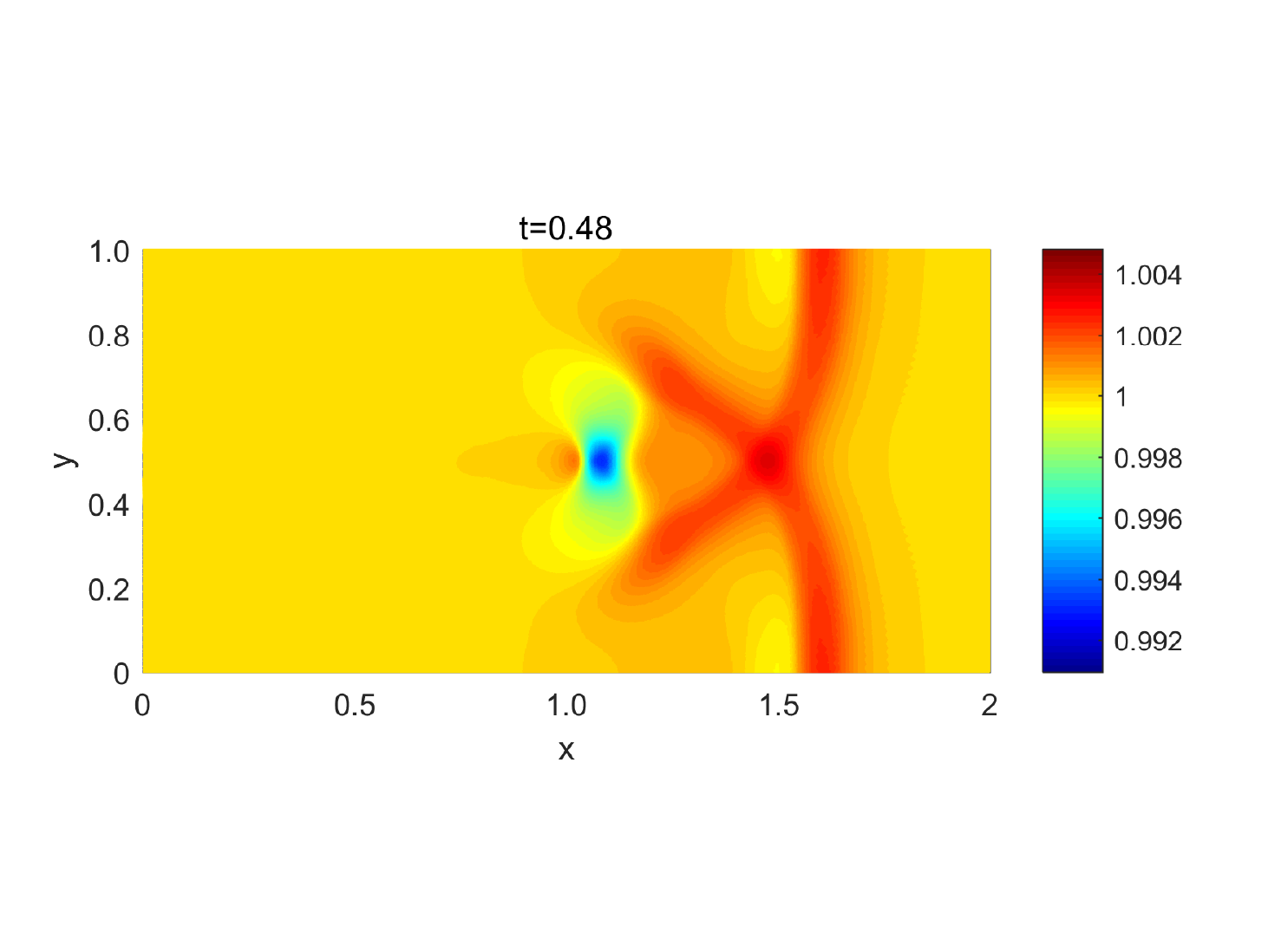}}
\subfigure[Entropy: $h+B$ at $t=0.48$]{
\includegraphics[width=0.45\textwidth, trim=15 60 15 60, clip]
{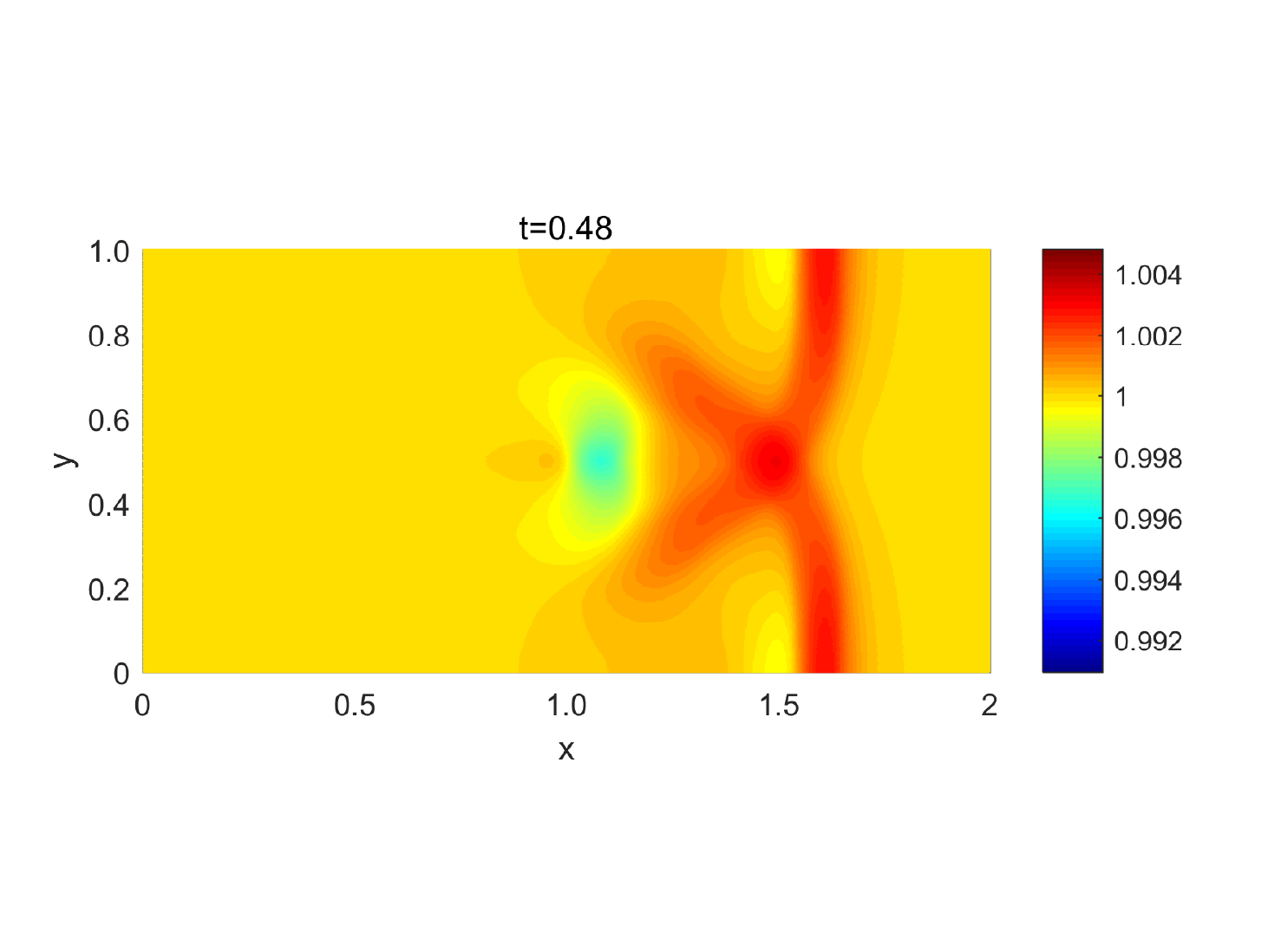}}
\subfigure[$\mathcal{E}$ and $h$: $hu$ at $t=0.48$]{
\includegraphics[width=0.45\textwidth, trim=20 50 0 60, clip]
{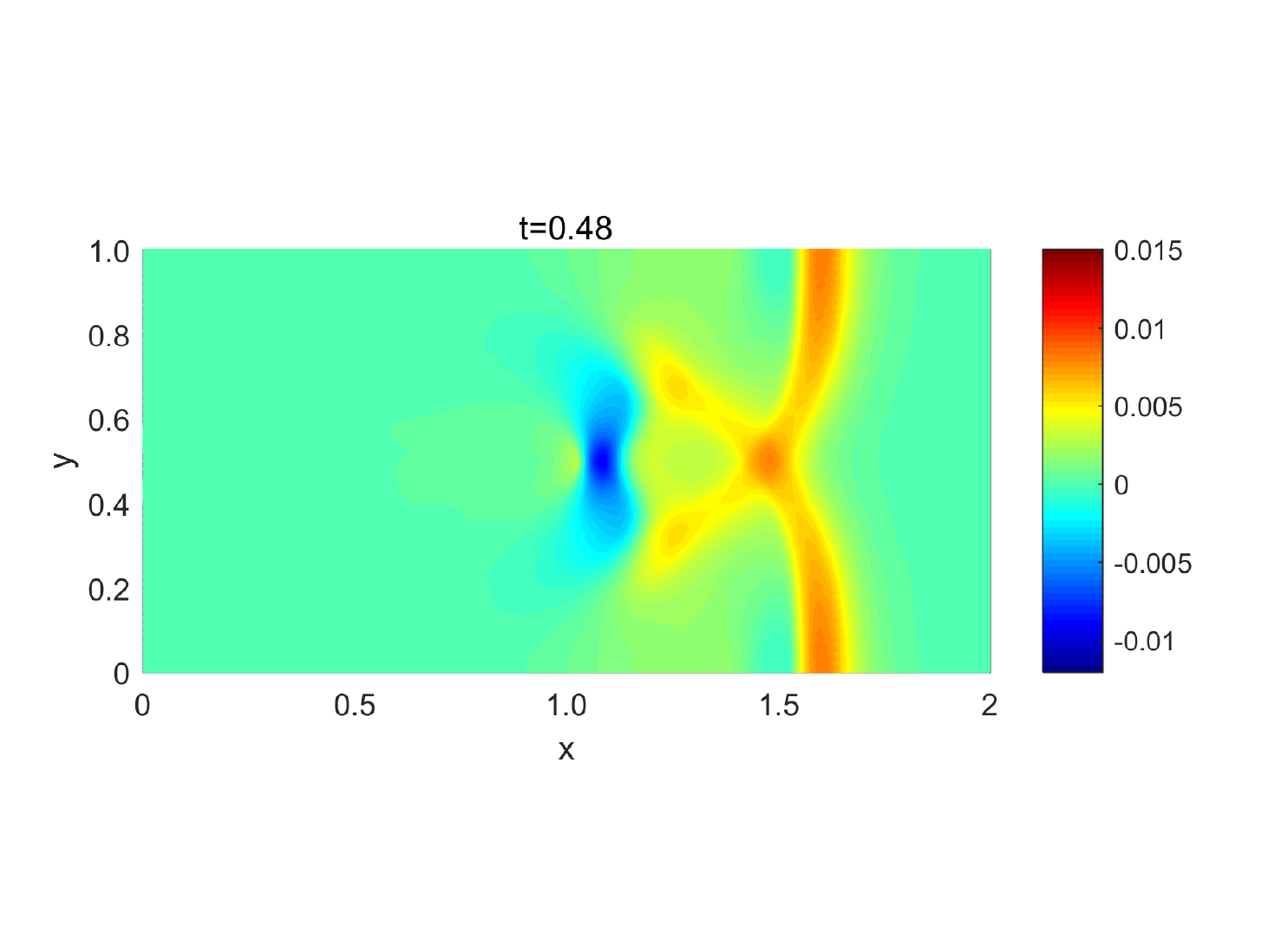}}
\subfigure[Entropy: $hu$ at $t=0.48$]{
\includegraphics[width=0.45\textwidth, trim=15 60 15 60, clip]
{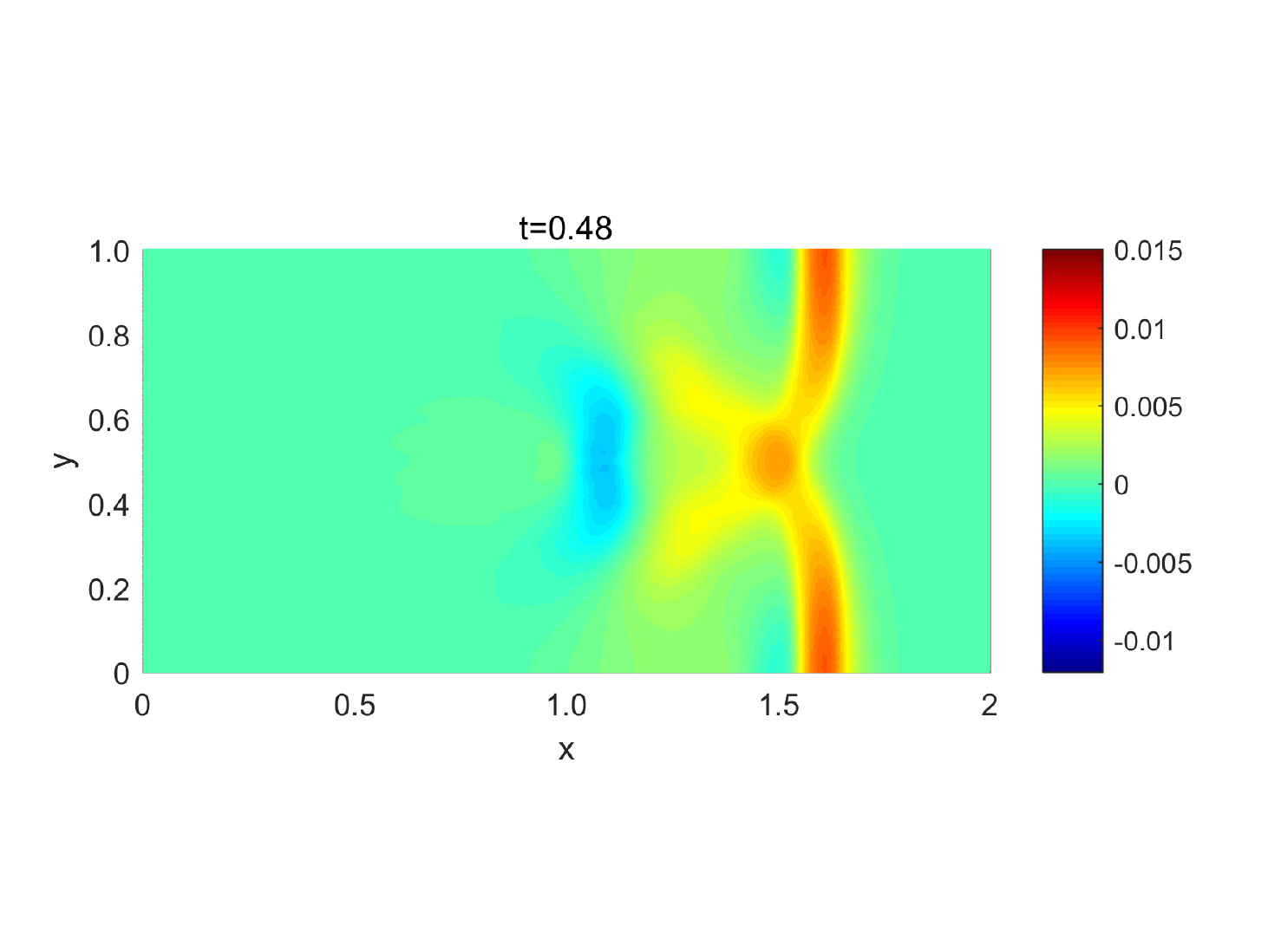}}
\subfigure[$\mathcal{E}$ and $h$: $hv$ at $t=0.48$]{
\includegraphics[width=0.45\textwidth, trim=15 60 15 60, clip]
{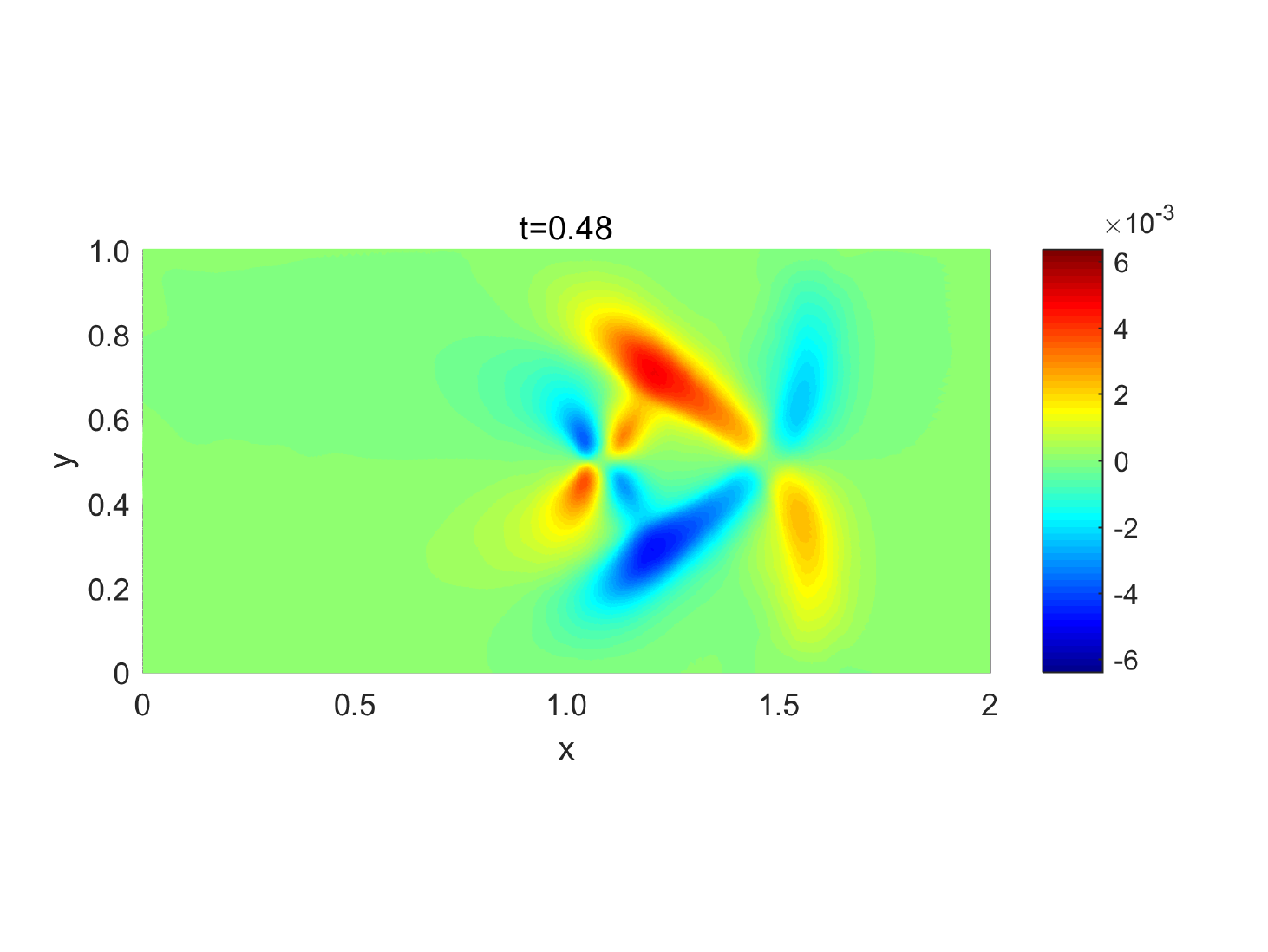}}
\subfigure[Entropy: $hv$ at $t=0.48$]{
\includegraphics[width=0.45\textwidth, trim=15 60 15 60, clip]
{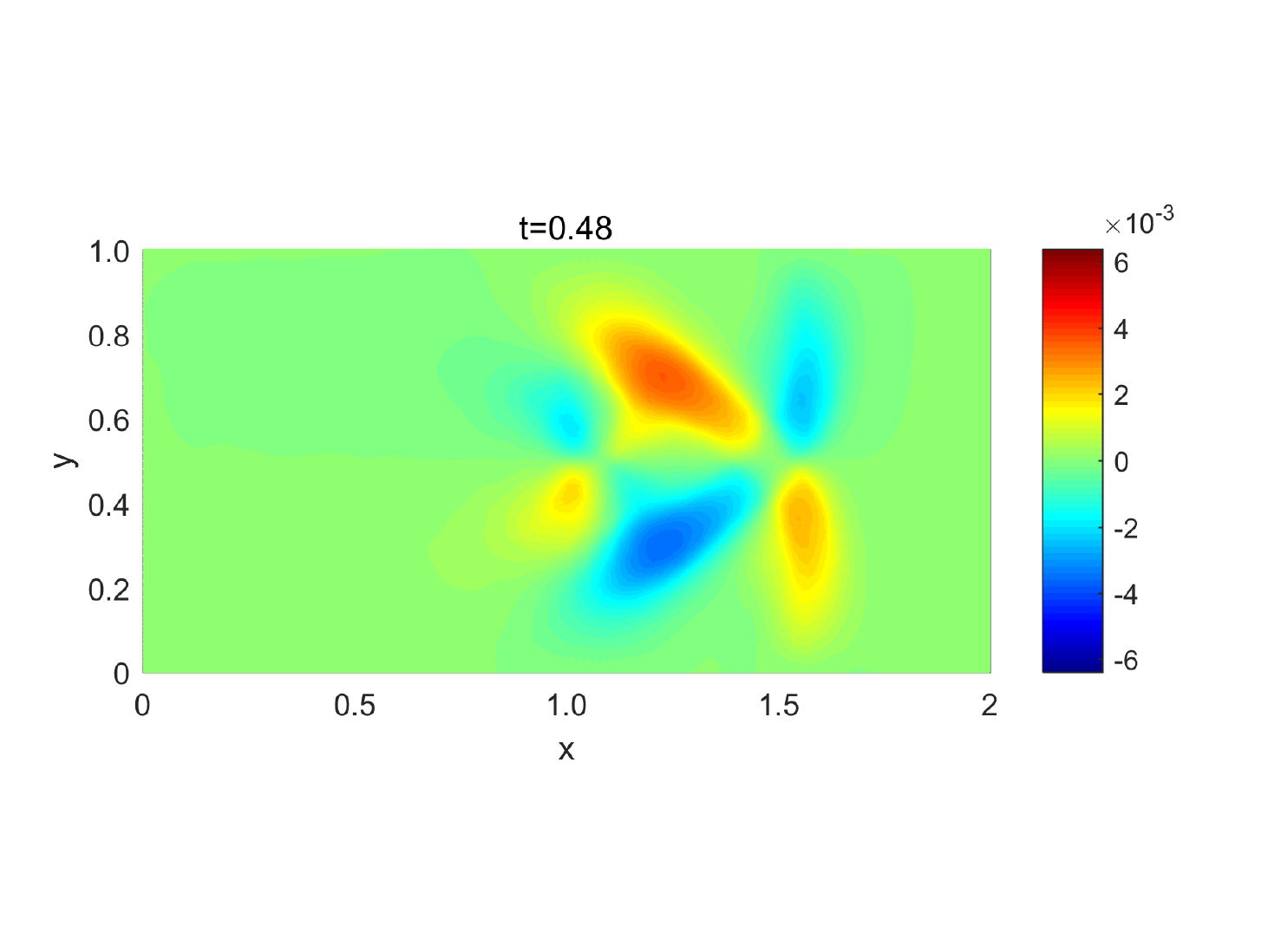}}
\caption{Continuation of Fig.~\ref{Fig:test4-2d-comparison-12}: $t = 0.48$.}
\label{Fig:test4-2d-comparison-48}
\end{figure}
%

\begin{figure}[H]
\centering
\subfigure[$h+B$: MM $N=150\times 50\times4$]{
\includegraphics[width=0.30\textwidth, trim=15 60 15 60, clip]{R_test4_2d_P2_Eh_Bph_M50_t12-eps-converted-to.pdf}}
\subfigure[$hu$: MM $N=150\times 50\times4$]{
\includegraphics[width=0.30\textwidth, trim=15 60 15 60, clip]{R_test4_2d_P2_Eh_hu_M50_t12-eps-converted-to.pdf}}
\subfigure[$hv$: MM $N=150\times 50\times4$]{
\includegraphics[width=0.30\textwidth, trim=15 60 15 60, clip]{R_test4_2d_P2_Eh_hv_M50_t12-eps-converted-to.pdf}}
\subfigure[$h+B$: FM $N=150\times 50\times4$]{
\includegraphics[width=0.30\textwidth, trim=15 60 15 60, clip]{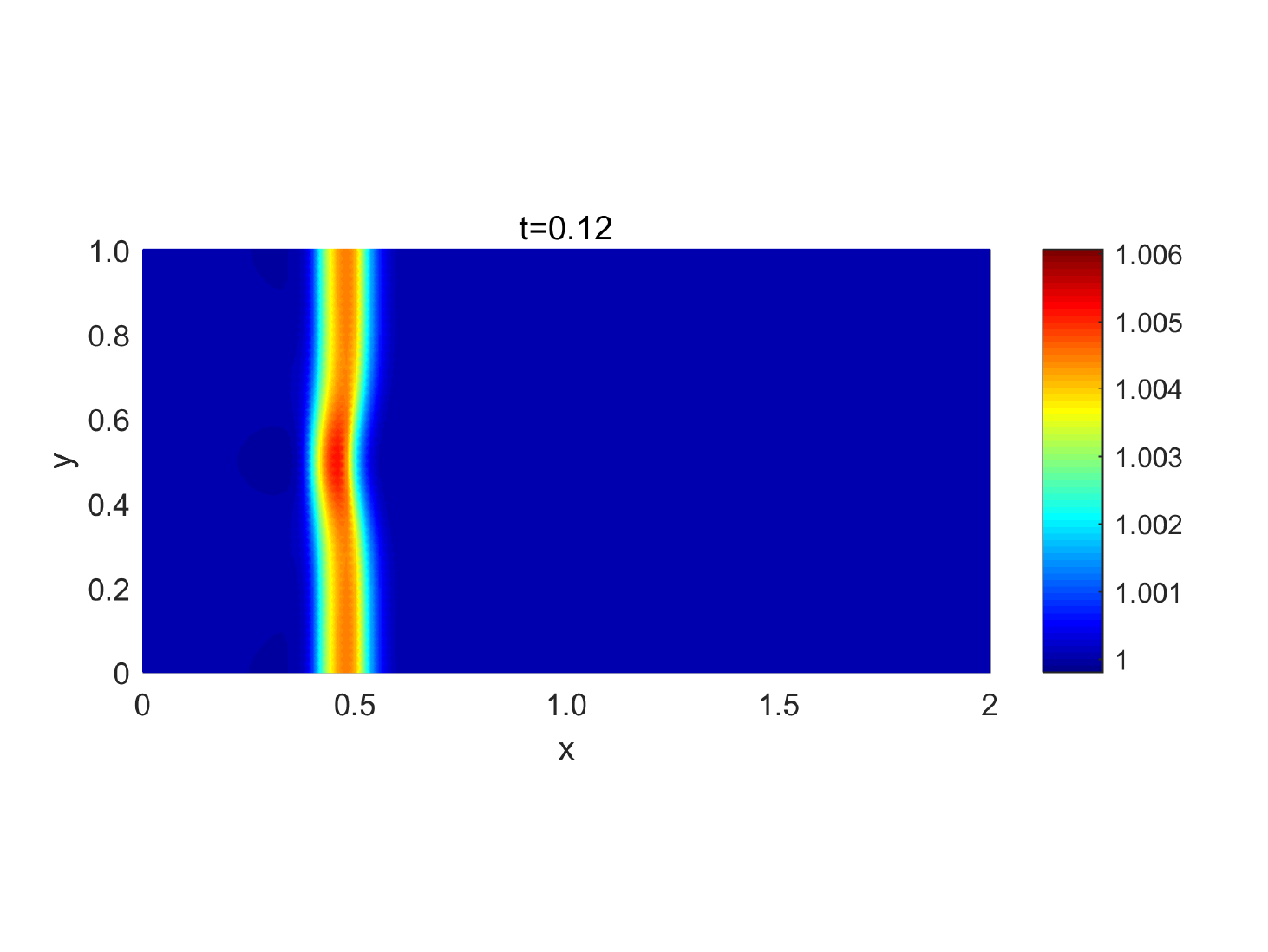}}
\subfigure[$hu$: FM $N=150\times 50\times4$]{
\includegraphics[width=0.30\textwidth, trim=15 60 15 60,
clip]{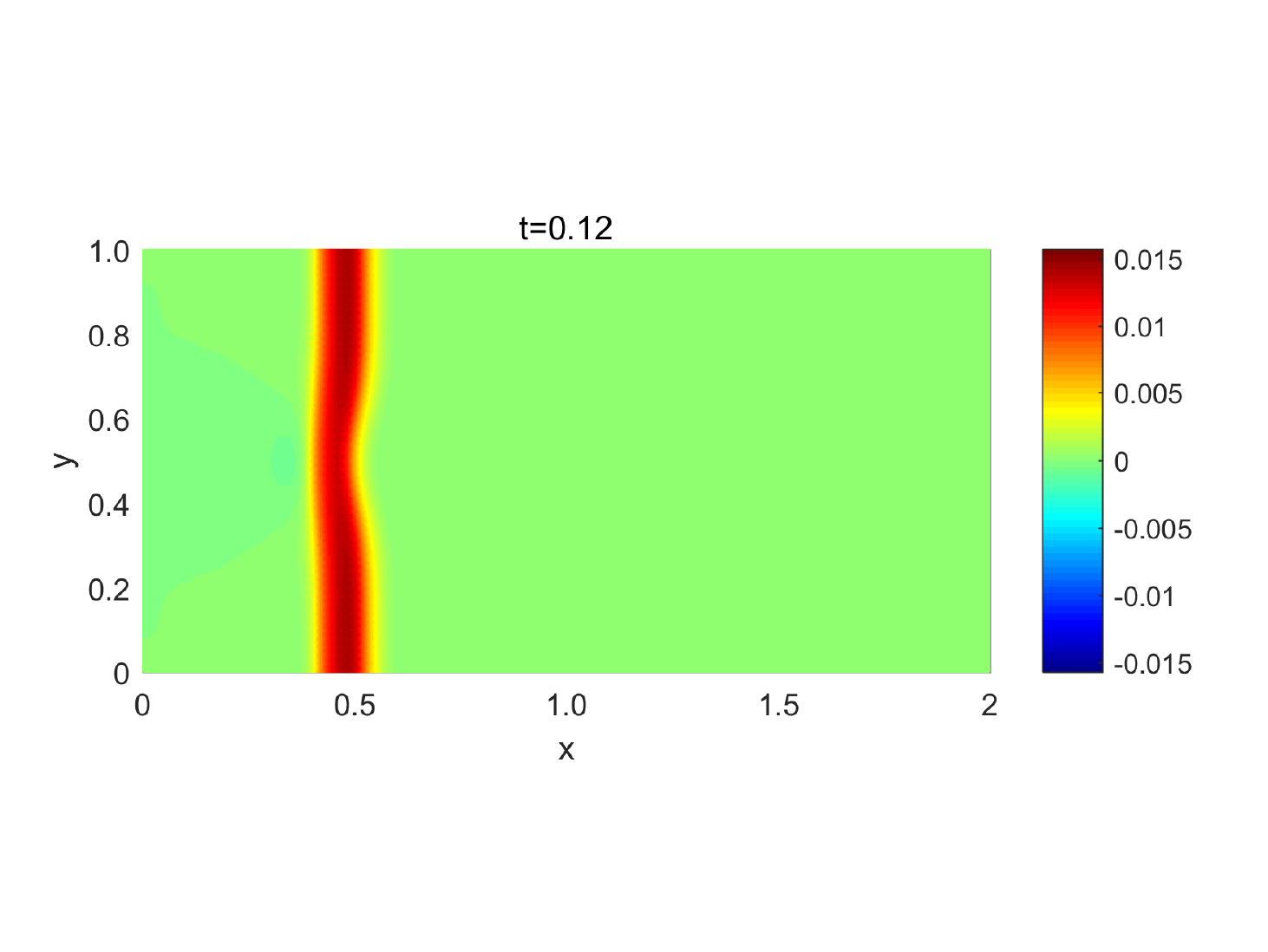}}
\subfigure[$hv$: FM $N=150\times 50\times4$]{
\includegraphics[width=0.30\textwidth, trim=15 60 15 60, clip]{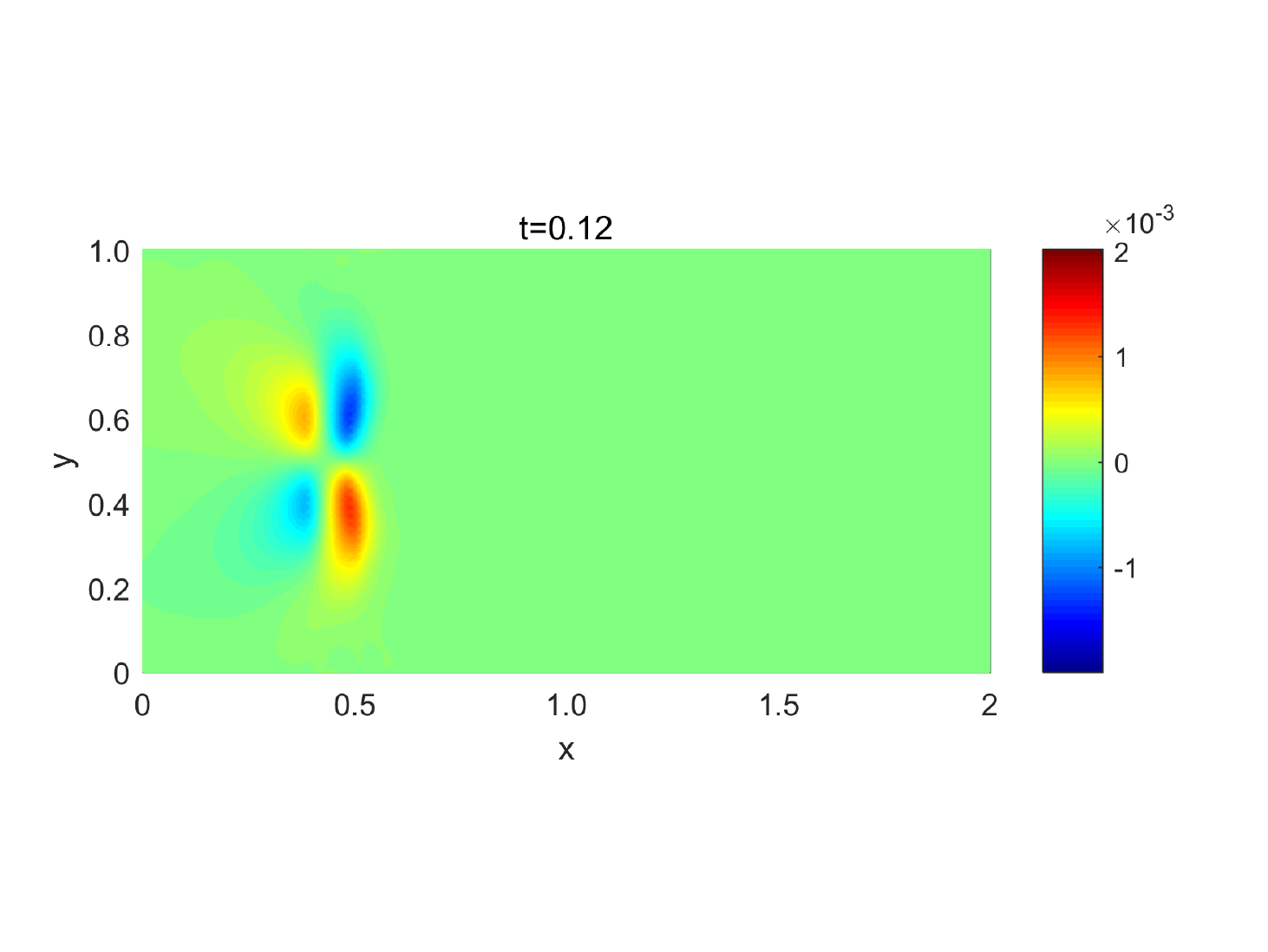}}
\subfigure[$h+B$: FM $600\times 200\times4$]{
\includegraphics[width=0.30\textwidth, trim=15 60 15 60, clip]{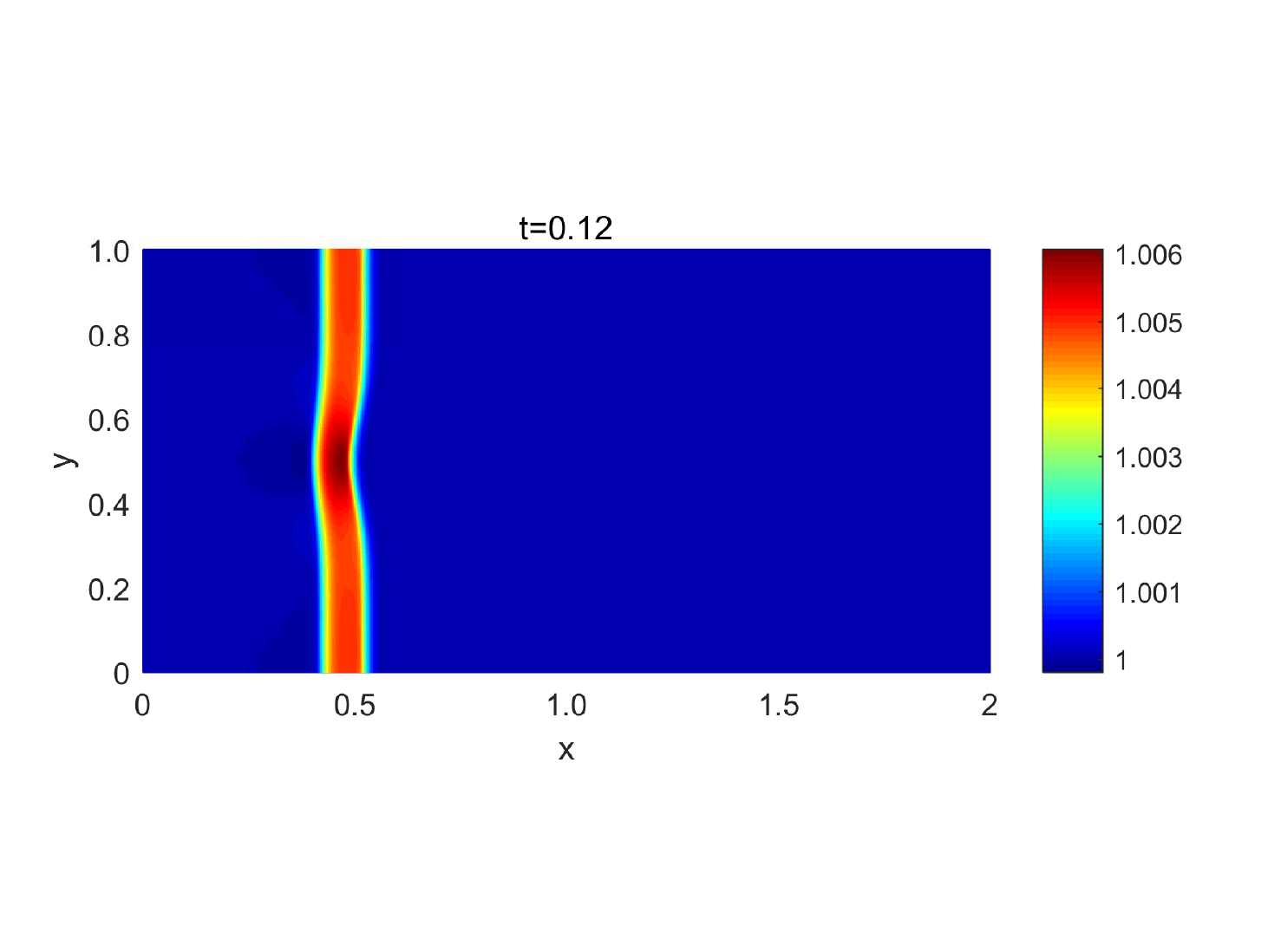}}
\subfigure[$hu$: FM $N=600\times 200\times4$]{
\includegraphics[width=0.30\textwidth, trim=15 60 15 60, clip]{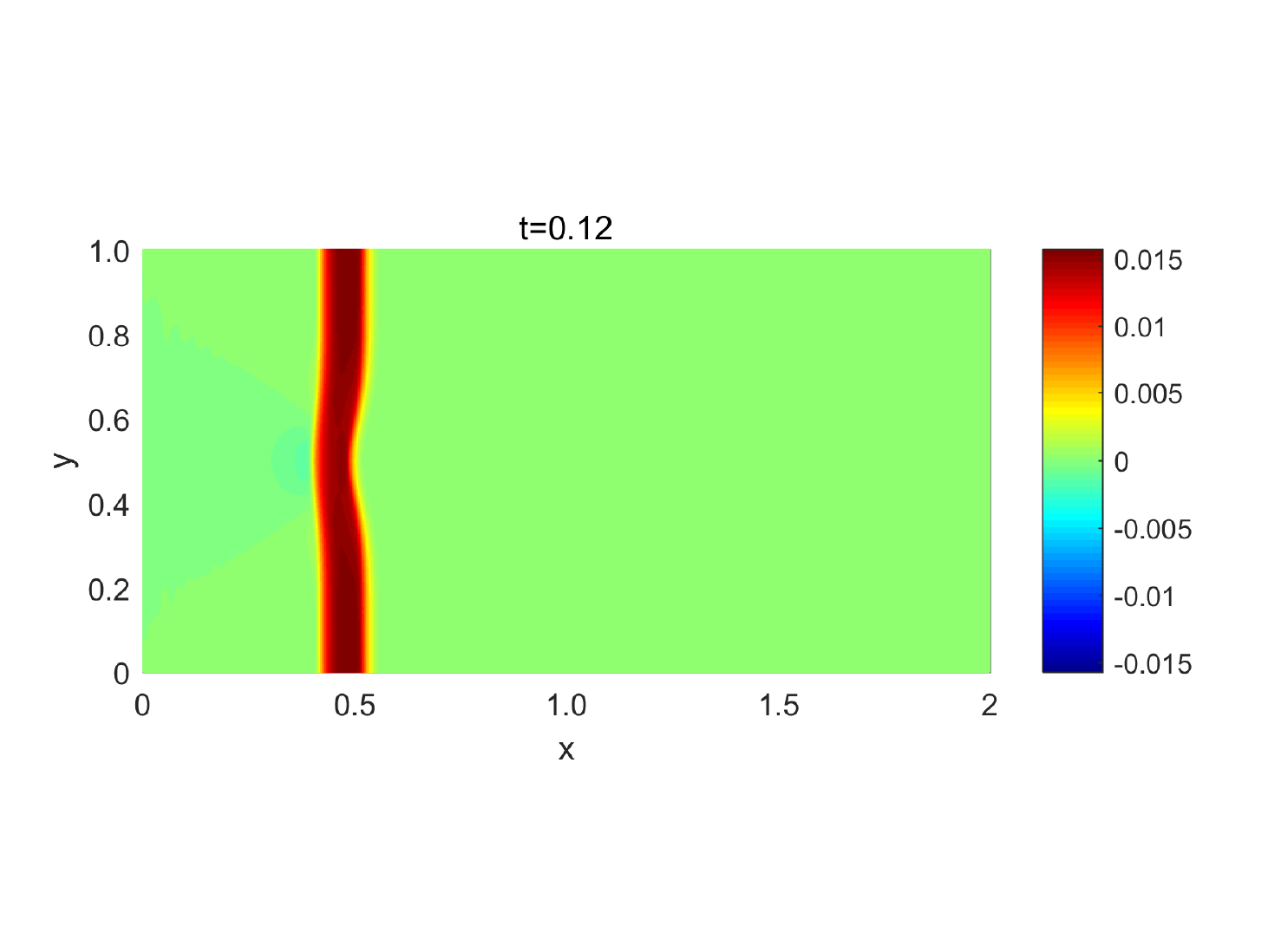}}
\subfigure[$hv$: FM $N=600\times 200\times4$]{
\includegraphics[width=0.30\textwidth, trim=15 60 15 60, clip]{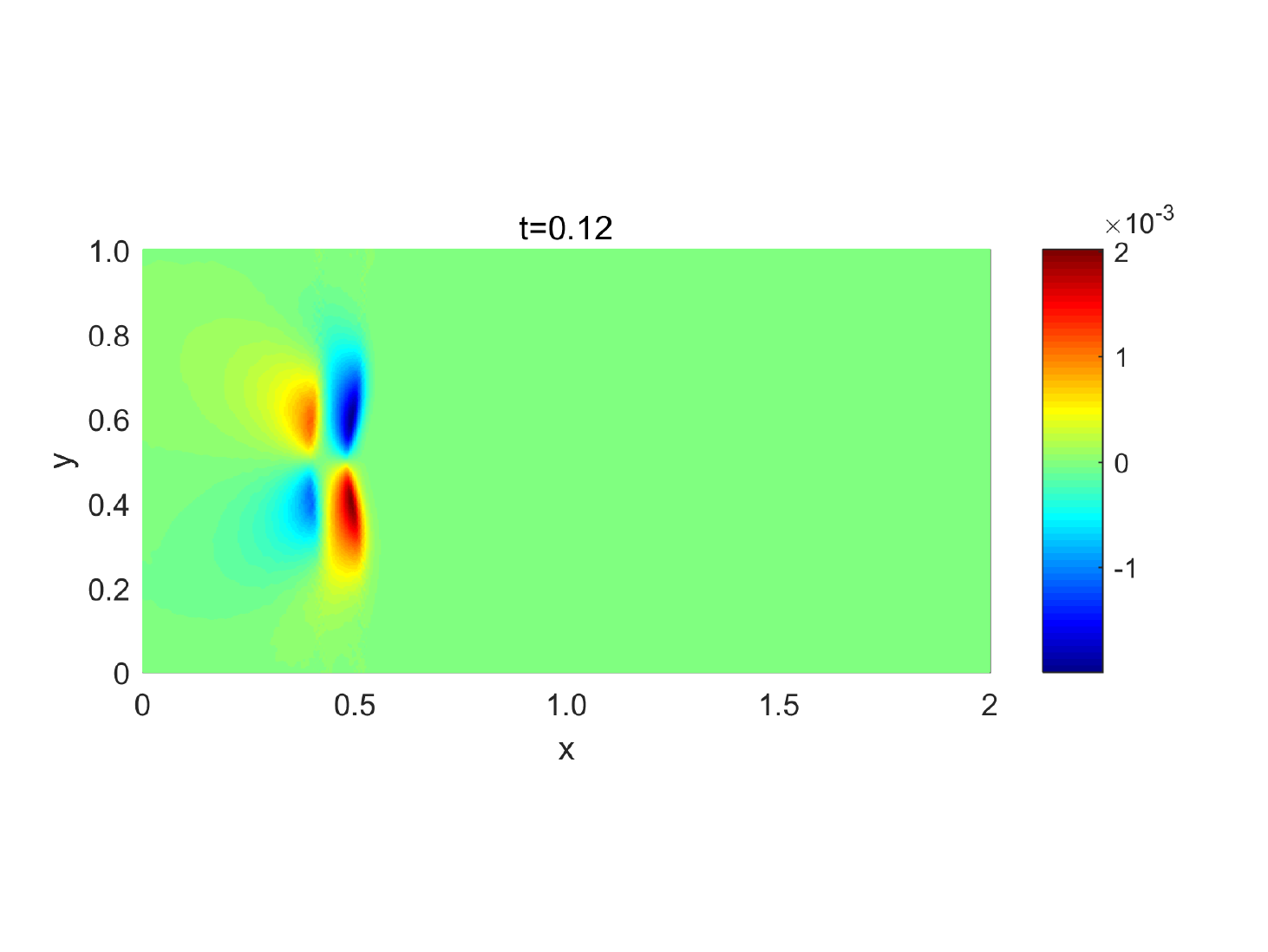}}
\caption{Example \ref{test4-2d}. The contours of $h+B$, $hu$, and $hv$ at $t=0.12$ are obtained with the $P^2$ MM-DG method and a moving mesh of $N=150\times 50\times4$ and fixed meshes of $N=150\times 50\times4$ and $N=600\times 200\times4$.}
\label{Fig:test4-2d-h-hu-hv-t12}
\end{figure}

\begin{figure}[H]
\centering
\subfigure[$h+B$: MM $N=150\times 50\times4$]{
\includegraphics[width=0.30\textwidth, trim=15 60 15 60, clip]{R_test4_2d_P2_Eh_Bph_M50_t24-eps-converted-to.pdf}}
\subfigure[$hu$: MM $N=150\times 50\times4$]{
\includegraphics[width=0.30\textwidth, trim=15 60 15 60, clip]{R_test4_2d_P2_Eh_hu_M50_t24-eps-converted-to.pdf}}
\subfigure[$hv$: MM $N=150\times 50\times4$]{
\includegraphics[width=0.30\textwidth, trim=15 60 15 60, clip]{R_test4_2d_P2_Eh_hv_M50_t24-eps-converted-to.pdf}}
\subfigure[$h+B$: FM $N=150\times 50\times4$]{
\includegraphics[width=0.30\textwidth, trim=15 60 15 60, clip]{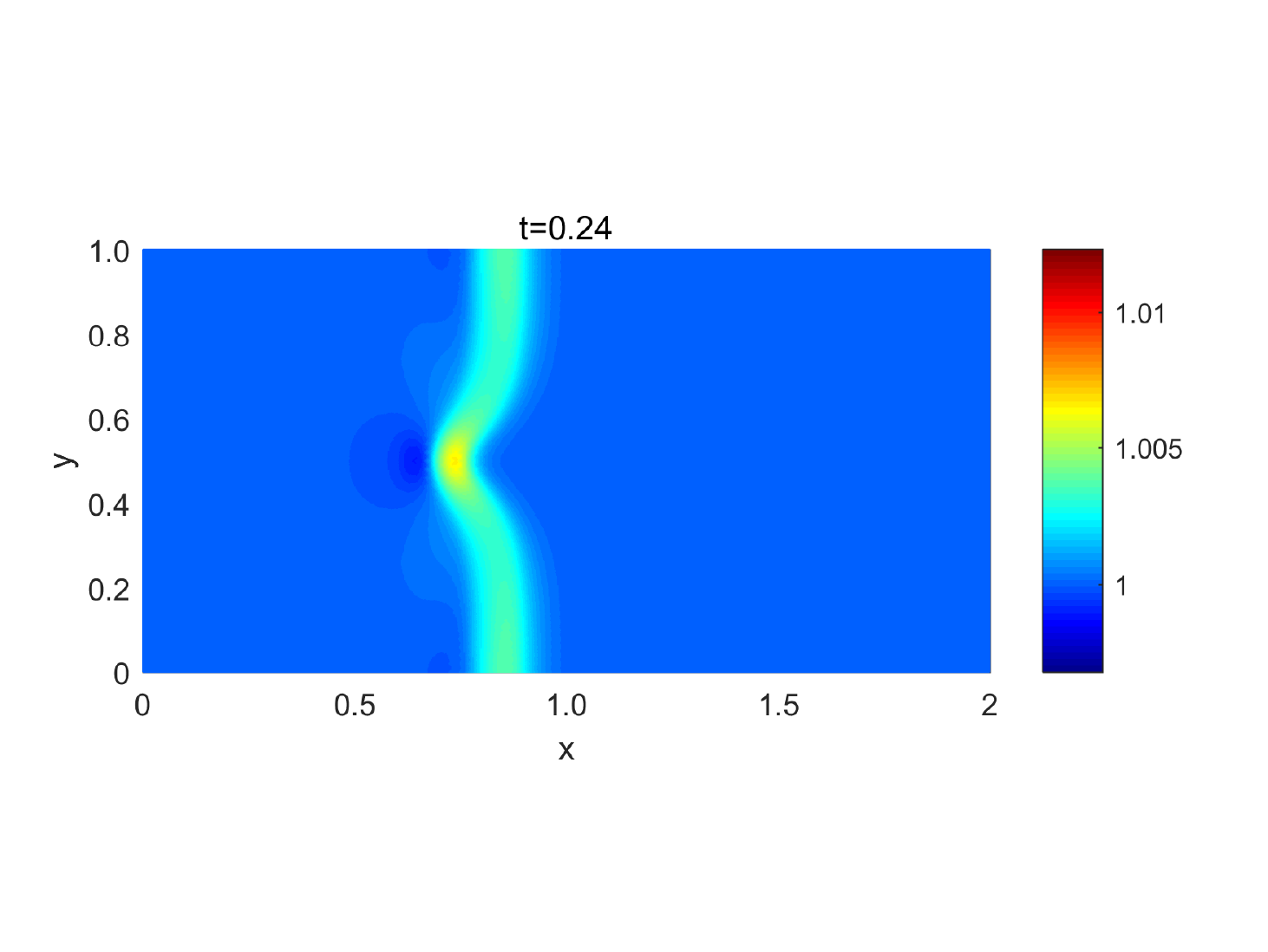}}
\subfigure[$hu$: FM $N=150\times 50\times4$]{
\includegraphics[width=0.30\textwidth, trim=15 60 15 60, clip]{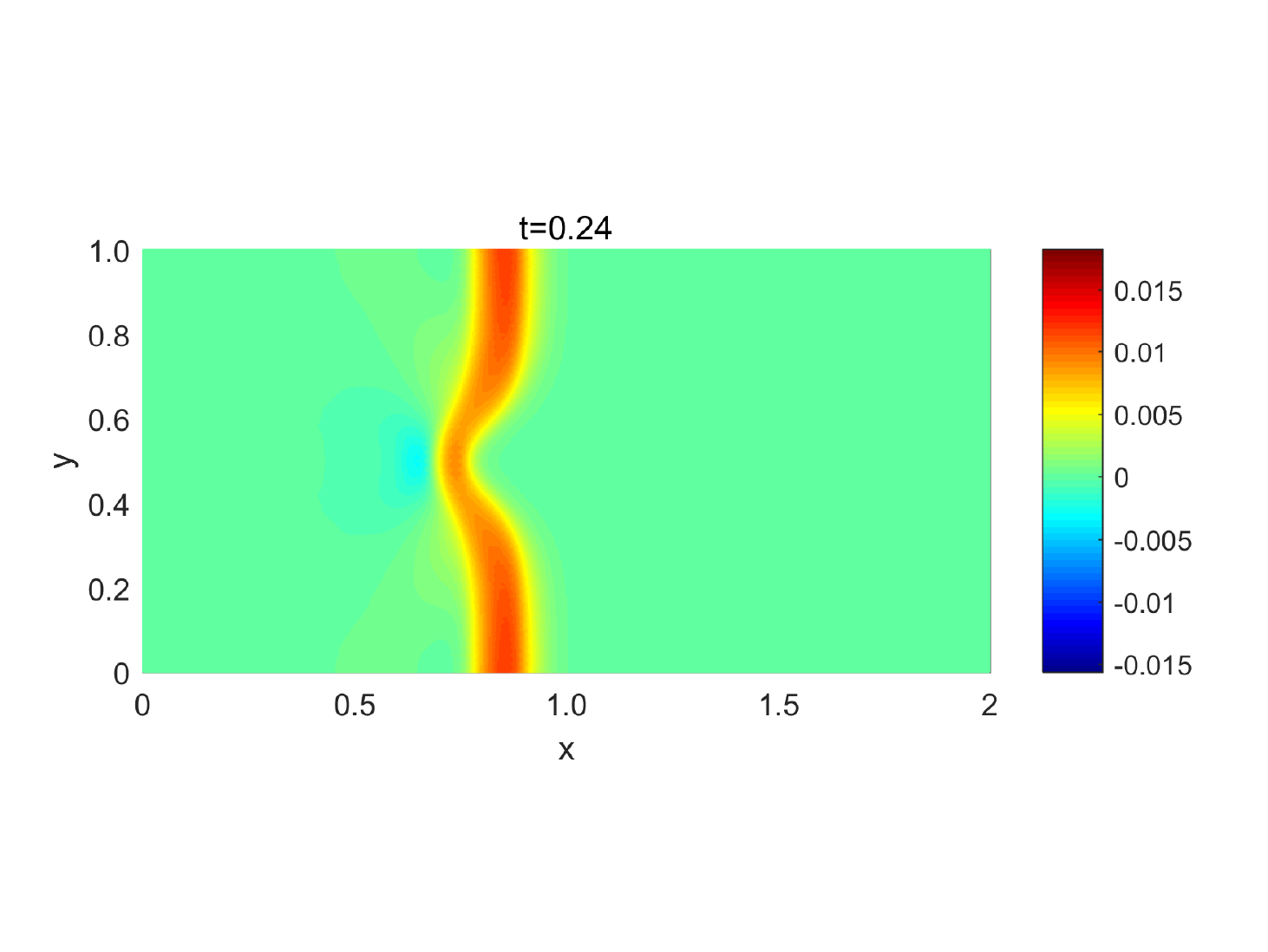}}
\subfigure[$hv$: FM $N=150\times 50\times4$]{
\includegraphics[width=0.30\textwidth, trim=15 60 15 60, clip]{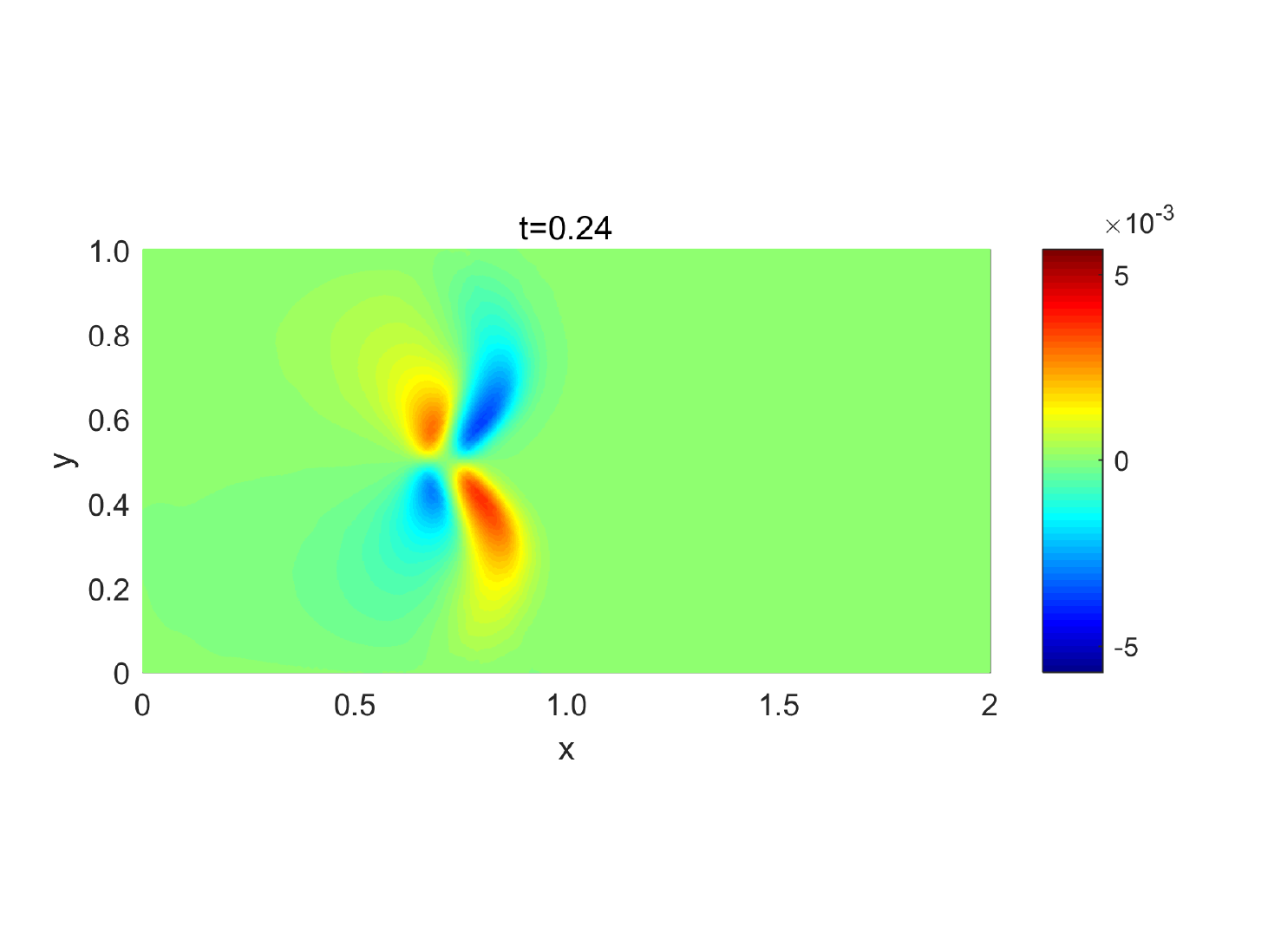}}
\subfigure[$h+B$: FM $600\times 200\times4$]{
\includegraphics[width=0.30\textwidth, trim=15 60 15 60, clip]{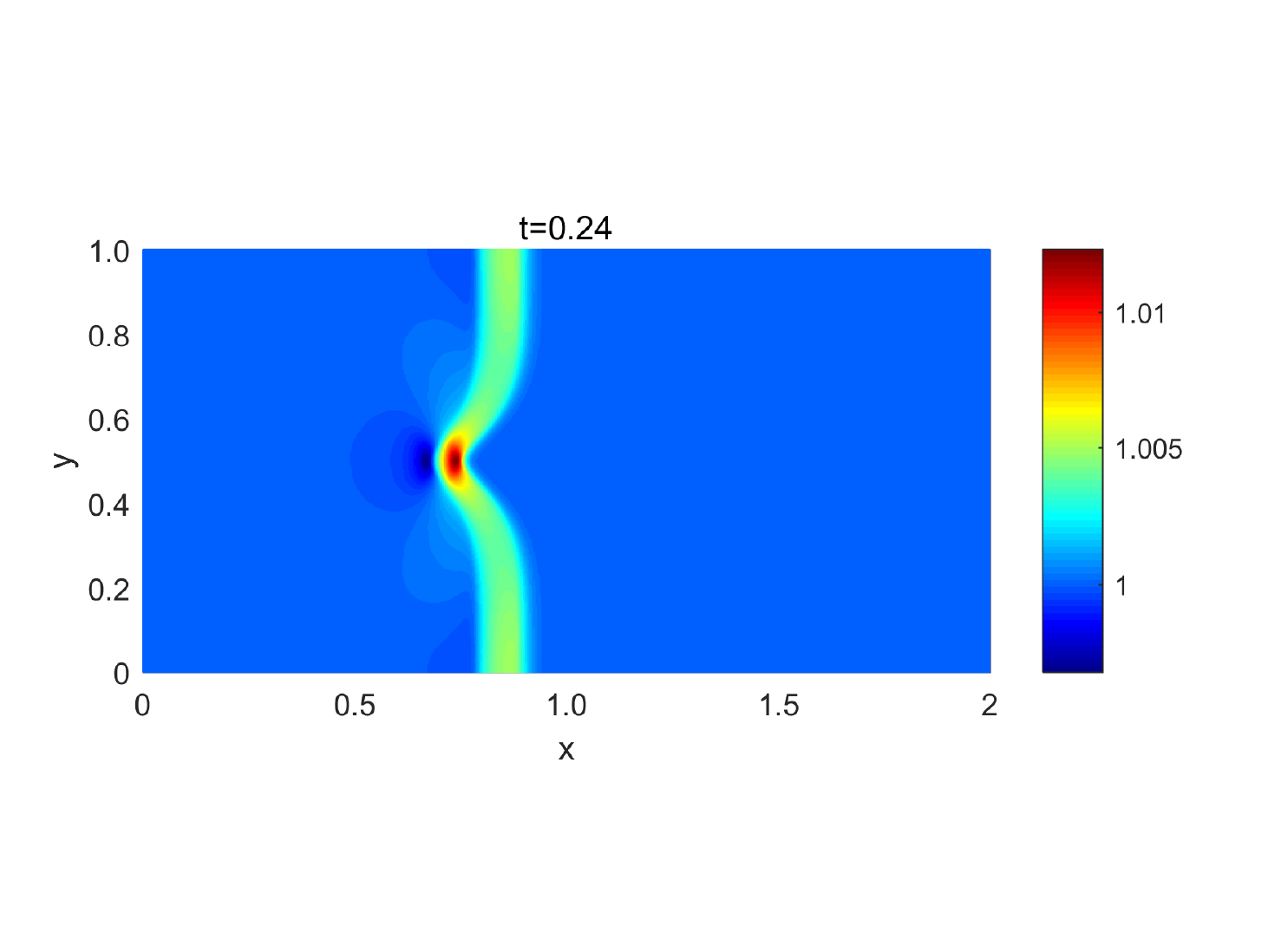}}
\subfigure[$hu$: FM $N=600\times 200\times4$]{
\includegraphics[width=0.30\textwidth, trim=15 60 15 60, clip]{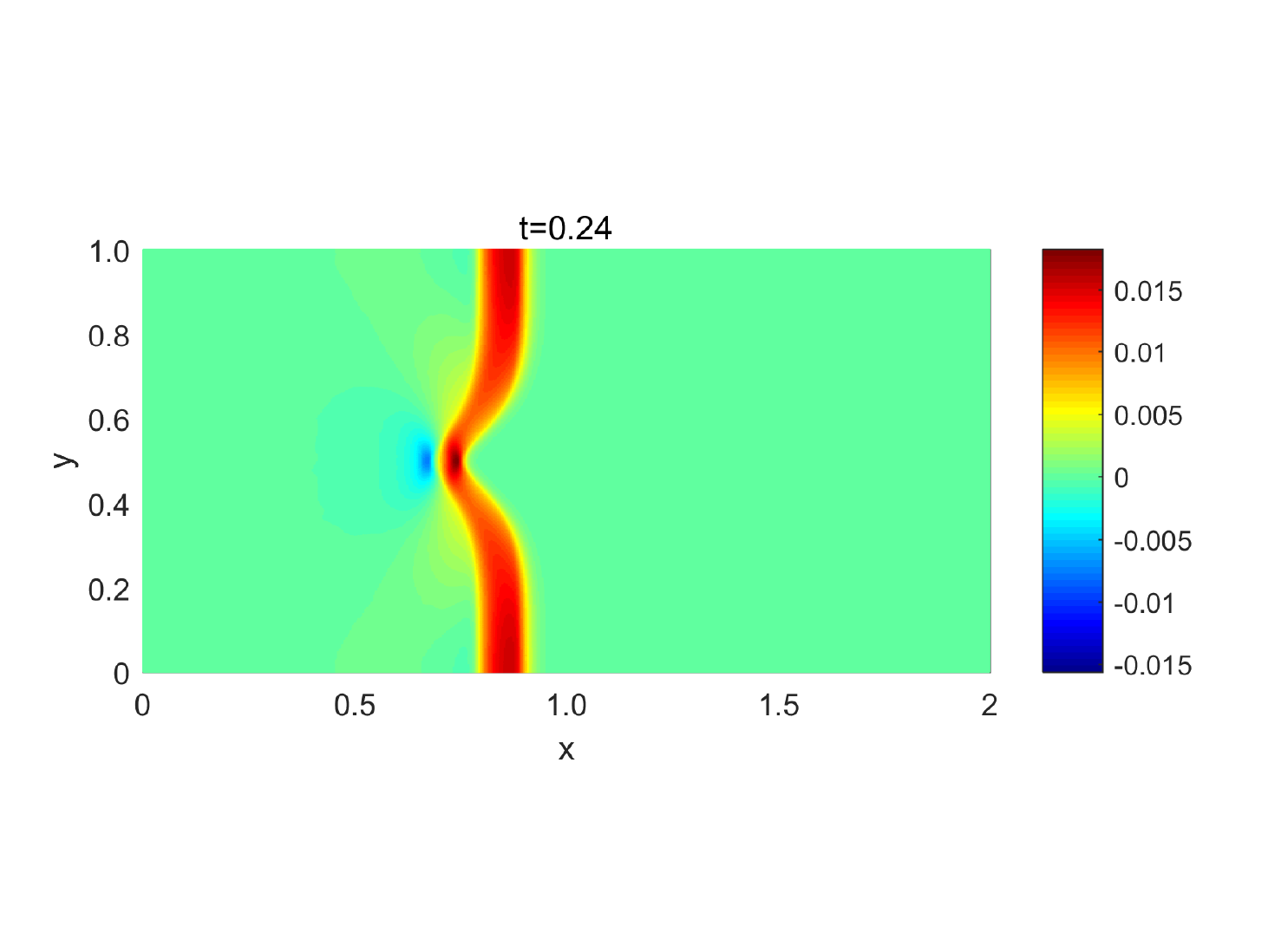}}
\subfigure[$hv$: FM $N=600\times 200\times4$]{
\includegraphics[width=0.30\textwidth, trim=15 60 15 60, clip]{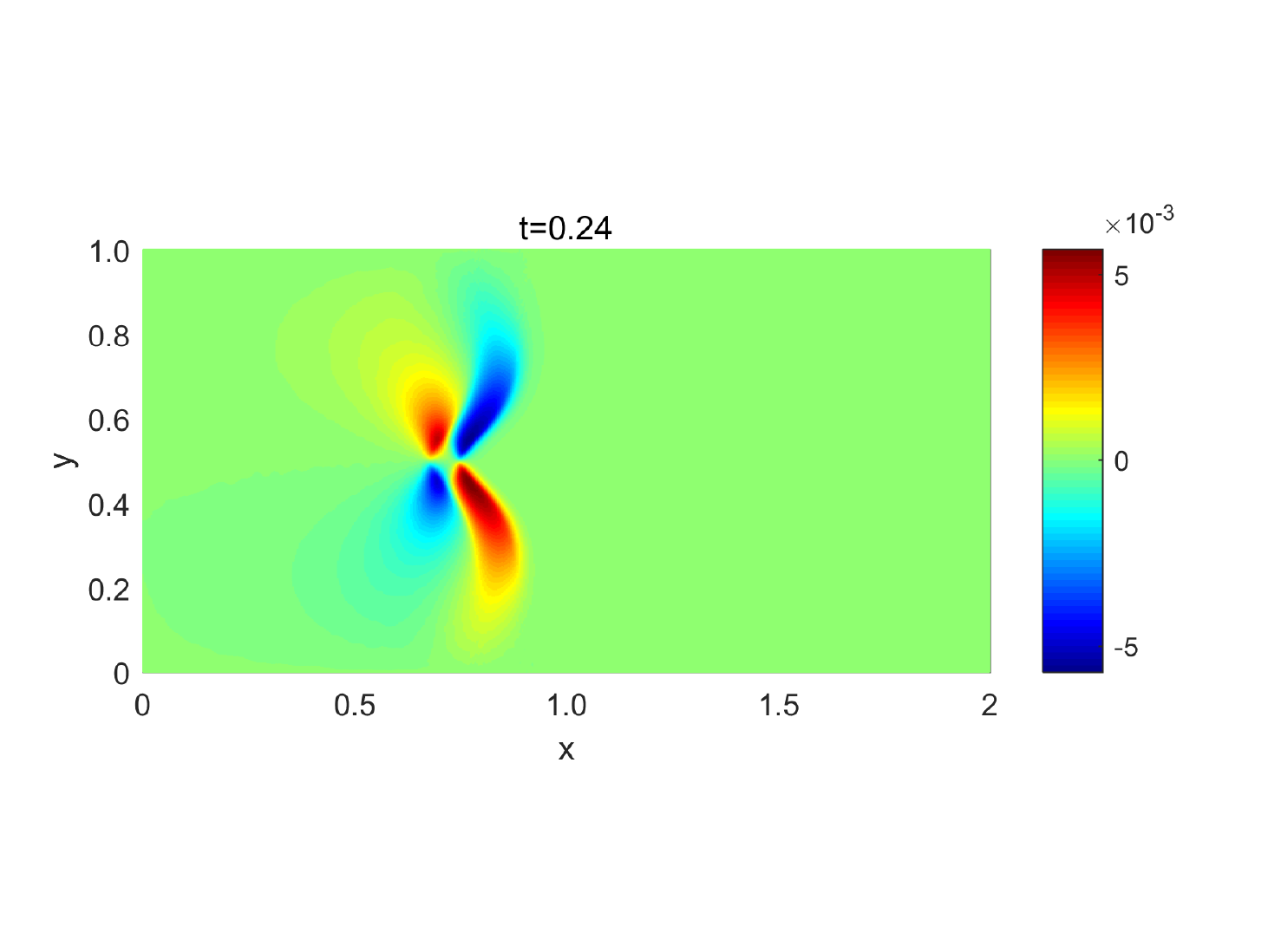}}
\caption{Continuation of Fig.~\ref{Fig:test4-2d-h-hu-hv-t12}: $t = 0.24$.}
\label{Fig:test4-2d-h-hu-hv-t24}
\end{figure}

\begin{figure}[H]
\centering
\subfigure[$h+B$: MM $N=150\times 50\times4$]{
\includegraphics[width=0.30\textwidth, trim=15 60 15 60, clip]{R_test4_2d_P2_Eh_Bph_M50_t36-eps-converted-to.pdf}}
\subfigure[$hu$: MM $N=150\times 50\times4$]{
\includegraphics[width=0.30\textwidth, trim=15 60 15 60, clip]{R_test4_2d_P2_Eh_hu_M50_t36-eps-converted-to.pdf}}
\subfigure[$hv$: MM $N=150\times 50\times4$]{
\includegraphics[width=0.30\textwidth, trim=15 60 15 60, clip]{R_test4_2d_P2_Eh_hv_M50_t36-eps-converted-to.pdf}}
\subfigure[$h+B$: FM $N=150\times 50\times4$]{
\includegraphics[width=0.30\textwidth, trim=15 60 15 60, clip]{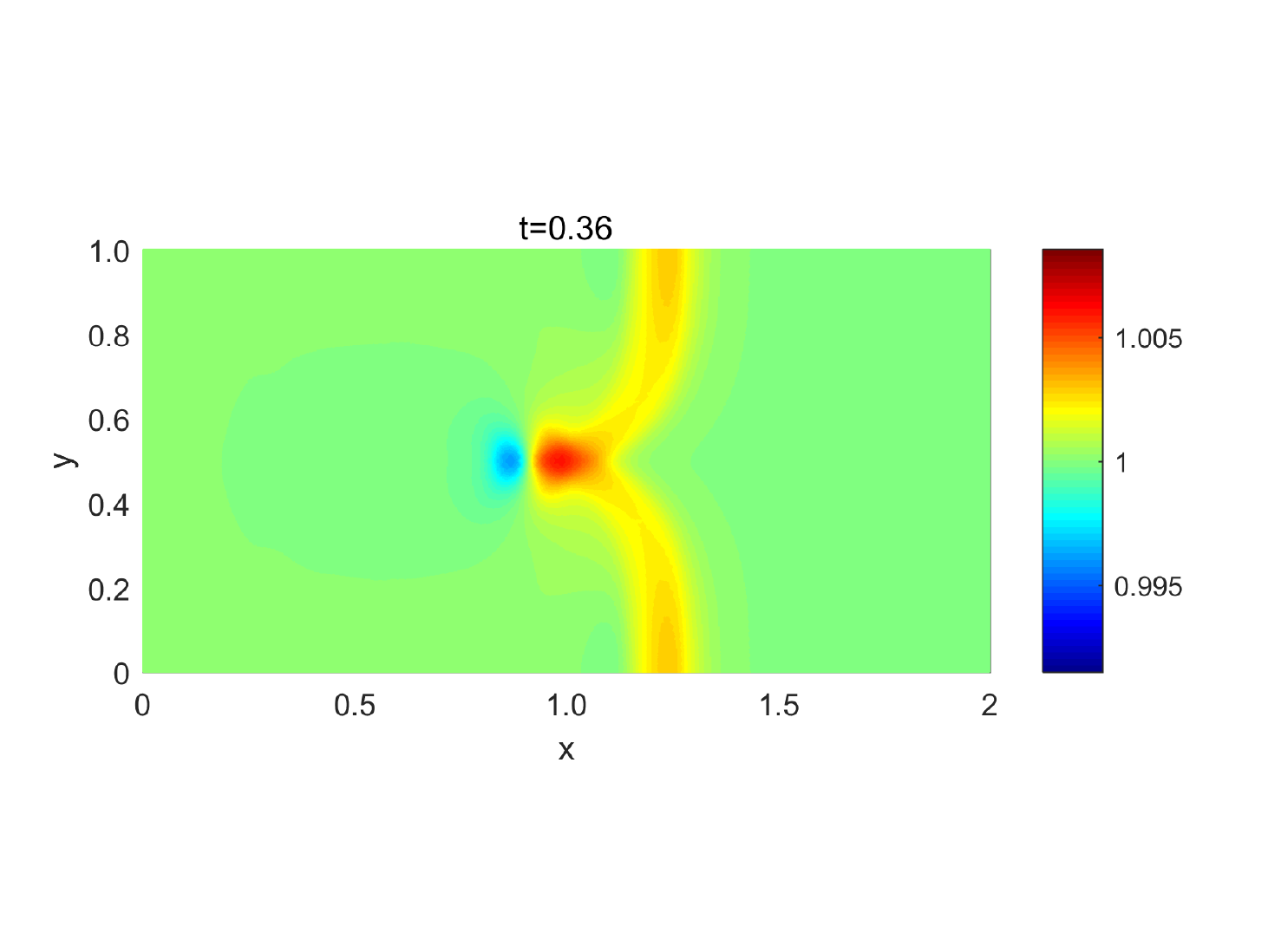}}
\subfigure[$hu$: FM $N=150\times 50\times4$]{
\includegraphics[width=0.30\textwidth, trim=15 60 15 60, clip]{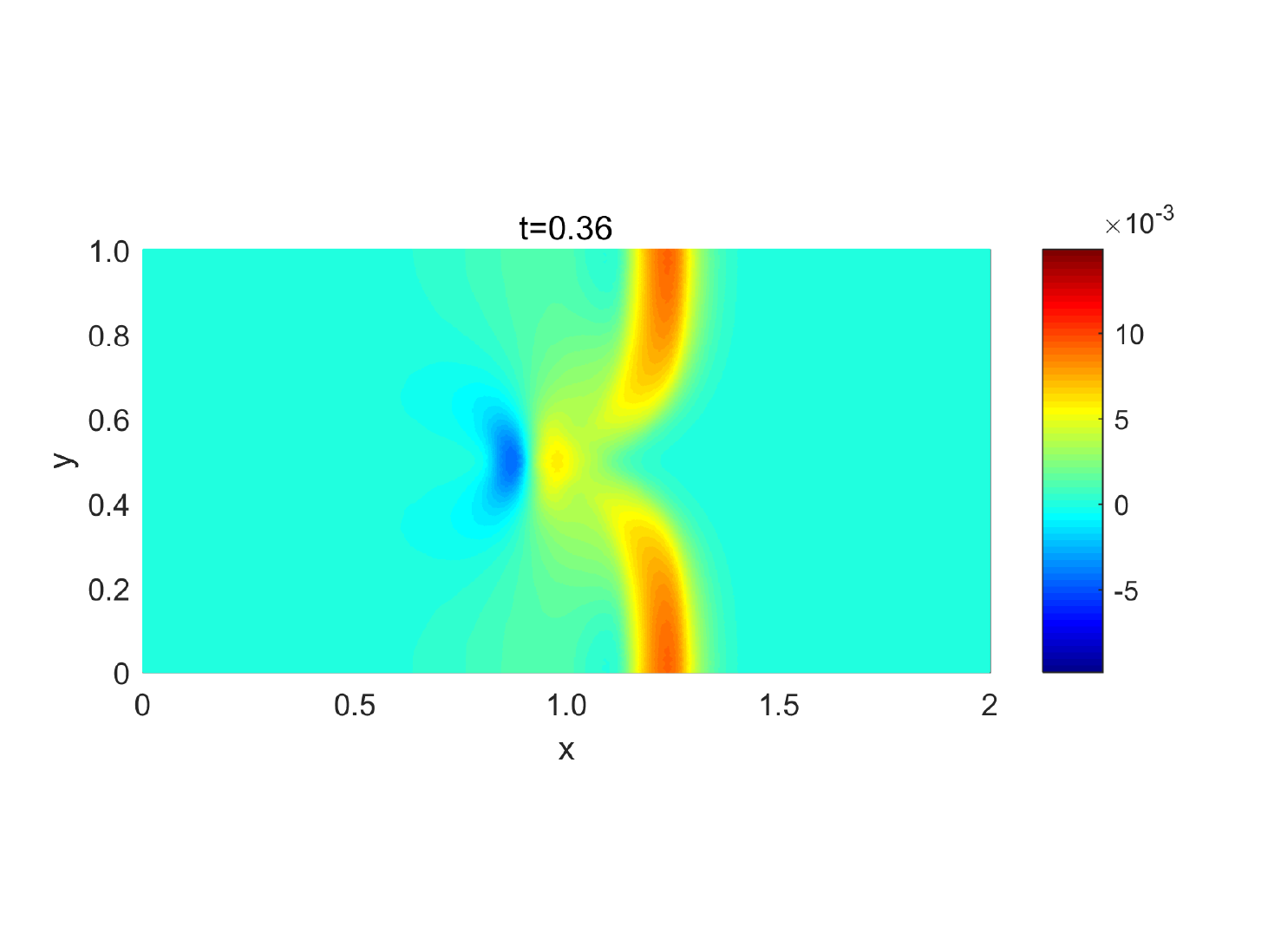}}
\subfigure[$hv$: FM $N=150\times 50\times4$]{
\includegraphics[width=0.30\textwidth, trim=15 60 15 60, clip]{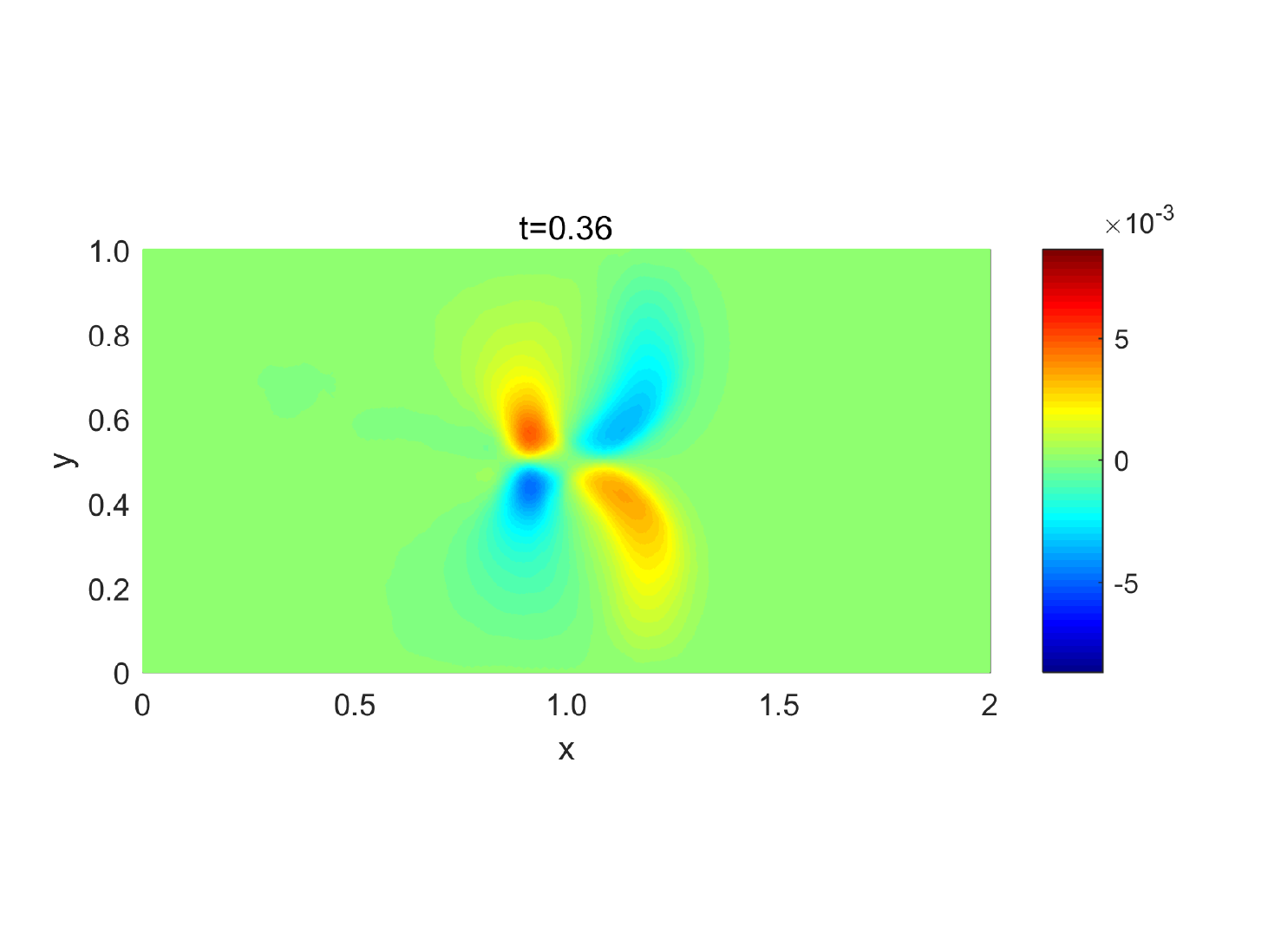}}
\subfigure[$h+B$: FM $600\times 200\times4$]{
\includegraphics[width=0.30\textwidth, trim=15 60 15 60, clip]{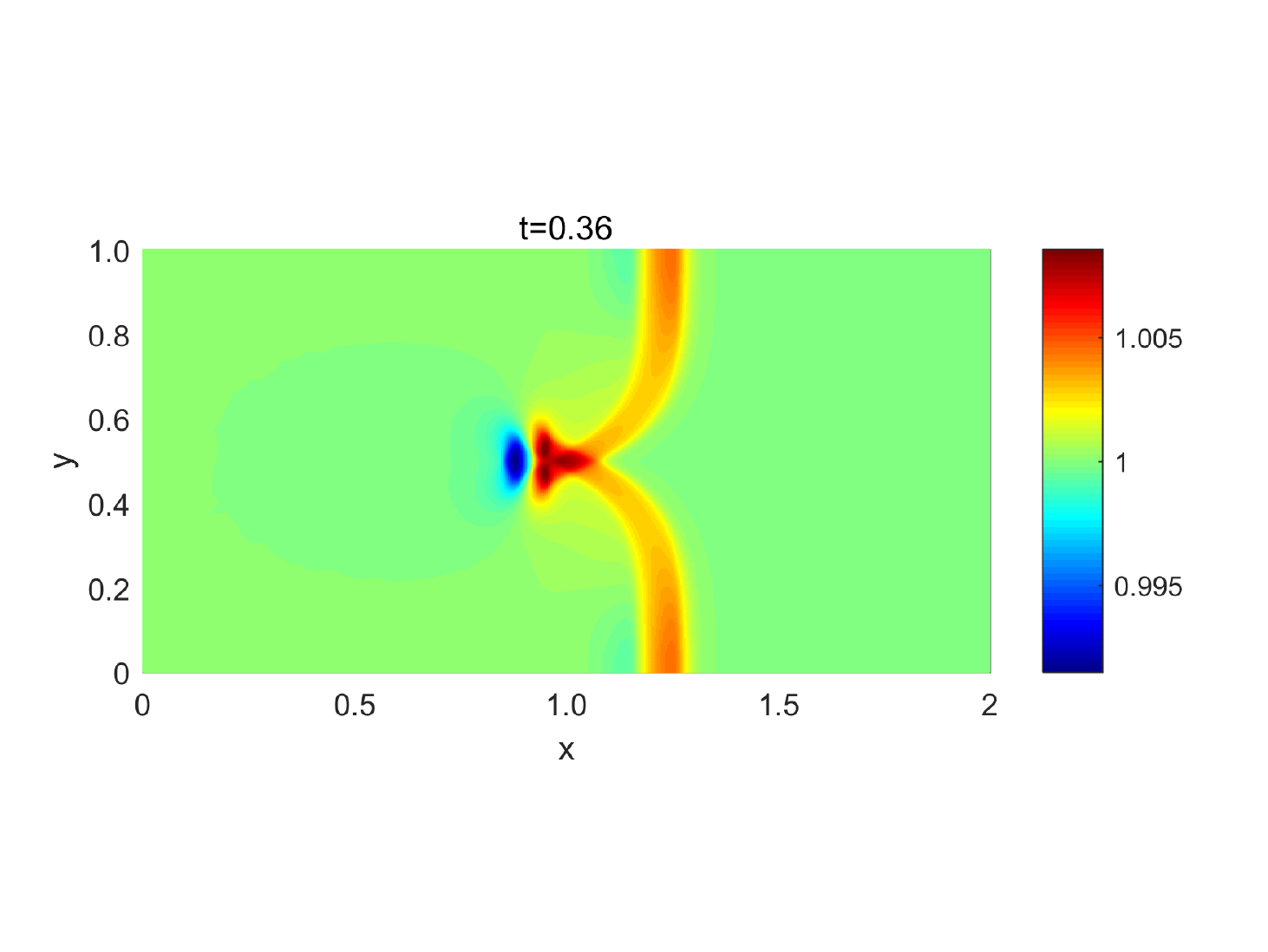}}
\subfigure[$hu$: FM $N=600\times 200\times4$]{
\includegraphics[width=0.30\textwidth, trim=15 60 15 60, clip]{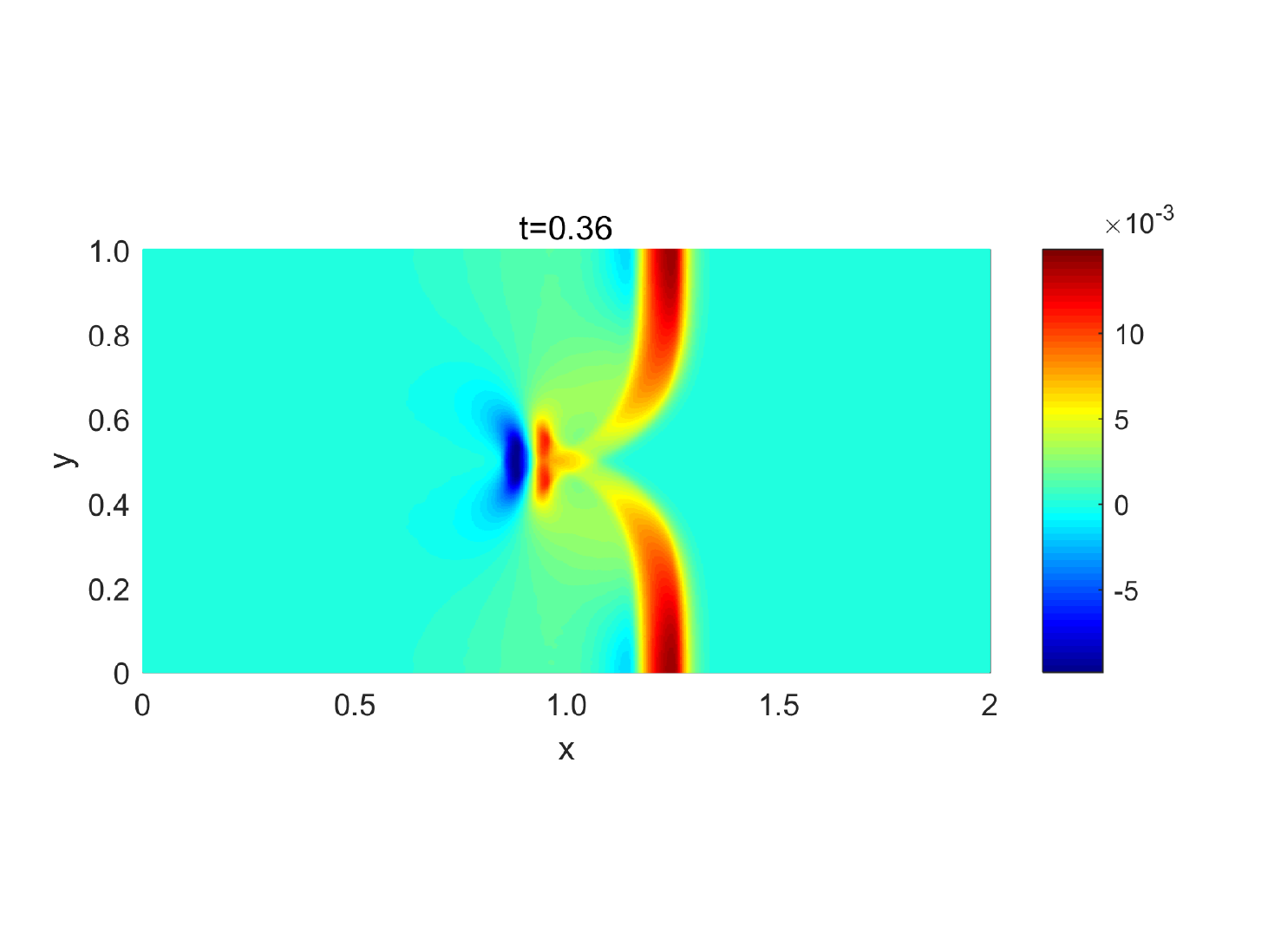}}
\subfigure[$hv$: FM $N=600\times 200\times4$]{
\includegraphics[width=0.30\textwidth, trim=15 60 15 60, clip]{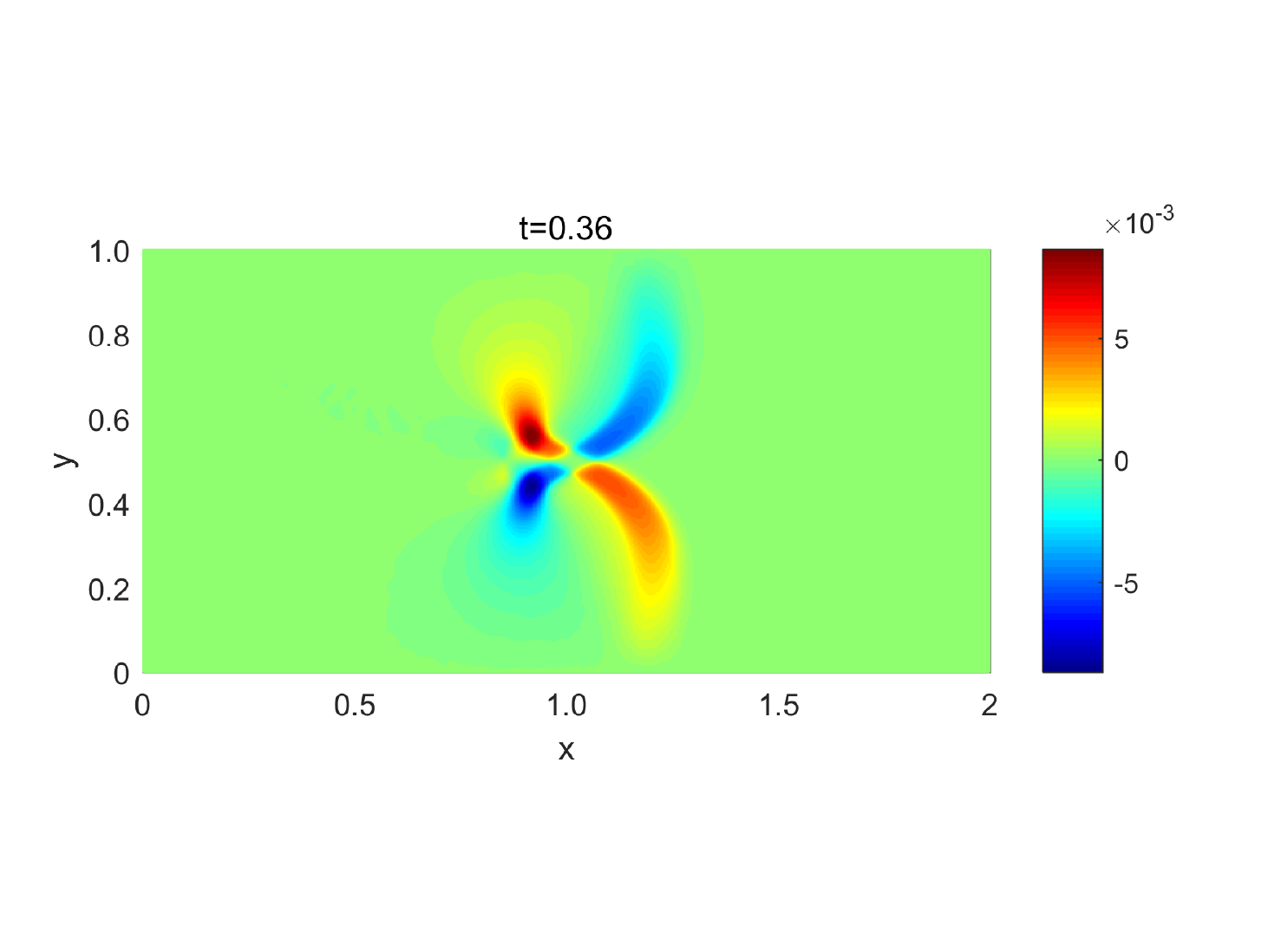}}
\caption{Continuation of Fig.~\ref{Fig:test4-2d-h-hu-hv-t12}: $t = 0.36$.}
\label{Fig:test4-2d-h-hu-hv-t36}
\end{figure}

\begin{figure}[H]
\centering
\subfigure[$h+B$: MM $N=150\times 50\times4$]{
\includegraphics[width=0.30\textwidth, trim=15 60 15 60, clip]{R_test4_2d_P2_Eh_Bph_M50_t48-eps-converted-to.pdf}}
\subfigure[$hu$: MM $N=150\times 50\times4$]{
\includegraphics[width=0.30\textwidth, trim=15 60 15 60, clip]{R_test4_2d_P2_Eh_hu_M50_t48-eps-converted-to.pdf}}
\subfigure[$hv$: MM $N=150\times 50\times4$]{
\includegraphics[width=0.30\textwidth, trim=15 60 15 60, clip]{R_test4_2d_P2_Eh_hv_M50_t48-eps-converted-to.pdf}}
\subfigure[$h+B$: FM $N=150\times 50\times4$]{
\includegraphics[width=0.30\textwidth, trim=15 60 15 60, clip]{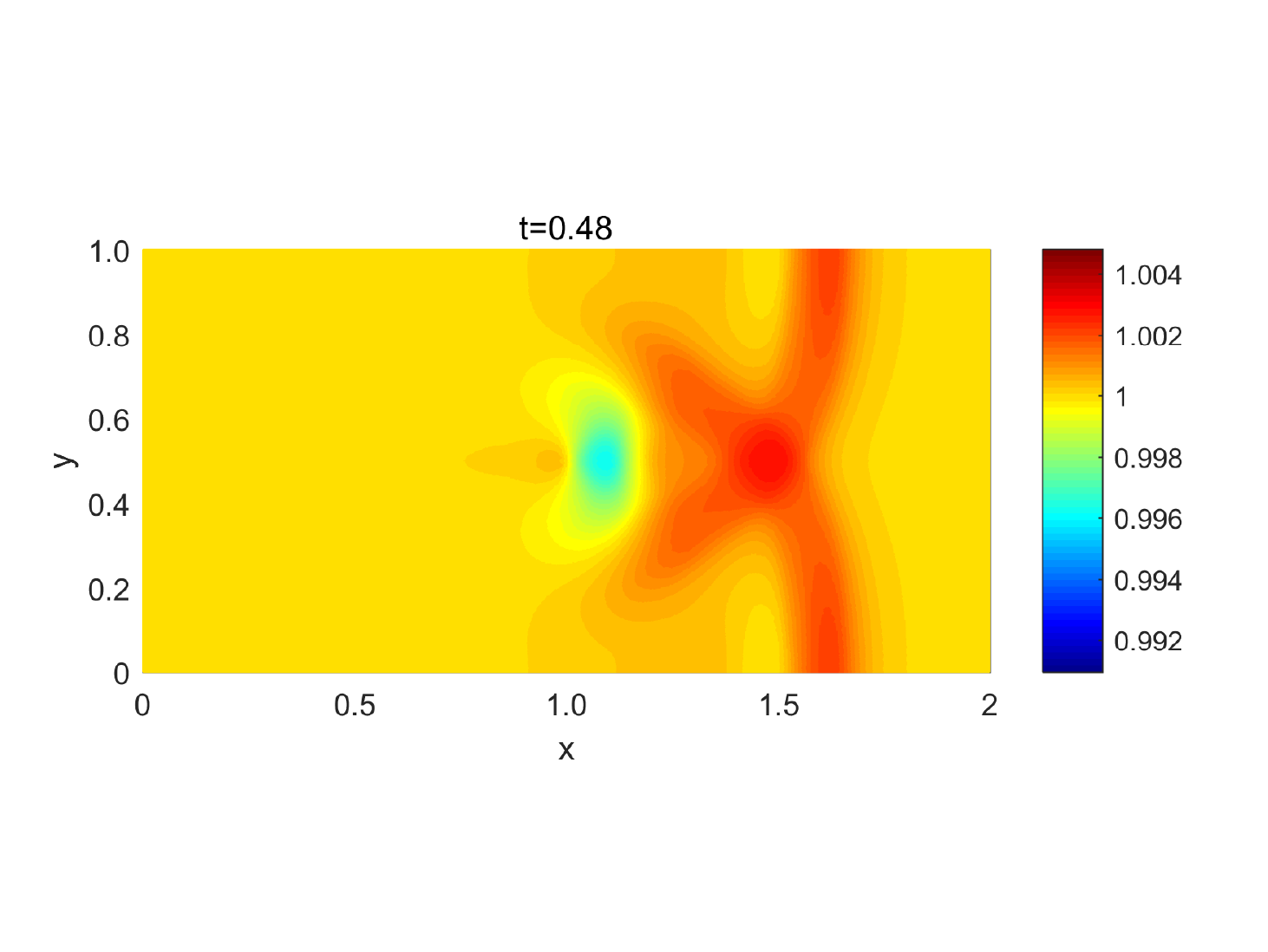}}
\subfigure[$hu$: FM $N=150\times 50\times4$]{
\includegraphics[width=0.30\textwidth, trim=15 60 15 60, clip]{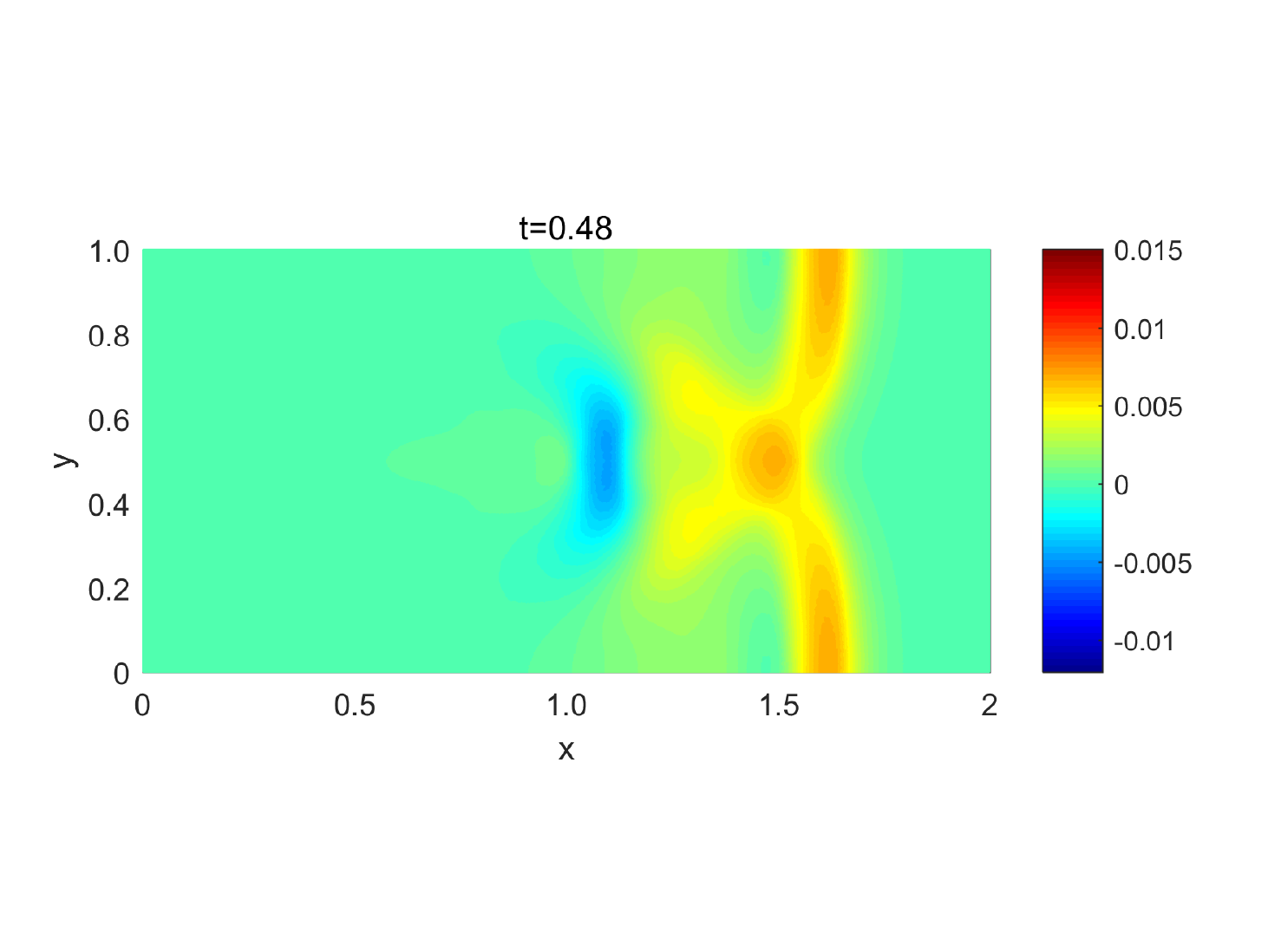}}
\subfigure[$hv$: FM $N=150\times 50\times4$]{
\includegraphics[width=0.30\textwidth, trim=15 60 15 60, clip]{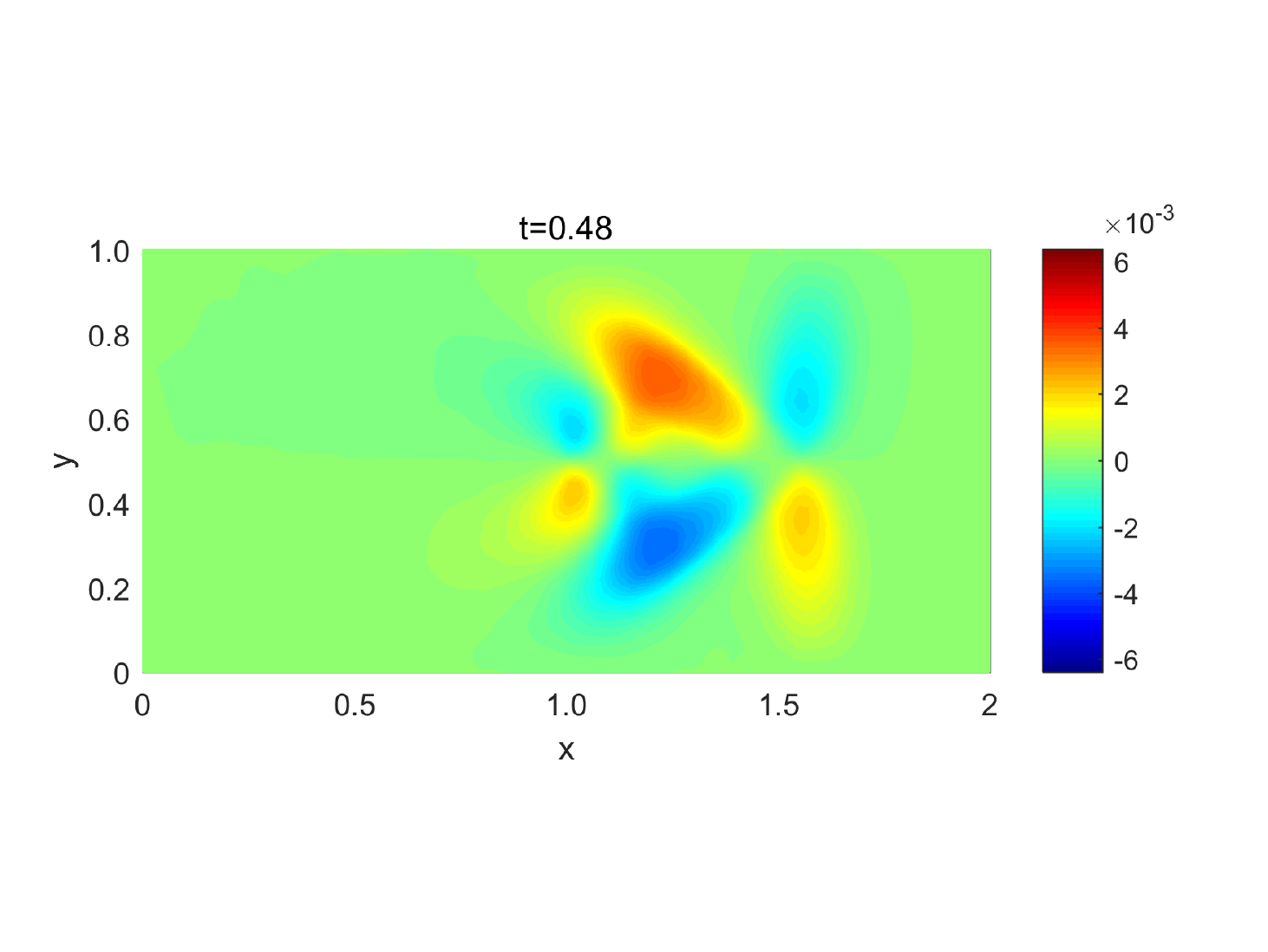}}
\subfigure[$h+B$: FM $600\times 200\times4$]{
\includegraphics[width=0.30\textwidth, trim=15 60 15 60, clip]{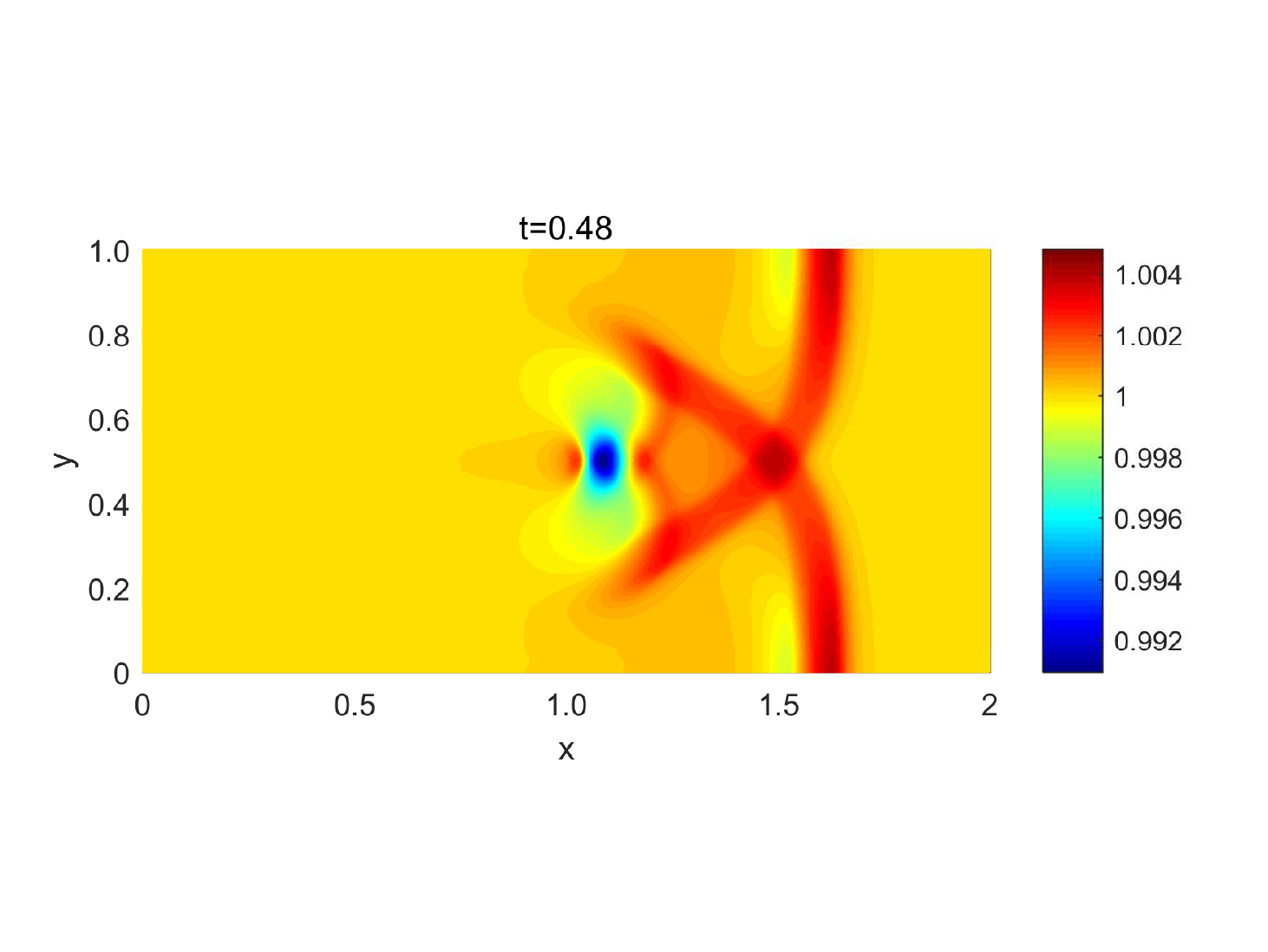}}
\subfigure[$hu$: FM $N=600\times 200\times4$]{
\includegraphics[width=0.30\textwidth, trim=15 60 15 60, clip]{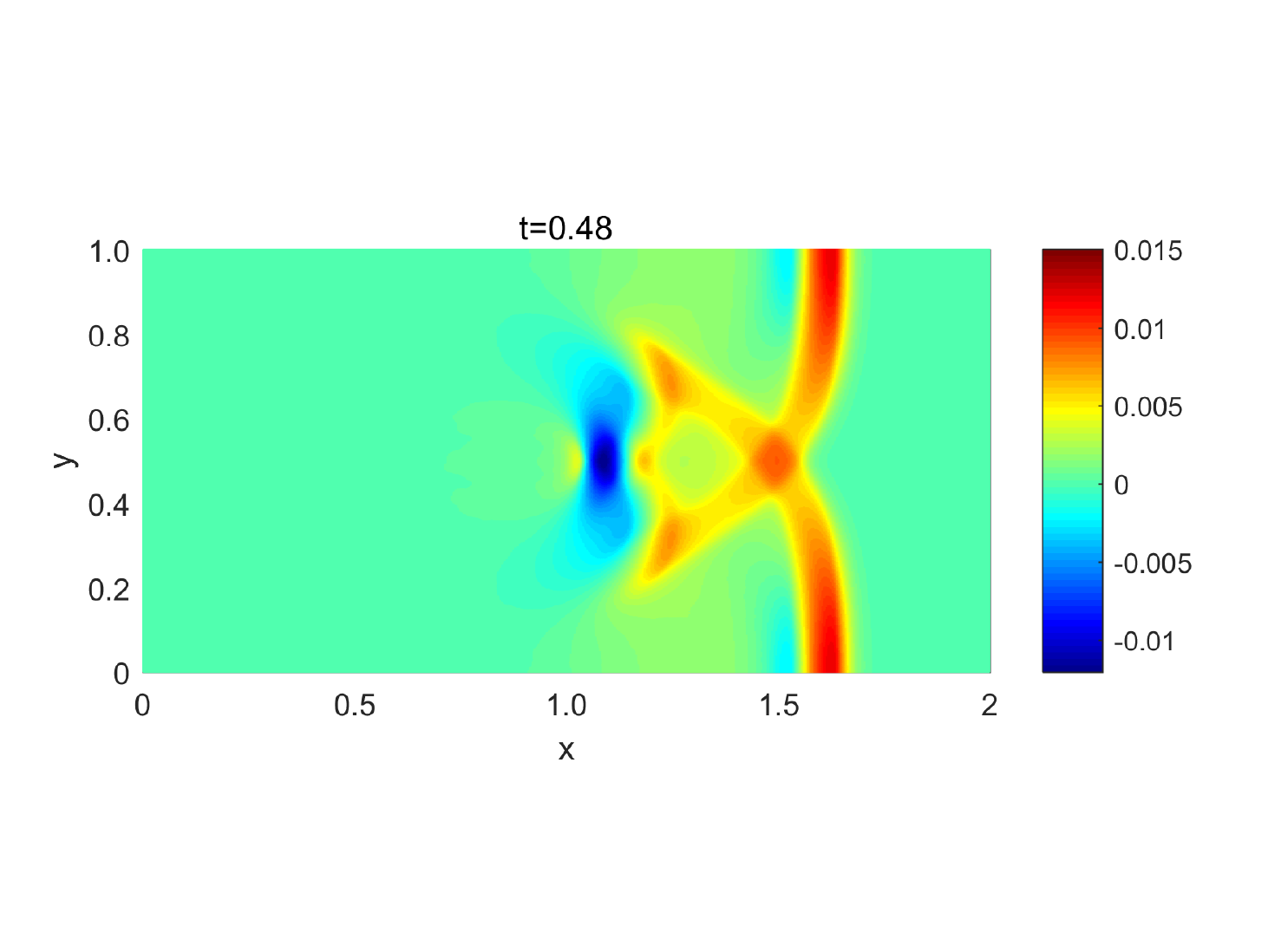}}
\subfigure[$hv$: FM $N=600\times 200\times4$]{
\includegraphics[width=0.30\textwidth, trim=15 60 15 60, clip]{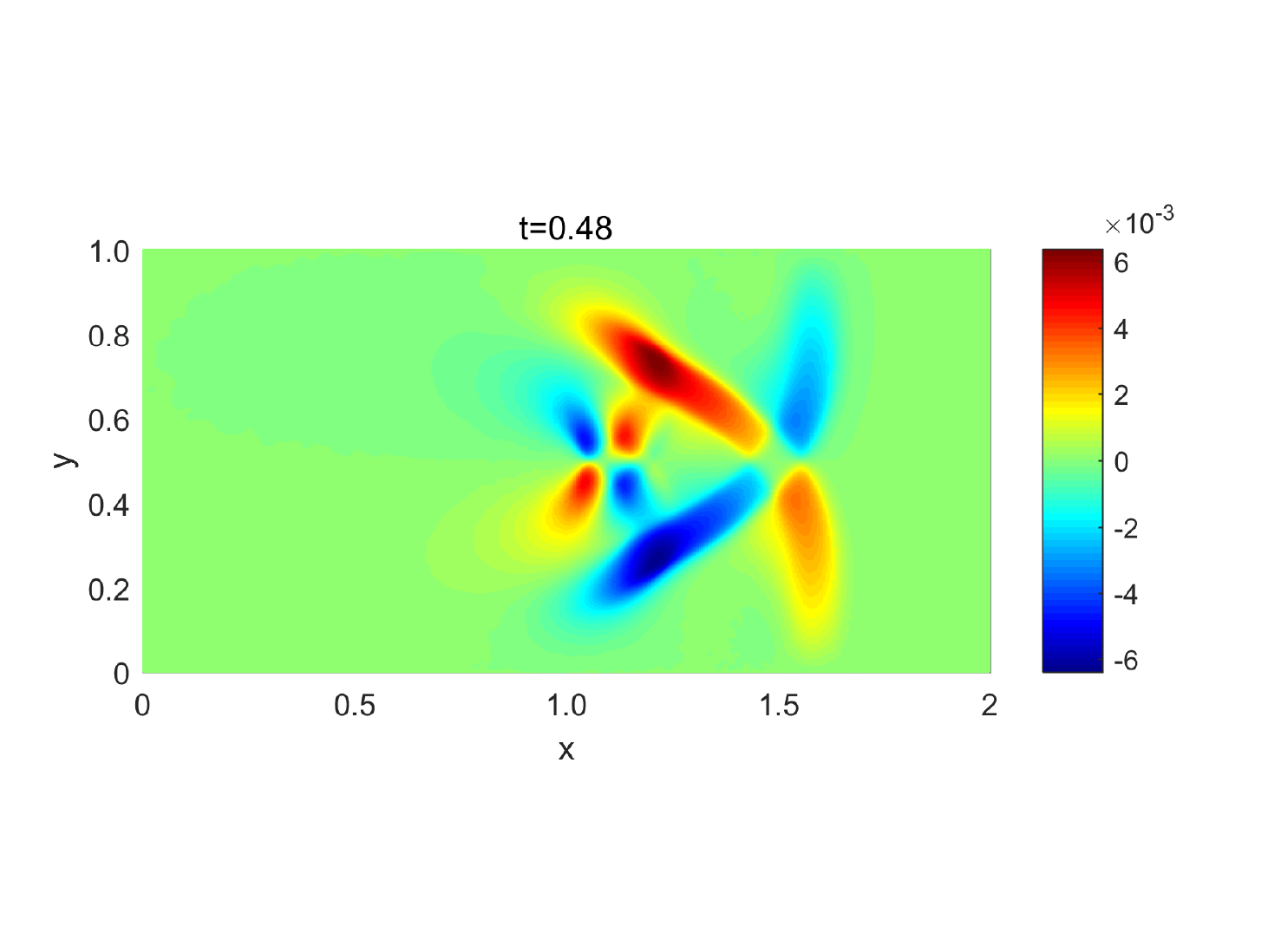}}
\caption{Continuation of Fig.~\ref{Fig:test4-2d-h-hu-hv-t12}: $t = 0.48$.}
\label{Fig:test4-2d-h-hu-hv-t48}
\end{figure}

\section{Conclusions and further comments}
\label{sec:conclusions}
We have presented a high-order well-balanced positivity-preserving adaptive moving mesh DG method
for the numerical solution of the SWEs with non-flat bottom topography.
The method is of rezoning type and contains three main components at each time step,
the generation of the new mesh by node redistribution/movement,
the interpolation of the physical variables from the old mesh to the new one,
and the numerical solution of the SWEs on the new mesh that is fixed
over the time step. A focus of the study is the well-balance and positivity-preserving properties of the method
that are crucial to its ability to simulate perturbation waves to the lake-at-rest steady state
such as waves on a lake or tsunami waves in deep ocean.

We have employed the MMPDE moving mesh scheme to generate the new mesh
at each time step. A key of the MMPDE scheme is to define the metric tensor
that provides the information needed for controlling the size, shape, and orientation
of mesh elements over the whole spatial domain. The numerical examples have shown
that the metric tensor based on the entropy/total energy
$E =\frac{1}{2}(hu^2+hv^2)+\frac{1}{2}gh^2+ghB$, a common choice in the adaptive
mesh simulation of shock waves, does not lead to fully correct mesh concentration
for waves of small magnitude.
Instead, we have proposed to use the equilibrium variable $\mathcal{E}
=\frac{1}{2}(u^2+v^2)+g(h+B)$ combined with the water depth $h$
in the computation of the metric tensor
and demonstrated numerically that they are well suited for the simulation of
the lake-at-rest steady state and its perturbations.

We have used a fixed mesh well-balanced DG scheme for solving the SWEs on the new mesh.
However, to ensure the well-balance property of the overall MM-DG method, we need to pay special
attention to the interpolation of the flow variables and bottom topography,
the slope limiting, and the positivity-preserving (PP) limiting.
We have proposed to use a DG-interpolation scheme (cf. \S\ref{sec:DG-interp})
for the purpose. The scheme has high-order accuracy, preserves constant solutions, and
is conservative while being economic to implement: it requires $\mathcal{O}(N_v)$ operations in each use,
where $N_v$ is the number of vertices of the mesh.
We note that although the bottom topography is a given time-independent function,
it needs to be updated at each time step from the old mesh to the new one
due to the movement of the mesh. Moreover, the same scheme should be used
for updating both the bottom topography and flow variables
in order to attain the well-balance property of the MM-DG method. We have used
the DG-interpolation scheme for both.

To ensure the nonnegativity of the water depth, we use the PP-DG-interpolation (DG-interpolation with PP limiter) to interpolate the water depth from the old mesh to the new one, and apply a scaling PP limiter to the water depth every time after we use the TVB limiter.
However, the PP limiter will destroy the well-balance property.
To restore the property, we have proposed
to make a high-order correction to the approximation of the bottom topography
according to the modifications in the water depth due to the PP limiting;
see \eqref{B-update-1} and \ref{B-update-2}.

A selection of numerical examples in one and two dimensions have been presented
to demonstrate the well-balance property, positivity-preserving, and high-order accuracy of the MM-DG method.
They have also shown that the method is well suited for the numerical simulation of the lake-at-rest
steady state and its perturbations. Particularly, the mesh concentration correctly reflects
structures in the flow variables and bottom topography and leads to more accurate numerical
solutions than a fixed mesh with the same number of elements.

\vspace{20pt}

\section*{Acknowledgments}
M. Zhang and J. Qiu were supported partly by Science Challenge Project (China), No. TZ 2016002 and
National Natural Science Foundation--Joint Fund (China) grant U1630247.
This work was carried out while M. Zhang was visiting the Department of Mathematics, the University of Kansas
under the support by the China Scholarship Council (CSC: 201806310065).



\end{document}